\documentclass[11pt, twoside]{article}
\usepackage{amsfonts}
\usepackage{bbm}
\usepackage{mathrsfs}
\usepackage{amssymb,amsmath,amsfonts,amsthm}
\usepackage{color}

\allowdisplaybreaks

\numberwithin{equation}{section}

\newtheorem{theorem}{Theorem}[section]
\newtheorem{prop}[theorem]{Proposition}
\newtheorem{lemma}[theorem]{Lemma}
\newtheorem{definition}[theorem]{Definition}
\newtheorem{corollary}[theorem]{Corollary}

\newtheorem{remark}[theorem]{Remark}

\newcommand{\R}{\mathbb{R}}

\begin{document}

\title{Maximal function, Littlewood--Paley theory, Riesz transform and atomic decomposition in the multi-parameter flag setting}

\author{Yongsheng Han, Ming-Yi Lee, Ji Li and Brett D. Wick}

\date{}
\maketitle

\begin{abstract}
In this paper, we develop via real variable methods various characterisations of the Hardy spaces in the multi-parameter flag setting.  These characterisations include those via, the non-tangential and radial maximal function, the Littlewood--Paley square function and area integral, Riesz transforms and the atomic decomposition in the multi-parameter flag setting.  The novel ingredients in this paper include (1) establishing appropriate discrete Calder\'on reproducing formulae in the flag setting and a version of the Plancherel--P\'olya inequalities for flag quadratic forms; (2) introducing the maximal function and area function via flag Poisson kernels and flag version of harmonic functions; (3) developing an atomic decomposition via the finite speed propagation and area function in terms of flag heat semigroups. As a consequence of these real variable methods, we obtain the full characterisations of the multi-parameter Hardy space with the flag structure.
\end{abstract}

\bigskip
\setcounter{tocdepth}{2}
\tableofcontents
\bigskip

\bigskip


\medskip


\newpage

\hskip5.7cm{\bf Notation}
{\small
\begin{itemize}
\item $\|\cdot\|_2$:\ the $L^2$ norm on $\mathbb R^{n+m}$, i.e., $\|\cdot\|_{L^2(\mathbb R^{n+m})}$;
\item $\|\cdot\|_1$:\ the $L^1$ norm on $\mathbb R^{n+m}$, i.e., $\|\cdot\|_{L^1(\mathbb R^{n+m})}$;
\item $\mathcal S(\mathbb R^n)$:\ Schwartz test function space on $\mathbb R^{n}$;
\item $M_s$: the strong maximal function on $\Bbb R^n\times \Bbb R^m$, see definition in \eqref{strong maximal};
\item $M_F$: the flag maximal function on $\Bbb R^n\times \Bbb R^m$, see Definition \ref{def flag HL};
\item $g_F(f)$: flag Littlewood--Paley square function via flag Schwartz function $\psi$, see Definition \ref{def-of-gf};
\item $S_F(f)$: flag Littlewood--Paley area function via flag Schwartz function $\psi$, see Definition \ref{def-of-S-function-Sf};
\item $M^{*}_{\phi}(f)$ : flag non-tangential maximal function via flag Schwartz function $\phi$, see Definition \ref{non-tangential-maximal of f};
\item $M^{+}_{\phi}(f)$ : flag radial maximal function via flag Schwartz function $\phi$, see Definition \ref{def-of-radial-maximal-function};
\item $H_F^1(\mathbb{R}^n\times\mathbb{R}^m)$ : The flag Hardy spaces, see Definition \ref{def-of-hardy-by-han};
\item $\triangle^{(1)}$ :  the Laplacian on $\Bbb R^{n+m}$, see Definition \ref{def of product atom};
\item $\triangle^{(2)}$ :  the Laplacian on $\Bbb R^{m}$, see Definition \ref{def of product atom};
\item $\ell(Q)$ : the sidelength of the cube $Q$;
\item $H^1_{F,at,M}(\mathbb{R}^{n}\times\mathbb{R}^{m})$: the atomic Hardy space, see Definition \ref{def-of-atomic-product-Hardy-space};
\item $R^{(1)}_j$: the $j$-th Riesz transform on $\mathbb R^{n+m}$, $j=1,2,\ldots,n+m$;
\item $R^{(2)}_k$: the $k$-th Riesz transform on $\R^m$, $k=1,2,\ldots,m$;
\item $R_{j,k}=R^{(1)}_j\ast_{\mathbb R^m} R^{(2)}_k$ : the flag Riesz transforms;
\item $ P(x,y)=P^{(1)}\ast_{\mathbb R^m}P^{(2)}(x,y)$ : the flag Poisson kernel, $P^{(1)}(x,y)$ and $P^{(2)}(z)$ are the classical Poisson kernels on $\mathbb{R}^{n+m}$ and $\mathbb{R}^m$, respectively;
\item $S_F(U)$: the flag Lusin area integral via flag Poisson kernel, that is $U(x,y,t,s)=P_{t,s}\ast f(x,y)$, see Definition \ref{def-of-S-function-Su};
\item $M_1$ : the Hardy--Littlewood maximal function on $\mathbb R^{n+m}$;
\item $M_2$ : the Hardy--Littlewood maximal function on $\mathbb R^m$;
\item $S_{F,\triangle^{(1)},\triangle^{(2)}}(f)$: the area function associated with $\triangle^{(1)}$ and $\triangle^{(2)}$, see Definition \ref{def-of-S-function-Sf L};
\item $\Omega$: an open set in $\mathbb R^n\times\mathbb R^m$ with finite measure;
\item $m(\Omega)$: the set of all maximal dyadic subrectangles contained in $\Omega$.
\end{itemize}
}

\newpage

\section{Introduction and statement of main results, applications}
\setcounter{equation}{0}

\subsection{Background and main results}

It was well-known that techniques from Fourier series and methods of complex analysis played a seminal role in the classical harmonic analysis. After many improvements, mostly achieved by the
Calder\'on--Zygmund school, the real variable methods, such as, maximal function, Littlewood--Paley square function, Lusin area integral, singular integrals and atomic decomposition have come to more prominence.

For the classical one parameter case, the Hardy--Littlewood maximal function and Calder\'{o}n--Zygmund singular integrals commute with
the usual dilations on $\mathbb{R}^{n}$, $\delta \cdot x=(\delta
x_{1},\ldots,\delta x_{n})$ for $\delta >0$. This theory has been extensively studied and is by now well understood, see for example the monograph \cite{St}.  On the other hand, the product theory began with the \emph{strong} maximal function and continued with the Marcienkiewicz multiplier. They commute with the
multi-parameter dilations on $\mathbb{R}^{n}$, $\delta \cdot x=(\delta
_{1}x_{1},\ldots,\delta _{n}x_{n})$ for $\delta =(\delta _{1},\ldots,\delta
_{n})\in \mathbb{R}_{+}^{n}$.
Product theory has been studied, for
example, in Gundy and Stein \cite{GS}, R. Fefferman and Stein \cite{FS}, R.
Fefferman \cite{F1, F2, F3}, Chang and R. Fefferman \cite{CF1, CF2, CF3}, Journ\'{e} \cite{J1},
and Pipher \cite{P}. More precisely, R. Fefferman and Stein \cite{FS} studied the $
L^{p}$ boundedness ($1<p<\infty $) for the product convolution singular
integral operators. Journ\'{e} in \cite{J1} introduced non-convolution
product singular integral operators, established the product $T1$ theorem
and proved the $L^{\infty }\rightarrow {\rm BMO}$ boundedness for such operators.
The product Hardy space $H^{p}\left( \mathbb{R}^{n}\times \mathbb{R}
^{m}\right) $ was first introduced by Gundy and Stein \cite{GS}. Later, Chang and
R. Fefferman \cite{CF1, CF2, CF3} developed the atomic
decomposition and established the dual space of the Hardy space $H^{1}\left(
\mathbb{R}^{n}\times \mathbb{R}^{m}\right) $, namely the product BMO space, denoted by ${\rm BMO}\left(
\mathbb{R}^{n}\times \mathbb{R}^{m}\right) $.

Note that the product theory has an \emph{explicit} underlying multi-parameter product structure. However, when the underlying multi-parameter structure is not explicit, but only \emph{implicit}, an appropriate $L^p$ theory, with $1<p<\infty$, has only recently been developed.  To be precise, in \cite{MRS, MRS2}, Muller, Ricci and Stein studied the Marcinkiewicz multipliers on the Heisenberg group $\mathbb H^n$  associated with the sub-Laplacian on $\mathbb H^n$ and the central invariant vector field, and obtained the $L^p$ boundedness for $1<p<\infty$. This is surprising since
these multipliers are invariant under a two parameter
group of dilations on $\mathbb{C}^{n}\times \mathbb{R}$, while there is
\emph{no} two parameter group of \emph{automorphic} dilations on $\mathbb{H}
^{n}$. Moreover, they showed that Marcinkiewicz multipliers can be characterized by
a convolution operator of the form $f\ast K$ where, $K$ is a {\it flag} convolution kernel, which satisfies the size, smoothness conditions lying in between the one-parameter and product singular integrals.
The crucial idea is that Muller, Ricci and Stein introduced and studied
the natural implicit structure on the Heisenberg group
$\mathbb H^n$, named {\it flag setting}, given via the following projection $\pi$ from $\mathbb H^n\times \mathbb R$ onto $\mathbb H^n$:
$f = \pi F$ with $F\big( (z,t),s\big) {\rm\ \  on\ \ } \mathbb H^{n}\times \mathbb R$  and $f$ on $\mathbb H^n$ as follows:
\begin{align*}
F\big( (z,t),s\big) &{\rm\ \  on\ \ } \mathbb H^{n}\times \mathbb R {\rm\ \  with\ \ } (z,t)\in \mathbb H^n=\mathbb C^n\times\mathbb R\ \ {\rm and\ \ } s\in \mathbb R\\
\Big\downarrow \pi\\
f(z,t) &{\rm\ \  on\ \ } \mathbb H^n,
\end{align*}
where the projection
$\pi$, ($f:=\pi F$), is defined as
$$ f(z,t) =\pi F(z,t)= \int_{\mathbb R} F\big( (z,t-s),s\big)\, ds.$$ 

Later, Nagel, Ricci and Stein \cite{NRS} studied the flag singular integrals on Euclidean space and applications
on certain quadratic CR submanifolds
of ${\Bbb{C}}^n.$ Nagel, Ricci, Stein and Wainger \cite{NRSW1, NRSW2} further generalised the
theory of singular integrals with flag kernels to a more general setting, namely, that of homogeneous groups.
They proved that on a homogeneous group singular integral operators with flag kernels are bounded on  $L^p,1<p<\infty,$ and form an algebra.
See also \cite{G1, G2, G3, DLOPW} for related work.

At the endpoint, it is natural to expect that certain Hardy space and $\mathrm{BMO}$ bounds
are available. However, the lack of automorphic dilations underlies the failure of such
multipliers to be in general bounded on the classical Hardy space and also precludes
a pure product Hardy space theory on the Heisenberg group.
This was the original motivation in \cite{HLS} to develop a theory of \emph{flag} Hardy
spaces $H_{flag}^{p}, 0<p\leq 1$ on the Heisenberg group $\mathbb H^n,$ that is, in a
sense `intermediate' between the classical Hardy spaces $H^{p}(\mathbb{H}
^{n})$ and the product Hardy spaces $H_{product}^{p}(\mathbb{C}^{n}\times \mathbb{R})$. The flag $H^p$ theory on the Heisenberg group developed in \cite{HLS} includes the discrete version of the Calder\'on reproducing formula
associated with the given multi-parameter structure and the Plancherel--P\'olya
type inequality in this setting.
They established the flag Hardy spaces $H^p_{flag}(\mathbb H^n)$ via the discrete Littlewood--Paley
square function, and then studied the dual space $CMO^p_{flag}(\mathbb H^n)$ using the corresponding
Carleson measures.  Calder\'on--Zygmund decomposition in terms of functions in $H^p_{flag}(\mathbb H^n)$ and
interpolation has also been developed. In \cite{HHLT} they showed that singular integrals with flag kernels, which include the aforementioned Marcinkiewicz multipliers, are bounded on $H_{flag}^{p}(\mathbb H^n)$, as
well as from $H_{flag}^{p}(\mathbb H^n)$ to $L^{p}(\mathbb H^n)$, for $0<p\leq 1$. Moreover, in \cite{HLS} they constructed a singular
integral with a flag kernel on the Heisenberg group, which is not bounded on the classical Hardy space $H^1(\mathbb{H}^n).$ Since, as pointed out in \cite{HLS}, the flag Hardy space $H_{flag}^p(\mathbb{H}^n)$ is contained in the classical Hardy space $H^{p}(\mathbb{H}^{n}),$
this counterexample implies that $H_{flag}^{1}(\mathbb{H}^{n})\subsetneqq H^{1}(\mathbb{H}^{n}).$

It was well-known that both of the classical and product multi-parameter Hardy spaces can be characterized by the real variable methods, such as, Riesz
transforms, maximal functions, the Littlewood--Paley square function and Lusin area integrals, as well as atomic decompositions, see \cite{FS}.
Thus, a natural question arises:

{\it {\bf \it Q}: Can one develop all these real variable methods in the multi-parameter flag structure setting}?

The main purpose of this paper is to address this question, focusing on the case of  $\mathbb R^n\times \mathbb R^m$  associated with flag structure induced by a projection, which was a simplified model of M\"uller, Ricci and Stein, and was studied by Nagel, Ricci and Stein \cite{NRS}, as well as Nagel, Ricci, Stein and Wainger \cite{NRSW1,NRSW2}
\begin{align*}
F\big( (x,y),z\big) &{\rm\ \  on\ \ } \mathbb R^{n+m}\times \mathbb R^m\\
\Big\downarrow \pi\\
f(x,y) &{\rm\ \  on\ \ } \mathbb R^n\times \mathbb R^m,
\end{align*}
where the projection $\pi$, ($f:=\pi F$), is defined as
$$ f(x,y) =\pi F(x,y):= \int_{\mathbb R^m} F\big( (x,y-z),z\big)\, dz. $$

To be precise, the main results of this paper develop the real variable methods, maximal functions, the Littlewood--Paley square function and the Lusin area integrals, Riesz transforms, as well as  atomic decompositions, in the more complicated multi-parameter flag structure setting. As a consequence, using these real variable methods, we obtain the full characterisations of flag Hardy space $H^1_F(\mathbb R^n\times\mathbb R^m)$.

\subsection{Statement of main results}

To state the main results of this paper, one requires several definitions.  To begin with, we first introduce the Littlewood--Paley square function and Lusin area integrals associated with the flag structure on $\R^n\times \R^m$. For this purpose,
let $\psi^{(1)}\in\mathcal{ S}(\mathbb{ R}^{n+m})$ with  supp $\widehat{\psi^{(1)}}\subset\Big\{\xi:\frac{1}{2}\le |\xi|\le 2\Big\}$ and
\begin{eqnarray*}
\label{calderon condition 1}
\int^\infty_0 |\widehat{\psi^{(1)}}(t\xi)|^2\frac{dt}{t}=1\text{ for all }\ \xi\in
\Bbb R^n\times\Bbb R^m\setminus\{(0,0)\}.
\end{eqnarray*}
Let $\psi^{(2)}\in\mathcal{ S}(\mathbb{ R}^{m})$ with ${\rm supp}\ \widehat{\psi^{(2)}}\subset\Big\{\eta:\frac{1}{2}\le |\eta|\le 2\Big\}$ and
\begin{eqnarray*}
\label{calderon condition 2}
\int^\infty_0|\widehat{\psi^{(2)}}(s\eta)|^2\frac{ds}{s}=1\text{ for all }\ \eta\in\Bbb R^m\setminus\{0\}.
\end{eqnarray*}
Applying the projection of M\"uller, Ricci and Stein, we set
\begin{eqnarray}
\label{psi jk}
\psi_{t,s}(x,y)=\psi^{(1)}_t\ast_{_{\mathbb R^m}}\psi^{(2)}_s(x,y):=\int_{\mathbb R^m}\psi^{(1)}_t(x,y-z)\psi^{(2)}_s(z)dz,
\end{eqnarray}
where $\psi^{(1)}_t(x,y)=t^{-(n+m)}\psi^{(1)}(\frac{x}{t},\frac{y}{t})$ and $\psi^{(2)}_s(z)=s^{-m}\psi^{(2)}(\frac{z}{s})$.

\begin{definition}\label{def-of-gf}
For $f\in L^1(\Bbb R^{n+m})$, the Littlewood--Paley square function $g_F(f)$ is
defined by
\begin{eqnarray*}\label{Flag-g-func}
g_F(f)(x,y)=\bigg\{\int^\infty_0\int^\infty_0 \Big|\psi_{t,s}\ast f(x,y)\Big|^2\frac{dt}{t}\frac{ds}{s}\bigg\}^{1/2},
\end{eqnarray*}
where $\psi_{t,s}(x,y)$ is the same as in \eqref{psi jk}.
\end{definition}

We now introduce the Lusin area integral associated with the flag structure.
\begin{definition}\label{def-of-S-function-Sf}
For $f\in L^1(\mathbb{R}^{n+m})$, the Lusin area integral of $f$ is defined by
\begin{eqnarray*}
S_F(f)(x,y)=\bigg\{\int_{\mathbb{R}_{+}^{n+1}}\int_{\mathbb{R}_{+}^{m+1}}\chi_{t,s}(x-x_1,y-y_1)|\psi_{t,s}\ast
f(x_1,y_1)|^2 {dx_1dt \over t^{n+m+1}}{dy_1ds \over s^{m+1}} \bigg\}^{1/2},\ \
\end{eqnarray*}
where $\chi_{t,s}(x,y)=\chi_t^{(1)}\ast_{\mathbb R^m} \chi_s^{(2)}(x,y)$,
$\chi_t^{(1)}(x,y)=\chi^{(1)}({x\over t},{y\over t})$,
$\chi_s^{(2)}(z)=\chi^{(2)}({z\over s})$, $\chi^{(1)}(x,y)$
and $\chi^{(2)}(z)$ are the indicator functions of the unit balls of
$\mathbb{R}^{n+m}$ and $\mathbb{R}^m$, respectively.
\end{definition}
Note that the projection of M\"uller, Ricci and Stein is involved in
$\chi_{t,s}(x,y)$.

To define maximal functions associated with the flag structure, applying the projection of M\"uller, Ricci and Stein,
we first introduce the following collection of functions that will be used to build the maximal functions.
\begin{definition}\label{def-of-DF}
Let $\phi(x,y)=\phi^{(1)}\ast_{\mathbb R^m}\phi^{(2)}(x,y)$, where
$\phi^{(1)}\in\mathcal {S}(\mathbb{R}^{n+m})$ and
$\phi^{(2)}\in\mathcal{S}(\mathbb{R}^m)$ satisfying
\begin{eqnarray*}
\int_{\mathbb{R}^{n+m}}\phi^{(1)}(x,y)dxdy=\int_{\mathbb{R}^m}\phi^{(2)}(z)dz=1.
\end{eqnarray*}
We denote $\mathscr{D}_F(\mathbb{R}^n\times\mathbb{R}^m)$ by the collection of all functions $\phi$ that satisfies the above conditions.
\end{definition}
The non-tangential maximal function is defined by
\begin{definition}\label{def-of-maximal-function}
Let $\phi\in \mathscr{D}_F(\mathbb{R}^n\times\mathbb{R}^m)$.
For each $f\in L^1(\mathbb{R}^{n+m}) $,
the non-tangential maximal function of $f$ is defined by
\begin{eqnarray*}\label{non-tangential-maximal of f}
M^{*}_{\phi}(f)(x,y)=\sup_{(x_1,y_1,t,s)\in\Gamma(x,y)}|\phi_{t,s}\ast
f(x_1,y_1)|,
\end{eqnarray*}
where $\phi_{t,s}(x,y)=\phi_t^{(1)}\ast_{\Bbb R^m} \phi_s^{(2)}(x,y)$,
$\phi_t^{(1)}(x,y)=t^{-(m+n)}\phi^{(1)}({x\over t}, {y\over t} ),$
$\phi_s^{(2)}(z)=s^{-m}\phi^{(2)}({z\over s})$ and
$\Gamma(x,y)=\{(x_1,y_1,t,s):\ |x-x_1|\leq t,\ |y-y_1|\leq t+s \}$.
\end{definition}
Similarly, we define the radial maximal function as follows.
\begin{definition}\label{def-of-radial-maximal-function}
Let $\phi\in \mathscr{D}_F(\mathbb{R}^n\times\mathbb{R}^m)$.
For any $f\in L^1(\mathbb{R}^{n+m}) $,
the radial maximal function of $f$ is defined by
\begin{eqnarray*}\label{radial-maximal of f}
M^{+}_{\phi}(f)(x,y)=\sup_{t,s>0}|\phi_{t,s}\ast
f(x,y)|,
\end{eqnarray*}
where $\phi_{t,s}(x,y)$ is defined as in Definition \ref{def-of-maximal-function}.
\end{definition}

One of the main results of this paper is the following theorem.
\begin{theorem}\label{main}
All the following norms
$$\|g_F(f)\|_1,\ \|S_F(f)\|_1,\ \|M^{*}_{\phi}(f)\|_1,\ \|M^{+}_{\phi}(f)\|_1$$ are equivalent for $f\in L^1(\R^{n+m}).$
\end{theorem}

As a consequence of Theorem \ref{main}, it is natural to introduce the flag Hardy space as follows.
\begin{definition}\label{def-of-hardy-by-han}
The flag Hardy spaces $H_F^1(\mathbb{R}^n\times\mathbb{R}^m)$ is defined to be the collection of  $f\in
L^1(\mathbb{R}^{n+m}) $ such that $ g_F(f)\in
L^1(\mathbb{R}^{n+m}) $. The norm of $H_F^1(\mathbb{R}^n\times\mathbb{R}^m)$  is
defined by
\begin{eqnarray*}\label{Hp norm}
\|f\|_{H_F^1(\mathbb{R}^n\times\mathbb{R}^m)}=\|g_F(f)\|_{1}.
\end{eqnarray*}
\end{definition}

\begin{remark}
Note that the multi-parameter flag structure is involved in the Littlewood--Paley square function $g_F(f)$, the Lusin area integral of $S_F(f)$, the non-tangential and radial maximal function $M^{*}_{\phi}(f)$ and $M^{+}_{\phi}(f).$
Therefore, the multi-parameter flag structure is involved in the flag Hardy space.
Moreover, the flag Hardy space
$H_F^1(\mathbb{R}^n\times\mathbb{R}^m)$ can also be characterized by
the maximal functions, the Littlewood--Paley square function and the Lusin area integrals. We would like to point out that the main results in this paper still hold for all $0<p\leq 1.$ The reason this paper only deals with the case $p=1$ is that we would like to keep the length of this paper more reasonable and present the main ideas necessary for the case $0<p\leq 1$.  The extension to the case for $0<p<1$ is a lengthy technical exercise best left to the interested reader.
\end{remark}

It was well-known that the atomic decomposition is a very important tool to study the boundedness of singular integrals for the classical one parameter and product multi-parameter Hardy spaces. However, the lack of the cancellation was a major difficulty in providing the atomic decomposition for the flag Hardy space. In this paper, we develop a new approach to provide an atomic decomposition for the flag Hardy space, namely the functional calculus, the finite speed propagation and the flag heat semigroups are involved. The one-parameter result was obtain in \cite{HLMMY}. To do this, we introduce the atom as follows.

\begin{definition}\rm\label{def of product atom}
Let $\triangle^{(1)}$ be the Laplacian on $\Bbb R^{n+m}$ and $\triangle^{(2)}$ be the Laplacian on $\Bbb R^m$ and let $M$ be a positive integer.
A function $a(x_1, x_2)\in L^2({\Bbb R}^{n+m})$ is called a
$(1, 2, M)$-atom if it satisfies
\begin{enumerate}
\item[1)] {\rm supp} $a\subset \Omega$, where $\Omega$
is an open set of ${\Bbb R}^{n}\times{\Bbb R}^{m}$ with finite measure;
\item[2)] $a$ can be further decomposed into
$$
a=\sum\limits_{{R=I\times J\in m(\Omega)} \atop {\ell(I)\le \ell(J)}} a_R
$$
\end{enumerate}
where the summation is running over all dyadic rectangles $R=I\times J\subset \mathbb R^n\times \mathbb R^m$ such that $R$ is contained in  $m(\Omega)$ and $\ell(I)\leq\ell(J)$, with $m(\Omega)$ denoting the set of all maximal dyadic
subrectangles contained in $\Omega$, and there exists a series of function $b_R $ belonging to the domain of
$({\triangle^{(1)}})^{k_1}\otimes_2({\triangle^{(2)}})^{k_2}$ in $L^2({\Bbb R}^{n+m})$, for each
$k_1, k_2=1, \cdots, M,$   such that
\begin{enumerate}
\item[(i)] $a_R=\big(({\triangle^{(1)}})^M\otimes_2({\triangle^{(2)}})^M\big) b_R$;
\item[(ii)] supp $\big(({\triangle^{(1)}})^{k_1}\otimes_2({\triangle^{(2)}})^{k_2}\big)b_R\subset
10R$, \ \ $ k_1, k_2=0, 1, \cdots, M$;
\item[(iii)] $\|a\|_2\leq
|\Omega|^{-{1\over 2}}$ and $ k_1, k_2=0, 1, \cdots, M,$
\begin{align*}
\sum_{R=I\times J\in m(\Omega)}\ell(I)^{-4M} \ell(J)^{-4M}
\Big\|\big(\ell(I)^2 \, \triangle^{(1)}\big)^{k_1}\otimes_2\big(\ell(J)^2 \,
\triangle^{(2)}\big)^{k_2} b_R\Big\|_2^2\leq |\Omega|^{-1}.
\end{align*}
\end{enumerate}
\end{definition}

The atomic decomposition for the flag Hardy space is given by the following definition.

\begin{definition}\rm\label{def-of-atomic-product-Hardy-space}
Let $M>m/2$. The Hardy spaces $H^1_{F,at,M}(\mathbb{R}^{n}\times\mathbb{R}^{m})$
is defined as follows.
For $f\in L^2(\mathbb{R}^{n+m}),$ we say that $f=\sum_{j}\lambda_ja_j$ is an atomic
$(1, 2, M)$-representation of $f$ if $\{\lambda_j\}_{j=0}^\infty\in
\ell^1$, each $a_j$ is a $(1, 2, M)$-atom, and the sum converges in
$L^2(\mathbb{R}^{n+m})$. The space $\Bbb H^1_{F,at,M}(\Omega)$ is defined to be
$$\Bbb H^1_{F,at,M}(\mathbb{R}^{n}\times\mathbb{R}^{m})=\{f\in L^2(\mathbb{R}^{n+m}): f\ \mbox{has an atomic $(1,2,M)$-representation}\}$$
with the norm
$$\|f\|_{\Bbb H^1_{F,at,M}}:=\inf\bigg\{\sum_{j=0}^\infty |\lambda_j| :f=\sum_{j=0}^\infty \lambda_ja_j\ \ \text{is an atomic $(1,2,M)$-representation}\bigg\}.$$
The atomic Hardy space $H^1_{F,at,M}(\mathbb{R}^{n}\times\mathbb{R}^{m})$ is defined as the completion of $\Bbb H^1_{F,at,M}(\mathbb{R}^{n}\times\mathbb{R}^{m})$ with respect to this norm.
\end{definition}

\begin{theorem}\label{theorem of Hardy space atomic decom}
Suppose that  $M>m/2$. Then
$$
H^1_F(\mathbb{R}^{n}\times\mathbb{R}^{m})=H^1_{F,at,M}(\mathbb{R}^{n}\times\mathbb{R}^{m}).
$$
Moreover,
$$
\|f\|_{H^1_F(\mathbb{R}^{n}\times\mathbb{R}^{m})}\approx
\|f\|_{\Bbb H^1_{F,at,M}(\mathbb{R}^{n}\times\mathbb{R}^{m})},
$$
where the implicit constants depend only on $M, n$ and
$m$.
\end{theorem}

As a consequence of Theorem \ref{theorem of Hardy space atomic decom}, we obtain the Riesz transform characterisation of the  flag Hardy space. For this purpose, we first introduce the flag Riesz transforms. To do this, let $R^{(1)}_j$ be the $j$-th Riesz transform on $\R^{n+m}$, $j=1,2,\ldots,n+m$, and $R^{(2)}_k$ be the $k$-th Riesz transform on $\R^m$, $k=1,2,\ldots,m$, respectively.  Namely, for each $f\in L^1(\mathbb{R}^{n+m})$
$$
R^{(1)}_jf(x)={\rm p.v.}\ c_{n+m}\int_{\mathbb{R}^{n+m}} \frac{x_j-y_j}{\left\vert x-y \right\vert^{n+m+1}} f(y)dy,\quad x\in\mathbb R^{n+m} $$
and  for each $f\in L^1(\mathbb{R}^{m})$
$$R^{(2)}_kf(z)={\rm p.v.}\ c_m\int_{\mathbb{R}^m} \frac{w_j-z_j}{\left\vert w-z \right\vert^{m+1}} f(w)dw, \quad z\in\mathbb R^{m}.
$$
Again applying the projection of M\"uller, Ricci and Stein, we set $R_{j,k}=R^{(1)}_j\ast_{\mathbb R^m} R^{(2)}_k,$ that is, $R_{j,k}$ is the composition of $R^{(1)}_j$ and $R^{(2)}_k$ on ${\mathbb R^m}.$ Notice that the flag structure is involved in the Riesz transforms $R_{j,k}$ for $j=1,2,\ldots,n+m$ and $k=1,2,\ldots,m.$

\begin{theorem}\label{boundedness of Riesz transforms}
$f\in H^1_F(\R^n\times\R^m)$ if and only if $\sum_{j=1}^{n+m}\sum_{k=1}^m\|R_{j,k}(f)\|_1+\|f\|_1<\infty.$ Moreover,
$$\sum_{j=1}^{n+m}\sum_{k=1}^m\|R_{j,k}(f)\|_1+\|f\|_1\approx  \|f\|_{H^1_F(\R^n\times\R^m)}.$$
\end{theorem}

As a corollary to the above theorems, we conclude following:
\begin{corollary}
\label{c:equivalentnorms}
Let all the notation be the same as above. The following norms
$$\|g_F(f)\|_1,\ \|S_F(f)\|_1,\ \|M^{*}_{\phi}(f)\|_1,\ \|M^{+}_{\phi}(f)\|_1, \ \ \sum_{j=1}^{n+m}\sum_{k=1}^m\|R_{j,k}(f)\|_1+\|f\|_1 $$ are equivalent for $f\in L^1(\R^{n+m}).$
\end{corollary}

\subsection{Strategy of proofs of the main results}

In Section \ref{s:2} we prove the equivalence between $\|g_F(f)\|_1$ and $ \|S_F(f)\|_1$ as a first step of this paper.

We recall that
in the classical case to show that the $L^p$ norms, with $p\leq 1,$ of the Littlewood--Paley square function and Lusin area integral are equivalent, the crucial tool is the sup-inf inequality, namely the Plancherel--P\'olya type
inequality. In order to establish such an inequality, one needs to develop the discrete Calder\'on reproducing formula. See \cite{H2} for more details in the setting of spaces of homogeneous type in the sense of Coifman and Weiss. In the present flag setting, to obtain the equivalence between the square function and Lusin area integral,  we will first establish a discrete Calder\'on reproducing formula and then prove the Plancherel--P\'olya type
inequality associated with the flag structure. As a consequence,  we obtain that
\begin{enumerate}
\item[(I)]\quad $\|g_F(f)\|_1\approx \|S_F(f)\|_1$.
\end{enumerate}
Moreover, following the same approach of developing a discrete reproducing formula and Plancherel--P\'olya type inequality, we also obtain
\begin{enumerate}
\item[(II)]\quad $\|S_F(f)\|_1 \lesssim \|S_F(U)\|_1,$ where $S_F(U)$ is defined below in the Definition \ref{def-of-S-function-Su}.
\end{enumerate}

As the second step, we provide the equivalences of the norms among $\|S_F(f)\|_1$, $\|M^*_\phi(f)\|_1$ and $\sum_{j=1}^{n+m}\sum_{k=1}^m\|R_{j,k}(f)\|_1+\|f\|_1$.

We will introduce the Lusin area integral, the non-tangential maximal function, and the radial maximal function via flag Poisson integrals. We introduce  the flag Poisson kernel by
\begin{eqnarray*}
 P(x,y)=P^{(1)}\ast_{\mathbb R^m}
P^{(2)}(x,y)=\int_{\mathbb{R}^m}P^{(1)}(x,y-z)P^{(2)}(z)dz,
\end{eqnarray*}
where, using the projection of M\"uller, Ricci and Stein,
\begin{eqnarray*}
P^{(1)}(x,y)=\frac{\displaystyle c_{n+m}} {\displaystyle
(1+|x|^2+|y|^2)^{(n+m+1)/2} }\ \ {\rm and}\ \
P^{(2)}(z)=\frac{\displaystyle c_{m}} {\displaystyle
(1+|z|^2)^{(m+1)/2} }\ \ \ \
\end{eqnarray*}
are the classical Poisson kernels on $\mathbb{R}^{n+m}$ and $\mathbb{R}^m$,
respectively.

For any $f\in L^1(\mathbb{R}^{n+m})$, we
define the flag Poisson integral of $f$ by
\begin{align}\label{Poisson integral}
U(x,y,t,s):=P_{t,s}\ast f(x,y),
\end{align}
where
\begin{align}\label{Pts}
P_{t,s}(x,y)=P^{(1)}_t\ast_{\mathbb R^m}P^{(2)}_s(x,y).
\end{align}

Since $P_{t,s}(x,y)\in L^1(\mathbb{R}^{n+m})$, it is easy to see
that $U(x,y,t,s)$ is well-defined. Moreover, for any
fixed $t$ and $s$, $P_{t,s}\ast f$ is a bounded $C^{\infty}$
function and the
function $U(x,y,t,s)$ is harmonic
in $(x,y,t)$ and $(y,s)$, respectively.

We now define the flag Lusin area integral of
$U$ as follows.
\begin{definition}\label{def-of-S-function-Su}
For $f\in L^1(\mathbb{R}^{n+m})$ and $U(x,y,t,s)=P_{t,s}\ast f(x,y),$
$S_F(U)$, the flag Lusin area integral of the flag Poisson integral $U(x,y,t,s)$ is
defined by
\begin{eqnarray*}\label{S-function-Su}
S_F(U)(x,y)=\bigg\{\int_{\mathbb{R}_{+}^{n+1}}\int_{\mathbb{R}_{+}^{m+1}}\chi_{t,s}(x-x_1,y-y_1)
|t\nabla^{(1)}s\nabla^{(2)}U(x_1,y_1,t,s)|^2 {dx_1dt \over t^{n+m+1}}{dy_1ds
\over s^{m+1}} \bigg\}_,^{{1\over 2}}\ \ \
\end{eqnarray*}
where $\chi_{t,s}(x,y)$ is the same as in Definition
\ref{def-of-S-function-Sf},
$\nabla^{(1)}=\big(\partial_t,\partial_{x_1},\ldots,\partial_{x_n},\partial_{y_1},\ldots,\partial_{y_m}\big)$
and
$\nabla^{(2)}=\big(\partial_s,\partial_{y_1},\ldots,\partial_{y_m}\big)$.
\end{definition}
Next, we define the non-tangential maximal
function of $U.$
\begin{definition} \label{def-of-non-tangential-maximal of u}
Let $f\in L^1(\mathbb{R}^{n+m})$, the
non-tangential maximal function of $U$
is defined by
\begin{eqnarray*}
U^{*}(x,y)=\sup_{(x_1,y_1,t,s)\in\Gamma(x,y)}|P_{t,s}\ast
f(x_1,y_1)|,
\end{eqnarray*}
where $\Gamma(x,y)=\{(x_1,y_1,t,s):\ |x-x_1|\leq t,\ |y-y_1|\leq t+s
\}$.
\end{definition}
Similarly, the radial maximal
function of $U$ is given by the following
\begin{definition} \label{def-of-radial-maximal of u}
Let $f\in L^1(\mathbb{R}^{n+m})$, the
radial maximal function of $U$
is defined by
\begin{eqnarray*}
U^{+}(x,y)=\sup_{t>0,s>0}|P_{t,s}\ast
f(x,y)|.
\end{eqnarray*}
\end{definition}
In Section 3, we will show the following inequalities:
\begin{equation*}
\begin{aligned}
\|S_F(U)\|_1 \lesssim \|U^{*}\|_1 \lesssim \|U^{+}\|_1
\lesssim \sum_{j=1}^{n+m}\sum_{k=1}^m\|R_{j,k}(f)\|_1+\|f\|_1,
\end{aligned}
\end{equation*}
where the first inequality follows from establishing a flag-type Merryfield's lemma (which is a suitable substitution of the ``good $\lambda$ inequality'' in the one-parameter setting), the second inequality is obtained by combining two approaches: the harmonic function technique (sub-harmonic inequality) and the grand maximal function technique (this is natural here since part of the behaviour of the flag maximal function is like the one-parameter case while the other part is like the tensor product case), and the last inequality follows from flag-type generalised Cauchy--Riemann equations via the techniques of Poisson kernel and conjugate Poisson kernels.

In Section 4, the following estimates will be concluded:
\begin{enumerate}
\item[(III)] $\|U^{*}\|_1\approx \|M^{*}_\phi(f)\|_1,$

\item[(IV)] $\|U^{+}\|_1\approx \|M^{+}_\phi(f)\|_1$,
\end{enumerate}
The main approach here is to introduce a suitable flag-type grand maximal function and to apply a suitable decomposition of Poisson kernel into a series of Schwartz functions, as well as forming a Schwartz function from the Poisson kernel.

In Section 5, the main breakthrough is the flag-type atomic decomposition. We introduce a suitable version of flag atoms, and define the Hardy space via atoms. Then we prove its equivalence with the Hardy space defined via area function, where the key approach is to use functional calculus and the semigroup technique.
Then as a direct application of the atomic decomposition,
we obtain the following estimate
\begin{enumerate}
\item[(V)] $\sum_{j=1}^{n+m}\sum_{k=1}^m\|R_{j,k}(f)\|_1+\|f\|_1\lesssim \|g_F(f)\|_1.$
\end{enumerate}

Indeed, for each $f\in L^1(\mathbb R^{n+m}),$ 
from all the above estimates (I)---(V), we have
the following chain of inequalities: 
\begin{align*}
&\| g_F(f)\|_1\approx \|S_F(f)\|_1 \lesssim  \|S_F(U)\|_1 \hskip6cm{\rm(Littlewood-Paley)}\\
& \lesssim \|U^*\|_1\lesssim \|M^{*}_\phi(f)\|_1
\lesssim  \|U^*\|_1\lesssim \|U^+\|_1\lesssim \|M^{+}_\phi(f)\|_1
\lesssim  \|U^+\|_1 \hskip1cm{\rm(maximal\ function)}\\
&\lesssim \sum_{j=1}^{n+m}\sum_{k=1}^m\|R_{j,k}(f)\|_1+\|f\|_1 \hskip6.85cm{\rm(Riesz\ transform)}\\
&\lesssim \|S_F(f)\|_1.\hskip9.58cm{\rm(Littlewood-Paley)}
\end{align*}
This implies the main result Theorem \ref{main}; it also gives Theorem \ref{boundedness of Riesz transforms} and Corollary \ref{c:equivalentnorms}.

\subsection{Applications and related open questions}

{\bf Application I}:
\smallskip

As a first direct application of our Theorem \ref{boundedness of Riesz transforms},
i.e., the flag Riesz transform characterisation of $H_F^1(\mathbb{R}^n\times\mathbb{R}^m)$, and the duality of  $H^1$ with BMO space studied in \cite{HLS},
we obtain the decomposition of flag BMO space.

\begin{corollary}\label{cor BMO decom}
The following two statements are equivalent.

$(i)$ $\varphi \in {\rm BMO}_F(\R^n\times \R^m) $;

$(ii)$ There exist $g_{j,k}\in L^\infty( \R^{n+m} )$, $j=0,1,\ldots,n+m$, $k=0,1,\ldots,m$, such that
$$ \varphi=\sum_{j=0}^{n+m}\sum_{k=0}^m R_{j,k}(g_{j,k}). \ $$
\end{corollary}

This provides a soft proof of the decomposition of flag BMO space ${\rm BMO}_F(\mathbb{R}^n\times\mathbb{R}^m)$.

\bigskip
\noindent {\bf Application II}:
\smallskip

In this multi-parameter flag setting, concerning the space  ${\rm BMO}_F(\mathbb{R}^n\times\mathbb{R}^m)$ and the flag Riesz transforms, it is natural to study the commutator of $b\in {\rm BMO}_F(\mathbb{R}^n\times\mathbb{R}^m)$ and the flag Riesz transforms.

Here we recall that the classical commutator of a symbol $b$ and the Hilbert transform was first introduced by A. Calder\'on.
Later, Coifman, Rochberg and Weiss \cite{CRW} established the equivalent characterisation of BMO and the boundedness of commutator, which recovers a well-known result of Nehari \cite{Ne}  on Hankel operators in complex analysis. Later,  Ferguson and Lacey \cite{FLa} established the equivalent characterisation of product BMO and the iterated commutator of Hilbert transforms in each variable separately, which connects to the little Hankel operator on the bi-disc setting (see \cite{FSa}). Then Lacey, Petermichl, Pipher and the last author \cite{LPPW} further generalised this to the product setting of $\mathbb R^n\times \mathbb R^m$ for  iterated commutator of Riesz transforms, which bypassed the use of analyticity in \cite{FLa}.

Based on the decomposition of flag BMO space ${\rm BMO}_F(\mathbb{R}^n\times\mathbb{R}^m)$ via flag Riesz transforms as in Corollary \ref{cor BMO decom} above,
we see that the suitable definition of flag iterated commutator is as follows.

Given two functions $b, f\in L^2(\mathbb{R}^{n+m})$,
we first recall the usual definition of commutator
\begin{align}
[b,R_j^{(1)}](f)(x_1,x_2) := b(x_1,x_2) R_j^{(1)} \ast f (x_1,x_2) - R_j^{(1)} \ast (bf)(x_1,x_2).
\end{align}
The commutator can also act only on the second variable:
\begin{align}
[b,R_k^{(2)}]_{2}(f)(x_1,x_2):= b(x_1,x_2) R_k^{(2)} \ast_2 f (x_1,x_2) - R_k^{(2)} \ast_2 (bf)(x_1,x_2).
\end{align}

Our iterated commutator takes the following form:

\begin{definition}\label{def-of-flag-commutator}
Given two functions $b, f\in L^2(\mathbb{R}^{n+m})$, the iterated commutator in the flag setting of $\R^n\times\R^m$ is defined as
\begin{align*}
[[b,R_j^{(1)}],R^{(2)}_k]_2(f) &:= b(x_1,x_2) R_j^{(1)} \ast R^{(2)}_k\ast_2 f(x_1,x_2) - R_j^{(1)} \ast (b\cdot R^{(2)}_k\ast_2 f) (x_1,x_2)\\
&\quad - R^{(2)}_k \ast_2 \big( b\cdot R_j^{(1)} \ast f\big) (x_1,x_2) + R^{(2)}_k\ast_2 R_j^{(1)}\ast (b\cdot f) (x_1,x_2).
\end{align*}
\end{definition}

Parallel to the result in one-parameter setting by Coifman, Rochberg and Weiss \cite{CRW} and tensor product setting by Ferguson and Lacey \cite{FLa} as well as Lacey, Petermichl, Pipher and the last author \cite{LPPW},
it is natural to explore the following:

\smallskip
{\bf Q1:} Is there a weak factorisation of flag Hardy space $H_F^1(\mathbb R^n\times \mathbb R^m)$?

\smallskip
{\bf Q2:} Is there an equivalent characterisation of the flag BMO space ${\rm BMO}_F(\mathbb R^n\times \mathbb R^m)$ via the iterated commutator $[[b,R_j^{(1)}],R^{(2)}_k]_2$? That is, suppose $b\in {\rm BMO}_F(\mathbb R^n\times \mathbb R^m)$,
$$ \|b\|_{{\rm BMO}_F(\mathbb R^n\times \mathbb R^m)}\approx \big\|[[b,R_j^{(1)}],R^{(2)}_k]_2:\ \ L^2(\mathbb R^{n+m})\to L^2(\mathbb R^{n+m})\big\|?  $$

\bigskip

For  {\bf Q1}, we point out that in one parameter there is a soft proof via the lower bound of the commutator \cite{CRW}, and there is a direct proof via construction in terms of atomic decomposition of Hardy space \cite{U}. Whether our flag atoms can be a suitable candidate for weak factorisation is an open question.

\bigskip

For {\bf Q2}, we point out that recently Duong, Ou, Pipher, and the third and fourth authors \cite{DLOPW} have proved the upper bound of the iterated commutator, i.e.,
$$ \big\|[[b,R_j^{(1)}],R^{(2)}_k]_2:\ \ L^2(\mathbb R^{n+m})\to L^2(\mathbb R^{n+m})\big\|\lesssim\|b\|_{{\rm BMO}_F(\mathbb R^n\times \mathbb R^m)}.  $$

\bigskip

We also point out that the lower bound of the iterated commutator is equivalent to Question 1, however, both remain open.

\newpage
\section{Flag Littlewood--Paley estimate: $\|g_F(f)\|_1$, $ \|S_F(f)\|_1$ and $ \|S_F(U)\|_1$ }
\label{s:2}

\subsection{Discrete Calder\'on reproducing formula}

We first recall the following test function space $\widetilde{\mathcal M}_{d}$   with the size and smoothness conditions on $\Bbb R^d$ for arbitrary positive integer $d$,
which was introduced in \cite{H1}.

\begin{definition}\label{pre test}
Fix two exponents $0<\beta < 1$ and $\gamma>0$. We say that $f$ defined on $\mathbb R^d,$ belongs to  $\widetilde{\mathcal M}_{d}(\beta,\gamma,r,x_0)$, $r>0$ and $x_0\in \Bbb R^d$, if
\begin{align}
|f(x)|&\le C\frac {r^\gamma}{(r+|x-x_0|)^{d+\gamma}},  \label{eq05}\\
|f(x) -f(x')|&\le C\Big(\frac{|x-x'|}{r+|x-x_0|}\Big)^\beta
  \frac {r^\gamma}{(r+|x-x_0|)^{d+\gamma}}\label{eq06}
\end{align}
for $|x-x'|\le \frac{r+|x-x_0|}2$.
If $f\in \widetilde{\mathcal M}_d(\beta,\gamma,r,x_0)$,
  then the norm of $f$ is defined by
$$\|f\|_{\widetilde{\mathcal M}_d(\beta,\gamma,r,x_0)}
  =\inf\{C: \eqref{eq05}\ \text{and}\ \eqref{eq06}\ \text{hold}\}.$$
\end{definition}

Then we recall the test function space ${\mathcal M}_{d}(\beta,\gamma,r,x_0)\subset \widetilde{\mathcal M}_{d}(\beta,\gamma,r,x_0)$ on $\Bbb R^d$ with a cancellation condition.
\begin{definition}\label{test}
Fix two exponents $0<\beta < 1$ and $\gamma>0$. We say that $f$ defined on $\mathbb R^d,$ belongs to  ${\mathcal M}_{d}(\beta,\gamma,r,x_0)$, $r>0$ and $x_0\in \Bbb R^d$, if
$f\in \widetilde{\mathcal M}_{d}(\beta,\gamma,r,x_0)$
 and
$$\int_{\Bbb R^d} f(x)\,dx = 0. $$
\end{definition}
 If $f\in {\mathcal M}_d(\beta,\gamma,r,x_0)$,
  then the norm of $f$ is defined by
$$\|f\|_{{\mathcal M}_d(\beta,\gamma,r,x_0)}=\|f\|_{\widetilde{\mathcal M}_d(\beta,\gamma,r,x_0)}
.$$
We now define the test function space on $\Bbb R^{n+m}\times \Bbb R^m$ as follows.

\begin{definition}\label{product test}
Fix two exponents $0<\beta < 1$ and $\gamma>0$. We say that $f$ defined on $\Bbb R^{n+m}\times \Bbb R^m$ belongs to $\widetilde{\mathcal M}_{(n+m)\times m}(\beta,\gamma,r_1,r_2,$ $x_0,y_0,z_0)$, $r_1,r_2>0$ and $(x_0,y_0,z_0)
\in \Bbb R^{n+m}\times \Bbb R^m$, if for each fixed $z\in \mathbb R^m, f(\cdot,\cdot,z)\in \widetilde{\mathcal M}_{n+m}(\beta,\gamma,r_1,x_0,y_0)$
and for each $(x,y)\in \mathbb R^{n+m}, f(x,y,\cdot)\in \widetilde{\mathcal M}_{m}(\beta,\gamma,r_2,z_0)$ and satisfies the following conditions:
\begin{enumerate}
\item[\textnormal{(1)}] $\|f(\cdot,\cdot,z)\|_{\widetilde{\mathcal M}_{n+m}(\beta,\gamma,r_1,x_0,y_0)}\le C\frac{\displaystyle r_2^\gamma}{\displaystyle (r_2+|z-z_0|)^{m+\gamma}}$,
\item[\textnormal{(2)}] $\|f(x,y,\cdot)\|_{\widetilde{\mathcal M}_{m}(\beta,\gamma,r_2,z_0)} \le C\frac {\displaystyle r_1^\gamma}{\displaystyle (r_1+|x-x_0|+|y-y_0|)^{n+m+\gamma}}$,
\item[\textnormal{(3)}] $\|f(\cdot,\cdot,z)-f(\cdot,\cdot,z')\|_{\widetilde{\mathcal M}_{n+m}(\beta,\gamma,r_1,x_0,y_0)}\le C\Big(\frac{\displaystyle |z-z'|}{\displaystyle r_2+|z-z_0|}\Big)^\beta\frac {\displaystyle r_2^\gamma}{\displaystyle (r_2+|z-z_0|)^{m+\gamma}}$
\item[]    for $|z-z'|\le \frac{r_2+|z-z_0|}2$,
\item[\textnormal{(4)}] $\|f(x,y,\cdot)-f(x',y',\cdot)\|_{\widetilde{\mathcal M}_{m}(\beta,\gamma,r_2,z_0)}$
\item[]\quad\quad$\le C\Big(\frac{\displaystyle |x-x'|+|y-y'|}{\displaystyle r_1+|x-x_0|+|y-y_0|}\Big)^\beta\frac {\displaystyle r_1^\gamma}{\displaystyle (r_1+|x-x_0|+|y-y_0|)^{n+m+\gamma}}$
\item[]    for $|x-x'|+|y-y'|\le \frac{r_1+|x-x_0|+|y-y_0|}2$.
\end{enumerate}
If $f\in \widetilde{\mathcal M}_{(n+m)\times m}(\beta,\gamma,r_1,r_2,$ $x_0,y_0,z_0)$, the norm of $f$ is defined by
$$\|f\|_{\widetilde{\mathcal M}_{(n+m)\times m}(\beta,\gamma,r_1,r_2,x_0,y_0,z_0)}
  =\inf\{C: (1)- (4)\  \text{hold}\}.$$

Similarly we have the definition for the test function space ${\mathcal M}_{(n+m)\times m}(\beta,\gamma,r_1,r_2,x_0,y_0,z_0)$
as a subset in $\widetilde{\mathcal M}_{(n+m)\times m}(\beta,\gamma,r_1,r_2,x_0,y_0,z_0)$ and satisfies the corresponding cancellation conditions for the variables $(x,y)$ and for $z$, respectively.
\end{definition}

We would like to point out that if $f_1\in {\mathcal M}_{n+m}(\beta,\gamma,r_1,x_0,y_0)$ and $f_2\in {\mathcal M}_{m}(\beta,\gamma,r_2,z_0)$ then $f(x,y,z)
=f_1(x,y)f_2(z)\in {\mathcal M}_{(n+m)\times m}(\beta,\gamma,r_1,r_2,x_0,y_0,z_0).$

The flag test function is defined by:
\begin{definition}\label{flag test function}
Let $0<\beta,\gamma<1$, $r_1,r_2>0$ and $x_0\in \Bbb R^n$, $y_0\in \Bbb R^m$. We say that a function $f$ defined on $\Bbb R^n\times \Bbb R^m$ belongs to
the flag test function space $\widetilde{\mathcal M}_{flag}(\beta,\gamma,r_1,r_2,x_0,y_0)$ if there exists a function $f^\sharp(x,y,z)\in \widetilde{\mathcal M}_{(n+m)\times m}(\beta,\gamma,r_1,r_2,x_0,y_0,z_0)$
such that
$$f(x,y)=\int_{\Bbb R^m} f^\sharp(x,y-z,z)dz.$$
If $f\in \widetilde{\mathcal M}_{flag}(\beta,\gamma,r_1,r_2,x_0,y_0)$, the norm of $f$ is defined by
$$\|f\|_{\widetilde{\mathcal M}_{flag}(\beta,\gamma,r_1,r_2,x_0,y_0)}=\inf\bigg\{\|f^\sharp\|_{{\mathcal M}_{(n+m)\times m}(\beta,\gamma,r_1,r_2,x_0,y_0,z_0)}: f(x,y)=\int_{\Bbb R^m} f^\sharp(x,y-z,z)dz\bigg\}.$$
Similarly we can define the test function space ${\mathcal M}_{flag}(\beta,\gamma,r_1,r_2,x_0,y_0)$ with the flag cancellation condition as a subset in $ \widetilde{\mathcal M}_{flag}(\beta,\gamma,r_1,r_2,x_0,y_0)$, which is projected from the product test function space ${\mathcal M}_{(n+m)\times m}(\beta,\gamma,r_1,r_2,x_0,y_0,z_0)$.
\end{definition}

Observe that the flag structure is involved in the structure of the flag test function space ${\mathcal M}_{flag}(\beta,\gamma,r_1,r_2,x_0,y_0)$.
We now prove the following discrete Calder\'on reproducing formula.

\begin{theorem}\label{dcrf}
Let $\beta,\gamma\in(0,1)$ and $r_1,r_2,r_3>0$, $x_0$ be fixed point in $\mathbb R^n$, $y_0$ and $z_0$ be fixed points in $\mathbb R^m$. For $j,k\in\mathbb Z$ and a fixed small positive number $\alpha$, let $\psi_{j,k} = \psi_{2^{-\alpha j},2^{-\alpha k}}$, whose definition is given in \eqref{psi jk}. Then there exist functions $\phi_{j,k}(x,y, x_I,y_J)\in {\mathcal M}_{flag}(\beta,\gamma,2^{-j},2^{-k},x_I,y_J)$ and a fixed large integer $N$ such that for the flag test function $f(x,y)=\int_{\mathbb R^m}f_1(x,y-z)f_2(z)dz$ with $f_1\in {\mathcal M}_{n+m}(\beta,\gamma,r_1,r_2,x_0,y_0)$ and $f_2\in {\mathcal M}_{m}(\beta,\gamma,r_3,z_0),$
\begin{align}\label{discrf}
f(x,y)=c_\alpha\sum_{ j,k\in\mathbb Z} \sum_{I,J}|I||J|\ {\phi}_{j,k}(x,y, x_I,y_J)\ \psi_{j,k}\ast f(x_I,y_J),
\end{align}
where the series converges in $L^2(\mathbb R^{n+m})$ and in ${\mathcal M}_{flag}(\beta,\gamma,r_1,r_2,x_0,y_0)$, $I\subset \mathbb{R}^n$ and $J\subset \mathbb{R}^m$ are dyadic cubes with side-lengths $\ell(I)=2^{-j-N}$ and $\ell(J)=2^{-(j\wedge k-N)}$, and $x_I$ and $y_J$ are {\bf any fixed points} in $I$ and $J,$ respectively.
\end{theorem}
Note that for each $f\in L^1(\mathbb R^{n+m}),$ $f\in ({\mathcal M}_{flag}(\beta,\gamma,r_1,r_2,x_0,y_0))^\prime.$ As a consequence of Theorem \ref {dcrf}, by duality, if $\psi_{t,s}$ is the same as in (\ref{psi jk}), $h\in {\mathcal M}_{flag}(\beta,\gamma,r_1,r_2,x_0,y_0)$ and $f\in L^1(\mathbb R^{n+m}),$
\begin{align}\label{discrf1}
\langle f,h \rangle=\bigg\langle c_\alpha\sum_{ j,k\in\mathbb Z} \sum_{I,J}|R|\ {\phi}_{j,k}(\cdot,\cdot, x_I,y_J)\ \psi_{j,k}\ast f(x_I,y_J),h\bigg\rangle.
\end{align}

\begin{remark}\label{remark1}
Indeed, the series in the right-hand side of \eqref{discrf} converges in the test function space ${\mathcal M}_{flag}(\beta,\gamma,r_1,r_2,x_0,y_0)$ and in the distribution space $({\mathcal M}_{flag}(\beta,\gamma,r_1,r_2,x_0,y_0))^\prime,$ the dual of ${\mathcal M}_{flag}(\beta,\gamma,r_1,r_2,x_0,y_0).$
However, the proofs of such results are a little bit complicated. In this paper, we focus only on the Hardy space with $p=1.$ Thus, for our purpose, we only need the convergence in the distribution sense as given in \eqref{discrf1}.
\end{remark}

\begin{proof}[Proof of Theorem \ref{dcrf}]
To show Theorem \ref{dcrf}, observe that if $\psi_{t,s}$ are as in (\ref{psi jk}), by taking the Fourier transform, we have the following Calder\'on's reproducing formula, namely for all $f\in L^2(\mathbb{R}^{n+m}),$
\begin{eqnarray}\label{continuous-Caldeorn-reproducing-formula}
f(x,y)=\int^\infty_0\int^\infty_0\psi_{t,s}\ast\psi_{t,s}\ast f(x,y)\frac{dt}{t}\frac{ds}{s},
\end{eqnarray}
where the series converges in $L^2(\mathbb{R}^{n+m}).$ One can also show this reproducing formula holds in ${\mathcal M}_{flag}(\beta,\gamma,r_1,r_2,x_0,y_0)$ by analysis on the convergence.

Suppose that $f\in {\mathcal M}_{flag}(\beta,\gamma,r_1,r_2,x_0,y_0)$ with $f(x,y)=\int_{\mathbb R^{m}}f_1(x,y-z)f_2(z)dz$ where $f_1\in {\mathcal M}_{n+m}(\beta,\gamma,r_1,r_2,x_0,y_0)$ and $f_2\in {\mathcal M}_{m}(\beta,\gamma,r_3,z_0).$  Fix an arbitrary contant $\alpha>0$.
We first split the continuous reproducing formula \eqref{continuous-Caldeorn-reproducing-formula} into three parts
\begin{eqnarray}\label{disc-Caldeorn-reproducing-formula1 pre}
f(x,y)&=&\int^\infty_0 \int^\infty_0\psi_{t,s}\ast\psi_{t,s}\ast f(x,y)\frac{dt}{t}\frac{ds}{s} \\
&=&\sum_{ j,k\in\mathbb Z} \int^{2^{-\alpha(k-1)}}_{2^{-\alpha k}}\int^{2^{-\alpha(j-1)}}_{2^{-\alpha j}} \psi_{t,s}\ast\psi_{t,s}\ast f(x,y)\frac{dt}{t}\frac{ds}{s} \nonumber\\
&=&c_\alpha\bigg(\sum_{ \substack{j,k\in\mathbb Z\\ j\leq k}}+\sum_{ \substack{j,k\in\mathbb Z\\ j> k}}\bigg)  \psi_{j,k}\ast\psi_{j,k}\ast f(x,y)\nonumber\\
&&+\sum_{ j,k\in\mathbb Z} \int^{2^{-\alpha(k-1)}}_{2^{-\alpha k}}\int^{2^{-\alpha(j-1)}}_{2^{-\alpha j}} (\psi_{t,s}\ast\psi_{t,s} -\psi_{j,k}\ast\psi_{j,k})\ast f(x,y)\frac{dt}{t}\frac{ds}{s}\nonumber\\
&=:&\mathcal T_1 (f)(x,y) +\mathcal T_2 (f)(x,y)+\mathcal R_\alpha (f)(x,y),\nonumber
\end{eqnarray}
where $\psi_{j,k}= \psi_{2^{-\alpha j},2^{-\alpha k}}$ and
$c_\alpha =\int^{2^{-\alpha(k-1)}}_{2^{-\alpha k}}\int^{2^{-\alpha(j-1)}}_{2^{-\alpha j}} \frac{dt}{t}\frac{ds}{s}= 2(\ln2)^2\alpha^2 $.

For the first operator $\mathcal T_1 (f)(x,y)$.  Recalling the definition of $\psi_{t,s} $ in \eqref{psi jk}, we now rewrite
$\mathcal T_1 (f)(x,y)$ as follows.
\begin{eqnarray*}
\mathcal T_1 (f)(x,y)
&=&c_\alpha\sum_{ j\in\mathbb Z} \sum_{ \substack{k\in\mathbb Z\\ j\leq k}}\psi^{(1)}_{j} \ast_{\mathbb R^m} \psi^{(2)}_{k}\ast \psi^{(1)}_{j} \ast_{\mathbb R^m} \psi^{(2)}_{k}\ast f(x,y)\nonumber\\
&=&c_\alpha\sum_{ j\in\mathbb Z} \bigg(\sum_{ \substack{k\in\mathbb Z\\ j\leq k}} \psi^{(1)}_{j} \ast_{\mathbb R^m} (\psi^{(2)}_{k}\ast_{\mathbb R^m}   \psi^{(2)}_{k})\bigg) \ast \psi^{(1)}_{j} \ast f(x,y) \nonumber\\
&=&c_\alpha\sum_{ j\in\mathbb Z}   \tilde\psi^{(1)}_{j}  \ast \psi^{(1)}_{j} \ast f(x,y) ,\nonumber
\end{eqnarray*}
where we denote $ \tilde\psi^{(1)}_{j} := \sum_{ \substack{k\in\mathbb Z\\ j\leq k}} \psi^{(1)}_{j} \ast_{\mathbb R^m} (\psi^{(2)}_{k}\ast_{\mathbb R^m}   \psi^{(2)}_{k}) $, and it is easy to verify that $ \tilde\psi^{(1)}_{j}$ satisfies the same conditions as $\psi^{(1)}_{j}$ does on $\mathbb R^{n+m}$.
Hence,  we see that the performance of the operator $\mathcal T_1$ is a one parameter analogous singular integral operator.

Hence, we now decompose $\mathbb R^n\times \mathbb R^m$ into dyadic cubes of the form $R= I\times J$ with $\ell(I) = 2^{-j-N}$ and $\ell(J) = 2^{-j-N}$, where $N$ is a large fixed positive integer.
Applying Coifman's decomposition of the identity yields
\begin{eqnarray*}\label{main term T1}
\mathcal T_1 (f)(x,y)
&=& c_\alpha\sum_{ \substack{j,k\in\mathbb Z\\ j\leq k}}\psi_{j,k}\ast\psi_{j,k}\ast f(x,y)\\
&= & c_\alpha\sum_{ \substack{j,k\in\mathbb Z\\ j\leq k}} \sum_{I,J} \int_{I\times J}\psi_{j,k}(x-u,y-v)\, \psi_{j,k}\ast f(u,v)\, dudv\\
&= & c_\alpha\sum_{ \substack{j,k\in\mathbb Z\\ j\leq k}} \sum_{I,J}|R|\ \Big({1\over |R|} \int_{I\times J}\psi_{j,k}(x-u,y-v)\,dudv\Big)\ \psi_{j,k}\ast f(x_I,y_J)+\mathcal{R}_N^{(1)}(f)(x,y),
\end{eqnarray*}
where $x_I$ and $y_J$ are {\bf any fixed points} in $I$ and $J$, respectively, and
\begin{eqnarray*}\label{main term R1}
\mathcal{R}_N^{(1)}(f)(x,y)
&=&c_\alpha\sum_{ j} \sum_{I,J} \int_{I\times J} \sum_{ \substack{k\in\mathbb Z\\ j\leq k}} \psi_{j,k}(x-u,y-v)\,    (\psi_{j,k} \ast f(u,v)-\psi_{j,k} \ast f(x_I,y_J))\,dudv .\nonumber
\end{eqnarray*}
We now need to show that
$\mathcal{R}_N^{(1)}(f)$ is bounded on
${\mathcal M}_{flag}(\beta,\gamma,r_1,r_2,x_0,y_0)$
with small norm. To see this, we need to consider the lifting of
$\mathcal{R}_N^{(1)}(f)$ onto the product setting $\mathbb R^{n+m}\times\mathbb R^m$ and then estimate its norm with respect to
$\widetilde{\mathcal M}_{(n+m)\times m}(\beta,\gamma,r_1,r_2,x_0,y_0,z_0)$.

To see this, by observing that
\begin{align*}
&\psi_{j,k}\ast f(u,v)-\psi_{j,k}\ast f(x_I,y_J) \\
&=\int_{\mathbb R^{n+m}}[\psi_{j,k}(u-u',v-v')-\psi_{j,k}(x_I-u',y_J-v')]f(u',v')du'dv',
\end{align*}
we can write
\begin{eqnarray*}
\mathcal{R}_N^{(1)}(f)(x,y)&=&c_\alpha\sum_{ j} \sum_{I,J} \int_{I\times J} \sum_{ \substack{k\in\mathbb Z\\ j\leq k}}  \int_{\mathbb R^{n+m}}\psi_{j,k}(x-u,y-v)\\
&&\times[\psi_{j,k}(u-u',v-v')-\psi_{j,k}(x_I-u',y_J-v')]\  f(u',v')du'dv'\, dudv.
\end{eqnarray*}
Note that  $f(u',v')=\int_{\mathbb R^m}f_1(u',v'-w')f_2( w')dw'$ with $f_1\cdot f_2 \in  {\mathcal M}_{(n+m)\times m}(\beta,\gamma,r_1,r_2,x_0,y_0,z_0)$, and that
\begin{align*}
&\psi_{j,k}(x-u,y-v)=\int_{\mathbb R^m}\psi^{(1)}_j(x-u,y-v-z)\psi^{(2)}_k(z)dz\\
&\psi_{j,k}(u-u',v-v')=\int_{\mathbb R^m}\psi^{(1)}_j(u-u',v-v'-w)\psi^{(2)}_k(w)dw\\
&\psi_{j,k}(x_I-u',y_J-v')= \int_{\mathbb R^m}\psi^{(1)}_j(x_I-u',y_J-v'-w)\psi^{(2)}_k(w)dw
\end{align*}
 Thus, we have
\begin{align*}
&\mathcal{R}_N^{(1)}(f)(x,y)\\
&=\int_{\Bbb R^m} \int_{\mathbb R^{n+m}}\int_{\Bbb R^m}\int_{\Bbb R^m} c_\alpha\sum_{ j} \sum_{I,J}   \int_{I\times J} \sum_{ \substack{k\in\mathbb Z\\ j\leq k}} \psi^{(1)}_{j}(x-u,y-v-z)\psi^{(2)}_k(z)\\
&\qquad\times[\psi^{(1)}_{j}(u-u',v-v'-w)-\psi^{(1)}_{j}(x_I-u',y_J-v'-w)]\psi^{(2)}_k(w)dudv \ dw\\
&\qquad\qquad \times f^\sharp(u',v'-w',w')du'dv'dw'dz\\
&=\int_{\Bbb R^m} \int_{\mathbb R^{n+m}}\int_{\Bbb R^m}\int_{\Bbb R^m} c_\alpha\sum_{ j} \sum_{I,J}   \int_{I\times J} \sum_{ \substack{k\in\mathbb Z\\ j\leq k}} \psi^{(1)}_{j}(x-u,y-v-z)\psi^{(2)}_k(z)\\
&\qquad\times[\psi^{(1)}_{j}(u-u',v-v'-w-w')-\psi^{(1)}_{j}(x_I-u',y_J-v'-w-w')]\psi^{(2)}_k(w)dudv \ dw\\
&\qquad\qquad \times f^\sharp(u',v',w')du'dv'dw'dz,
\end{align*}
where the second equality follows from changing of variable with respect to $v'$.
Next, by changing the variable with respect to $w$, we have
\begin{align*}
&\mathcal{R}_N^{(1)}(f)(x,y)\\
&=\int_{\Bbb R^m} \int_{\mathbb R^{n+m}}\int_{\Bbb R^m}\int_{\Bbb R^m} c_\alpha\sum_{ j} \sum_{I,J}   \int_{I\times J} \sum_{ \substack{k\in\mathbb Z\\ j\leq k}} \psi^{(1)}_{j}(x-u,y-v-z)\psi^{(2)}_k(z)\\
&\qquad\times[\psi^{(1)}_{j}(u-u',v-v'-w)-\psi^{(1)}_{j}(x_I-u',y_J-v'-w)]\psi^{(2)}_k(w-w')dudv \ dw\\
&\qquad\qquad \times f^\sharp(u',v',w')du'dv'dw'dz.
\end{align*}
Then we continue to change the variable with respect to $z$ to get
\begin{align*}
&\mathcal{R}_N^{(1)}(f)(x,y)\\
&=\int_{\Bbb R^m} \int_{\mathbb R^{n+m}}\int_{\Bbb R^m}\int_{\Bbb R^m} c_\alpha\sum_{ j} \sum_{I,J}   \int_{I\times J} \sum_{ \substack{k\in\mathbb Z\\ j\leq k}} \psi^{(1)}_{j}(x-u,y-v-z+w)\psi^{(2)}_k(z-w)\\
&\qquad\times[\psi^{(1)}_{j}(u-u',v-v'-w)-\psi^{(1)}_{j}(x_I-u',y_J-v'-w)]\psi^{(2)}_k(w-w')dudv \ dw\\
&\qquad\qquad \times f^\sharp(u',v',w')du'dv'dw'dz.
\end{align*}
Then we can write
\begin{align*}
\mathcal{R}_N^{(1)}(f)(x,y)&=  \int_{\mathbb R^m} (\mathcal{R}_N^{(1)})^{\sharp}(f^{\sharp})(x,y-z,z) dz,
\end{align*}
where the kernel of $(\mathcal{R}_N^{(1)})^{\sharp}$ is given by
\begin{align}\label{R1sharp}
&(\mathcal{R}_N^{(1)})^{\sharp}(x,y,z;u',v',w')\\
&=\int_{\Bbb R^m} c_\alpha\sum_{ j} \sum_{I,J}   \int_{I\times J} \sum_{ \substack{k\in\mathbb Z\\ j\leq k}} \psi^{(1)}_{j}(x-u,y-v-z+w)\psi^{(2)}_k(z-w)\nonumber\\
&\qquad\times[\psi^{(1)}_{j}(u-u',v-v'-w)-\psi^{(1)}_{j}(x_I-u',y_J-v'-w)]\psi^{(2)}_k(w-w')dudv \ dw.\nonumber
 \end{align}
Similarly, for $\mathcal T_2$, since $j>k$, we point out that this is an analogous of operator of product type. We now decompose $\mathbb R^n\times \mathbb R^m$ into dyadic rectangles of the form $R= I\times J$ with $\ell(I) = 2^{-j-N}$ and $\ell(J) = 2^{-k-N}$, where $N$ is a large fixed positive integer.
\begin{eqnarray*}\label{main term T2}
\mathcal T_2 (f)(x,y)
&=& c_\alpha\sum_{ \substack{j,k\in\mathbb Z\\ j> k}}\psi_{j,k}\ast\psi_{j,k}\ast f(x,y)\\
&= & c_\alpha\sum_{ \substack{j,k\in\mathbb Z\\ j> k}} \sum_{I,J} \int_{I\times J}\psi_{j,k}(x-u,y-v)\, \psi_{j,k}\ast f(u,v)\, dudv\\
&= & c_\alpha\sum_{ \substack{j,k\in\mathbb Z\\ j> k}} \sum_{I,J}|R|\ \Big({1\over |R|} \int_{I\times J}\psi_{j,k}(x-u,y-v)\,dudv\Big)\ \psi_{j,k}\ast f(x_I,y_J)+\mathcal{R}_N^{(2)}(f)(x,y),
\end{eqnarray*}
where $x_I$ and $y_J$ are {\bf any fixed points} in $I$ and $J$, respectively, and
\begin{eqnarray*}\label{main term R2}
\mathcal{R}_N^{(2)}(f)(x,y)
&=&c_\alpha\sum_{ \substack{j,k\in\mathbb Z\\ j> k}}  \sum_{I,J} \int_{I\times J}  \psi_{j,k}(x-u,y-v)\,    (\psi_{j,k} \ast f(u,v)-\psi_{j,k} \ast f(x_I,y_J))\,dudv .\nonumber
\end{eqnarray*}
Then similarly we can write
\begin{align*}
\mathcal{R}_N^{(2)}(f)(x,y)&=  \int_{\mathbb R^m} (\mathcal{R}_N^{(2)})^{\sharp}(f^{\sharp})(x,y-z,z) dz,
\end{align*}
where the kernel of $(\mathcal{R}_N^{(2)})^{\sharp}$ is given by
\begin{align}\label{R2sharp}
&(\mathcal{R}_N^{(2)})^{\sharp}(x,y,z;u',v',w')\\
&=\int_{\Bbb R^m} c_\alpha\sum_{ \substack{j,k\in\mathbb Z\\ j> k}} \sum_{I,J}   \int_{I\times J}  \psi^{(1)}_{j}(x-u,y-v-z+w)\psi^{(2)}_k(z-w)\nonumber\\
&\qquad\times[\psi^{(1)}_{j}(u-u',v-v'-w)-\psi^{(1)}_{j}(x_I-u',y_J-v'-w)]\psi^{(2)}_k(w-w')dudv \ dw.\nonumber
 \end{align}

We need an estimate on $(\mathcal{R}_N^{(1)})^{\sharp}$ and $(\mathcal{R}_N^{(2)})^{\sharp}$ that is contained in the following lemma:
\begin{lemma}
\label{l:claim}
Suppose  $f^\sharp(x,y,z)=f_1(x,y)f_2(z)$, where  $f_1\in {\mathcal M}_{n+m}(\beta,\gamma,r_1,r_2,x_0,y_0)$ and $f_2\in {\mathcal M}_{m}(\beta,\gamma,r_3,z_0)$, then for $i=1,2$,
$(\mathcal{R}_N^{(i)})^{\sharp}(f^\sharp)$ is in ${\mathcal M}_{(n+m)\times m}(\beta,\gamma,r_1,r_2,x_0,y_0,z_0)$ with
\begin{eqnarray}\label{productboundedness}
||(\mathcal{R}_N^{(i)})^{\sharp}(f^\sharp)||_{{\mathcal M}_{(n+m)\times m}(\beta,\gamma,r_1,r_2,x_0,y_0,z_0)}
\le C2^{-N}||f^\sharp||_{{\mathcal M}_{(n+m)\times m}(\beta,\gamma,r_1,r_2,x_0,y_0,z_0)},
\end{eqnarray}
where $N$ is the large fixed positive integer as in Theorem \ref{dcrf} and $C$ is an absolute constant depending only dimensions $n,m$.
\end{lemma}

Assuming Lemma \ref{l:claim} for the moment, it implies that for $f(x,y)=\int_{\mathbb R^m}f_1(x,y-z)f_2(z)dz$ with $f_1\in {\mathcal M}_{n+m}(\beta,\gamma,r_1,x_0,y_0)$ and $f_2\in {\mathcal M}_{m}(\beta,\gamma,r_2,z_0)$, for $i=1,2$,  the function
$\mathcal{R}_N^{(i)}(f)$ is in ${\mathcal M_{flag}}(\beta,\gamma,r_1,r_2,x_0,y_0)$ with
\begin{align*}
\|\mathcal{R}_N^{(i)}(f)\|_{\mathcal M_{flag}(\beta,\gamma,r_1,r_2,x_0,y_0)}\le C2^{-N}\|f\|_{\mathcal M_{flag}(\beta,\gamma,r_1,r_2,x_0,y_0)}.
\end{align*}
Hence, by noting that $C$ is an absolute constant depending only dimensions $n,m$ and by choosing $N$ large such that
\begin{align}\label{choice of N}
C2^{-N}<{1\over 8},
\end{align}
we obtain that for $i=1,2$,
\begin{align}\label{remainder R}
\|\mathcal{R}_N^{(i)}(f)\|_{\mathcal M_{flag}(\beta,\gamma,r_1,r_2,x_0,y_0)}\le {1\over 8}\|f\|_{\mathcal M_{flag}(\beta,\gamma,r_1,r_2,x_0,y_0)}.
\end{align}
Following the above approach, we can also establish the estimate for
$\mathcal R_\alpha(f)$:  there exists a small positive number $\alpha$ such that
\begin{align}\label{remainder Ralpha}
\|\mathcal{R}_\alpha(f)\|_{\mathcal M_{flag}(\beta,\gamma,r_1,r_2,x_0,y_0)}\le {1\over 8}\|f\|_{\mathcal M_{flag}(\beta,\gamma,r_1,r_2,x_0,y_0)}.
\end{align}
Thus, from the decomposition of $f$ as in \eqref{disc-Caldeorn-reproducing-formula1 pre} and the split of $\mathcal T_1$ and $\mathcal T_2$, we get that
\begin{eqnarray}\label{disc-Caldeorn-reproducing-formula2 pre}
f(x,y)&=&c_\alpha\sum_{ \substack{j,k\in\mathbb Z\\ j\leq k}} \sum_{\substack{I,J\\ \ell(I)=2^{-j-N}\\ \ell(J)=2^{-j-N} }}|R|\ \Big({1\over |R|} \int_{I\times J}\psi_{j,k}(x-u,y-v)\,dudv\Big)\ \psi_{j,k}\ast f(x_I,y_J)
\\
&&+ c_\alpha\sum_{ \substack{j,k\in\mathbb Z\\ j> k}} \sum_{\substack{I,J\\ \ell(I)=2^{-j-N}\\ \ell(J)=2^{-k-N} }}|R|\ \Big({1\over |R|} \int_{I\times J}\psi_{j,k}(x-u,y-v)\,dudv\Big)\ \psi_{j,k}\ast f(x_I,y_J)\nonumber\\
&&+ \mathcal{R}_N^{(1)}(f)(x,y)+\mathcal{R}_N^{(2)}(f)(x,y)+\mathcal R_\alpha (f)(x,y)\nonumber\\
&=:&\mathcal T(f)(x_1,x_2) + \mathcal{R}_N^{(1)}(f)(x,y)+\mathcal{R}_N^{(2)}(f)(x,y)+\mathcal R_\alpha (f)(x,y) \nonumber,
\end{eqnarray}
which implies that
$Id = \mathcal T+ \mathcal{R}_N^{(1)}+ \mathcal{R}_N^{(2)} + \mathcal R_\alpha$ with
$$\| \mathcal{R}_N^{(1)}+ \mathcal{R}_N^{(2)}+\mathcal{R}_\alpha\|_{\mathcal M_{flag}(\beta,\gamma,r_1,r_2,x_0,y_0)\to \mathcal M_{flag}(\beta,\gamma,r_1,r_2,x_0,y_0)}<{1\over2},$$
and hence $\mathcal T$ is invertible with $$\| \mathcal{T}^{-1}\|_{\mathcal M_{flag}(\beta,\gamma,r_1,r_2,x_0,y_0)\to \mathcal M_{flag}(\beta,\gamma,r_1,r_2,x_0,y_0)}\leq C.$$

Next, by noting that $\Big({1\over |R|} \int_{I\times J}\psi_{j,k}(x-u,y-v)\,dudv\Big)$ is in $ {\mathcal M}_{flag}(\beta,\gamma,2^{-j},2^{-k},x_I,y_J)$, we have that
$\mathcal T^{-1}  \Big({1\over |R|} \int_{I\times J}\psi_{j,k}(\cdot-u,\cdot-v)\,dudv\Big)(x,y)$ is also in $ {\mathcal M}_{flag}(\beta,\gamma,2^{-j},2^{-k},x_I,y_J)$, and we denote it by ${\phi}_{j,k}(x,y, x_I,y_J)$. Hence,
 we get
\begin{eqnarray}\label{disc-Caldeorn-reproducing-formula3 pre}
f(x,y)= \mathcal T^{-1} \cdot \mathcal T f(x,y)
&=&c_\alpha\sum_{ \substack{j,k\in\mathbb Z\\ j\leq k}} \sum_{\substack{I,J\\ \ell(I)=2^{-j-N}\\ \ell(J)=2^{-j-N} }}|R|\ {\phi}_{j,k}(x,y, x_I,y_J)\ \psi_{j,k}\ast f(x_I,y_J)
\nonumber\\
&&+ c_\alpha\sum_{ \substack{j,k\in\mathbb Z\\ j> k}} \sum_{\substack{I,J\\ \ell(I)=2^{-k-N}\\ \ell(J)=2^{-j-N} }}|R|\ {\phi}_{j,k}(x,y, x_I,y_J)\ \psi_{j,k}\ast f(x_I,y_J).\nonumber
\end{eqnarray}

 This then gives the proof of Theorem \ref{dcrf} (assuming Lemma \ref{l:claim}).
\end{proof}

We now turn to demonstrating Lemma \ref{l:claim}. To do this, we introduce the following definition and key estimates.

\begin{definition}\label{operator1}
Let $T$ be a
bounded linear operator on $L^2(\mathbb R^n)$ associated with a
kernel $K(x,y)$ defined on $\{(x,y)\in\mathbb R^n\times \mathbb R^n:\ x\not=y\}$, given initially by
$$ Tf(x)=\int_{\mathbb R^n} K(x,y)f(y)dy,\qquad x\not\in {\rm supp} f$$
for $f\in C^\infty(\mathbb R^n)$ with compact support,
where $K(x,y)$ satisfies the following conditions:
there exists a constant $C>0$ such that for all $x\not=y$,
\begin{enumerate}
\item[\textnormal{(i)}] $|K(x,y)|\le C|x-y|^{-n},$
\item[\textnormal{(ii)}] $|K(x,y)-K(x',y)|\le C |x-x'||x-y|^{-n-1}$
             \qquad if $|x-x'|\le |x-y|/2,$
\item[\textnormal{(iii)}] $|K(x,y)-K(x,y')|\le C |y-y'||x-y|^{-n-1}$
             \qquad if $|y-y'|\le |x-y|/2,$
\item[\textnormal{(iv)}] $|K(x,y)-K(x',y)-K(x,y')+K(x',y')|\le C|x-x'||y-y'||x-y|^{-n-2}$
   \item[] if $|x-x'|\le |x-y|/2$ and $|y-y'|\le |x-y|/2$.
\end{enumerate}
We denote by $\| K\|_{\mathbb R^n}$ the smallest constant $C$ that satisfies ${\rm (i)}$--${\rm (iv)}$ above. The operator norm of $T$ is defined by
$|||T|||:=\|T\|_{L^2(\mathbb R^n)\mapsto L^2(\mathbb R^n)}+ \| K\|_{\mathbb R^n}.$ Here we use $n$ to denote arbitrary positive integer.
\end{definition}
We would like to point out that the classical Calder\'on--Zygmund kernel $K(x,y)$ only needs to satisfy the conditions (i), (ii) and (iii).
For our purpose, namely the boundedness of operators on test function space, condition (iv) is required, see \cite[Chapter 2, Theorem 2.4]{DH} for the classical one parameter case. More precisely, we have the following:
\begin{lemma}\label{boundedness1}
Suppose that $T$ is an operator as in Definition  \ref{operator1} and $T(1)=T^*(1)=0.$ Then $T$ is bounded on the test function space
$\mathcal M_n(\alpha,\beta, r, x_0)$ for $\alpha, \beta\in (0,1), r>0$ and $x_0\in \mathbb R^n.$ Moreover, there exists a constant $C$ such that
$$\|T(f)\|_{\mathcal M_n(\alpha,\beta, r, x_0)}\leq C|||T|||\ \|f\|_{\mathcal M_n(\alpha,\beta, r, x_0)}.$$
\end{lemma}
See \cite{H1} for the definition of $T(1)=T^*(1)=0$ and the proof of Lemma \ref{boundedness1}.  We now define the product operator as follows.
\begin{definition}\label{product CZ}
The operator $T$ is said to be a product operator on $\Bbb R^{n+m}\times \Bbb R^m$ if $T$ is a bounded linear operator on $L^2(\Bbb R^{n+m}\times \Bbb R^m)$ associated with a kernel $K(x,y,z,u,v,w)$ defined on $\{((x,y,z),(u,v,w))\in (\Bbb R^{n+m}\times \Bbb R^m) \times (\Bbb R^{n+m}\times \Bbb R^m): (x,y)\not=(u,v), z\not=w\}$,
and
$$Tf(x,y,z)=\int_{\Bbb R^{n+m}\times\Bbb R^m} K(x,y,z,u,v,w)f(u,v,w)dudvdw,\quad (x,y,z)\not\in {\rm supp} f$$
for $f\in C^\infty(\Bbb R^{n+m}\times \Bbb R^m)$ with compact support,
where $K(x,y,z,u,v,w)$ satisfies the following conditions: for all $(x,y)\not=(u,v)$ and $z\not=w$
\begin{enumerate}
\item[\textnormal{(1)}] $\|K(\cdot,\cdot,z,\cdot,\cdot,w)\|_{\Bbb R^{n+m}}\le C|z-w|^{-m},$
\item[\textnormal{(2)}] $\|K(x,y,\cdot,u,v,\cdot)\|_{\Bbb R^{m}}\le C(|x-u|+|y-v|)^{-(n+m)},$
\item[\textnormal{(3)}] $\|K(\cdot,\cdot,z,\cdot,\cdot,w)-K(\cdot,\cdot,z',\cdot,\cdot,w)\|_{\Bbb R^{n+m}}\le C\frac{\displaystyle |z-z'|}{\displaystyle |z-w|^{m+1}}$\quad for $|z-z'|\le |z-w|/2,$
\item[\textnormal{(4)}] $\|K(\cdot,\cdot,z,\cdot,\cdot,w)-K(\cdot,\cdot,z,\cdot,\cdot,w')\|_{\Bbb R^{n+m}}\le C\frac{\displaystyle |w-w'|}{\displaystyle |z-w|^{m+1}}$\quad for $|w-w'|\le |z-w|/2,$
\item[\textnormal{(5)}] $\|K(\cdot,\cdot,z,\cdot,\cdot,w)-K(\cdot,\cdot,z',\cdot,\cdot,w)-K(\cdot,\cdot,z,\cdot,\cdot,w)+K(\cdot,\cdot,z',\cdot,\cdot,w')\|_{\Bbb R^{n+m}}\\[4pt]
 \le C\frac{\displaystyle |z-z'||w-w'|}{\displaystyle  |z-w|^{m+2}}\\[3pt] $ \quad for $|z-z'|\le |z-w|/2$ and $|w-w'|\le |z-w|/2,$
\item[\textnormal{(6)}] $\|K(x,y,\cdot,u,v,\cdot)-K(x',y',\cdot,u,v,\cdot)\|_{\Bbb R^{m}}\le C\frac{\displaystyle |x-x'|+|y-y'|}{\displaystyle (|x-u|+|y-v|)^{n+m+1}}\\
           $\quad for $|x-x'|+|y-y'|\le (|x-u|+|y-v|)/2,$
\item[\textnormal{(7)}] $\|K(x,y,\cdot,u,v,\cdot)-K(x,y,\cdot,u',v',\cdot)\|_{\Bbb R^{m}}\le C\frac{\displaystyle |u-u'|+|v-v'|}{\displaystyle (|x-u|+|y-v|)^{n+m+1}}\\
           $\quad for $|u-u'|+|v-v'|\le (|x-u|+|y-v|)/2,$
\item[\textnormal{(8)}] $\|K(x,y,\cdot,u,v,\cdot)-K(x',y',\cdot,u,v,\cdot)-K(x,y,\cdot,u',v',\cdot)+K(x',y',\cdot,u',v',\cdot)\|_{\Bbb R^{m}}\\[4pt]
\le C\frac{\displaystyle (|x-x'|+|y-y'|)(|u-u'+|v-v'|)}{\displaystyle (|x-u|+|y-v|)^{n+m+2}}\\[3pt]
           $\quad for $|x-x'|+|y-y'|\le (|x-u|+|y-v|)/2$ and $|u-u'|+|v-v'|\le (|x-u|+|y-v|)/2.$
\end{enumerate}
We denote by $\|K\|$ the smallest constant $C$ that satisfies ${\rm (1)}$---${\rm (8)}$ above.  The operator norm of $T$ is defined by
$|||T|||=\|T\|_{L^2(\mathbb R^{n+m}\times\mathbb R^m)\mapsto L^2(\mathbb R^{n+m}\times\mathbb R^m)}+ \|K\|.$
\end{definition}

Before stating the result, we first recall the cancellation condition from Journ\'e for the product singular integral $T$:  $$T_1(1)=T_2(1)=T_1^*(1)=T_2^*(1)=0$$
(see \cite[Section 3, page 64--65]{J1} for definitions).
Due to the length of this definition, we do not repeat it here.

\begin{prop}\label{prop1}
Let $\beta,\gamma\in(0,1)$ and $r_1,r_2,r_3>0$, $x_0$ be fixed point in $\mathbb R^n$, $y_0$ and $z_0$ be fixed points in $\mathbb R^m$.
If $T$ is a product operator as in Definition \ref{product CZ} and $T$ satisfies the product type cancellation condition\ $T_1(1)=T_2(1)=T_1^*(1)=T_2^*(1)=0$, then
$$\|Tf\|_{{\mathcal M}_{(n+m)\times m}(\beta,\gamma,r_1,r_2,x_0,y_0,z_0)}\le C|||T|||\ \|f\|_{{\mathcal M}_{(n+m)\times m}(\beta,\gamma,r_1,r_2,x_0,y_0,z_0)}$$
for all $f(x,y,z)=f_1(x,y)f_2(z)$ with $f_1\in {\mathcal M}_{n+m}(\beta,\gamma,r_1,x_0,y_0)$ and $f_2\in {\mathcal M}_{m}(\beta,\gamma,r_2,z_0)$.
\end{prop}
\begin{remark}
Indeed, Proposition \ref{prop1} holds for all $f\in {\mathcal M}_{n+m,m}(\beta,\gamma,r_1,r_2,x_0,y_0,z_0).$ The proof for such a result is a little bit complicated. However, Proposition \ref{prop1} is enough to provide a proof for Lemma \ref{l:claim}, which was one of the main ingredients in Theorem \ref{dcrf}.
\end{remark}
\begin{proof}[Proof of Proposition \ref{prop1}]
 Before jumping into the proof, we would like to highlight that
the cancellation condition $T_1(1)=T_2(1)=T_1^*(1)=T_2^*(1)=0$ plays a key role here, without which, the argument
in Proposition \ref{prop1} is not true.

Suppose that $f(x,y,z)=f_1(x,y)f_2(z)$ with $$||f_1||_{{\mathcal M}_{n+m}(\beta,\gamma,r_1,x_0,y_0)}=||f_2||_{{\mathcal M}_{m}(\beta,\gamma,r_2,z_0)}=1.$$ We write
\begin{align*}
Tf(x,y,z)&=\int_{\Bbb R^{n+m+m}}  K(x,y,z,u,v,w)f(u,v,w)dudvdw\\
&= \int_{\Bbb R^{m}} \int_{\Bbb R^{n+m}}  K(x,y,z,u,v,w)f_1(u,v)dudvf_2(w)dw\\
&= \int_{\Bbb R^{m}}  S(z,w)f_2(w)dw,
\end{align*}
where $x,y$ and $f_1$ are fixed, and $S(z,w)=\int_{\Bbb R^{n+m}}  K(x,y,z,u,v,w)f_1(u,v)dudv.$

We claim that for fixed $x\in \Bbb R^n$, $y\in \Bbb R^m$, $S(g)(z)=\int S(z,w)g(w)dw$
is an operator bounded on $\mathcal M_m(\beta,\gamma,r_2,z_0)$ with the kernel $S(z,w)$ satisfying Lemma \ref{boundedness1}. Moreover,
 \begin{enumerate}
\item[(1)] $|S(z,w)|\le C|z-w|^{-m}|||T|||\frac {\displaystyle r_1^\gamma}{\displaystyle (r_1+|x-x_0|+|y-y_0|)^{n+m+\gamma}}$,
\item[(2)] $|S(z,w)-S(z',w)|\le C\frac{\displaystyle |z-z'|}{\displaystyle |z-w|^{m+1}}|||T|||\frac {\displaystyle r_1^\gamma}{\displaystyle (r_1+|x-x_0|+|y-y_0|)^{n+m+\gamma}}$
\item[] for $|z-z'|\le |z-w|/2$,
\item[(3)]  $|S(z,w)-S(z,w')|\le C\frac{\displaystyle |w-w'|}{\displaystyle |z-w|^{m+1}}|||T|||\frac {\displaystyle r_1^\gamma}{\displaystyle (r_1+|x-x_0|+|y-y_0|)^{n+m+\gamma}}$
\item[] for $|w-w'|\le |z-w|/2$,
\item[(4)]  $|S(z,w)-S(z',w)-S(z,w')+S(z',w')|$
\item[]\quad\quad $\le C\frac{\displaystyle |z-z'||w-w'|}{\displaystyle |z-w|^{m+2}}|||T|||\frac {\displaystyle r_1^\gamma}{\displaystyle (r_1+|x-x_0|+|y-y_0|)^{n+m+\gamma}}$
\item[]  for $|z-z'|,|w-w'|\le |z-w|/2$,
\item[(5)]  $S(1)=S^*(1)=0$.
\end{enumerate}
The proof of the claim follows from Lemma \ref{boundedness1}. Indeed, for fixed $z, w\in \mathbb R^m,$ the operator $L$ with the kernel $K(x, y, z,u,v,w)$ is given by
$$L(f_1)(x,y,z,w)=\int_{\Bbb R^{n+m}}  K(x,y,z,u,v,w)f_1(u,v)dudv.$$
By the condition (1) in Definition \ref{product CZ} together with Lemma \ref{boundedness1}, the operator $L$
is bounded on $\mathcal M_{\mathbb R^{n+m}}(\beta,\gamma,r_1,x_0,y_0).$ Thus,
\begin{eqnarray*}
|L(f_1)(x,y,z,w)|\leq C|||T|||\ |z-w|^{-m}\frac {r_1^\gamma}{(r_1+|x-x_0|+|y-y_0|)^{n+m+\gamma}},
\end{eqnarray*}
which implies that $S(z,w)$ satisfies estimate (1) in the above claim, that is,
\begin{eqnarray*}
|S(z,w)|\leq C |z-w|^{-m}|||T|||\frac {r_1^\gamma}{(r_1+|x-x_0|+|y-y_0|)^{n+m+\gamma}}.
\end{eqnarray*}
Similarly, applying conditions (3) and (4) in Definition \ref{product CZ} together with Lemma \ref{boundedness1}, respectively, we conclude that $S(z,w)$ satisfies the estimates in (2) and (3) in the above claim, respectively. The condition (5) in Definition \ref{product CZ} together with Lemma \ref{boundedness1} yields the estimate (5) in the above claim for $S(z,w).$

Based on the estimates on $S(z,w),$ the kernel of $S,$ applying Lemma \ref{boundedness1} gives that the operator $S$ is bounded on ${\mathcal M}_{\mathbb R^m}(\beta,\gamma,r_2,z_0)$ and hence
$$|Tf(x,y,z)|=|S(f_2)(z)|\leq C|||T|||\frac {r_1^\gamma}{(r_1+|x-x_0|+|y-y_0|)^{n+m+\gamma}}\frac {r_2^\gamma}{(r_2+|z-z_0|)^{m+\gamma}}$$
and
\begin{align*}
&|Tf(x,y,z)-T(x,y,z')|=|S(f_2)(z)-S(f_2)(z')|\\
&\leq C|||T|||\frac {r_1^\gamma}{(r_1+|x-x_0|+|y-y_0|)^{n+m+\gamma}}\Big(\frac{|z-z'|}{r_2+|z-z_0|}\Big)^\beta\frac {r_2^\gamma}{(r_2+|z-z_0|)^{m+\gamma}}
\end{align*}
for $|z-z'|\le \frac{r_2+|z-z_0|}{2}.$

Similarly, if write
\begin{align*}
Tf(x,y,z)&=\int_{\Bbb R^{n+m}\times \Bbb R^m}  K(x,y,z,u,v,w)f(u,v,w)dudvdw\\
&= \int_{\Bbb R^{n+m}}  \int_{\Bbb R^{m}}  K(x,y,z,u,v,w)f_2(w)dwf_1(u,v)dudv\\
&= \int_{\Bbb R^{n+m}}  R(x,y,z,u,v)f_1(u,v)dudv,
\end{align*}
where $z$ and $f_2$ are fixed, and $R(x,y,z,u,v)=\int_{\Bbb R^{m}}  K(x,y,z,u,v,w)f_2(w)dw,$ then applying the same proof implies that the operator $R$ is bounded on
$\mathcal M_{n+m}(\beta,\gamma,r_1,x_0,y_0)$ and moreover,
\begin{align*}
&|Tf(x,y,z)-Tf(x',y',z)|\\
&=|R(f_1)(x,y)-R(f_1)(x',y')|\\
&\leq C|||T|||\Big(\frac{|x-x'|+|y-y'|}{r_1+|x-x_0|+|y-y_0|}\Big)^\beta\frac {r_1^\gamma}{(r_1+|x-x_0|+|y-y_0|)^{n+m+\gamma}}\frac {r_2^\gamma}{(r_2+|z-z_0|)^{m+\gamma}}
\end{align*}
for $|x-x'|+|y-y'|\leq (r_1+|x-x_0|+|y-y_0|)/2.$

It remains to show the following estimate:
\begin{align*}
&|Tf(x,y,z)-Tf(x',y',z)- Tf(x,y,z')+Tf(x',y',z')|\\
&\leq C|||T|||\Big(\frac{|x-x'|+|y-y'|}{r_1+|x-x_0|+|y-y_0|}\Big)^\beta\\
&\quad\times\Big(\frac{|z-z'|}{r_2+|z-z_0|}\Big)^\beta\frac {r_1^\gamma}{(r_1+|x-x_0|+|y-y_0|)^{n+m+\gamma}}\frac {r_2^\gamma}{(r_2+|z-z_0|)^{m+\gamma}}
\end{align*}
for $|x-x'|+|y-y'|\leq (
r_1+|x-x_0|+|y-y_0|)/2$ and $|z-z'|\leq (r_2+|z-z_0|)/2.$
To do this, write
\begin{align*}
&Tf(x,y,z)-Tf(x,y,z')\\
&=\int_{\Bbb R^{n+m+m}}  [K(x,y,z,u,v,w)-K(x,y,z',u,v,w)]f(u,v,w)dudvdw\\
&= \int_{\Bbb R^{n+m}}  \int_{\Bbb R^{m}}  [K(x,y,z,u,v,w)-K(x,y,z',u,v,w)]f_2(w)dwf_1(u,v)dudv\\
&= \int_{\Bbb R^{n+m}}  H(x,y,z,z',u,v)f_1(u,v)dudv\\
&=H(f_1)(x,y,z,z'),
\end{align*}
where $z,z'$ and $f_2$ are fixed, and $$H(x,y,z,z',u,v):=\int_{\Bbb R^{m}}  [K(x,y,z,u,v,w)-K(x,y,z',u,v,w)]f_2(w)dw.$$

We claim that the operator $H$ with the kernel $H(x,y,z,z',u,v)$ defined above is bounded on $\mathcal M_{n+m}(\beta,\gamma,r_1,x_0,y_0)$ and moreover,
\begin{align*}
&|Tf(x,y,z)-Tf(x',y',z)- Tf(x,y,z')+Tf(x',y',z')|\\
&=|H(f_1)(x,y,z,z')-H(f_1)(x',y',z,z')|\\
&\leq C|||T|||\Big(\frac{|x-x'|+|y-y'|}{r_1+|x-x_0|+|y-y_0|}\Big)^\beta
\Big(\frac{|z-z'|}{r_2+|z-z_0|}\Big)^\beta\\
&\qquad\times\frac {r_1^\gamma}{(r_1+|x-x_0|+|y-y_0|)^{n+m+\gamma}}\frac {r_2^\gamma}{(r_2+|z-z_0|)^{m+\gamma}}
\end{align*}
for $|x-x'|+|y-y'|\leq (r_1+|x-x_0|+|y-y_0|)/2$ and $|z-z'|\leq (r_2+|z-z_0|)/2.$

To see the claim, note first that by condition (2) in Definition \ref{product CZ} together with Lemma \ref{boundedness1}, for fixed $x,y,u$ and $v,$ the operator $$\int_{\Bbb R^{m}}  K(x,y,z,u,v,w)f_2(w)dw$$ is bounded on $\mathcal M_{ m}(\beta,\gamma,r_2,z_0)$ and hence, for $|z-z'|\leq (r_2+|z-z_0|)/2,$
\begin{align*}
&\bigg|\int_{\Bbb R^{m}}  [K(x,y,z,u,v,w)-K(x,y,z',u,v,w)]f_2(w)dw\bigg|\\
&\leq C(|x-u|+|y-v|)^{-(n+m)}\Big(\frac{|z-z'|}{r_2+|z-z_0|}\Big)^\beta\frac {r_2^\gamma}{(r_2+|z-z_0|)^{m+\gamma}},
\end{align*}
which implies that for $|z-z'|\leq (r_2+|z-z_0|)/2,$
$$|H(x,y,z,z',u,v)|\leq C(|x-u|+|y-v|)^{-(n+m)}\Big(\frac{|z-z'|}{r_2+|z-z_0|}\Big)^\beta\frac {r_2^\gamma}{(r_2+|z-z_0|)^{m+\gamma}}.$$
Write
\begin{align*}
&H(x,y,z,z',u,v)-H(x',y',z,z',u,v)\\
&=\int_{\Bbb R^{m}}  [K(x,y,z,u,v,w)-K(x',y',z,u,v,w)]f_2(w)dw  \\
&\qquad -\int_{\Bbb R^{m}}  [K(x,y,z',u,v,w)-K(x',y',z',u,v,w)]f_2(w)dw.
\end{align*}
By condition (6) in Definition \ref{product CZ} together with Lemma \ref{boundedness1}, for fixed $x,y,x',y',u$ and $v$ with $|x-x'|+|y-y'|\leq (|x-u|+|y-v|)/2,$ the operator
$$\int_{\Bbb R^{m}}  [K(x,y,z,u,v,w)-K(x',y',z,u,v,w)]f_2(w)dw$$ is bounded on $\mathcal M_{m}(\beta,\gamma,r_2,z_0)$
and hence, for $|z-z'|\leq (r_2+|z-z_0|)/2$ and $|x-x'|+|y-y'|\leq (|x-u|+|y-v|)/2,$
\begin{align*}
&|H(x,y,z,z',u,v)-H(x',y',z,z',u,v)|\\
&\leq C\frac{|x-x'|+|y-y'|}{(|x-u|+|y-v|)^{n+m+1}}\Big(\frac{|z-z'|}{r_2+|z-z_0|}\Big)^\beta\frac {r_2^\gamma}{(r_2+|z-z_0|)^{m+\gamma}}.
\end{align*}
Similarly, for $|z-z'|\leq (r_2+|z-z_0|)/2$ and $|u-u'|+|v-v'|\leq (|x-u|+|y-v|)/2,$
\begin{align*}
&|H(x,y,z,z',u,v)-H(x,y,z,z',u',v')|\\
&\leq C\frac{|u-u'|+|v-v'|}{(|x-u|+|y-v|)^{n+m+1}}\Big(\frac{|z-z'|}{r_2+|z-z_0|}\Big)^\beta\frac {r_2^\gamma}{(r_2+|z-z_0|)^{m+\gamma}}.
\end{align*}
Finally, we write
\begin{align*}
&H(x,y,z,z',u,v)-H(x',y',z,z',u,v)-H(x,y,z,z',u',v')+H(x',y',z,z',u',v')\\
&=\int_{\R^m} [K(x,y,z,u,v,w)-K(x',y',z,u,v,w)\\
&\quad\quad\quad\quad\quad\quad\quad\quad-K(x,y,z,u',v',w)+K(x',y',z,u',v',w)]f_2(w)dw\\
&\quad\quad\quad-\int_{\R^m} \Big[K(x,y,z',u,v,w)-K(x',y',z',u,v,w)\\
&\quad\quad\quad\quad\quad\quad\quad\quad-K(x,y,z',u',v',w)+K(x',y',z',u',v',w)\Big]f_2(w)dw.
\end{align*}
Applying condition (7) in Definition \ref{product CZ} together with Lemma \ref{boundedness1}, for fixed $x,y,x',y',u,v,u'$ and $v'$ with $|x-x'|+|y-y'|\leq (|x-u|+|y-v|)/2$ and $|u-u'|+|v-v'|\leq (|x-u|+|y-v|)/2,$ the operator $$\int_{\R^m} [K(x,y,z,u,v,w)-K(x',y',z,u,v,w)-K(x,y,z,u',v',w)+K(x',y',z,u',v',w)]f_2(w)dw$$ is bounded on on $\mathcal M_{m}(\beta,\gamma,r_2,z_0)$
and hence, for $|z-z'|\leq (r_2+|z-z_0|)/2,|x-x'|+|y-y'|\leq (|x-u|+|y-v|)/2$ and $|u-u'|+|v-v'|\leq (|x-u|+|y-v|)/2,$
\begin{align*}
&|H(x,y,z,z',u,v)-H(x',y',z,z',u,v)-H(x,y,z,z',u',v')+H(x',y',z,z',u',v')|\\
&\leq C\frac{(|x-x'|+|y-y'|)(|u-u'|+|v-v'|)}{(|x-u|+|y-v|)^{n+m+2}}\Big(\frac{|z-z'|}{r_2+|z-z_0|}\Big)^\beta\frac {r_2^\gamma}{(r_2+|z-z_0|)^{m+\gamma}}.
\end{align*}
Therefore, the operator
$$\int_{\R^m} [K(x,y,z,u,v,w)-K(x',y',z,u,v,w)-K(x,y,z,u',v',w)+K(x',y',z,u',v',w)]f_2(w)dw$$
is bounded on $\mathcal M_{n+m}(\beta,\gamma,r_1,x_0,y_0)$ and this yields the claim. The proof of Proposition \ref{prop1} is concluded.
\end{proof}

\begin{proof}[Proof of Lemma \ref{l:claim}]
Suppose  $f^\sharp(x,y,z)=f_1(x,y)f_2(z)$, where  $f_1\in {\mathcal M}_{n+m}(\beta,\gamma,r_1,r_2,x_0,y_0)$ and $f_2\in {\mathcal M}_{m}(\beta,\gamma,r_3,z_0)$.

We first consider the estimate for $(\mathcal{R}_N^{(1)})^{\sharp}(f^{\sharp})$.
To verify the norm of $(\mathcal{R}_N^{(1)})^{\sharp}(f^{\sharp})$ with respect to $\widetilde{\mathcal M}_{(n+m)\times m}(\beta,\gamma,r_1,r_2,x_0,y_0,z_0)$, we point out that the key ingredient  is to use
 the size and smoothness of the functions
 $\psi^{(1)}$ and $\psi^{(2)}$ appearing in the kernel of $(\mathcal{R}_N^{(1)})^{\sharp}$ given in
\eqref{R1sharp}.

Based on this observation, we point out that
 $[\psi^{(1)}_{j}(u-u',v-v'-w)-\psi^{(1)}_{j}(x_I-u',y_J-v'-w)] \sim 2^{-N} \psi^{(1)}_{j}(u-u',v-v'-w)$ in terms of the size and smoothness condition, since
 $(u,v)$ and $(x_I,y_J)$ are both in the cube $I\times J$ with side-length $\ell(I)=\ell(J)=2^{-j-N}$.
 As a consequence, we get
\begin{align*}
(\mathcal{R}_N^{(1)})^{\sharp}(x,y,z;u',v',w')
&\sim 2^{-N}\int_{\Bbb R^m} c_\alpha\sum_{ j} \sum_{I,J}   \int_{I\times J} \sum_{ \substack{k\in\mathbb Z\\ j\leq k}} \psi^{(1)}_{j}(x-u,y-v-z+w)\psi^{(2)}_k(z-w)\nonumber\\
&\qquad\times \psi^{(1)}_{j}(u-u',v-v'-w)\psi^{(2)}_k(w-w')dudv \ dw\\
&\sim 2^{-N} c_\alpha\sum_{ j}   \psi^{(1)}_{j}(x-u',y-v')  \sum_{ \substack{k\in\mathbb Z\\ j\leq k}}\int_{\Bbb R^m} \psi^{(2)}_k(z-w)\psi^{(2)}_k(w-w') dw\\
&\sim 2^{-N} c_\alpha\sum_{ j}   \psi^{(1)}_{j}(x-u',y-v')  \psi^{(2)}_j(z-w'),
 \end{align*}
where $\sim$ denotes the equivalence in terms of estimating the size and smoothness conditions of $(\mathcal{R}_N^{(1)})^{\sharp}(x,y,z;u',v',w')$. Note that this is in fact a one-parameter structure with respect to the variables $((x,y),z)$ and $((u',v'),w')$ in $\mathbb R^{n+m}\times\mathbb R^m$, estimating mainly in terms of cubes,  which is a special case of the tensor product setting. Moreover, $(\mathcal{R}_N^{(1)})^{\sharp}(x,y,z;u',v',w')$ satisfies the cancellation in terms of the tensor product setting $\mathbb R^{n+m}\times\mathbb R^m$ with respect to $(x,y)$, $z$, $(u',v')$ and $w'$, respectively. Hence, by applying
Proposition \ref{prop1} we obtain that \eqref{productboundedness} holds for $(\mathcal{R}_N^{(1)})^{\sharp}$.

Next, we consider the estimate for $(\mathcal{R}_N^{(2)})^{\sharp}(f^{\sharp})$.
Again, following similar estimate as above for $(\mathcal{R}_N^{(1)})^{\sharp}(f^{\sharp})$, we obtain that
\begin{align*}
&(\mathcal{R}_N^{(2)})^{\sharp}(x,y,z;u',v',w')\\
&\sim 2^{-N} c_\alpha\sum_{ j}  \sum_{ \substack{k\in\mathbb Z\\ j> k}} \psi^{(1)}_{j}(x-u',y-v')  \int_{\Bbb R^m} \psi^{(2)}_k(z-w)\psi^{(2)}_k(w-w') dw\\
&\sim 2^{-N} c_\alpha\sum_{ j} \sum_{ \substack{k\in\mathbb Z\\ j> k}}  \psi^{(1)}_{j}(x-u',y-v')  \psi^{(2)}_k(z-w').
 \end{align*}
Note that this is a typical tensor product structure with respect to the variables $((x,y),z)$ and $((u',v'),w')$ in $\mathbb R^{n+m}\times\mathbb R^m$, estimating mainly in terms of rectangles $R=I\times J$ with $\ell(I)\leq \ell(J)$. Moreover, $(\mathcal{R}_N^{(2)})^{\sharp}(x,y,z;u',v',w')$ satisfies the cancellation in terms of the tensor product setting $\mathbb R^{n+m}\times\mathbb R^m$ with respect to $(x,y)$, $z$, $(u',v')$ and $w'$, respectively. Hence, by applying
Proposition \ref{prop1} we obtain that \eqref{productboundedness} holds for $(\mathcal{R}_N^{(2)})^{\sharp}$.
The proof of  Lemma \ref{l:claim} is complete.
\end{proof}

Similar to Theorem \ref{dcrf}, one can also establish the following discrete reproducing formula via a modification of the process of the discretization.
\begin{theorem}\label{dcrf 1}
Let $\beta,\gamma,r_1,r_2,r_3>0$, $x_0$ be fixed point in $\mathbb R^n$, $y_0$ and $z_0$ be fixed points in $\mathbb R^m$. Let $\psi_{t,s}$ be the same as in \eqref{psi jk}. Then there exist functions $\phi_{j,k}(x,y, x_I,y_J)$ in the test function space  ${\mathcal M}_{flag}(\beta,\gamma,2^{-j},2^{-k},x_I,y_J)$ and a fixed large integer $N$ such that for $f(x,y)=\int_{\mathbb R^m}f_1(x,y-z)f_2(z)dz$ with $f_1\in {\mathcal M}_{n+m}(\beta,\gamma,r_1,r_2,x_0,y_0)$ and $f_2\in {\mathcal M}_{m}(\beta,\gamma,r_3,z_0),$
\begin{eqnarray}\label{discrf 1}
f(x,y)=\sum_j\sum_k\sum_I\sum_J|I||J|\phi_{j,k}(x,y,x_I,y_J)\int^{2^{-\alpha(k-1)}}_{2^{-\alpha k}}\!\!\!\int^{2^{-\alpha(j-1)}}_{2^{-\alpha j}}\psi_{t,s}\ast f(x_I,y_J)\frac{dt}{t}\frac{ds}{s},\ \ \ \ \
\end{eqnarray}
where the series converges in $L^2(\mathbb R^{n+m})$ and in ${\mathcal M}_{flag}(\beta,\gamma,r_1,r_2,x_0,y_0)$, $I\subset \mathbb{R}^n$ and $J\subset \mathbb{R}^m$ are dyadic cubes with side-lengths $\ell(I)=2^{-j-N}$ and $\ell(J)=2^{-(j\wedge k-N)}$, and $x_I$ and $y_J$ are {\bf any fixed points} in $I$ and $J,$ respectively.
\end{theorem}

\subsection{Flag Plancherel--P\'olya type inequalities}
\label{s:2.2}

Applying the discrete Calder\'on reproducing formula in \eqref{discrf 1} provides the following Plancherel--P\'olya type
inequalities.
\begin{theorem}\label{pp-inequality1}
Suppose $\psi_{t,s}$ is as in \eqref{psi jk}. Let $\alpha$ and $N$ be chosen the same as in Theorem \ref{dcrf}. Then for $f\in L^1(\mathbb R^{n+m}),$
\begin{eqnarray}
&&\bigg\Vert \bigg\lbrace \sum \limits_j\sum \limits_k\int^{2^{-\alpha(k-1)}}_{2^{-\alpha k}}\int^{2^{-\alpha(j-1)}}_{2^{-\alpha j}}\sum
\limits_J\sum \limits_I \sup_{\substack{u\in I\\v\in J}}\vert \psi _{t,s}*f(u,v)
\vert ^2\chi _I(x)\chi _J(y)\frac{dt}{t}\frac{ds}{s}\bigg\rbrace ^{{{1}\over{2}}}\bigg\Vert _{1}\nonumber\\
&\approx&
\bigg\Vert \bigg\lbrace \sum \limits_j\sum \limits_k\int^{2^{-\alpha(k-1)}}_{2^{-\alpha k}}\int^{2^{-\alpha(j-1)}}_{2^{-\alpha j}}\sum \limits_J
\sum \limits_I \inf_{\substack{u\in I\\v\in J}}\vert \psi _{t,s}*f(u,v)
\vert ^2\chi _I(x)\chi _J(y)\frac{dt}{t}\frac{ds}{s}\bigg\rbrace^{{{1}\over{2}}}\bigg\|_{1},\nonumber
\end{eqnarray}
where $I\subset \R^n, J\subset \R^m$ are dyadic cubes with side-lengths $\ell(I)=2^{-j-N}$ and $\ell(J)=2^{-(j\wedge k-N)}$  (the same as in Theorems \ref{dcrf} and \ref{dcrf 1}) and $\chi _I$ and $\chi _J$ are indicator
functions of $I$ and $J$, respectively.
\end{theorem}

\begin{proof}
For $f\in L^1(\mathbb R^{n+m}),$ by using Theorem \ref{dcrf 1}, we get that
\begin{align*}
&\psi _{t,s}*f(u,v)\\
&=\sum_{j'}\sum_{k'} \sum_{I'}\sum_{J'}|I'||J'|\psi _{t,s}*\phi_{j',k'}(u,v)\int^{2^{-\alpha(k'-1)}}_{2^{-\alpha k'}}\int^{2^{-\alpha(j'-1)}}_{2^{-\alpha j'}}\psi_{t',s'}\ast f(x_{I'},y_{J'})\frac{dt'}{t'}\frac{ds'}{s'}.
\end{align*}
For $2^{-\alpha j}<t<2^{-\alpha (j-1)}$ and $2^{-\alpha  k}<s<2^{-\alpha (k-1)}$, from Theorem \ref{dcrf 1}, we see that, as a function of $(x,y)$,
$\phi_{j',k'}(x,y, x_{I'},y_{J'})$ is in the test function space  ${\mathcal M}_{flag}(\beta,\gamma,2^{-j'},2^{-k'},x_{I'},y_{J'})$.
Then we have the following almost orthogonality estimate for $\phi_{j',k'}$ and $\psi_{t,s}$, see \cite[Lemma 6]{HLS}.
\begin{align*}
&|\psi _{t,s}*\big(\phi_{j',k'}(\cdot,\cdot, x_{I'},y_{J'}) \big)(u,v)|\\
&\le C2^{-|j-j'|\beta}2^{-|k-k'|\beta}\frac {2^{-(j\wedge j')\gamma}}{(2^{-(j\wedge j')}+|x_{I'}-u|)^{n+\gamma}}\frac {2^{-[(k\wedge k')\wedge(j\wedge j')])\gamma}}{(2^{-[(k\wedge k')\wedge(j\wedge j')]}+|x_{J'}-v|)^{m+\gamma}}.
\end{align*}

Observe that
\begin{align*}
&\sum_{I'}\sum_{J'}|I'||J'| \frac {2^{-(j\wedge j')\gamma}}{(2^{-(j\wedge j')}+|x_{I'}-u|)^{n+\gamma}}\frac {2^{-[(k\wedge k')\wedge(j\wedge j')])\gamma}}{(2^{-[(k\wedge k')\wedge(j\wedge j')]}+|x_{J'}-v|)^{m+\gamma}}\\
&\qquad\times \int^{2^{-\alpha(k'-1)}}_{2^{-\alpha k'}}\int^{2^{-\alpha(j'-1)}}_{2^{-\alpha j'}}\psi_{t',s'}\ast f(x_{I'},y_{J'})\frac{dt'}{t'}\frac{ds'}{s'}\\
&\le C\bigg\{M_s\bigg(\int^{2^{-\alpha(k'-1)}}_{2^{-\alpha k'}}\int^{2^{-\alpha(j'-1)}}_{2^{-\alpha j'}}\sum_{I'}\sum_{J'}\psi_{t',s'}\ast f(x_{I'},y_{J'})\chi_{I'}\chi_{J'}\frac{dt'}{t'}\frac{ds'}{s'}\bigg)^r(u,v)\bigg\}^{1/r},
\end{align*}
where $\frac {n+m}{n+m+\beta}<r<1$, and $M_s$ is the strong maximal function on $\Bbb R^n\times \Bbb R^m$ defined as
\begin{align}\label{strong maximal}
   M_s(f)(x_1,x_2):=\sup_{R:\ {\rm \ rectangles\ in\ } {\mathbb R}^{n}\times\mathbb{R}^{m},\ (x_1,x_2)\in R}{1\over |R|}\int_R|f(y_1,y_2)|dy_1dy_2.
\end{align}

See \cite[pages 147--148]{FJ} for the proof of the classical case. Note that $x_{I'}$ and $y_{J'}$ are arbitrary points in $I'$ and $J'$, respectively. We have that
\begin{align*}
&\sup_{\substack{u\in I\\v\in J}}|\psi _{t,s}*f(u,v)|\\
&\le C_N\sum_{j'}\sum_{k'}2^{-|j-j'|\beta}2^{-|k-k'|\beta} 2^{-j'n(1-\frac1r)}2^{[(j\wedge j')-j']n(1-{1\over r})}2^{-k'n(1-\frac1r)}2^{[(k\wedge k')-k']m(1-{1\over r})}  \\
&\quad\times\bigg\{M_s\bigg(
  \int^{2^{-\alpha(k'-1)}}_{2^{-\alpha k'}}\int^{2^{-\alpha(j'-1)}}_{2^{-\alpha j'}}\sum_{I'}\sum_{J'}\inf_{\substack{u'\in I'\\v'\in J'}}\psi_{t',s'}\ast f(u',v')\chi_{I'}\chi{J'}\frac{dt'}{t'}\frac{ds'}{s'}\bigg)^r(u,v)\bigg\}^{1/r}.
\end{align*}
Applying H\"older's inequality, together with the facts that
$$\sum \limits_j\sum \limits_k 2^{-|j-j'|\beta}2^{-|k-k'|\beta} 2^{-j'n(1-\frac1r)}2^{[(j\wedge j')-j']n(1-{1\over r})}2^{-k'n(1-\frac1r)}2^{[(k\wedge k')-k']m(1-{1\over r})} \leq C,$$
$$ \sum\limits_J\sum \limits_I \chi _I(x)\chi _J(y)\leq C$$
and
$$\int^{2^{-\alpha(k-1)}}_{2^{-\alpha k}}\int^{2^{-\alpha(j-1)}}_{2^{-\alpha j}}\frac{dt}{t}\frac{ds}{s}\leq C,$$
gives
\begin{align*}
&\bigg\Vert\bigg\lbrace \sum \limits_j\sum \limits_k\int^{2^{-\alpha(k-1)}}_{2^{-\alpha k}}\int^{2^{-\alpha(j-1)}}_{2^{-\alpha j}}\sum
\limits_J\sum \limits_I \sup_{\substack{u\in I\\v\in J}}\vert \psi _{t,s}*f(u,v)
\vert ^2\chi _I(x)\chi _J(y)\frac{dt}{t}\frac{ds}{s}\bigg\rbrace ^{{{1}\over{2}}}\bigg\Vert_{1}\\
&\lesssim \bigg\Vert\sum_{j'}\sum_{k'}\bigg\{M_s\bigg(
  \int^{2^{-\alpha(k'-1)}}_{2^{-\alpha k'}}\!\!\!\int^{2^{-\alpha(j'-1)}}_{2^{-\alpha j'}}\!\sum_{I'}\sum_{J'}\inf_{\substack{u'\in I'\\v'\in J'}}\psi_{t',s'}\ast f(u',v')\chi_{I'}\chi_{J'}\frac{dt'}{t'}\frac{ds'}{s'}\bigg)^r(u,v)\bigg\}^{1\over r}\bigg\Vert_{1}.
\end{align*}
By using the Fefferman--Stein vector-valued maximal function inequality on $L^{1\over r}(\mathbb R^{n+m})$, we get
\begin{eqnarray}
&&\bigg\Vert \bigg\lbrace \sum \limits_j\sum \limits_k\int^{2^{-\alpha(k-1)}}_{2^{-\alpha k}}\int^{2^{-\alpha(j-1)}}_{2^{-\alpha j}}\sum
\limits_J\sum \limits_I \sup_{\substack{u\in I\\v\in J}}\vert \psi _{t,s}*f(u,v)
\vert ^2\chi _I(x)\chi _J(y)\frac{dt}{t}\frac{ds}{s}\bigg\rbrace ^{{{1}\over{2}}}\bigg\Vert _{1}\nonumber\\
&\approx&
\bigg\Vert \bigg\lbrace \sum \limits_j\sum \limits_k\int^{2^{-k+1}}_{2^{-k}}\int^{2^{-j+1}}_{2^{-j}}\sum \limits_J
\sum \limits_I \inf_{\substack{u\in I\\v\in J}}\vert \psi _{t,s}*f(u,v)
\vert ^2\chi _I(x)\chi _J(y)\frac{dt}{t}\frac{ds}{s}\bigg\rbrace^{{{1}\over{2}}}\bigg\|_{1}.\nonumber
\end{eqnarray}
The proof is completed.
\end{proof}

We remark that applying a similar proof, for any fixed constant $C_0$ one can get the following Plancherel--P\'olya type
inequalities:
\begin{eqnarray}\label{P-P inequality}
&&\bigg\Vert \bigg\lbrace \sum \limits_j\sum \limits_k\int^{2^{-\alpha(k-1)}}_{2^{-\alpha k}}\int^{2^{-\alpha(j-1)}}_{2^{-\alpha j}}\sum
\limits_{J}\sum \limits_{I} \sup_{\substack{u\in {C_0I}\\v\in {C_0J}}}\vert \psi _{t,s}*f(u,v)
\vert ^2\chi _I(x)\chi _J(y)\frac{dt}{t}\frac{ds}{s}\bigg\rbrace ^{{{1}\over{2}}}\bigg\Vert _{1}\nonumber\\
&\approx&
\bigg\Vert \bigg\lbrace \sum \limits_j\sum \limits_k\int^{2^{-\alpha(k-1)}}_{2^{-\alpha k}}\int^{2^{-\alpha(j-1)}}_{2^{-\alpha j}}\sum \limits_J
\sum \limits_I \inf_{\substack{u\in I\\v\in J}}\vert \psi _{t,s}*f(u,v)
\vert ^2\chi _I(x)\chi _J(y)\frac{dt}{t}\frac{ds}{s}\bigg\rbrace^{{{1}\over{2}}}\bigg\|_{1},
\end{eqnarray}
where $C_0I\subset \Bbb R^n$ and $C_0J\subset \Bbb R^m,$ are cubes
with side-length $\ell(C_0I)=C_02^{-j-N}$ and
$\ell(C_0J)=C_0\ell(J)=2^{-(j\wedge k)-N},$ respectively.

\subsection{The equivalence of $\|g_F(f)\|_{1}$ and $ \|S_F(f)\|_{1}$ }
\label{s:2.3}
\subsubsection{The proof that $\|S_F(f)\|_{1}\lesssim \|g_F(f)\|_{1}$}

We write
\begin{eqnarray*}
\|S_F(f)(x,y)\|_1
&=&
\bigg\|\bigg\{\sum_{j,k}\sum_{I,J}\int^{2^{-\alpha(k-1)}}_{2^{-\alpha k}}\int^{2^{-\alpha(j-1)}}_{2^{-\alpha j}}
\int_{\mathbb{R}^{n}}\int_{\mathbb{R}^{m}}\chi_{t,s}(x-x_1,y-y_1) \\
&&\qquad\times |\psi_{t,s}\ast
f(x_1,y_1)|^2 \chi _I(x)\chi _J(y){dx_1dt \over t^{n+m+1}}{dy_1ds \over s^{m+1}} \bigg\}^{1/2}\bigg\|_1
\end{eqnarray*}
where $N$ is a fixed large integer as in the Plancherel--P\'olya type
inequalities and $I\subset \R^n$ and $J\subset \R^m$ are dyadic cubes
with side-length $\ell(I)=2^{-j-N}, \ell(J)=\ell(J)=2^{-(j\wedge k)-N},$ and $\chi _I$ and $\chi _J$ are indicator
functions of $I$ and $J$, respectively.

Observe that there exists a fixed constant $C_0$ such that for $2^{-\alpha j}\leq t\leq 2^{-\alpha(j-1)}, 2^{-\alpha k}\leq s\leq 2^{-\alpha(k-1)}$ and $x_1\in \mathbb R^n$ and $y_1\in \mathbb R^m,$
\begin{eqnarray*}
&&\chi_{t,s}(x-x_1,y-y_1)|\psi_{t,s}\ast
f(x_1,y_1)|^2\chi _I(x)\chi _J(y)\\
&&\leq \chi_{t,s}(x-x_1,y-y_1)\sup_{\substack{u\in C_0I\\ v\in C_0J}}|\psi_{t,s}\ast
f(u,v)|^2\chi _I(x)\chi _J(y).
\end{eqnarray*}
Therefore,
\begin{eqnarray*}
\|S_F(f)(x,y)\|_1\leq \bigg\|\bigg\{\sum_{j,k}\sum_{I,J}\int^{2^{-\alpha(k-1)}}_{2^{-\alpha k}}\int^{2^{-\alpha(j-1)}}_{2^{-\alpha j}}
\int_{\mathbb{R}^{n}}\int_{\mathbb{R}^{m}}\chi_{t,s}(x-x_1,y-y_1)\\
\times\sup_{\substack{u\in C_0I\\ v\in C_0J}}|\psi_{t,s}\ast
f(u,v)|^2\chi _I(x)\chi _J(y){dx_1dt \over t^{n+m+1}}{dy_1ds \over s^{m+1}} \bigg\}^{1/2}\bigg\|_1.
\end{eqnarray*}
Applying the estimate $\int_{\mathbb{R}^{n}}\int_{\mathbb{R}^{m}}\chi_{t,s}(x-x_1,y-y_1)dx_1dy_1\leq Ct^{n+m}s^m$ together with the
Plancherel--P\'olya type inequalities in (\ref{P-P inequality}) yields
\begin{eqnarray*}
&&\|S_F(f)(x,y)\|_1 \\
&&\leq \bigg\|\bigg\{\sum_{j,k}\sum_{I,J}\int^{2^{-\alpha(k-1)}}_{2^{-\alpha k}}\int^{2^{-\alpha(j-1)}}_{2^{-\alpha j}}\sup_{\substack{u\in C_0I\\ v\in C_0J}}
|\psi_{t,s}\ast
f(u,v)|^2\chi _I(x)\chi _J(y)\frac{dt}{t}\frac{ds}{s}\bigg\}^{1/2}\bigg\|_1\\
&&\lesssim\bigg\|\bigg\{\sum_{j,k}\sum_{I,J}\int^{2^{-\alpha (k-1)}}_{2^{-\alpha k}}\int^{2^{-\alpha(j-1)}}_{2^{-\alpha j}}|\psi_{t,s}\ast
f(x,y)|^2\chi _I(x)\chi _J(y)\frac{dt}{t}\frac{ds}{s}\bigg\}^{1/2}\bigg\|_1\\
&&=\|g_F(f)\|_1.
\end{eqnarray*}

\subsubsection{The proof that $\|g_F(f)\|_{1}\lesssim\|S_F(f)\|_{1}$}

\medskip
The proof of this part is similar. To see this, write
\begin{eqnarray*}
\|g_F(f)\|_1&&=
\bigg\|\bigg\{\int_0^{\infty}\int_0^{\infty}|\psi_{t,s}\ast
f(x,y)|^2\frac{dt}{t}\frac{ds}{s}\bigg\}^{1/2}\bigg\|_1\\
&&=\bigg\|\bigg\{\sum_{j,k}\sum_{I,J}\int^{2^{-\alpha(k-1)}}_{2^{-\alpha k}}\int^{2^{-\alpha(j-1)}}_{2^{-\alpha j}}|
\psi_{t,s}\ast
f(x,y)|^2\chi _I(x)\chi _J(y)\frac{dt}{t}\frac{ds}{s}\bigg\}^{1/2}\bigg\|_1.
\end{eqnarray*}
By the Plancherel--P\'olya type inequalities in (\ref{P-P inequality}), the last term above is dominated by
\begin{eqnarray*}
&&C\bigg\|\bigg\{\sum_{j,k}\sum_{I,J}\int^{2^{-\alpha(k-1)}}_{2^{-\alpha k}}\int^{2^{-\alpha(j-1)}}_{2^{-\alpha j}}
\inf_{\substack{u\in C_0I\\ v\in C_0J}}|\psi_{t,s}\ast f(u,v)|^2\chi _I(x)\chi _J(y){dt \over t}{ds \over s} \bigg\}^{1/2}\bigg\|_1\\
&&\leq C\bigg\|\bigg\{\sum_{j,k}\sum_{I,J}\int^{2^{-\alpha(k-1)}}_{2^{-\alpha k}}\int^{2^{-\alpha(j-1)}}_{2^{-\alpha j}}
\int_{\mathbb{R}^{n}}\int_{\mathbb{R}^{m}}\chi_{t,s}(x-x_1,y-y_1)\times\\
&&\hskip4cm\inf_{\substack{u\in C_0I\\ v\in C_0J}}|\psi_{t,s}\ast
f(u,v)|^2\chi _I(x)\chi _J(y){dx_1dt \over t^{n+m+1}}{dy_1ds \over s^{m+1}} \bigg\}^{1/2}\bigg\|_1\\
&&\leq C\|S_F(f)\|_1.
\end{eqnarray*}

\subsubsection{The two norms $\|g_F(f)\|_{1}$ and $\|S_F(f)\|_{1}$ are well-defined}

Based on the Plancherel--P\'olya type inequalities in (\ref{P-P inequality}) and the proof of $\|g_F(f)\|_{1}\approx\|S_F(f)\|_{1}$ as in Subsection 2.3.1 and 2.3.2, we also obtain that the two norms $\|g_F(f)\|_{1}$ and $\|S_F(f)\|_{1}$ are well-defined. To be more precise, we have the following
\begin{prop}\label{equivalence of psi varphi}
Suppose $\psi_{t,s}$ is as in \eqref{psi jk}, and $\varphi_{t,s}$ is as in \eqref{psi jk}.
Then we have the flag Littlewood--Paley area function and square functions defined via $\psi_{t,s}$, denoted by
$S_{F,\psi}(f)$ and $g_{F,\psi}(f)$, respectively. And we also  have the flag Littlewood--Paley area function and square functions defined via $\varphi_{t,s}$, denoted by
$S_{F,\varphi}(f)$ and $g_{F,\varphi}(f)$, respectively. Then we have
$$ \|g_{F,\psi}(f)\|_{1} \approx\|g_{F,\varphi}(f)\|_{1} \approx\|S_{F,\psi}(f)\|_{1} \approx\|S_{F,\varphi}(f)\|_{1},$$
where the implicit constants are independent of $\psi_{t,s}$ and $\varphi_{t,s}$.
\end{prop}

As a consequence, we see that the Hardy space $H^1_F(\mathbb R^n\times\mathbb R^m)$ given in Definition \ref{def-of-hardy-by-han} is well-defined, and can be characterized equivalently by $S_F(f)$,

\subsection{The estimate $\|S_F(f)\|_{1} \lesssim \|S_F(U)\|_{1}$}

The estimate $\|S_F(f)\|_1 \lesssim \|S_F(U)\|_1,$
follows from the same ideas in Section \ref{s:2.2} and \ref{s:2.3}.  More precisely, we first need to establish the following discrete Calder\'on reproducing formula.
For this purpose, let $\phi^{(1)}(x,y)\in
\mathcal {S}(\mathbb{R}^{n+m})$ be radial and satisfy the following
conditions:
\begin{enumerate}
\item[$(i)$] ${\rm supp}\ \phi^{(1)} \subset B(0,1)$, where $B(0,1)$ is the
unit ball in $\mathbb{R}^{n+m}$;
\item[$(ii)$]
$\int_{\mathbb{R}^{n+m}}x^{\alpha}y^{\beta}\phi^{(1)}(x,y)dxdy=0$,
where $|\alpha|+|\beta|\leq 2(n \vee m)$;
\item[$(iii)$] $\int_0^{\infty}e^{-u}\widehat{\phi^{(1)}}(u)du=-1$.
\end{enumerate}
In fact, $\phi^{(1)}(x,y)$ can be constructed as follows. Choose
$h^{(1)}\in \mathcal{S}(\mathbb{R}^{n+m})$, radial and supported in
$B(0,1)$. Let $k=4(n \vee m)$ and $\phi^{(1)}(x,y)= \Delta^k h^{(1)}(x,y)$.
Multiplying by an appropriate constant, we can see that such
$\phi^{(1)}(x,y)$ satisfies all the conditions above.

Similarly, choosing $h^{(2)}\in \mathcal{S}( \mathbb{R}^m)$, radial and
supported in $B(0,1)$ and $\phi^{(2)}(z)= \Delta^k h^{(2)}(z)$. Multiplying by an appropriate constant, we obtain that $\phi^{(2)}(z)\in
\mathcal{S}( \mathbb{R}^m)$, is radial and satisfies the following
conditions:
\begin{enumerate}
\item[$(i)$] ${\rm supp}\ \phi^{(2)} \subset B(0,1)$, where $B(0,1)$ is the
unit ball in $\mathbb{R}^m$;
\item[$(ii)$] $\int_{\mathbb{R}^m}z^{\gamma}\phi^{(2)}(z)dz=0$, where
$|\gamma|\leq  2(n \vee m)$;
\item[$(iii)$] $\int_0^{\infty}e^{-u}\widehat{\phi^{(2)}}(u)du=-1$.
\end{enumerate}

Let $\phi(x,y)=\phi^{(1)}*_{\mathbb R^m} \phi^{(2)}(x,y)$ and
$\phi_{t,s}(x,y)=\phi^{(1)}_t*_{\mathbb R^m} \phi^{(2)}_s(x,y)$. Repeating the same proof as in Theorem \ref{dcrf}, leads to
the following statement.

\begin{theorem}\label{prop-of-spcial-discrete-Calderon-reproducing-formula}
There exist $\phi_{j,k,I,J}(x,y)\in {\mathcal M}_{flag}(\beta,\gamma,2^{-j},2^{-k},x_I,y_J)$ and a fixed large integer $N$ such that
\begin{align*}
&f(x,y)\\
&=\sum_j\sum_k\sum_I\sum_J|I||J|\phi_{j,k,I,J}(x,y)\int^{2^{-\alpha(k-1)}}_{2^{-\alpha k}}\int^{2^{-v\alpha(j-1)}}_{2^{-\alpha j}}
\phi_{t,s}*\Big(ts{\partial \over {\partial t}}{\partial \over {\partial s}}P_{t,s}\Big)*f(x_I,y_J)\frac{dt}{t}\frac{ds}{s},
\end{align*}
where $I\subset \mathbb{R}^n$ and $J\subset \mathbb{R}^m$ are dyadic cubes with
side-lengths $\ell(I)=2^{-j-N}$ and $\ell(J)=2^{-(j\wedge k)-N}$, $x_I$ and $y_J$ are {\bf any fixed points} in $I$ and $J,$ respectively. Moreover,
for $f\in L^1(\mathbb R^{n+m})$ and $h\in {\mathcal M}_{flag}(\beta,\gamma,r_1,r_2,x_0,y_0)$,
\begin{equation*}
\begin{aligned}
&\langle f,h \rangle \\
&=\bigg\langle \sum_j\sum_k \sum_I\sum_J|I||J|\phi_{j,k,I,J}(\cdot,\cdot)\int^{2^{-\alpha(k-1)}}_{2^{-\alpha k}}\int^{2^{-\alpha(j-1)}}_{2^{-\alpha j}}\Big(ts{\partial
\over {\partial t}}{\partial \over {\partial
s}}P_{t,s}\Big)*f(x_I,y_J)\frac{dt}{t}\frac{ds}{s},h\bigg\rangle.
\end{aligned}
\end{equation*}
\end{theorem}
Applying the same proof as in Section \ref{s:2.2} gives the following.

\begin{theorem}\label{prop-of-special-P-P}
Let $f\in L^1(\mathbb{R}^{n+m})$, we have
\begin{eqnarray*}
\lefteqn{ \hspace{.5cm}  \bigg\|\sum_j\sum_k\sum_I\sum_J
\int_{2^{-\alpha j}}^{2^{-\alpha(j-1)}}\int_{2^{-\alpha k}}^{2^{-\alpha(k-1)}}\sup_{u\in I,v\in
J}|\psi_{t,s}\ast f(u,v)|^2 {dt\over t}{ds\over s}
\chi_I(x)\chi_J(y)\bigg\|_1         }  \\
&\approx&  \bigg\|\sum_j\sum_k\sum_I\sum_J
\int_{2^{-\alpha j}}^{2^{-\alpha(j-1)}}\int_{2^{-\alpha k}}^{2^{-\alpha (k-1)}}\inf_{u\in I,v\in
J}\Big|\Big(ts{\partial\over{\partial t}}{\partial\over{\partial
s}}P_{t,s}\Big)\ast f(u,v)\Big|^2 {dt\over t}{ds\over s}
\chi_I(x)\chi_J(y)\bigg\|_1,\nonumber
\end{eqnarray*}
where
$I\subset
\mathbb{R}^n$ and $J\subset \mathbb{R}^m$ are dyadic cubes with
side-lengths $\ell(I)=2^{-j-N}$ and $\ell(J)=2^{-(j\wedge k)-N}$, 
respectively.
\end{theorem}

The estimate $\|S_F(f)\|_{1} \lesssim \|S_F(U)\|_{1}$
then follows from Theorem \ref{prop-of-special-P-P} as in Section \ref{s:2.3}. We leave the details to the reader.

\section{Estimates of area function, maximal function and Riesz transform via flag Poisson integral technique}

In this section, we will show the following estimates: let $f\in L^1(\mathbb R^{n+m})$ and $U$ be the flag Poisson integral of $f$ as in \eqref{Poisson integral}, then
\begin{eqnarray*}
\|S_F(U)\|_{1}\lesssim \|U^{*}\|_{1}\lesssim \|U^{+}\|_{1}\lesssim \sum_{j=1}^{n+m}\sum_{k=1}^m\|R_{j,k}(f)\|_1+\|f\|_1.
\end{eqnarray*}

\subsection{The estimate $\|S_F(U)\|_1\lesssim \|U^{*}\|_1$}

We first introduce the following  maximal function associated with
the flag structure.
\begin{definition}\label{def flag HL}
For locally integrable function $f$ on $\mathbb{R}^{n+m}$, we define
the flag-type maximal function $M_F(f)$ by
\begin{eqnarray*}
M_F(f)(x,y):=\sup_{t,s>0,\ (x,y)\in R} \frac{1}{|R|}\int_R
|f(u,v)|dudv,
\end{eqnarray*}
where $R=I\times J$ run over all rectangles with sides parallel to the axes
and $\ell(I)=t$, $\ell(J)=t+s$.
\end{definition}

We now recall the lemma of K. Merryfield.
\begin{lemma}[\cite{M}]\label{lemma-of-Merryfield}
Let $\varphi\in C_0^{\infty}(\mathbb{R}^n)$ satisfy
\begin{enumerate}
\item[\textnormal{(1)}] $\varphi(-x)=\varphi(x)$;
\item[\textnormal{(2)}] ${\rm supp}\ \varphi\subset B_n(0,1) $, where $B_n(0,1)$ is the unit ball
in $\mathbb{R}^n$;
\item[\textnormal{(3)}] $\int_{\mathbb{R}^n}\varphi(x)dx=1$.
\end{enumerate}
Then there exists a function $\psi \in C_0^{\infty}(\mathbb{R}^n)$
that satisfies $supp\ \psi\subset B_n(0,1) $ and
$\int\limits_{\mathbb{R}^n}\psi(x)dx=0$,
such that
\begin{eqnarray*} \label{Merryfield inequality}
\lefteqn{   \int_{\mathbb{R}_+^{n+1}}|\nabla \mathcal P_t\ast f(x)|^2|g\ast
\varphi_t(x)|^2tdxdt    }\\
&\leq& C \int_{\mathbb{R}^n}f(x)^2g(x)^2dx+
\int_{\mathbb{R}_{+}^{n+1}}(\mathcal P_t\ast f(x))^2|g\ast \psi_t(x)|^2
\frac{dxdt}{t},\nonumber
\end{eqnarray*}
where $C$ is independent of $f$ and $g$. Here $\mathcal P(x)$ is the Poisson kernel on $\mathbb R^n$
with $\mathcal P(x)= c_n {1\over (1+|x|^2)^{n+1\over2}}$ and $\mathcal P_t(x) = t^{-n} \mathcal P(x/t)$.
\end{lemma}
We also point out that in the application in \cite{M},  the function $f$ is in $L^2(\mathbb R^n)\cap L^1(\mathbb R^n)$ and the function $g$ is a characteristic function of a measurable set in $\mathbb R^n$ such that
$\int_{\mathbb{R}_{+}^{n+1}}|g\ast \psi_t(x)|^2
\frac{dxdt}{t}$ is finite. For the specific functions $f$ and $g$ in application in \cite{M}, we see that the right-hand side of the above inequality is finite.

Now we establish a  Merryfield type lemma in this flag setting as follows.  Let $\varphi^{(1)}(x,y)\in
C_0^{\infty}(\mathbb{R}^{n+m})$ satisfy

\begin{enumerate}
\item[\textnormal{(1)}] $\varphi^{(1)}(-x,-y)=\varphi^{(1)}(x,y)$;
\item[\textnormal{(2)}] ${\rm supp}\ \varphi^{(1)}\subset B_{n+m}(0,1) $, where $B_{n+m}(0,1)$ is the unit
ball in $\mathbb{R}^{n+m}$;
\item[\textnormal{(3)}] $\int_{\mathbb{R}^{n+m}}\varphi(x,y)dxdy=1$.
\end{enumerate}
Let $\varphi^{(2)}(z)\in C_0^{\infty}(\mathbb{R}^{m})$ satisfy
the same conditions as in Lemma \ref{lemma-of-Merryfield}, and $\varphi(x,y)=\varphi^{(1)}\ast_{\mathbb R^m}\varphi^{(2)}(x,y)$. We, applying the projection of M\"uller, Ricci and Stein, define $\varphi_{t,s}(x,y)=\varphi^{(1)}_t\ast_{\mathbb R^m}\varphi^{(2)}_s(x,y)$.

Similarly, we can obtain two functions $\psi^{(1)}(x,y)$ and $\psi^{(2)}(z)$ such that
$\psi^{(1)} \in C_0^{\infty}(\mathbb{R}^{n+m})$
that satisfies $\operatorname{supp}\ \psi^{(1)} \subset B_{n+m}(0,1) $ and
$$\int_{\mathbb{R}^{n+m}} \psi^{(1)}(x,y) dxdy=0,$$ and
$\psi^{(2)} \in C_0^{\infty}(\mathbb{R}^m)$
that satisfies $\operatorname{supp}\ \psi^{(2)}\subset B_{m}(0,1) $ and
$$\int_{\mathbb{R}^m}\psi^{(2)}(z)dz=0.$$
Then we define
$\psi(x,y):=\psi^{(1)}\ast_{\mathbb R^m}\psi^{(2)}(x,y)$ and $\psi_{t,s}(x,y):=\psi^{(1)}_t\ast_{\mathbb R^m}\psi^{(2)}_s(x,y)$.  We arrive at the following technical lemma.
\begin{lemma}\label{lemma-of-Merryfield-flag}
Let all the notation be the same as above, and recall that $P_{t,s}$ is the flag Poisson kernel as defined in \eqref{Pts}. Then there exists a positive absolute constant $C$ such that
\begin{eqnarray}\label{Flag Merryfield inequality}
&&\hskip-.7cm\int_{\mathbb{R}_+^{n+1}}\int_{\mathbb{R}_+^{m+1}}\big|t\nabla^{(1)}s\nabla^{(2)}
P_{t,s}\ast f(x,y)\big|^2\big|g\ast
\varphi_{t,s}(x,y)\big|^2 \frac{dyds}{s}\frac{dxdt}{t}   \\
&\leq& C\bigg\{\int_{\mathbb{R}^n}\int_{\mathbb{R}^m}f(x,y)^2g(x,y)^2dxdy \nonumber\\
&&\ \ \ \ \ +
\int_{\mathbb{R}^n}\int_{\mathbb{R}_{+}^{m+1}}\big|P_s^{(2)}\ast_{\mathbb R^m}
f(x,y)\big|^2\big|\psi_s^{(2)}\ast_{\mathbb R^m} g(x,y)\big|^2
\frac{dyds}{s}dx\nonumber\\
&&\ \ \ \ \
+\int_{\mathbb{R}^m}\int_{\mathbb{R}_+^{n+1}}\big|P_t^{(1)}\ast
f(x,y)\big|^2
\big|\psi_t^{(1)}\ast g(x,y)\big|^2 \frac{dxdt}{t}dy\nonumber\\
&&\ \ \ \ \ +
\int_{\mathbb{R}_{+}^{n+1}}\int_{\mathbb{R}_{+}^{m+1}}|P_{t,s}\ast
f(x,y)|^2|\psi_{t,s}\ast g(x,y)|^2
\frac{dyds}{s}\frac{dxdt}{t}\bigg\},\nonumber
\end{eqnarray}
where $f\in L^1(R^{n+m})\cap  L^2(R^{n+m})$ with $\|U^*\|_1<\infty$ and $g$ is a characteristic function of a measurable set in $\mathbb R^{n+m} $ such that the integrals
$$ \int_{\mathbb{R}_{+}^{m+1}}\big|\psi_s^{(2)}\ast_{\mathbb R^m} g(x,y)\big|^2
\frac{dyds}{s},  \int_{\mathbb{R}_+^{n+1}}
\big|\psi_t^{(1)}\ast g(x,y)\big|^2 \frac{dxdt}{t}, \int_{\mathbb{R}_{+}^{n+1}}\int_{\mathbb{R}_{+}^{m+1}}|\psi_{t,s}\ast g(x,y)|^2
\frac{dyds}{s}\frac{dxdt}{t}$$
are all finite.
\end{lemma}

We point out that in the proof of $\|S_F(U)\|_1\lesssim \|U^{*}\|_1$ below, we will choose two specific functions $f$ and $g$
such that they satisfy above conditions, and that the right-hand side of \eqref{Flag Merryfield inequality} is finite

\begin{proof}[Proof of Lemma \ref{lemma-of-Merryfield-flag}]
Applying Lemma \ref{lemma-of-Merryfield} with $n$ replaced  by $n+m$ gives
\begin{eqnarray*}
\lefteqn{
\int_{\mathbb{R}_+^{n+1}}\int_{\mathbb{R}_+^{m+1}}\Big|t\nabla^{(1)}s\nabla^{(2)}
P_{t,s}\ast f(x,y)\Big|^2\Big|g\ast
\varphi_{t,s}(x,y)\Big|^2 \frac{dyds}{s}\frac{dxdt}{t}   }\\
&=&\int_0^{\infty}\int_0^{\infty}\int_{\mathbb{R}^n}\int_{\mathbb{R}^m}\Big|t\nabla^{(1)}P_t^{(1)}\ast
\big( (s\nabla^{(2)}P_s^{(2)})\ast_{\mathbb R^m} f
\big)(x,y)\Big|^2\Big|\varphi_t^{(1)}\ast\big(\varphi_s^{(2)}\ast_{\mathbb R^m}
g\big)(x,y)\Big|^2 dydx \frac{dt}{t}\frac{ds}{s} \\
&\lesssim&
\int_0^{\infty}\int_{\mathbb{R}^n\times\mathbb{R}^m}|F_s(x,y)|^2|G_s(x,y)|^2dydx
\frac{ds}{s}\\
&&+\int_0^{\infty}\int_0^{\infty}\int_{\mathbb{R}^n\times\mathbb{R}^m}|P_t^{(1)}\ast
F_s(x,y)|^2|\psi_t^{(1)}\ast G_s(x,y)|^2dydx
\frac{dt}{t}\frac{ds}{s}\\
&\triangleq& I_1+I_2,
\end{eqnarray*}
where $F_s(x,y)= (s\nabla^{(2)}P_s^{(2)})\ast_{\mathbb R^m} f (x,y)$
and $G_s(x,y)=\varphi_s^{(2)}\ast_{\mathbb R^m} g(x,y)$, and the implicit constant depends on the constant from Lemma \ref{lemma-of-Merryfield}.

To estimate $I_1$, by using Lemma \ref{lemma-of-Merryfield} on $\mathbb R^m$ we have
\begin{eqnarray*}
\lefteqn{
\int_0^{\infty}\int_{\mathbb{R}^n\times\mathbb{R}^m}|F_s(x,y)|^2|G_s(x,y)|^2dxdy
\frac{ds}{s}   }\\
&=&\int_{\mathbb{R}^n}\int_{\mathbb{R}_{+}^{m+1}}\Big|\Big(s\nabla^{(2)}P_s^{(2)}\ast_{\mathbb R^m}
f (x,\cdot)\Big)(y)\Big|^2\Big|\Big(\varphi_s^{(2)}\ast_{\mathbb R^m}
g(x,\cdot)\Big)(y)\Big|^2
\frac{dyds}{s}dx\\
&\lesssim& \int_{\mathbb{R}^n}\int_{\mathbb{R}^m}f(x,y)^2g(x,y)^2dydx\\
&&+
\int_{\mathbb{R}^n}\int_{\mathbb{R}_{+}^{m+1}}\Big|P_s^{(2)}\ast_{\mathbb R^m}
f(x,y)\Big|^2\Big|\psi_s^{(2)}\ast_{\mathbb R^m} g(x,y)\Big|^2
\frac{dyds}{s}dx,
\end{eqnarray*}
where  the implicit constant depends on the constant from Lemma \ref{lemma-of-Merryfield}.

Similarly, we have
\begin{eqnarray*}
I_2
&=&\int_{\mathbb{R}_{+}^{n+1}}\int_{\mathbb{R}_{+}^{m+1}}\Big|s\nabla^{(2)}P_s^{(2)}\ast_{\mathbb R^m}
\big(P_t^{(1)}\ast f
(x,\cdot)\big)(y)\Big|^2\Big|\varphi_s^{(2)}\ast_{\mathbb R^m}
\big(\psi_t^{(1)}\ast g(x,\cdot)\big)(y)\Big|^2
\frac{dyds}{s}\frac{dxdt}{t}\\
&\lesssim& \int_{\mathbb{R}^m}\int_{\mathbb{R}_+^{n+1}}\big|P_t^{(1)}\ast f(x,y)\big|^2
\big|\psi_t^{(1)}\ast g(x,y)\big|^2 \frac{dxdt}{t}dy\\
&&+
\int_{\mathbb{R}_{+}^{n+1}}\int_{\mathbb{R}_{+}^{m+1}}|P_s^{(2)}\ast_{\mathbb R^m}
P_t^{(1)}\ast f(x,y)|^2|\psi_s^{(2)}\ast_{\mathbb R^m}
\psi_t^{(1)}\ast g(x,y)|^2 \frac{dyds}{s}\frac{dxdt}{t},
\end{eqnarray*}
where  the implicit constant depends on the constant from Lemma \ref{lemma-of-Merryfield}.
 The estimates of term $I_1$ and term $I_2$ yield
\eqref{Flag Merryfield inequality}.

The proof of Lemma \ref{lemma-of-Merryfield-flag} is complete.
\end{proof}

We now begin to prove $\|S_F(U)\|_1\lesssim \|U^{*}\|_1$.  For any $\alpha>0$ and each $f\in L^1(\Bbb R^{n+m})$ satisfying $\|
U^{*}\|_1<\infty$, define
\begin{eqnarray*}
A(\alpha)=\Big\{(x,y)\in \mathbb{R}^n\times\mathbb{R}^m:\
M_F\big(\chi_{\{U^{*}>\alpha\}}\big)(x,y)<{1\over 2000}\Big\}.
\end{eqnarray*}
Then, by changing the order of the integration, we have
\begin{eqnarray*}
\lefteqn{\int_{A(\alpha)}S_F(u)(x,y)^2dxdy}\\
&\leq& \int_{\mathbb{R}_+^{n+1}\times\mathbb{R}_+^{m+1}}
\int_{A(\alpha)}\chi_{t,s}(x-x_1,y-y_1)dxdy \big|t\nabla^{(1)}
s\nabla^{(2)}U(x_1,y_1,t,s)\big|^2{dx_1dt \over t^{n+m+1}}{dy_1ds
\over s^{m+1}}.
\end{eqnarray*}
By the definition of $\chi_{t,s}(x-x_1,y-y_1)$, for
any fixed $(x_1,y_1,t,s)$, if $\chi_{t,s}(x-x_1,y-y_1)\neq 0,$ then $(x,y)$ belongs to $R$, where
$R=R(x_1,y_1,t,s)$ is a rectangle centered at $(x_1,y_1)$ and with
side-length $2t$ and $2t+2s$. This means that to estimate $\int_{A(\alpha)}\chi_{t,s}(x-x_1,y-y_1)dxdy,$ we only need to consider
those $(x,y)\in A(\alpha)\bigcap R(x_1,y_1,t,s).$ As a consequence,
\begin{eqnarray*}
M_F\big(\chi_{\{U^{*}>\alpha\}}\big)(x,y)<{1\over 2000}.
\end{eqnarray*}
Hence for such fixed $(x_1,y_1,t,s)$ mentioned above, we have
\begin{eqnarray*}
\frac{\displaystyle 1}{\displaystyle
|R(x_1,y_1,t,s)|}|A(\alpha)\bigcap R(x_1,y_1,t,s)|<{1\over 2000}.
\end{eqnarray*}
Let $R^{*}=\Big\{(x_1,y_1,t,s):\ \frac{\displaystyle
1}{\displaystyle \ |R(x_1,y_1,t,s)|\ }|A(\alpha)\bigcap
R(x_1,y_1,t,s)|<{1\over 2000} \Big\},$ then we have
\begin{eqnarray}\label{step1}
\int_{A(\alpha)}S_F(U)(x,y)^2dxdy
&\leq&\int_{R^{*}}\big|t\nabla^{(1)}
s\nabla^{(2)}U(x_1,y_1,t,s)\big|^2dx_1dy_1\frac{dt}{t}\frac{ds}{s}.
\end{eqnarray}
Let $ g(x,y)=\chi_{\{u^{*}\leq \alpha\}}(x,y)$ and
$\varphi^{(1)}(x,y)\in C_0^{\infty}(\mathbb{R}^{n+m})$ be a non-negative function satisfying
\begin{enumerate}
\item[\textnormal{(1)}] $\varphi^{(1)}(-x,-y)=\varphi^{(1)}(x,y)$;
\item[\textnormal{(2)}] ${\rm supp}\ \varphi^{(1)}\subset B_{n+m}(0,1) $, where $B_{n+m}(0,1)$ is the unit
ball in $\mathbb{R}^{n+m}$;
\item[\textnormal{(3)}] $\int_{\mathbb{R}^{n+m}}\varphi(x,y)dxdy=1$;
\item[\textnormal{(4)}] $\varphi^{(1)}(x,y)=1$ when $|(x,y)|\leq \frac{1}{3}$.
\end{enumerate}

\noindent Similarly, let $\varphi^{(2)}(z)\in C_0^{\infty}(\mathbb{R}^{m})$ be a non-negative function satisfying
\begin{enumerate}
\item[\textnormal{(1)}] $\varphi^{(2)}(-z)=\varphi^{(2)}(z)$;
\item[\textnormal{(2)}] ${\rm supp}\ \varphi^{(2)}\subset B_m(0,1) $, where $B_m(0,1)$ is the unit
ball in $\mathbb{R}^{m}$;
\item[\textnormal{(3)}] $\int_{\mathbb{R}^{m}}\varphi^{(2)}(z)dz=1$;
\item[\textnormal{(4)}]
$\varphi^{(2)}(z)=1$ when $|z|\leq \frac{1}{3}$.
\end{enumerate}

Set, using the projection of M\"uller, Ricci and Stein, $\varphi(x,y)=\varphi^{(1)}\ast_{\mathbb R^m} \varphi^{(2)}(x,y)$ and $\varphi_{t,s}(x,y)=\varphi_t^{(1)}\ast_{\mathbb R^m} \varphi_s^{(2)}(x,y)$.

We now {\bf claim} that there exists a positive constant $C$ such that for every for
$(x_1,y_1)\in R^{*}$,
\begin{eqnarray} \label{step2}
\varphi_{t,s}\ast g(x_1,y_1)\geq C.
\end{eqnarray}

To see this claim,  for
$(x_1,y_1)\in R^{*}$, we consider
\begin{eqnarray*}
\varphi_{t,s}\ast g(x_1,y_1)&=&\int_{\{u^{*}\leq
\alpha\}}\varphi_{t,s}(x_1-x'_1,y_1-y'_1)dx'_1dy'_1.
\end{eqnarray*}

From the properties of $\varphi^{(1)}$ and $\varphi^{(2)}$ it follows that in the last inequality above,
$|x_1-x'_1|<t$, $|y_1-y'_1-z|<t$ and $|z|<s$, which gives $|y_1-y'_1|<t+s$. Hence, for this fixed
$(x_1,y_1)\in R^{*}$, we see that $(x'_1,y'_1)\in R(x_1,y_1,t,s)$, which shows that
$\varphi_{t,s}(x_1-x'_1,y_1-y'_1)$ as a function of $(x'_1,y'_1)$ is supported in $R(x_1,y_1,t,s)$. Hence,
\begin{eqnarray}\label{step2 pre}
\varphi_{t,s}\ast g(x_1,y_1)&=&\int_{\{u^{*}\leq \alpha\}\bigcap
R(x_1,y_1,t,s)}\varphi_{t,s}(x_1-x'_1,y_1-y'_1)dx'_1dy'_1.
\end{eqnarray}
Also, since $\varphi^{(1)}(x,y)=1$ when $|(x,y)|\leq \frac{1}{3}$ and $\varphi^{(2)}(z)=1$ when $|z|\leq \frac{1}{3}$, we consider the following two cases:

Case i):  $s\leq t$.

In this case, we have  $ |R(x_1,y_1,t,s)|\approx t^{n+m}$.

Since
$$\varphi_{t,s}(x_1-x'_1,y_1-y'_1) = \int_{\mathbb R^m}\varphi^{(1)}_{t}(x_1-x'_1,y_1-y'_1-z) \varphi^{(2)}_s(z)dz,$$
we see that $\varphi^{(1)}_{t}(x_1-x'_1,y_1-y'_1-z)=t^{-n-m}$ when $|x_1-x'_1|\leq t/3$, $|y_1-y'_1-z|\leq t/3$ and that
$\varphi^{(2)}_s(z) =s^{-m}$ when $|z|<s/3$.
Thus, if we choose $x'_1$ with $|x_1-x'_1|\leq t/3$ and choose  $y'_1$ and $z$ with $|y_1-y'_1|<t/6$ and $|z|<s/6$, respectively, then it is direct to get $\varphi^{(1)}_{t}(x_1-x'_1,y_1-y'_1-z) = t^{-n-m}$ and $\varphi^{(2)}_s(z)=s^{-m}$.
Hence, from \eqref{step2 pre} we have
\begin{align*}
&\varphi_{t,s}\ast g(x_1,y_1)\\
&\geq \int_{\{U^{*}\leq \alpha\}\bigcap \{ (x'_1,y'_1): |x_1-x'_1|\leq {t\over3}, |y_1-y'_1|<{t\over6} \} }\int_{|z|<{s\over 3}}\varphi^{(1)}_{t}(x_1-x'_1,y_1-y'_1-z) \varphi^{(2)}_s(z)dzdx'_1dy'_1\nonumber\\
&= \int_{|z|<{s\over 3}} \int_{\{U^{*}\leq \alpha\}\bigcap \{ (x'_1,y'_1): |x_1-x'_1|\leq {t\over3}, |y_1-y'_1|<{t\over6} \} }\varphi^{(1)}_{t}(x_1-x'_1,y_1-y'_1-z)\,dx'_1dy'_1\ \  \varphi^{(2)}_s(z)dz\nonumber\\
& \approx { |\{U^{*}\leq \alpha\} \cap \{ (x'_1,y'_1): |x_1-x'_1|\leq {t/3}, |y_1-y'_1|<{t/6} \} | \over t^{n+m}}\\
& \gtrsim {|\{U^{*}\leq \alpha\} \cap \{ (x'_1,y'_1): |x_1-x'_1|\leq {t/12}, |y_1-y'_1|<{(t+s)/12} \}| \over |R(x_1,y_1,t,s) |}\\
&= {|\{U^{*}\leq \alpha\} \cap {1\over 12}R(x_1,y_1,t,s) | \over |R(x_1,y_1,t,s) |}.
\end{align*}
Since $(x_1,y_1)\in R^{*}$, from the definition of $R^*$ we obtain that
$|\{U^{*}\leq \alpha\} \cap R(x_1,y_1,t,s)  | > {1999\over2000}|R(x_1,y_1,t,s)|$, which gives that \eqref{step2} holds.

\medskip
Case ii):  $s> t$.

In this case, we have  $ |R(x_1,y_1,t,s)|\approx t^{n} s^m$.
By writing
$$\varphi_{t,s}(x_1-x'_1,y_1-y'_1) = \int_{\mathbb R^m}\varphi^{(1)}_{t}(x_1-x'_1,z) \varphi^{(2)}_s(y_1-y'_1-z)dz,$$
we see that $\varphi^{(1)}_{t}(x_1-x'_1,z)=t^{-n-m}$ when $|x_1-x'_1|\leq t/3$, $|z|\leq t/3$ and that
$\varphi^{(2)}_s(y_1-y'_1-z) =s^{-m}$ when $|y_1-y'_1-z|<s/3$.
Thus, if we choose $x'_1$ and $z$ with $|x_1-x'_1|\leq t/3$ and $|z|<t/6$, respectively,  and choose  $y'_1$ with $|y_1-y'_1|<s/6$, then we get $\varphi^{(1)}_{t}(x_1-x'_1,y_1-y'_1-z) =t^{-n-m}$ and $\varphi^{(2)}_s(z)=s^{-m}$.

Again, from \eqref{step2 pre} we have
\begin{align*}
&\varphi_{t,s}\ast g(x_1,y_1)\\
&\geq \int_{|x_1-x'_1|\leq {t\over3}}\int_{|y_1-y'_1|<{s\over6}  }\int_{|z|<{t\over 3}}\chi_{\{U^{*}\leq \alpha\}}(x'_1,y'_1) \varphi^{(1)}_{t}(x_1-x'_1,z) \varphi^{(2)}_s(y_1-y'_1-z)dzdx'_1dy'_1\nonumber\\
& \approx { |\{U^{*}\leq \alpha\} \cap \{ (x'_1,y'_1): |x_1-x'_1|\leq {t/3}, |y_1-y'_1|<{s/6} \} | \over t^{n}s^m}\\
& \gtrsim {|\{U^{*}\leq \alpha\} \cap \{ (x'_1,y'_1): |x_1-x'_1|\leq {t/12}, |y_1-y'_1|<{(t+s)/12} \}| \over |R(x_1,y_1,t,s) |}\\
&= {|\{U^{*}\leq \alpha\} \cap {1\over 12}R(x_1,y_1,t,s) | \over |R(x_1,y_1,t,s) |}
\end{align*}
which, similar to Case i),  gives that \eqref{step2} holds.

Thus, combining these two cases implies that the claim  \eqref{step2} holds.

Next, combining \eqref{step1} and \eqref{step2}, and recalling that  $U(x_1,y_1,t,s)= P_{t,s}\ast f(x_1,y_1)$ we have
\begin{eqnarray*}\label{step3}
\lefteqn{\int_{A(\alpha)}S_F(u)(x,y)^2dxdy}\\
&\lesssim&  \int_{R^{*}}\big|t\nabla^{(1)}
s\nabla^{(2)}U(x_1,y_1,t,s)\big|^2|\varphi_{t,s}\ast
g(x_1,y_1)|^2dx_1dy_1\frac{dt}{t}\frac{ds}{s}\nonumber\\
&\lesssim& \lefteqn{
\int_{\mathbb{R}_+^{n+1}}\int_{\mathbb{R}_+^{m+1}}\big|t\nabla^{(1)}s\nabla^{(2)}
U(x_1,y_1,t,s)\big|^2\big|g\ast
\varphi_{t,s}(x_1,y_1)\big|^2 \frac{dy_1ds}{s}\frac{dx_1dt}{t}   }\nonumber\\
&\lesssim& \bigg\{\int_{\mathbb{R}^n}\int_{\mathbb{R}^m}f(x_1,y_1)^2g(x_1,y_1)^2dy_1dx_1 \nonumber\\
&&\ \ \ \ \ +
\int_{\mathbb{R}^n}\int_{\mathbb{R}_{+}^{m+1}}\big|P_s^{(2)}\ast_{\mathbb R^m}
f(x_1,y_1)\big|^2\big|\psi_s^{(2)}\ast_{\mathbb R^m} g(x_1,y_1)\big|^2
\frac{dy_1ds}{s}dx_1\nonumber\\
&&\ \ \ \ \
+\int_{\mathbb{R}^m}\int_{\mathbb{R}_+^{n+1}}\big|P_t^{(1)}\ast
f(x_1,y_1)\big|^2
\big|\psi_t^{(1)}\ast g(x_1,y_1)\big|^2 \frac{dx_1dt}{t}dy_1\nonumber\\
&&\ \ \ \ \ +
\int_{\mathbb{R}_{+}^{n+1}}\int_{\mathbb{R}_{+}^{m+1}}|P_{t,s}\ast
f(x_1,y_1)|^2|\psi_{t,s}\ast g(x_1,y_1)|^2
\frac{dy_1ds}{s}\frac{dx_1dt}{t}\bigg\}\nonumber\\
&=:& II_1+II_2+II_3+II_4,\nonumber
\end{eqnarray*}
where the last inequality follows from Lemma \ref{lemma-of-Merryfield-flag}, and the implicit constants depend on the constants in \eqref{step2} and in Lemma \ref{lemma-of-Merryfield-flag}.

For the term $II_1$, from the definition of the function $g$ and the non-tangential maximal function $U^*$, we obtain that
\begin{eqnarray*}\label{term I}
|II_1|
&\leq& \int_{\{U^{*}\leq \alpha\}}f(x_1,y_1)^2dx_1dy_1
 \leq
\int_{\{U^{*}\leq \alpha\}}|U^{*}(x_1,y_1)|^2dx_1dy_1.
\end{eqnarray*}

We now consider the term $II_2$.
Note that if $\psi_s^{(2)}\ast_{\mathbb R^m} g(x_1,y_1)=\int
\psi_s^{(2)}(y_1-w)g(x_1,w)dw\neq 0$, then there exists some $w$
such that $|y_1-w|< s$ and $(x_1,w)\in \{u^{*}\leq \alpha\}$. Hence we get that
$|P_s^{(2)}\ast_{\mathbb R^m} f(x_1,y_1)|\leq \alpha$.
Also note that $\psi_s^{(2)}$ satisfies the cancellation condition, so for the constant function $1$ on $\mathbb R^n\times\mathbb R^m$, we have
$$ \psi_s^{(2)}\ast_{\mathbb R^m}1(x_1,y_1) = \int_{\mathbb R^m} \psi_s^{(2)}(y_1-y'_1)dy'_1=0. $$
As a consequence, we see that
\begin{eqnarray}\label{term II}
|II_2|&\leq&  \alpha^2 \int_{\mathbb{R}^n}\int_{\mathbb{R}_+^{m+1}}
|\psi_s^{(2)}\ast_{\mathbb R^m} g(x_1,y_1)|^2 \frac{dy_1ds}{s}dx_1 \\
&=& \alpha^2 \int_{\mathbb{R}^n}\int_{\mathbb{R}_+^{m+1}}
\big|\psi_s^{(2)}\ast_{\mathbb R^m}\big(1-g\big)(x_1,y_1)\big|^2 \frac{dy_1ds}{s}dx_1 \nonumber\\
&\leq& \alpha^2 \int_{\mathbb{R}^n}\int_{\mathbb{R}^{m}} |1-
g(x_1,y_1)|^2 dy_1dx_1 \nonumber\\
&\leq &  \alpha^2 |\{U^{*}>\alpha\}|, \nonumber
\end{eqnarray}
where the equality there follows from the cancellation condition of $\psi_s^{(2)}$.

The estimate for $II_3$ is similar to that of $II_2$ and we omit the details here.

For the last term $II_4$, if $\psi_{t,s}\ast g(x_1,y_1)=\int
\psi_{t,s}(x_1-v,y_1-w)g(v,w)dvdw\neq 0$, similarly as term $II_2$,
there exists $(v,w)$ such that $(v,w)\in
\{U^{*}\leq \alpha\}$ and $|x_1-v|<t$, $|y_1-w|<t+s$. Hence
$|P_{t,s}\ast f(x_1,y_1)|\leq \alpha$. Following the same strategy in the proof of
\eqref{term II}, we have
\begin{eqnarray}\label{term IV}
|II_4|\leq \int_{\mathbb{R}_{+}^{n+1}}\int_{\mathbb{R}_{+}^{m+1}}|P_{t,s}\ast
f(x_1,y_1)|^2|\psi_{t,s}\ast g(x_1,y_1)|^2
\frac{dy_1ds}{s}\frac{dx_1dt}{t}\leq   \alpha^2 |\{U^{*}>\alpha\}|. \nonumber
\end{eqnarray}
Combining all estimates above implies that
\begin{align}\label{part I}
\int_{\{M_F(\chi_{\{U^{*}>\alpha\}})\leq\frac{1}{200}\}}S_F(U)(x,y)^2dxdy
\lesssim \bigg( \alpha^2 |\{U^{*}>\alpha\}|+ \int_{\{U^{*}\leq
\alpha\}}|U^{*}(x_1,y_1)|^2dx_1dy_1\bigg),
\end{align}
where the implicit constant depends on the constants in \eqref{step2} and in Lemma \ref{lemma-of-Merryfield-flag}.

By the definition of the maximal function $M_F$  (Definition \ref{def flag HL}), we have that $M_F(f) \leq M_s(f)$, the strong maximal function on $\Bbb R^n\times \Bbb R^m$ defined as in \eqref{strong maximal},
 and hence, by using the $L^2$ boundedness of the strong maximal function, we have
\begin{eqnarray} \label{part II}
\Big|\Big\{(x,y):\ M_F(\chi_{\{U^{*}>\alpha\}})>\frac{1}{2000}
\Big\}\Big| &\leq&  \Big|\Big\{(x,y):\
M_s(\chi_{\{U^{*}>\alpha\}})(x,y)>\frac{1}{2000} \Big\}\Big|\\
&\lesssim &  \int_{\mathbb R^{n+m}} M_s(\chi_{\{U^{*}>\alpha\}})(x,y)^2dxdy\nonumber\\
&\lesssim &  \int_{\mathbb R^{n+m}}\chi_{\{U^{*}>\alpha\}}(x,y)^2dxdy\nonumber\\
&\lesssim & |\{U^{*}>\alpha\}|.\nonumber
\end{eqnarray}
The estimates in (\ref{part I}) and (\ref{part II}) yield
\begin{align*}
&|\{(x,y):\ S_F(U)(x,y)>\alpha\}|\\
&\leq \Big|\Big\{(x,y):\
M_F(\chi_{\{U^{*}>\alpha\}})>\frac{1}{200}\ {\rm and}\
S_F(U)(x,y)>\alpha\Big\}\Big| \\
&\quad+ \Big|\Big\{(x,y):\ M_F(\chi_{\{U^{*}>\alpha\}})\leq\frac{1}{200}\ {\rm and}\
S_F(U)(x,y)>\alpha\Big\}\Big|\nonumber\\
&\lesssim |\{U^{*}>\alpha\}| + {1\over \alpha^2} \int_{\{M_F(\chi_{\{U^{*}>\alpha\}})\leq\frac{1}{200}\}}S_F(U)(x,y)^2dxdy \\
&\lesssim   |\{U^{*}>\alpha\}| + \alpha^{-2} \int_{\{U^{*}\leq
\alpha\}}U^{*}(x_1,y_1)^2dx_1dy_1,
\end{align*}
which implies that $\|S_F(U)\|_1 \lesssim \|U^{*}\|_1$.

\subsection{The estimate $\|U^{*}\|_1\lesssim \|U^{+}\|_1$}

As mentioned in the introduction, the flag Hardy space is, in some sense, intermediate between the classical one parameter and the
product Hardy spaces. To deal with the flag non-tangential maximal function, we decompose it into the classical one parameter and the product cases. More precisely, we write, for any $(\bar x, \bar y)\in\mathbb R^{n+m}$, that
\begin{eqnarray*}
U^{*}(\bar x,\bar        y)&=& \sup_{(x,y,t,s)\in\Gamma(\bar x,\bar y)}|U(x,y,t,s)|\\
&\leq&  \sup_{(x,y,t,s)\in\Gamma(\bar x,\bar y),\ s\leq t}|U(x,y,t,s)| +\sup_{(x,y,t,s)\in\Gamma(\bar x,\bar y),\ s> t}|U(x,y,t,s)|\\
&=:& U_1^{*}(\bar x,\bar y) + U_2^{*}(\bar x,\bar y),
\end{eqnarray*}
where $\Gamma(\bar x,\bar y)=\{(x,y,t,s):\ |x-\bar x|\leq t,\ |y-\bar y|\leq t+s
\}$.

The main idea to show $\|U_1^{*}\|_1\lesssim \|U^{+}\|_1$ is the following lemma which was proved by Fefferman and Stein in \cite{FS} for the
classical one parameter Hardy space.

\begin{lemma}\label{FSlemma}
Suppose $B$ is a ball in $\mathbb R^{d+1},$ with center $(x_0,t_0).$ Let $u$ be harmonic in $B$ and continuous on the closure of $B.$ For any $p>0,$
$$|u(x_0,t_0)|^p\leq C_p\frac{1}{|B|}\int_B|u(x,t)|^pdxdt.$$
\end{lemma}
Suppose $f\in L^1(\mathbb{R}^{n+m})$ and $U(x,y,t,s)=P_{t,s}\ast f(x,y).$ Note that
$U(x,y,t,s),$ as a function of $(x,y,t)$ with a fixed $s,$ is harmonic on $\mathbb R_+^{n+m+1}.$  Lemma \ref{FSlemma} implies that for any $r>0$ and $ s\leq t,$
\begin{align*}\label{P1f}
 |U(x,y,t,s)|^r \leq C_r {1\over |B_1|} \int_{B_1} |U(x_1,y_1,t,s)|^rdx_1dy_1dt,
\end{align*}
where $B_1$ is any ball in $\mathbb R^{n+m+1}_+$ with the radius $t$ and the center $(x,y,t)\in \Gamma_1(\bar x, \bar y)$, where
$$\Gamma_1(\bar x, \bar y)=\{(x_1,y_1,t):\ |\bar x-x_1|\leq 2t,\ |\bar y-y_1|\leq 2t \}.$$
Note that the projection of $B_1$ on $\mathbb R^{n+m}$ is contained in the ball centered at $(\bar x, \bar y)$ with radius $4t$. Therefore,
\begin{eqnarray*}
|U(x,y,t,s)|^r &\leq& C_r t^{-n-m} \int_{B((\bar x,\bar y), 4t)} |U(x_1,y_1,t_1,s)|^rdx_1dy_1\\
&\leq & C_r t^{-n-m} \int_{B((\bar x,\bar y), 4t)} |U^+(x_1,y_1)|^rdx_1dy_1\\
&\leq& C_r M_1(|U^+|^r)(\bar x,\bar y),
\end{eqnarray*}
where $M_1$ is the standard Hardy--Littlewood maximal function on $\R^{n+m}$.

As a consequence, this implies that
\begin{eqnarray*}
U_1^{*}(\bar x,\bar y)\leq C\bigg( M_1\big( |U^+|^r \big)(\bar x,\bar y)\bigg)^{1\over r},
\end{eqnarray*}
which, together with the $L^{1\over r}, 0<r<1,$ boundedness of the Hardy--Littlewood maximal function
$M_1(f)$,  implies that
\begin{eqnarray*}\label{eee u2}
 \|U_1^*\|_{1}\leq C\|U^+\|_{1}.
 \end{eqnarray*}

Now we estimate $U_2^{*}(\bar x,\bar y)$. Observe that when $s>t$ the cone $\Gamma(\bar x,\bar y)=\{(x_1,y_1,t):\ |\bar x-x_1|\leq t,\ |\bar y-y_1|\leq t+s \}$ essentially is the cone in the product setting. Therefore, we write that
\begin{eqnarray*}
U_2^{*}(\bar x,\bar y)&=&\sup_{(x,y,t,s)\in\Gamma(\bar x,\bar y), s>t}|P_{t,s}\ast
f(x,y)|\\
&\leq &\sup_{(x,y,t,s)\in\Gamma_2(\bar x,\bar y)}\bigg| \int_{\mathbb R^n\times \mathbb R^m} \ \int_{\mathbb R^m} P^{(1)}_{t}(x-x_1,z)  P^{(2)}_{s}(y-y_1-z)dz \
f(x_1,y_1)\, dx_1dy_1 \bigg|,
\end{eqnarray*}
where
$$\Gamma_2(\bar x, \bar y)=\{(x,y,t,s):\ |\bar x-x|\leq 2t,\ |\bar y-y|\leq 2s \}.$$

The main idea to estimate the last term above is to introduce the following flag grand maximal function $\mathcal G_{\beta,\gamma}(f)(x_0,y_0):$ for $f\in L^1(\mathbb R^{n+m})$ and $(x_0,y_0)\in\mathbb R^n\times \mathbb R^m$,
\begin{equation*}
\mathcal G_{\beta,\gamma}(f)(x_0,y_0):=\sup\{ |\langle f, \varphi \rangle|:\  \|\varphi\|_{\widetilde{\mathcal M}_{flag}(\beta,\gamma,r_1,r_2,x_0,y_0)}\leq 1,\ r_1,r_2>0\}.
\end{equation*}
By Definition \ref{flag test function}, it is easy to see that as a function of $(x_1,y_1)$,
$$\int_{\mathbb R^m} P^{(1)}_{t}(x-x_1,z)  P^{(2)}_{s}(y-y_1-z)dz$$
is in $\widetilde{\mathcal M}_{flag}(1,1,t,s,\bar x,\bar y)$ with $(x,y,t,s)\in\Gamma_2(\bar x,\bar y)$ since $P^{(1)}_{t}(x-x_1,z)\in \widetilde{\mathcal M}_{n+m}(1,1,t,\bar x,0)$ and $ P^{(2)}_{s}(y-y_1) \in  \widetilde{\mathcal M}_{m}(1,1,s,\bar y).$
Moreover, it is also easy to check that
$$\sup_{(x,y,t,s)\in\Gamma_2(\bar x,\bar y)} \bigg\| \int_{\mathbb R^m} P^{(1)}_{t}(x-x_1,z)  P^{(2)}_{s}(y-y_1-z)dz \bigg\|_{\widetilde{\mathcal M}_{flag}(1,1,t,s,\bar x,\bar y)} \leq C, $$
where $C$ is an absolute constant independent of $(\bar x,\bar y)$.

As a consequence, we obtain that
\begin{eqnarray*}
U_2^{*}(\bar x,\bar y)
&= &\sup_{(x,y,t,s)\in\Gamma_2(\bar x,\bar y)}\bigg| \bigg\langle \ \int_{\mathbb R^m} P^{(1)}_{t}(x-\cdot,z)  P^{(2)}_{s}(y-\cdot-z)dz, \
f(\cdot,\cdot)\, \bigg\rangle \bigg|\\
&\leq& C \mathcal G_{1,1}(f)(\bar x,\bar y).
\end{eqnarray*}

It suffices to prove that for $f\in L^1(\mathbb R^{n+m})$ and $r>0,$
\begin{eqnarray}\label{flag domination}
\mathcal G_{1,1}(f)(\bar x,\bar y)\leq C \bigg( M_1\Big(M_2\big( |U^+|^r \big)\Big)(\bar x,\bar y)\Big)\bigg)^{1\over r}+C \bigg( M_2\Big(M_1\big( |U^+|^r \big)\Big)(\bar x,\bar y)\Big)\bigg)^{1\over r},
\end{eqnarray}
where $M_1$ and $M_2$ are the Hardy--Littlewood maximal functions on $\mathbb R^{n+m}$ and $\mathbb R^m,$ respectively.

We first claim that
\begin{align}\label{moser 1}
 |\langle f,\psi\rangle| &\leq C \bigg( M_1\Big(M_2\big( |U^+|^r \big)\Big)(\bar x,\bar y)\Big)\bigg)^{1\over r}
 \end{align}
for $r<1$ and close to 1,  $f\in  L^1(\mathbb R^{n+m})$,
and for every  $\psi\in {\mathcal M}_{flag}(1,1,2^{-j_1},2^{-k_1},\bar x,\bar y)$ with the norm $\|\psi\|_{{\mathcal M}_{flag}(1,1,2^{-j_1},2^{-k_1},\bar x,\bar y)}\leq 1$.

The key idea to show the above claim is to apply the discrete Calder\'on reproducing formula. To see this, consider the following approximations to the identity on $\mathbb R^{n+m}$: For each $j\in\mathbb Z$, define the operator
\begin{equation*}\label{Pj}
{\mathcal P}^{(1)}_j:= P^{(1)}_{2^{-j}}
\end{equation*}
with the kernel ${\mathcal P}^{(1)}_j(x,y):= P^{(1)}_{2^{-j}}(x,y).$

It is easy to see that
$$\lim_{j\to\infty} {\mathcal P}^{(1)}_j = \lim_{j\to\infty}P^{(1)}_{2^{-j}}= Id\quad {\rm and }\quad \lim_{j\to-\infty} {\mathcal P}^{(1)}_j = \lim_{j\to-\infty}P^{(1)}_{2^{-j}}= 0$$
in the sense of $L^2(\mathbb R^{n+m})$. And we further have
$$\int_{\mathbb R^{n+m}}{\mathcal P}^{(1)}_j(x,y) dxdy=1.$$
Set
$Q^{(1)}_j:={\mathcal P}^{(1)}_j-{\mathcal P}^{(1)}_{j-1}.$ Then $Q^{(1)}_j(x,y),$ the kernel of $Q^{(1)}_j$ satisfies the same size and smoothness  conditions as ${\mathcal P}^{(1)}_j(x,y)$ does, and
$$\int_{\mathbb R^{n+m}}Q^{(1)}_j(x,y) dxdy=0.$$
The operators ${\mathcal P}^{(2)}_k$ and $Q^{(2)}_k$ on $\mathbb R^m$ are defined similarly.

Repeating the same proof as in Theorem \ref{dcrf}, we have the following reproducing formula:   there exist functions $\phi_{j,k}(x,y, x_I,y_J)\in {\mathcal M}_{flag}(\beta,\gamma,2^{-j},2^{-k},x_I,y_J)$ and a fixed large integer $N$ such that for $f(x,y)=\int_{\mathbb R^m}f_1(x,y-z)f_2(z)dz$ with $f_1\in {\mathcal M}_{n+m}(\beta,\gamma,r_1,x_0,y_0)$ and $f_2\in {\mathcal M}_{m}(\beta,\gamma,r_2,z_0),$
\begin{eqnarray}\label{discrf2}
f(x,y)=\sum_j\sum_k\sum_I\sum_J|I||J|\phi_{j,k}(x,y,x_I,y_J) Q_{j,k}(f)(x_I,y_J),
\end{eqnarray}
where the series converges in $L^2(\mathbb R^{n+m})$ and in the flag test function space, and $I\subset \mathbb{R}^n, J\subset \mathbb{R}^m$ are dyadic cubes with side-lengths $\ell(I)=2^{-j-N}, \ell(J)=2^{-(j\wedge k-N)}$, $x_I$ and $y_J$ are any fixed points in $I$ and $J,$ respectively, and
$$ Q_{j,k}(f)(x_I,y_J) = \int_{\mathbb R^{n+m}} Q_{j,k}(x_I-x,y_J-y)f(x,y)dxdy $$
with the kernel
$$ Q_{j,k}(x,y) = \int_{\mathbb R^m} Q^{(1)}_j(x,y-z)Q^{(2)}_k(z)dz. $$

Now applying \eqref{discrf2} to the left-hand side of \eqref{moser 1}, we have
\begin{align}\label{eeeee1}
|\langle f,\psi\rangle|&=\bigg|\sum_j\sum_k\sum_I\sum_J|I||J| \langle \psi, \phi_{j,k}(\cdot,\cdot,x_I,y_J) \rangle Q_{j,k}(f)(x_I,y_J)\bigg|\nonumber\\
&\leq C \sum_j\sum_k\sum_I\sum_J|I||J| 2^{-|j-j_1|\beta}2^{-|k-k_1|\beta}\frac {2^{-(j\wedge j_1)\gamma}}{(2^{-j\wedge j_1}+|x_{I}-\bar x|)^{n+\gamma}}\nonumber\\
&\quad\quad\times \frac {2^{-[(k\wedge k_1)\wedge (j\wedge j_1)]\gamma}}{(2^{-[(k\wedge k_1)\wedge(j\wedge j_1)]}+|y_{J}-\bar y|)^{m+\gamma}}\inf_{z_1\in I, z_2\in J}
|U^+(z_1,z_2)|.
 \end{align}
 Here in the last inequality we used the following estimates:
\begin{enumerate}
\item[\textnormal{(1)}] The almost orthogonality estimate:
\begin{align*}
&|\langle \psi, \phi_{j,k}(\cdot,\cdot,x_I,y_J) \rangle|\\
&\leq C2^{-|j-j_1|\beta}2^{-|k-k_1|\beta}\frac {2^{-(j\wedge j_1)\gamma}}{(2^{-j\wedge j_1}+|x_{I}-\bar x|)^{n+\gamma}}\frac {2^{-[(k\wedge k_1)\wedge (j\wedge j_1)]\gamma}}{(2^{-[(k\wedge k_1)\wedge (j\wedge j_1)]}+|y_{J}-\bar y|)^{m+\gamma}}
\end{align*}
for $\beta,\gamma\in(0,1)$. We refer to \cite[Lemma 6]{HLS} for the proof.

\medskip
\item[\textnormal{(2)}] The fact that $x_I$ and $y_J$ are any fixed points in $I$ and $J,$ implies that we can choose $x_I\in I$ and $y_J\in J$ such that
 \begin{align*}
&|Q_{j,k}(f)(x_I,y_J)|\\
 &\leq 2 \inf_{z_1\in I, z_2\in J}  |Q_{j,k}(f)(z_1,z_2)|\\
&= 2 \inf_{z_1\in I, z_2\in J}\bigg|\int_{\mathbb R^{n+m}} \int_{\mathbb R^m} Q^{(1)}_j(z_1-x,z_2-y-z)Q^{(2)}_k(z)dz\ f(x,y)dxdy  \bigg|
\\
&= 2 \inf_{z_1\in I, z_2\in J}\bigg|\int_{\mathbb R^{n+m}} \int_{\mathbb R^m} \bigg(P^{(1)}_j(z_1-x,z_2-y-z)-P^{(1)}_{j-1}(z_1-x,z_2-y-z)\bigg)\\
&\hskip5cm\times \bigg(P^{(2)}_k(z)-P^{(2)}_{k-1}(z)\bigg)dz\ f(x,y)dxdy  \bigg|
\\
&\leq 8 \inf_{z_1\in I, z_2\in J}
 |U^+(z_1,z_2)|.
 \end{align*}
 \end{enumerate}

To estimate the last term in (\ref{eeeee1}), observe that for $0<r<1,$
\begin{align*}
&\sum_j\sum_k\sum_I\sum_J |I||J|2^{-|j-j_1|\beta}2^{-|k-k_1|\beta}\frac {2^{-(j\wedge j_1)\gamma}}{(2^{-j\wedge j_1}+|x_{I}-\bar x|)^{n+\gamma}}\frac {2^{-[(k\wedge k_1)\wedge (j\wedge j_1)]\gamma}}{(2^{-[(k\wedge k_1)\wedge (j\wedge j_1)]}+|y_{J}-\bar y|)^{m+\gamma}} \\
&\hskip3cm \times \inf_{z_1\in I, z_2\in J}|U^+(z_1,z_2)|\\
&\le \bigg\{ \sum_j\sum_k\sum_I\sum_J |I|^r|J|^r2^{-|j-j_1|\beta r}2^{-|k-k_1|\beta r}\frac {2^{-(j\wedge j_1)\gamma r}}{(2^{-j\wedge j_1}+|x_{I}-\bar x|)^{(n+\gamma)r}}\\
&\hskip3cm \times \frac {2^{-[(k\wedge k_1)\wedge (j\wedge j_1)]\gamma r}}{(2^{-[(k\wedge k_1)\wedge (j\wedge j_1)]}+|y_{J}-\bar y|)^{(m+\gamma)r}} \inf_{z_1\in I, z_2\in J}|U^+(z_1,z_2)|^r\bigg\}^{1/r}.
\end{align*}

Note that $\ell(I)=2^{-j-N}$ and $\ell(J)=2^{-(j\wedge k-N)}$. Write
\begin{align*}
&\sum_I\sum_J |I|^r|J^r|\frac {2^{-(j\wedge j_1)\gamma r}}{(2^{-j\wedge j_1}+|x_{I}-\bar x|)^{(n+\gamma)r}}\frac {2^{-[(k\wedge k_1)\wedge (j\wedge j_1)]\gamma r}}{(2^{-[(k\wedge k_1)\wedge (j\wedge j_1)]}+|y_{J}-\bar y|)^{(m+\gamma)r}}\inf_{z_1\in I, z_2\in J}|U^+(z_1,z_2)|^r\\
&=C2^{-jn(r-1)}2^{(j \wedge k)m(r-1)}\sum_I\sum_J |I||J|\frac {2^{-(j\wedge j_1)\gamma r}}{(2^{-j\wedge j_1}+|x_{I}-\bar x|)^{(n+\gamma)r}}  \\
&\hskip5cm \times \frac {2^{-[(k\wedge k_1)\wedge (j\wedge j_1)]\gamma r}}{(2^{-[(k\wedge k_1)\wedge (j\wedge j_1)]}+|y_{J}-\bar y|)^{(m+\gamma)r}} \inf_{z_1\in I, z_2\in J}|U^+(z_1,z_2)|^r\\
&\le C2^{-jn(r-1)}2^{(j \wedge k)m(r-1)}\int_{\Bbb R^n\times \Bbb R^m} \frac {2^{-(j\wedge j_1)\gamma r}}{(2^{-j\wedge j_1}+|x-\bar x|)^{(n+\gamma)r}}\frac {2^{-[(k\wedge k_1)\wedge (j\wedge j_1)]\gamma r}}{(2^{-[(k\wedge k_1)\wedge (j\wedge j_1)]}+|y-\bar y|)^{(m+\gamma)r}}\\
&\hskip6cm \times |U^+(x,y)|^rdxdy\\
&\le C2^{-jn(r-1)}2^{(j \wedge k)m(r-1)}2^{-(j\wedge j_1)\gamma r}2^{-(j\wedge j_1)[n-(n+\gamma)r]}2^{-[(k\wedge k_1)\wedge (j\wedge j_1)]\gamma r} \\
&\qquad\times 2^{-[(k\wedge k_1)\wedge (j\wedge j_1)][m-(m+\gamma)r]}\Big(M_1\big(M_2\big(|U^+|^r\big)\big)\Big)(\bar x,\bar y).
\end{align*}
A direct computation shows that if $\frac{m+n}{m+n+\beta}<r<1$, then
$$\sum_j\sum_k 2^{-|j-j_1|\beta r}2^{-|k-k_1|\beta r}2^{-jn(r-1)}2^{-(j\wedge j_1)(n-nr)}2^{(j \wedge k)m(r-1)}2^{-[(k\wedge k_1)\wedge (j\wedge j_1)][m-mr]}\le C.$$
Thus, we obtain that
the right-hand side of \eqref{eeeee1} is bounded by
$$\bigg( M_1\Big(M_2\big( |U^+|^r \big)\Big)(\bar x,\bar y)\bigg)^{1\over r},$$
which implies \eqref{moser 1}.

We now prove \eqref{flag domination}. For every $\varphi$ with $$\varphi (x,y) = \int_{\mathbb R^m}  \varphi^{(1)} (x,y-z)  \varphi^{(2)} (z) dz,  $$
where
$\varphi^{(1)}(x,y) \in \widetilde{\mathcal M}_{n+m}(1,1,t,\bar x,0) $
with
$ \|\varphi^{(1)}\|_{\widetilde{\mathcal M}_{n+m}(1,1,t,\bar x,0)}\leq 1$,
and $\varphi^{(2)}(z)
\in \widetilde{\mathcal M}_{m}(1,1,s,\bar y) $
with
$ \|\varphi^{(2)}\|_{\widetilde{\mathcal M}_{m}(1,1,s,\bar y)}\leq 1$,
we have $\varphi(x,y) \in \widetilde{\mathcal M}_{flag}(1,1,t,s,\bar x,\bar y)$ with
$\|\varphi \|_{ \widetilde{\mathcal M}_{flag}(1,1,t,s,\bar x,\bar y)}\leq1$.

Let
$$\sigma_1:=\int_{\mathbb R^{n+m}}\varphi^{(1)} (x,y)dxdy,\quad \sigma_2:=\int_{\mathbb R^m} \varphi^{(2)} (z) dz.$$
It is obvious that $|\sigma_1|, |\sigma_2|\leq C$.
We set
$$\psi^{(1)}(x,y):={1\over 1+\sigma_1 C}\bigg[ \varphi^{(1)} (x,y)- \sigma_1 {\mathcal P}^{(1)}_{j_1}(\bar x- x, y) \bigg],$$
$$ \psi^{(2)}(z):={1\over 1+\sigma_2 C}\bigg[  \varphi^{(2)} (z)- \sigma_2 {\mathcal P}^{(2)}_{ k_1}(z-\bar y) \bigg],$$
where $j_1 := \lfloor\log_2 t \rfloor+1$ and $k_1 := \lfloor\log_2 s \rfloor+1$ .

Then for an appropriate constant $C,$ the function $\psi (x,y) = \int_{\mathbb R^m}  \psi^{(1)} (x,y-z)  \psi^{(2)} (z) dz$ is in  ${\mathcal M}_{flag}(1,1,t,s,\bar x,\bar y)$ with
$\|\psi \|_{ {\mathcal M}_{flag}(1,1,t,s,\bar x,\bar y)}\leq1.$

Based on the definition of $\psi$, we have
\begin{align*}
|\langle f, \varphi\rangle|&=\bigg| \int_{\mathbb R^{n+m}} f(x,y)\varphi (x,y)  dxdy \bigg|\\
&=\bigg| \int_{\mathbb R^{n+m}} f(x,y)\int_{\mathbb R^m}  \varphi^{(1)} (x,y-z)  \varphi^{(2)} (z) dz  dxdy \bigg|\\
&=\bigg| \int_{\mathbb R^{n+m}} f(x,y)\int_{\mathbb R^m} \bigg[ (1+\sigma_1 C)\psi^{(1)}(x,y-z) + \sigma_1 {\mathcal P}^{(1)}_{j_1}(\bar x- x, y-z) \bigg] \\
&\hskip2cm\times  \bigg[ (1+\sigma_2 C)\psi^{(2)}(z) + \sigma_2 {\mathcal P}^{(2)}_{k_1}(z-\bar  y) \bigg] dz  dxdy \bigg|\\
&\leq \bigg| \int_{\mathbb R^{n+m}} f(x,y)\int_{\mathbb R^m}  (1+\sigma_1 C)\psi^{(1)}(x,y-z)  (1+\sigma_2 C)\psi^{(2)}(z) dz  dxdy \bigg|\\
&\quad+\bigg| \int_{\mathbb R^{n+m}} f(x,y)\int_{\mathbb R^m}  \sigma_1 {\mathcal P}^{(1)}_{j_1}(\bar x- x, y-z)    (1+\sigma_2 C)\psi^{(2)}(z) dz  dxdy \bigg|\\
&\quad+\bigg| \int_{\mathbb R^{n+m}} f(x,y)\int_{\mathbb R^m}  (1+\sigma_1 C)\psi^{(1)}(x,y-z)  \sigma_2 {\mathcal P}^{(2)}_{k_1}(z-\bar  y)  dz  dxdy \bigg|\\
&\quad+\bigg| \int_{\mathbb R^{n+m}} f(x,y)\int_{\mathbb R^m}  \sigma_1 {\mathcal P}^{(1)}_{j_1}(\bar x- x, y-z)  \sigma_2 {\mathcal P}^{(2)}_{k_1}(z-\bar  y) \bigg] dz  dxdy \bigg|\\
&=:A_1+A_2+A_3+A_4.
\end{align*}

For the term $A_1$, from \eqref{moser 1} we obtain that
\begin{align*}
A_1 &\leq C \bigg( M_1\Big(M_2\big( |U^+|^r \big)\Big)(\bar x,\bar y)\bigg)^{1\over r}.
 \end{align*}

For the term $A_4$, by definition we have
\begin{align*}
A_4 \leq C U^+(\bar x, \bar y) = C\Big(|U^+(\bar x, \bar y)|^r \Big)^{1\over r}\leq C \bigg( M_1\Big(M_2\big( |U^+|^r \big)\Big)(\bar x,\bar y)\bigg)^{1\over r}.
 \end{align*}

As for $A_2$, we write
\begin{align*}
A_2&=\bigg| \int_{\mathbb R^m}\ \int_{\mathbb R^{n+m}} f(x,y)  \sigma_1 P^{(1)}_{j_1}(\bar x- x, y-z) dxdy\   (1+\sigma_2 C)\psi^{(2)}(z) dz   \bigg|\\
&=\bigg| \int_{\mathbb R^m}\ F_{\bar x, j_1}(z)  (1+\sigma_2 C)\psi^{(2)}(z) dz   \bigg|,
\end{align*}
where
$$F_{\bar x, j_1}(z):= \int_{\mathbb R^{n+m}} f(x,y)  \sigma_1 P^{(1)}_{j_1}(\bar x- x, y-z) dxdy.$$
Then following the same approach as above, by using the reproducing formula in terms of $Q^{(2)}_{k}$ and the almost orthogonality estimates, we obtain that
\begin{align*}
A_2&\leq C\Bigg(
M_2 \bigg(\sup_{s>0}\bigg| \int_{\mathbb R^m}\ F_{\bar x, j_1}(z)  P_{s}^{(2)}(z) dz   \bigg|^r\bigg)(\bar y) \Bigg)^{1\over r}\\
&\leq C\Bigg(
M_2 \bigg(\sup_{s>0}\bigg| \int_{\mathbb R^m} \int_{\mathbb R^{n+m}} f(x,y)  \sigma_1 P^{(1)}_{j_1}(\bar x- x, y-z) dxdy\  P_{s}^{(2)}(z) dz   \bigg|^r\bigg)(\bar y) \Bigg)^{1\over r}\\
&\leq C\bigg( M_2\Big(M_1\big( |U^+|^r \big)\Big)(\bar x,\bar y)\bigg)^{1\over r}.
\end{align*}
Symmetrically, we obtain that
\begin{align*}
A_3 &\leq C \bigg( M_1\Big(M_2\big( |U^+|^r \big)\Big)(\bar x,\bar y)\bigg)^{1\over r}.
 \end{align*}
Combining the estimates of $A_1$, $A_2$, $A_3$ and $A_4$, we obtain that \eqref{flag domination} holds.

\subsection{The estimate $\|U^{+}\|_{1}\lesssim \sum_{j=1}^{n+m}\sum_{k=1}^m\|R_{j,k}(f)\|_1+\|f\|_1$}

Let $P^{(1)}_t $ be the Poisson kernel on $\R^{n+m}$ and $Q^{(1)}_{j,t}$ be the $j$-th conjugate Poisson kernel on  $\R^{n+m}$.
Then following \cite[Section 8]{FS}, it is easy to verify that
$u:=u_0^{(1)}=P^{(1)}_t * f$, $u_j^{(1)}=Q^{(1)}_{j,t}* f = P^{(1)}_t * (R^{(1)}_j *f)$, $j=1,2,\ldots,n+m$ is a $(n+m+1)$-tuple of
harmonic functions that satisfy the following system of equations:
\begin{equation}\label{CR1}
\begin{cases}
\frac{\displaystyle \partial u^{(1)}_j}{\displaystyle \partial x_j}=\frac{\displaystyle\partial u^{(1)}_i}{\displaystyle\partial x_j}, \ \  0\leq i, j \leq n+m;  \\[10pt]
\sum\limits_{j=0}^{n+m}\frac{\displaystyle \partial u_j^{(1)}}{\displaystyle \partial x_j}=0.
\end{cases}
\end{equation}
Here we use $R^{(1)}_j $ to denote the $j$th Riesz transform on $\mathbb{R}^{n+m}$, $j=1,2,\ldots,n+m$.
Similarly, we use $P^{(2)}_s $ to denote the Poisson kernel on $\R^{m}$ and $Q^{(2)}_{k,s}$ to denote the $k$-th conjugate Poisson kernel on  $\R^{m}$.

Again, following \cite[Section 8]{FS}, we can verify that
$u:=u_0^{(2)}=P^{(2)}_s \ast_{\mathbb R^m} f$, $u_k^{(2)}=Q^{(2)}_{k,s}\ast_{\mathbb R^m} f= P^{(2)}_s \ast_{\mathbb R^m} (R^{(2)}_k \ast_{\mathbb R^m} f)$, $k=1,2,\ldots,m$ is a $(m+1)$-tuple of
harmonic functions that satisfy the following system of equations:
\begin{equation}\label{CR2}
\begin{cases}
\frac{\displaystyle \partial u^{(2)}_j}{\displaystyle \partial x_j}=\frac{\displaystyle\partial u^{(2)}_i}{\displaystyle\partial x_j}, \ \  0\leq i, j \leq m;  \\[10pt]
\sum\limits_{j=0}^{m}\frac{\displaystyle \partial u^{(2)}_j}{\displaystyle \partial x_j}=0.
\end{cases}
\end{equation}
Here we use $R^{(2)}_k $ to denote the $k$th Riesz transform on $\mathbb{R}^{m}$, $k=1,2,\ldots,m$.

We now set $U(x,y,t,s)=u_{0,0}(x,y,t,s)=P^{(1)}_t *_{\mathbb R^m}  P^{(2)}_s *f(x,y)$. Then we define
$$ u_{1,0}(x,y,t,s)=Q^{(1)}_{1,t} *_{\mathbb R^m} P^{(2)}_s *f(x,y)\quad{\rm and}\quad u_{0,1}(x,y,t,s)=P^{(1)}_{t} *_{\mathbb R^m}  Q^{(2)}_{1,s} *f(x,y), $$
and similarly,
$$ u_{j,k}(x,y,t,s)=Q^{(1)}_{j,t}  *_{\mathbb R^m}  Q^{(2)}_{k,s} *f(x,y), $$
for $j=1,\ldots,n+m$ and $k=1,\ldots,m$.

We first point out that for $k=1,\ldots,m$,  the tuple
$ (u_{0,k}, u_{1,k},\ldots, u_{n+m,k}) $
satisfies the Cauchy--Riemann equation in \eqref{CR1},  and that for $j=1,\ldots,n+m$
the tuple
$ (u_{j,0}, u_{j,1},\ldots, u_{j,m}) $
satisfies the Cauchy--Riemann equation in \eqref{CR2}.

Following the idea in \cite[Section 8]{FS}, we consider the matrix-valued function
\begin{align*}
F&=
  \begin{bmatrix}
    u_{0,0} &  \ldots  & u_{0,m} \\
    \ldots & \ldots & \ldots\\
    u_{n+m,0}\ \ & \ldots &\ \ u_{n+m,m}
  \end{bmatrix}
  =P^{(1)}_t *_{\mathbb R^m}  P^{(2)}_s * \widetilde F,
  \end{align*}
where we denote
\begin{align*}
\widetilde F=\begin{bmatrix}
    f &  \ldots  & R^{(2)}_m \ast_{\mathbb R^m}f \\
    \ldots & \ldots & \ldots\\
    R^{(1)}_{m+n} \ast f\ \ & \ldots &\ \ R^{(1)}_{m+n}*_{\mathbb R^m}  R^{(2)}_m *f
  \end{bmatrix}.
\end{align*}

We obtain
\begin{equation*}
\begin{aligned}
 &\sup_{t>0} \sup_{s>0}\int_{\mathbb R^{n}}\int_{\mathbb R^{m}} |F(x,y,t,s)|dxdy\\
 &\leq\sup_{t>0} \sup_{s>0}\int_{\mathbb R^{n}}\int_{\mathbb R^{m}} \Big( \sum_{j=0}^{n+m} \sum_{k=0}^{m} |u_{j,k}(x,y,t,s)|^2 \Big)^{1\over2}dxdy\\
 &\leq C \sum_{j=0}^{n+m} \sum_{k=0}^{m}\sup_{t>0} \sup_{s>0}\int_{\mathbb R^{n}}\int_{\mathbb R^{m}}   \Big|Q^{(1)}_{j,t}  *_{\mathbb R^m}  Q^{(2)}_{k,s} *f(x,y)\Big| dxdy \\
 &\leq C \sum_{j=0}^{n+m} \sum_{k=0}^{m}\sup_{t>0} \sup_{s>0}\int_{\mathbb R^{n}}\int_{\mathbb R^{m}}  \Big|P^{(1)}_{t}  *_{\mathbb R^m}  P^{(2)}_{s} *(R^{(1)}_j * R^{(2)}_k \ast_{\mathbb R^m} f) (x,y)\Big|dxdy \\
 &\leq C \sum_{j=0}^{n+m} \sum_{k=0}^{m}\int_{\mathbb R^{n}}\int_{\mathbb R^{m}}  \Big|(R^{(1)}_j \ast_{\R^m} R^{(2)}_k \ast) (f) (x,y)\Big|dxdy,
\end{aligned}
\end{equation*}
where the last inequality follows from the fact that
$$ \int_{\mathbb R^{n+m}} P^{(1)}_{t}(x-x_1,y-y_1)dxdy =C_{n+m} \quad {\rm\ and\ }\quad\int_{\mathbb R^{m}} P^{(2)}_{s}(y-y_1)dy =C_{m}$$
for all $t,s>0$, $x_1\in\mathbb R^n$ and $y_1\in \mathbb R^m$.

Next it suffices to show
\begin{align}\label{claim eee}
 \|U^+\|_{1} \leq C\sup_{t>0} \sup_{s>0}\int_{\R^{n}}\int_{\R^{m}} |F(x,y,t,s)|dxdy.
\end{align}
To see this, we have that for $q<1,$
\begin{align*}
|F(x,y,t+\epsilon,s+\epsilon_2)|^q&=\Big|P^{(1)}_{t+\epsilon_1} *  P^{(2)}_{s+\epsilon_2} *_{\mathbb R^m} \widetilde F(x,y)\Big|^q \\
&=\Big|P^{(1)}_{t} *P^{(1)}_{\epsilon_1} *  P^{(2)}_{s+\epsilon_2} *_{\mathbb R^m} \widetilde F(x,y)\Big|^q\\
&=\Big|P^{(1)}_t\ast F(x,y,\epsilon_1,s+\epsilon_2)\Big|^q\\
&\leq C_{q,m}\sum_{k=0}^m \Big|P^{(1)}_t\ast F_k(x,y,\epsilon_1,s+\epsilon_2)\Big|^q,
\end{align*}
where for each $k$, $F_k$ is the $k$th column in the matrix $F$. Since $P^{(1)}_t\ast F_k$ satisfies the generalised Cauchy--Riemann equations in \eqref{CR1} for the variable $(x,y,t)$, we get that $|P^{(1)}_t\ast F_k|^q$ is subharmonic for $q\geq {n+m-1\over n+m}$. Then from the subharmonic inequality \cite[Equation (59), Section 4.2, Chapter 3]{St} we have that for $q\geq {n+m-1\over n+m}$, $x\in \mathbb R^{n}, y\in\mathbb R^m,$ $t>0$ and $\epsilon_1>0$,
\begin{align*}
\Big|P^{(1)}_t\ast F_k(x,y,\epsilon_1,s+\epsilon_2)\Big|^q\leq P^{(1)}_t\ast |F_k(x,y,\epsilon_1,s+\epsilon_2)|^q,
\end{align*}
which implies that
\begin{align}\label{eee1}
|F(x,y,t+\epsilon_1,s+\epsilon_2)|^q
&\leq C_{q,m}\sum_{k=0}^m P^{(1)}_t\ast |F_k(x,y,\epsilon_1,s+\epsilon_2)|^q\\
&\leq C_{q,m} P^{(1)}_t\ast |F(x,y,\epsilon_1,s+\epsilon_2)|^q.\nonumber
\end{align}
And we use the basic fact that
$|F|^q=(\sum_{k=0}^m |F_k|^2)^{q\over 2} \approx  \sum_{k=0}^m|F_k|^q$.

Again, for $F(x,y,\epsilon_1,s+\epsilon_2)$, we have
\begin{align*}
|F(x,y,\epsilon_1,s+\epsilon_2)|^q&=|P^{(2)}_s\ast_{\mathbb R^m} F(x,y,\epsilon_1,\epsilon_2)|^q\\
&\leq C_{q,n+m}\sum_{j=0}^{n+m} \Big|P^{(2)}_s\ast_{\mathbb R^m} \widetilde F_j(x,y,\epsilon_1,\epsilon_2)\Big|^q,
\end{align*}
where for each $j$, $\widetilde F_j$ is the $j$th row in the matrix $F$. Since $P^{(2)}_s\ast_{\mathbb R^m}\widetilde F_j$ satisfies the generalised Cauchy--Riemann equations in \eqref{CR1} for the variable $(y,s)$, we get that $|P^{(2)}_s\ast_{\mathbb R^m}\widetilde F_j|^q$ is subharmonic for $q\geq {m-1\over m}$. Then again, from the subharmonic inequality \cite[Equation (59), Section 4.2, Chapter 3]{St} we have that for $q\geq {m-1\over m}$, $y\in\mathbb R^m,$ $s>0$ and $\epsilon_2>0$,
\begin{align}\label{eee2}
|F(x,y,\epsilon_1,s+\epsilon_2)|^q
&\leq C_{q,n+m}\sum_{j=0}^{n+m} P^{(2)}_s\ast_{\mathbb R^m} |\widetilde F_j(x,y,\epsilon_1,\epsilon_2)|^q\\
&\leq C_{q,n+m} P^{(2)}_s\ast_{\mathbb R^m} |F(x,y,\epsilon_1,\epsilon_2)|^q.\nonumber
\end{align}
And we use the basic fact that
$|F|^q=(\sum_{j=0}^{n+m} |\widetilde F_j|^2)^{q\over 2} \approx  \sum_{j=0}^{n+m}|\widetilde F_j|^q$.

Combining the estimates of \eqref{eee1} and \eqref{eee2}, we obtain that
\begin{align*}
|F(x,y,t+\epsilon_1,s+\epsilon_2)|^q \leq C_{q,n,m}P^{(1)}_t\ast P^{(2)}_s\ast_{\mathbb R^m} |F(x,y,\epsilon_1,\epsilon_2)|^q.
\end{align*}

Then, following the convergence argument in \cite[Section 8]{FS}, also in \cite[Section 4.2]{St}, we obtain that
\begin{align*}
\|U^+\|_1 \leq  C_{m,n}\sup_{t>0} \sup_{s>0}\int_{\R^{n}}\int_{\R^{m}} |F(x,y,t,s)|dydx.
\end{align*}
which implies that the claim \eqref{claim eee} holds.

\section{Flag maximal functions: from Poisson kernel to general Schwartz kernels }

In this section, the following estimates will be established:
\begin{enumerate}
\item[(II)] $\|U^{*}\|_1\approx \|M^{*}_\phi(f)\|_1,$

\item[(III)] $\|U^{+}\|_1\approx \|M^{+}_\phi(f)\|_1$,
\end{enumerate}

\subsection{The equivalence $\|U^{*}\|_{1}\approx \|M^{*}_\phi(f)\|_{1}$}

We first show
\begin{eqnarray*}
\label{u star-leq-M star f}
\|U^{*}\|_{1}\leq C\|M^{*}_\phi(f)\|_{1}.
\end{eqnarray*}
To do this, we introduce the ``tangential" maximal function
$M^{**}_N$ (depending on a parameter N) by
\begin{eqnarray*}
\label{tangential-maximal-function}
M_{N}^{**}(f)(x,y)=\sup_{u\in\mathbb{R}^n,v\in\mathbb{R}^m,t,s>0}|f\ast\phi_{t,s}(x-u,y-v)|
\frac{\displaystyle 1}{\displaystyle \Big(1+{|u|\over t}\Big)^N\Big(1+{|v|\over
{t+s}}\Big)^N}.
\end{eqnarray*}
Obviously,
$$M^+_{\phi}(f)(x,y)\leq M_{\phi}^{*}(f)(x,y)\leq 2^{2N}M_N^{**}(f)(x,y).$$

Next, we introduce the grand maximal functions. For this
purpose, we first note that on $\mathcal {S}(\mathbb{R}^{n+m})$ one
has a denumerable collection of seminorms
$\|\cdot\|_{\alpha_1,\alpha_2,\beta_1,\beta_2}$ given by
$$\|\phi\|_{\alpha_1,\alpha_2,\beta_1,\beta_2}=\sup_{(x,y)\in\mathbb{R}^{n+m}}\Big|x^{\alpha_1}y^{\alpha_2}
\partial_x^{\beta_1}\partial_y^{\beta_2}\phi(x,y)\Big|.$$
Similarly, on $\mathcal {S}(\mathbb{R}^{m})$, seminorms
$\|\cdot\|_{\alpha,\beta}$ are given by
$$\|\phi\|_{\alpha,\beta}=\sup_{z\in\mathbb{R}^{m}}\Big|z^{\alpha}
\partial_z^{\beta}\phi(z)\Big|.$$
Let
$\mathcal{F}^{(1)}=\{\|\cdot\|_{\alpha_1^i,\alpha_2^i,\beta_1^i,\beta_2^i}\}$
be any finite collections of seminorms on $\mathcal
{S}(\mathbb{R}^{n+m})$ and
$\mathcal{F}^{(2)}=\{\|\cdot\|_{\alpha^i,\beta^i}\}$ be any finite
collections of seminorms on $\mathcal {S}(\mathbb{R}^{m})$. Applying the projection of M\"uller, Ricci and Stein, set
\begin{align*}
\mathscr{F}=\Big\{\phi\in\mathscr{D}_F(\mathbb{R}^n\times\mathbb{R}^m):
{\rm for}\ {\rm all}\
\phi^{\sharp}\in\mathcal{S}(\mathbb{R}^{n+m}\times\mathbb{R}^m) \
{\rm satisfying}\
\phi(x,y)=\int_{\mathbb{R}^m}\phi^{\sharp}(x,y-z,z)dz,\\
\|\phi^{\sharp}(\cdot,\cdot,z)\|_{\alpha_1,\alpha_2,\beta_1,\beta_2}\leq
1 \ {\rm for }\ {\rm all}\ z\in\mathbb{R}^m {\rm and}
\|\cdot\|_{\alpha_1,\alpha_2,\beta_1,\beta_2}\in
\mathcal{F}^{(1)};\\
\|\phi^{\sharp}(x,y,\cdot)\|_{\alpha,\beta}\leq 1\
 {\rm for }\
{\rm all}\ (x,y)\in\mathbb{R}^{n+m}\ {\rm and}\
\|\cdot\|_{\alpha,\beta}\in \mathcal{F}^{(2)}\Big\}.
\end{align*}
We then define
\begin{eqnarray*}
M_{\mathscr{F}}(f)(x,y)=\sup_{\phi\in \mathscr{F}} M^+_{\phi}(f)(x,y).
\end{eqnarray*}
We need the following results.
\begin{lemma}\label{lemma-of-tangential-maximal-leq-nontangential-maximal}
If $M_{\phi}^{*}(f)\in L^1(\mathbb{R}^{n+m})$ and
$N>2(n\vee m)$, then $M_N^{**}(f)\in
L^1(\mathbb{R}^{n+m})$ with
\begin{eqnarray}
\label{tangential-maximal-leq-nontangential-maximal}
\|M_N^{**}(f)\|_1\leq
C_{N,p}\|M_{\phi}^{*}(f)\|_1.
\end{eqnarray}
\end{lemma}

\begin{proof}
We point out that if
$$M_{\phi,a,b}^{*}(f)(x,y)=\sup_{(x_1,y_1,t,s)\in\Gamma_{a,b}(x,y)}|\phi_{t,s}\ast
f(x_1,y_1)|,
$$
where $\Gamma_{a,b}(x,y)=\{(x_1,y_1,t,s):\ |x-x_1|\leq at,\
|y-y_1|\leq b(t+s) \}$, then
\begin{eqnarray}\label{comparison-of-nontangential-maximal-with-aperture}
\int_{\mathbb{R}^n\times\mathbb{R}^m}
|M_{\phi,a,b}^{*}(f)(x,y)|^pdxdy\leq
C_{n,m}(1+a)^n(1+b)^m\int_{\mathbb{R}^n\times\mathbb{R}^m}
|M_{\phi}^{*}(f)(x,y)|^pdxdy.
\end{eqnarray}
This can be obtained by mimicking the proof in \cite[$\S 2.5$, Chapter 2]{St}. Observing that
$$
\frac{|f\ast\phi_{t,s}(x-u,y-v)|}{\displaystyle \Big(1+{|u|\over
t}\Big)^{N}\Big(1+{|v|\over {t+s}}\Big)^{N}}\leq
\sum_{k=0}^{\infty}\sum_{\ell=0}^{\infty}2^{(1-k)N}2^{(1-\ell) N}
|M_{\phi,2^{k+1},2^{\ell+1}}^{*}(f)(x,y)|$$ for all $u\in
\mathbb{R}^n, v\in \mathbb{R}^m$, $t,s>0$ and $N>0$, and using
(\ref{comparison-of-nontangential-maximal-with-aperture}), we then
get (\ref{tangential-maximal-leq-nontangential-maximal}) with
$$C_{N}^p=c_{n,m}\sum_{k=0}^{\infty}\sum_{\ell=0}^{\infty}(1+2^k)^n\cdot(1+2^{\ell})^m\cdot
2^{(1-k)N}\cdot2^{(1-\ell) N},$$ which is finite if $N>2(n\vee
m)$. The proof of the Lemma 4.3 is concluded.
\end{proof}

Next we recall the following lemma from \cite{St} which will be used to pass from one approximation of the identity to
another.
\begin{lemma}[{\cite[Lemma 2, $\S 1.3$]{St}}]
\label{lemma-of-stein-decom-of-Schwartz-function}
Suppose we are given $\phi$ and $\psi\in \mathcal{S}(\mathbb R^d)$ with
$\int_{\mathbb R^d}\phi=1$. Then there is a sequence
$\{\eta^{(k)}\}\subset\mathcal{S}(\mathbb R^d)$ so that
\begin{eqnarray}\label{decom-of-schwartz-function}
\psi=\sum_{k=0}^{\infty} \eta^{(k)}\ast \phi_{2^{-k}}
\end{eqnarray}
with $\eta^{(k)}\rightarrow 0$ rapidly, in the sense that whenever
$\|\cdot\|_{\alpha,\beta}$ is a seminorm and $M\geq0$ is fixed, then
$$\|\eta^{(k)}\|_{\alpha,\beta}=O(2^{-kM})\ \ \ \ \ \ \ \ {\rm as}\ k\rightarrow \infty. $$
\end{lemma}
From Lemma \ref{lemma-of-stein-decom-of-Schwartz-function}, we obtain the following estimate
\begin{eqnarray}\label{grand-maximal-function-leq-nontangential-function}
\|M_{\mathscr{F}}(f)\|_{1}\leq
C\|M_{\phi}^{*}(f)\|_{1}.
\end{eqnarray}
Indeed, for any
$\phi=\phi^{(1)}\ast_{\mathbb R^m}\phi^{(2)}\in\mathcal{S}_F(\mathbb{R}^n\times\mathbb{R}^m)$,
by \eqref{decom-of-schwartz-function} on $\phi^{(1)}$ and
$\phi^{(2)}$ we have
\begin{eqnarray*}
M_{\phi}(f)(x,y)&\leq& \sup_{t,s>0}\sum_{k=0}^{\infty}\sum_{\ell=0}^{\infty}
\Big|f\ast \big(\phi_{2^{-k}t}^{(1)}\ast_{\mathbb R^m}
\phi_{2^{-\ell}s}^{(2)}\big) \ast\big(\eta^{(1),(k)}_t\ast_{\mathbb R^m}
\eta^{(2),(\ell)}_s\big) \Big|\\
&\leq& M_{N}^{**}(f)(x,y)
\sup_{t,s>0}\sum_{k=0}^{\infty}\sum_{\ell=0}^{\infty}
\int_{\mathbb{R}^n\times\mathbb{R}^m} \Big(1+{{|u|} \over
{2^{-k}t}}\Big)^N\Big(1+{{|v|} \over {2^{-\ell}(t+s)}}\Big)^N\\
&&\times \big|\eta^{(1),(k)}_t\ast_{\mathbb R^m} \eta^{(2),(\ell)}_s(u,v)\big|
dudv \\
&\leq& CM_{N}^{**}(f)(x,y),
\end{eqnarray*}
where the last inequality holds if $\phi$ belongs to an appropriate
chosen $\mathscr{F}$. Thus
$$M_{\mathscr{F}}(f)(x,y)=\sup_{\phi\in \mathscr{F}}M^+_\phi(f)(x,y)\leq CM_{N}^{**}(f)(x,y)$$
for all $x\in\mathbb{R}^n,y\in\mathbb{R}^m$; taking $N>{2(n\vee
m)}$ as in (\ref{tangential-maximal-leq-nontangential-maximal})
yields (\ref{grand-maximal-function-leq-nontangential-function}).

Next, we will show that
\begin{eqnarray}\label{nontangential-maximal-leq-verticle-maximal-function}
\|M_{\phi}^{*}(f)\|_{1}\leq
C\|M^+_{\phi}(f)\|_{1}.
\end{eqnarray}
Let $\mathscr{F}$ be the same as in
(\ref{grand-maximal-function-leq-nontangential-function}) and for
any fixed $\lambda>0$, let
$$F=F_{\lambda}=\big\{(x,y):M_{\mathscr{F}}(f)(x,y)\leq \lambda M_{\phi}^{*}(f)(x,y)\big\}.$$
We prove (\ref{nontangential-maximal-leq-verticle-maximal-function})
by showing that, for any $q>0$,
\begin{eqnarray*}
\label{nontangential-maximal-leq-HL andverticle-maximal}
M_{\phi}^{*}(f)(x,y)\leq C\, M_s\Big( \big(M^+_{\phi}(f)\big)^q\Big)^{1\over q}\ \ \ \ \ \
\ {\rm for}\ (x,y)\in F,
\end{eqnarray*}
 where $M_s$ is the strong maximal function. Now for any $(x,y)$,
 there exists $(x_1,y_1,t,s)$ with $|x-x_1|<t,$ $|y-y_1|<t+s$ and $f\ast \phi_{t,s}(x_1,y_1)
 \geq {1\over 2}M_{\phi}^{*}(f)(x,y) $. Choose $r_1$ small and
 consider the ball centered at $x_1$ of radius $r_1t$, i.e. the
 points $u$ so that $|x_1-u|<r_1t$. We have that
 $$|f\ast \phi_{t,s}(x_1,y_1)-f\ast \phi_{t,s}(u,y_1)|\leq r_1t\sup_{|u-x_1|<r_1t}|\nabla_u f\ast \phi_{t,s}(u,y_1)|.$$
Similarly, choose $r_2$ small and
 consider the ball centered at $y_1$ of radius $r_2(t+s)$, i.e. the
 points $v$ so that $|y_1-v|<r_2(t+s)$. We have that
 $$|f\ast \phi_{t,s}(x_1,y_1)-f\ast \phi_{t,s}(x_1,v)|\leq r_2(t+s)
 \sup_{|v-y_1|<r_2(t+s)}|\nabla_v f\ast \phi_{t,s}(x_1,v)|.$$
 Combining the above two cases, we have
\begin{eqnarray*}
\lefteqn{  \hspace{-.1cm}|f\ast \phi_{t,s}(x_1,y_1)-f\ast
\phi_{t,s}(u,y_1)-f\ast
\phi_{t,s}(x_1,v)+f\ast \phi_{t,s}(u,v)| }\\
&\leq& C r_1t \cdot r_2(t+s)\ \
 \sup_{|u-x_1|<r_1t,\ |v-y_1|<r_2(t+s)}|\nabla_u\nabla_v f\ast
 \phi_{t,s}(u,v)|.
\end{eqnarray*}
However, $\frac{\displaystyle
\partial}{\displaystyle \partial u_i} f\ast \phi_{t,s}(u,v)= f\ast \widetilde{\phi}^{\ i}_{t,s}(u,v),$ where
\begin{eqnarray*}
\widetilde{\phi}^{\
i}_{t,s}(u,v)=\int_{\mathbb{R}^m}\frac{\displaystyle
\partial}{\displaystyle \partial u_i}\phi^{(1)}_t(u,v-w)\phi^{(2)}_s(w)dw
= \frac{\displaystyle 1}{\displaystyle
t}\int_{\mathbb{R}^m}\big(\frac{\displaystyle
\partial}{\displaystyle \partial u_i}\phi^{(1)}\big)_t(u,v-w)\phi^{(2)}_s(w)dw.
\end{eqnarray*}
And
$\frac{\displaystyle
\partial}{\displaystyle \partial v_j}f\ast \phi_{t,s}(u,v)= f\ast \overline{\phi}^{\
j}_{t,s}(u,v),$ where
$$\overline{\phi}^{\
j}_{t,s}(u,v)={1\over t}\int_{\mathbb{R}^m}\big(\frac{\displaystyle
\partial}{\displaystyle \partial
v_j}\phi^{(1)}\big)_t(u,v-w)\phi^{(2)}_s(w)dw \ \ \ \ {\rm if}\ t>s;
$$
$$\overline{\phi}^{\
j}_{t,s}(u,v)={1\over
s}\int_{\mathbb{R}^m}\phi^{(1)}_t(u,w)\big(\frac{\displaystyle
\partial}{\displaystyle \partial
v_j}\phi^{(2)}\big)_s(v-w)dw \ \ \ \ {\rm if}\ t\leq s.
$$
Note that the set of functions of the form $\widetilde{\phi}^{\
i}(x+h_1,y+h_2)$ and $\overline{\phi}^{\ j}(x+h_1,y+h_2)$,
$|h_1|\leq 1+r_1$, $|h_2|\leq 1+r_2$, $i=1,\ldots,n$,
$j=1,\ldots,m$, is a compact set in
$\mathcal{S}_F(\mathbb{R}^n\times\mathbb{R}^m)$, hence we have
$c\widetilde{\phi}^{\ i}(x+h_1,y+h_2)$ and $c\overline{\phi}^{\
j}(x+h_1,y+h_2)\in \mathscr{F}$, where $c$ is a constant independent
of $\phi$, $h_1$ and $h_2$. Thus $|f\ast \phi_{t,s}(x_1,y_1)-f\ast
\phi_{t,s}(u,y_1)|\leq cr_1M_{\mathscr{F}}(f)(x,y)\leq cr_1\lambda
M_{\phi}^{*}(f)(x,y)$, if $(x,y)\in F$.

By considering the case $t>s$,  if $(x,y)\in F$ then we can obtain that  $$|f\ast
\phi_{t,s}(x_1,y_1)-f\ast \phi_{t,s}(x_1,v)|\leq cr_2 \lambda
M_{\phi}^{*}(f)(x,y),$$  and that $$ |f\ast
\phi_{t,s}(x_1,y_1)-f\ast \phi_{t,s}(u,y_1)-f\ast
\phi_{t,s}(x_1,v)+f\ast \phi_{t,s}(u,v)| \leq C r_1 \cdot r_2\
 \lambda
M_{\phi}^{*}(f)(x,y).$$

So if we take $r_1$ and $r_2$ so small that $cr_1\lambda$,
$cr_2\lambda$, $cr_1r_2\lambda< 1/16,$  then we have
$$|f\ast \phi_{t,s}(u,v)|>{1\over 4}M_{\phi}^{*}(f)(x,y)\ \ \ \ \ \ {\rm for\ all\ } u\in B(x_1,r_1t)\ {\rm
and}\ v\in B(y_1,r_2t).
$$
 Thus we get that
\begin{eqnarray*}
 {1\over 4^q}|M_{\phi}^{*}(f)(x,y)|^q&\leq& {1\over
{|B(x_1,r_1t)|\times |B(y_1,r_2t)|}}\int_{B(x_1,(1+r_1)t)\times
B(y_1,(1+r_2)t)}|f\ast \phi_{t,s}(u,v)|^qdudv\\
&\leq& \Big({1+r_1 \over r_1}\Big)^n \Big({1+r_2 \over r_2}\Big)^m
M_s[(M^+_{\phi}(f))^q](x,y),
\end{eqnarray*}
which is
(\ref{nontangential-maximal-leq-verticle-maximal-function}).
Similarly, we can obtain this result when considering the case
$t\leq s$.

Then using the maximal theorem (for $M_s$) with $q<1$ leads to
\begin{eqnarray}\label{nontangential-maximal-leqverticle-maximal-on-F}
\int_F M_{\phi}^{*}(f)(x,y)dxdy\leq\!
C\!\int_{\mathbb{R}^n\times\mathbb{R}^m}\!\!\!\big(M_s[(M^+_{\phi}(f))^q](x,y)\big)^{1\over
q}dxdy \leq \!C\!
\int_{\mathbb{R}^n\times\mathbb{R}^m}\!\!\!M^+_{\phi}(f)(x,y)dxdy.\
\ \ \ \ \ \
\end{eqnarray}
Hence, to prove (\ref{nontangential-maximal-leq-verticle-maximal-function}), it suffices to prove that the left-hand side of
\eqref{nontangential-maximal-leq-verticle-maximal-function} is controlled by the left-hand side of \eqref{nontangential-maximal-leqverticle-maximal-on-F}.
So, we now claim that
\begin{eqnarray}\label{claim}
\int_{\mathbb{R}^n\times \mathbb{R}^m}
 M_{\phi}^{*}(f)(x,y) dxdy\leq 2 \int_F
     M_{\phi}^{*}(f)(x,y)dxdy.
\end{eqnarray}
To see this, observe that 
\begin{align*}
\int_{F^c}
 M_{\phi}^{*}(f)(x,y)dxdy\leq \lambda^{-1}\int_{F^c}
        M_{\mathscr{F}}(f)(x,y) dxdy\leq
{\bar C}\lambda^{-1}\int_{\mathbb{R}^n\times \mathbb{R}^m}
   M_{\phi}^{*}(f)(x,y) dxdy,
\end{align*}
   where the last inequality
follows from
(\ref{grand-maximal-function-leq-nontangential-function}) and $\bar C$ is an absolute constant. Recall that $\lambda$ is any fixed positive constant. Thus, by taking $\lambda \geq 2\bar C$,
we see that
\begin{align*}
\int_{\mathbb{R}^n\times \mathbb{R}^m}
 M_{\phi}^{*}(f)(x,y) dxdy &\leq \int_F M_{\phi}^{*}(f)(x,y) dxdy  + \int_{F^c }M_{\phi}^{*}(f)(x,y) dxdy   \\
 &\leq \int_F M_{\phi}^{*}(f)(x,y) dxdy  + {1\over2}\int_{\mathbb{R}^n\times \mathbb{R}^m }M_{\phi}^{*}(f)(x,y) dxdy,
 \end{align*}
which shows the claim (\ref{claim}).
This, together with
(\ref{nontangential-maximal-leqverticle-maximal-on-F}), yields
(\ref{nontangential-maximal-leq-verticle-maximal-function}).

We recall the result that if $P^{(1)}(x,y)$ is the Poisson
kernel on $\mathbb{R}^{n+m}$, then
\begin{eqnarray*}
\label{P1}
P^{(1)}(x,y)=\frac{\displaystyle c_{n+m}} {\displaystyle
(1+|x|^2+|y|^2)^{(n+m+1)/2}
}=\sum_{k=0}^{\infty}2^{-k}\phi^{(1),(k)}_{2^k}(x,y),
\end{eqnarray*}
where
$\{\phi^{(1),(k)}\}$ is a bounded collection of functions in
$\mathcal{S}(\mathbb{R}^{n+m})$. Similarly, if $P^{(2)}(z)$ is the
Poisson kernel on $\mathbb{R}^{m}$, then
\begin{eqnarray*}\label{P2}
P^{(2)}(z)=\frac{\displaystyle c_{m}} {\displaystyle
(1+|z|^2)^{(m+1)/2}
}=\sum_{\ell=0}^{\infty}2^{-\ell}\phi^{(2),(\ell)}_{2^{\ell}}(x,y),
\end{eqnarray*}
where $\{\phi^{(2),(\ell)}\}$ is a bounded collection of functions
in $\mathcal{S}(\mathbb{R}^{m})$. Then for the Poisson kernel
$P_{t,s}(x,y)$, we have that
\begin{eqnarray*}
P_{t,s}(x,y)=P^{(1)}_t\ast_{\mathbb R^m}
P^{(2)}_s(x,y)=\sum_{k=0}^{\infty}\sum_{\ell=0}^{\infty}2^{-k}2^{-\ell}\phi^{(1),(k)}_{2^kt}
\ast_{\mathbb R^m}\phi^{(2),(\ell)}_{2^{\ell }s}(x,y),
\end{eqnarray*}
where obviously, $\Big\{\phi^{(k),(\ell)}_{2^kt,2^{\ell
}s}\Big\}=\Big\{\phi^{(1),(k)}_{2^kt} \ast_{\mathbb R^m}\phi^{(2),(\ell)}_{2^{\ell }s}\Big\}
$ is a bounded collection of functions in
$\mathscr{D}_F(\mathbb{R}^{n}\times\mathbb{R}^m)$. Thus, we have
\begin{align*}
\|U^*\|_{1}&\leq
\sum_{k=0}^{\infty}\sum_{\ell=0}^{\infty}2^{-k}2^{-\ell}\big\|M^*_{\phi^{(k),(\ell)}_{2^kt,2^{\ell
}s}} f\big\|_{1}
\leq C\|M_{\mathscr{F}}(f)\|_{1}
\leq C\|M_{\Phi}^*(f)\|_{1}.
\end{align*}

We now prove
\begin{eqnarray*}
\|M^{*}_\phi(f)\|_{1} \leq C \|U^{*}\|_{1}.
\end{eqnarray*}
Following \cite[Chapter III, $\S$ 1.7]{St}, for the Poisson kernel $P_t^{(1)}(x,y)$, there exists
a functions $\eta^{(1)}$ defined on $(1,\infty)$ such that
$$\int_1^\infty \eta^{(1)}(s)ds=1, \quad{\rm and}\quad \int_1^\infty s^k\eta^{(1)}(s)ds=0,\quad k=1,2,\ldots.$$
We now set
$$\Phi^{(1)}(x,y):=\int_1^\infty \eta^{(1)}(t)P_t^{(1)}(x,y)dt.$$

Similarly,
 for the Poisson kernel $P_t^{(2)}(z)$, there exists
a functions $\eta^{(2)}$ defined on $(1,\infty)$ such that
$$\int_1^\infty \eta^{(2)}(s)ds=1, \quad{\rm and}\quad \int_1^\infty s^k\eta^{(2)}(s)ds=0,\quad k=1,2,\ldots.$$
We now set
$$\Phi^{(2)}(z):=\int_1^\infty \eta^{(1)}(s)P_s^{(2)}(z)ds.$$

Then we have $\Phi^{(1)}(x,y) \in\mathcal{S}(\mathbb{R}^{n+m})$ and $\Phi^{(2)}(z) \in\mathcal{S}(\mathbb{R}^{m})$. Moreover, we have
$$\int_{\mathbb{R}^{n+m}} \Phi^{(1)}(x,y)dxdy=\int_1^\infty \eta^{(1)}(t)dt=1$$
and
$$\int_{\mathbb{R}^{m}} \Phi^{(2)}(z)dz=\int_1^\infty \eta^{(2)}(s)ds=1.$$
Hence, define
$$ \widetilde{\Phi}(x,y)= \Phi^{(1)}\ast_{\mathbb{R}^m} \Phi^{(2)} (x,y),$$
then
 we obtain that
$$M^*_{\widetilde{\Phi}}(f)(x,y)\le
     U^*(x,y) \int_1^\infty \eta(t)^{(1)}dt \int_1^\infty \eta^{(2)}(s)ds= U^*(x,y).$$
As a consequence, we obtain that for arbitrary $\phi\in \mathcal{D}_F(\mathbb{R}^n\times \mathbb{R}^m)$,
$$ \|M^*_{\phi}(f)\|_{1}\leq C \| M^*_{\widetilde{\Phi}}(f)\|_{1} \le  \|U^*\|_{1}.$$

\subsection{The equivalence $\| U^{+} \|_{1}\approx \| M^{+}_\phi(f) \|_{1}$}

It is clear that $U^+(x)\le U^*(x)$ for $x\in\Bbb R^n$. From Section 4.2, we get that $\|U^{*}\|_{1}\lesssim \|M^{*}_{\phi}(f)\|_{1}$, which together with \eqref{nontangential-maximal-leq-verticle-maximal-function}, gives
$$\| U^{+} \|_{1}\lesssim \| M^{+}_{\phi}(f) \|_{1}.$$
On the other hand, from the estimates
 $\|U^{*}\|_{1}\lesssim \|U^{+}\|_{1}$
 and $\|M^{*}_{\Phi}(f)\|_{1}\lesssim \|U^{*}\|_{1}$, we get
$$\| M^{+}_{\phi}(f) \|_{1}\lesssim \|M^{*}_{\phi}(f)\|_{1}\lesssim \| U^{+} \|_{1}.$$

\newpage
\section{Atomic decompositions of flag Hardy spaces}

\subsection{Heat kernel and finite speed propagation}

Assume that $L$ is a non-negative self-adjoint second order differential operator on
$L^2(\mathbb R^n)$, whose heat kernel $h_t(x,y)$ of $e^{-tL}$ satisfies the
Gaussian upper bound:
\begin{align}\label{ge123}
|h_t(x,y)|\leq {C\over t^n} e^{-{|x-y|^2\over ct}}, \quad t>0,
\end{align}
where $c$ and $C$ are two positive constants independent of $x,y$ and $t$.

Let $E_{L}(\lambda)$ denote its spectral decomposition. Then, for every bounded
Borel function $F:[0,\infty)\to{\mathbb{C}}$, one defines the bounded operator
$F(L): L^2(\mathbb R^n)\to L^2(\mathbb R^n)$ by the formula
$$
F(L)=\int_0^{\infty}F(\lambda)\,dE_{L}(\lambda).
$$
In particular, the operator $\cos(t\sqrt{L})$ is then well-defined and bounded
on $L^2(\mathbb R^n)$. Moreover, it follows from \cite[Theorem 3]{CS} that if the corresponding
heat kernels $p_{t}(x,y)$ of $e^{-tL}$ satisfy Gaussian bounds \eqref{ge123}, then there exists a finite
positive constant $c_0$ such that the Schwartz
kernel $K_{\cos(t\sqrt{L})}$ of $\cos(t\sqrt{L})$ satisfies
\begin{eqnarray}\label{e3.12} \hspace{1cm}
{\rm supp}K_{\cos(t\sqrt{L})}\subset
\big\{(x,y)\in \Omega\times \Omega: |x-y|\leq c_0 t\big\}.
\end{eqnarray}
See also \cite{Si2}. By the Fourier inversion
formula, whenever $F$ is an even, bounded, Borel function with its Fourier transform
$\widehat{F}\in L^1(\mathbb{R})$, we can write $F(\sqrt{L})$ in terms of
$\cos(t\sqrt{L})$. More specifically,  we have
$$
F(\sqrt{L})=(2\pi)^{-1}\int_{-\infty}^{\infty}{\widehat F}(t)\cos(t\sqrt{L})\,dt,
$$
which, combined with \eqref{e3.12}, gives
\begin{eqnarray}\label{e3.13} \hspace{1cm}
K_{F(\sqrt{L})}(x,y)=(2\pi)^{-1}\int_{|t|\geq c_0^{-1}|x-y|}{\widehat F}(t)
K_{\cos(t\sqrt{L})}(x,y)\,dt,\qquad \forall\,x,y\in\Omega.
\end{eqnarray}

The following  result
(see \cite[Lemma~3.5]{HLMMY}) is useful for certain estimates later.

\begin{lemma}\label{lemma finite speed} Let $\varphi\in C^{\infty}_0(\mathbb R)$ be
even and satisfy $\mbox{\rm supp}\,\varphi \subset (-c_0^{-1}, c_0^{-1})$, where $c_0$ is
the constant in  \eqref{e3.12}. Let $\Phi$ denote the Fourier transform of
$\varphi$. Then for every $\kappa=0,1,2,\dots$, and for every $t>0$,
the kernel $K_{(t^2L)^{\kappa}\Phi(t\sqrt{L})}(x,y)$ of the operator
$(t^2L)^{\kappa}\Phi(t\sqrt{L})$, defined by spectral theory, satisfies
$$
{\rm supp}\ \! K_{(t^2L)^{\kappa}\Phi(t\sqrt{L})}(x,y) \subset
\Big\{(x,y)\in \Bbb R^n\times \Bbb R^n: |x-y|\leq t\Big\}.
$$
\end{lemma}

For $s>0$, we define
$$
{\Bbb F}(s)=\Big\{\psi:{\Bbb C}\to{\Bbb C}\ {\rm measurable}: \
|\psi(z)|\leq C {|z|^s\over ({1+|z|^{2s}})}\Big\}.
$$
Then for any non-zero function $\psi\in {\Bbb F}(s)$, we have
$\int_0^{\infty}|{\psi}(t)|^2\frac{dt}{t}<\infty$.
Denote by  $\psi_t(z)=\psi(tz)$. It follows from the spectral theory
in \cite{Yo} that, for any $f\in L^2(\Bbb R^n)$,
\begin{align*}
\bigg\{\int_0^{\infty}\|\psi(t\sqrt{L})f\|_{L^2(\Bbb R^n)}^2{dt\over t}\bigg\}^{1/2}
&=\bigg\{\int_0^{\infty}\big\langle\,\overline{ \psi}(t\sqrt{L})\,
    \psi(t\sqrt{L})f, f\big\rangle_{L^2(\Bbb R^n)} {dt\over t}\bigg\}^{1/2}\\
&=\bigg\{\bigg\langle \int_0^{\infty}|\psi|^2(t\sqrt{L}) {dt\over t}f, f\bigg\rangle_{L^2(\Bbb R^n)}\bigg\}^{1/2}\\
&\leq \kappa \|f\|_{L^2(\Bbb R^n)},
\end{align*}
where $\kappa=C_L\big\{\int_0^{\infty}|{\psi}(t)|^2 {dt/t}\big\}^{1/2}$.

\subsection{Atomic decomposition for $H^1_F(\mathbb R^n\times \mathbb R^m)$.}

\begin{definition}\label{def-of-S-function-Sf L}
Let $\triangle^{(1)}$ be the Laplacian on $\Bbb R^{n+m}$ and $\triangle^{(2)}$ be the Laplacian on $\Bbb R^m$.
For $f\in L^1(\mathbb{R}^{n+m})$, the Lusin area integral of $f$ associated with these Laplacians is defined by
\begin{equation}\label{esf}
\begin{aligned}
&S_{F,\triangle^{(1)},\triangle^{(2)}}(f)(x_1,x_2)= \bigg( \int_{\Bbb R^{n+1}_+}\int_{\Bbb R^{m+1}_+} \chi_{t_1,t_2}(x_1-y_1,x_2-y_2)  \\
&\hskip2cm\times \big|(t_1^2\triangle^{(1)}e^{-t_1^2\triangle^{(1)}}\otimes_2t_2^2\triangle^{(2)}e^{-t_2^2\triangle^{(2)}} )f(y_1,y_2)\big|^2\
{dy_1 \ \! dt_1\over t_1^{n+m+1}}{dy_2 \ \! dt_2\over t_2^{m+1}}\bigg)^{1/2},
\end{aligned}
\end{equation}
where $\chi_{t_1,t_2}(x_1,x_2):=\chi_{t_1}^{(1)}\ast_{\Bbb R^m} \chi_{t_2}^{(2)}(x_1,x_2)$,
$\chi_{t_1}^{(1)}(x_1,x_2):=\chi^{(1)}({x_1\over t_1},{x_2\over t_1})$ and
$\chi_{t_2}^{(2)}(z):=\chi^{(2)}({z\over t_2})$, with $\chi^{(1)}(x_1,x_2)$
and $\chi^{(2)}(z)$  the indicator function of the unit balls of
$\mathbb{R}^{n+m}$ and $\mathbb{R}^{m}$, respectively.
\end{definition}

Based on the discrete reproducing formula as in Theorem \ref{dcrf} and the Plancherel--P\'olya type inequalities
as in Theorem \ref{pp-inequality1}, we can obtain the estimate
\begin{align}\label{SFLaplace SF}
\|S_{F,\triangle^{(1)},\triangle^{(2)}}(f)\|_{1}\lesssim \|S_F(f)\|_1.
\end{align}
Since this argument is similar to the estimates as in Section 2.3, we omit it here.

We now define the flag Hardy space $H^1_{F,\triangle^{(1)},\triangle^{(2)}}(\mathbb R^n\times \mathbb R^m)$
associated with $\triangle^{(1)}$ and $\triangle^{(2)}$ as follows.
\begin{definition}\label{def-of-Hardy-space L}
Let all the notation be the same as above. We define
$$H^1_{F,\triangle^{(1)},\triangle^{(2)}}(\mathbb R^n\times \mathbb R^m):=\{ f\in L^1(\Bbb R^{n+m}): \|S_{F,\triangle^{(1)},\triangle^{(2)}} f\|_{1}<\infty \}$$
with the norm
$$
\|f\|_{H^1_{F,\triangle^{(1)},\triangle^{(2)}}(\mathbb{R}^{n+m})}:=\|S_{F,\triangle^{(1)},\triangle^{(2)}}f\|_{1}.
$$
\end{definition}

We will later show that this Hardy space $H^1_{F,\triangle^{(1)},\triangle^{(2)}}(\mathbb R^n\times \mathbb R^m)$ is equivalent to $H_F^1(\mathbb R^n\times \mathbb R^m)$ as in Definition \ref{def-of-hardy-by-han}.

We are now recalling the  atomic Hardy space $H^1_{F,at,M}(\mathbb R^n\times \mathbb R^m)$ as in Definition \ref{def-of-atomic-product-Hardy-space},
and we will later prove that $H^1_{F,at,M}(\mathbb R^n\times \mathbb R^m)$ is equivalent to $H^1_{F,\triangle^{(1)},\triangle^{(2)}}(\mathbb R^n\times \mathbb R^m)$ above i.e., Theorem \ref{theorem of Hardy space atomic decom again} below. Then eventually we show that they are both equivalent to the space
$H^1_F(\mathbb R^n\times \mathbb R^m)$ via square functions, i.e., Theorem \ref{theorem of Hardy space atomic decom}.

For the convenience of the readers, we repeat the definition of $H^1_{F,at,M}(\mathbb R^n\times \mathbb R^m)$ here.

\begin{definition}\rm\label{def-of-atomic-product-Hardy-space again}
Let $M>m/2$. The Hardy spaces $H^1_{F,at,M}(\mathbb{R}^{n}\times\mathbb{R}^{m})$
is defined as follows.  For $f\in L^2(\mathbb{R}^{n+m}),$ we   say that $f=\sum_{j}\lambda_ja_j$ is an atomic
$(1, 2, M)$-representation of $f$ if $\{\lambda_j\}_{j=0}^\infty\in
\ell^1$, each $a_j$ is a $(1, 2, M)$-atom, and the sum converges in
$L^2(\mathbb{R}^{n+m})$. The space $\Bbb H^1_{F,at,M}(\mathbb R^n\times \mathbb R^m)$ is defined to be
$$\Bbb H^1_{F,at,M}(\mathbb{R}^{n}\times\mathbb{R}^{m})=\{f\in L^2(\mathbb{R}^{n+m}): f\ \mbox{has an atomic $(1,2,M)$-representation}\}$$
with the norm
$$\|f\|_{\Bbb H^1_{F,at,M}(\mathbb{R}^{n}\times\mathbb{R}^{m})}:=\inf\bigg\{\sum_{j=0}^\infty |\lambda_j| :f=\sum_{j=0}^\infty \lambda_ja_j\ \ \text{is an atomic $(1,2,M)$-representation}\bigg\}.$$
The atomic Hardy space $H^1_{F,at,M}(\mathbb{R}^{n}\times\mathbb{R}^{m})$ is defined as the completion of $\Bbb H^1_{F,at,M}(\mathbb{R}^{n}\times\mathbb{R}^{m})$ with respect to this norm.
\end{definition}

\begin{theorem}\label{theorem of Hardy space atomic decom again}
Suppose that  $M>m/2$. Then
$$
H^1_{F,\triangle^{(1)},\triangle^{(2)}}(\mathbb R^n\times \mathbb R^m)=H^1_{F,at,M}(\mathbb{R}^{n}\times\mathbb{R}^{m}).
$$
Moreover,
$$
\|f\|_{H^1_{F,\triangle^{(1)},\triangle^{(2)}}(\mathbb R^n\times \mathbb R^m)}\approx
\|f\|_{\Bbb H^1_{F,at,M}(\mathbb{R}^{n}\times\mathbb{R}^{m})},
$$
where the implicit constants depend only on $M, n$ and
$m$.
\end{theorem}

\subsection{Proof of  the atomic decomposition}  

We now proceed to the proof of Theorem \ref{theorem of Hardy space atomic decom again}. The basic strategy
is as follows: by density, it is enough to show that
$$\Bbb H^1_{F,at,M}(\mathbb R^n\times \mathbb R^m)= H^1_{F,\triangle^{(1)},\triangle^{(2)}}(\mathbb R^n\times \mathbb R^m) \cap L^2(\mathbb R^{n+m}) \qquad\text{for}\  M>m/2,$$
with equivalent of norms. The proof of this proceeds in two steps.

\medskip\noindent
{\bf Step 1.} \ ${\Bbb H}^1_{F,at,M}(\mathbb{R}^{n}\times\mathbb{R}^{m})\subset H^1_{F,\triangle^{(1)},\triangle^{(2)}}(\mathbb R^n\times \mathbb R^m)\cap L^2(\mathbb{R}^{n+m})$ for $M>m/2$.

\medskip\noindent
{\bf Step 2.} \ $H^1_{F,\triangle^{(1)},\triangle^{(2)}}(\mathbb R^n\times \mathbb R^m)\cap L^2(\mathbb{R}^{n+m})\subset {\Bbb H}^1_{F,at,M}(\mathbb{R}^{n}\times\mathbb{R}^{m})$ for every $M\in{\mathbb N}$.
\medskip

The conclusion of Step 1 is an immediate consequence of the following lemma and proposition.

\begin{lemma}\label{lemma-of-T-bd-on Hp}
Fix $M\in{\mathbb N}$. Assume that $T$ is a linear
operator or a nonnegative sublinear operator, satisfying the
weak-type {\rm (2,2)} bound
\begin{eqnarray*}
\big|\{x\in \Bbb R^n\times \Bbb R^m: |Tf(x)|>\eta\} \big|\le C_T\eta^{-2}\|f\|_2^2,\qquad \forall\ \eta>0.
\end{eqnarray*}
If there is an absolute constant $C>0$ such that
\begin{eqnarray}\label{e4.111}
\|Ta\|_1\leq C\qquad \text{for every $(1, 2, M)$-atom $a$},
\end{eqnarray}
then $T$ is bounded from
${\Bbb H}^1_{F,at,M}(\Bbb R^n\times \Bbb R^m)$ to $L^1(\Bbb R^{n+m})$ and
$$\|Tf\|_1\leq C\|f\|_{{\Bbb H}^1_{F,at,M}(\Bbb R^n\times \Bbb R^m)}.$$
Consequently, by density, $T$ extends to a bounded operator from
$H^1_{F,at,M}(\Bbb R^n\times \Bbb R^m)$ to $L^1(\Bbb R^{n+m})$.
\end{lemma}

\begin{proof}
Given $f \in \mathbb{H}^1_{F,at,M}(\mathbb{R}^{n}\times\mathbb{R}^{m})$.
Then $f = \sum_{j} \lambda_j a_j$
is an atomic $(1, 2, M)$-representation such that

$$
\|f\|_{ \mathbb{H}^1_{F,at,M}(\mathbb{R}^{n}\times\mathbb{R}^{m})} \approx  \sum_{j=0}^\infty |\lambda_j| .
$$

\noindent
 Since the sum converges in $L^2$
(by definition), and since $T$ is of weak-type $(2,2)$, we have that
at almost every point,

\begin{equation}\label{eq4.44}|T(f)| \leq \sum_{j=0}^\infty |\lambda_j| \,|T(a_j)|.
\end{equation}

\noindent
Indeed, for every $\eta >0$, we have that, if $f^N:= \sum_{j>N} \lambda_j a_j$, then,

 \begin{eqnarray*}
\big|\ \{x: |Tf(x)| - \sum_{j=0}^\infty |\lambda_j| \,|T a_j(x)| >\eta\}\big|\,&\leq& \limsup_{N\to \infty}
\big|  \{x:  |Tf^N(x)|>\eta\}\big|\\
&\leq& \,C_T\,\,\eta^{-2} \,\limsup_{N\to \infty} \|f^N\|_2^2 =0,
\end{eqnarray*}
from which (\ref{eq4.44}) follows.
In turn, (\ref{eq4.44}) and (\ref{e4.111}) imply the desired $L^1$ bound for $Tf$.
\end{proof}

We now provide the key Proposition of {\bf Step 1}.
\begin{prop}\label{leAtom}
Let $S_{F,\triangle^{(1)},\triangle^{(2)}}$ be the square function defined by \eqref{esf} and $M>m/2$. Then
$$
\|S_{F,\triangle^{(1)},\triangle^{(2)}}a\|_1\leq C\qquad \text{for every $(1, 2, M)$-atom $a$},
$$
where $C$ is a positive constant independent of $a$.
\end{prop}

By Proposition \ref{leAtom}, we may apply Lemma \ref{lemma-of-T-bd-on Hp} with $T=S_{F,\triangle^{(1)},\triangle^{(2)}}$ to obtain
$$
\|f\|_{H^1_{F,\triangle^{(1)},\triangle^{(2)}}(\Bbb R^n\times \Bbb R^m)}=\|S_{F,\triangle^{(1)},\triangle^{(2)}} f\|_1\leq
C\|f\|_{{\Bbb H}^1_{F,at,M}(\Bbb R^n\times \Bbb R^m)}$$
and Step 1 follows.

Suppose $\Omega\subset {\Bbb R}^{n}\times{\Bbb R}^{m}$
is open of finite measure. Denote by $m(\Omega)$ the maximal dyadic
subrectangles of $\Omega$. Let $m_1(\Omega)$  denote those dyadic
subrectangles $R\subseteq \Omega, R=I\times J$ that are maximal in
the $x_1$ direction. In other words if $S=I'\times J\supseteq R$ is
a dyadic subrectangle of $\Omega$, then $I=I'.$ Define $m_2(\Omega)$
similarly.   Let
$${\widetilde \Omega}=\big\{x\in {\mathbb R}^{n}\times\mathbb{R}^{m}:
M_s(\chi_{\Omega})(x)>{1\over 2}\big\},
$$
where $M_s$ is the strong maximal operator on ${\mathbb R}^{n}\times\mathbb{R}^{m}$  defined as in \eqref{strong maximal}.

\noindent For any $R=I\times J\in m_1(\Omega)$, we  set
$\gamma_1(R)=\gamma_1(R, \Omega)=\sup{|l|\over |I|},$ where the
supremum is taken over all dyadic intervals $l: I\subset l$ so that
$l\times J\subset {\widetilde \Omega}$. Define $\gamma_2$ similarly.
Then Journ\'e's lemma, (in one of its forms) says, for  any
$\delta>0$,
\begin{eqnarray*}
\sum_{R\in m_2(\Omega)} |R|\gamma_1^{-\delta}(R)\leq
c_{\delta}|\Omega|
 \ \ \ {\rm and}\ \ \
 \sum_{R\in m_1(\Omega)} |R|\gamma_2^{-\delta}(R)\leq c_{\delta}|\Omega|
 \end{eqnarray*}
for some  $c_{\delta}$ depending only on $\delta$, not on
$\Omega.$

\begin{proof}[Proof of Proposition \ref{leAtom}]
Given any $(1,2,M)$-atom $a$, suppose that
 $
a=\sum\limits_{R\in m(\Omega)} a_R
$  is supported in an open set $\Omega$ with finite measure.
For any $R=I \times J\in m(\Omega)$, let $\widetilde{I}$ be the
biggest dyadic cube containing  $I$, so that $\widetilde{I}\times
J\subset\widetilde{\Omega}$, where
$\widetilde{\Omega}=\{x\in\mathbb{R}^{n}\times\mathbb{R}^{m}:\
M_s(\chi_{\Omega})(x)>1/2\}$. Next, let $\widetilde{J}$ be the
biggest dyadic cube containing $J$, so that $\widetilde{I}\times
\widetilde{J}\subset\widetilde{\widetilde{\Omega}}$, where
$\widetilde{\widetilde{\Omega}}=\{x\in\mathbb{R}^{n}\times\mathbb{R}^{m}:\
M_s(\chi_{\widetilde{\Omega}})(x)>1/2\}$. Now let $\widetilde{R}$ be
the 100-fold dilate of $\widetilde{I}\times \widetilde{J}$ concentric with
$\widetilde{I}\times \widetilde{J}$. Clearly, an application of the
strong maximal function theorem shows that
$\big|\cup_{R\subset\Omega} \widetilde{R}\big|\leq
C|\widetilde{\widetilde{\Omega}}|\leq C|\widetilde{\Omega}|\leq
C|\Omega|$. From property (iii) of the
$(1,2,M)$-atom,

\begin{equation}\label{SL alpha uniformly bd on inside of Omega}
\begin{aligned}
\int_{\cup \widetilde{R}}|S_{F,\triangle^{(1)},\triangle^{(2)}}(a)(x_1,x_2)|dx_1dx_2
&\leq  |\cup\widetilde{R} |^{1/2}\|S_{F,\triangle^{(1)},\triangle^{(2)}}(a)\|_2\\
 &\leq  C |\Omega|^{1/2}\|a\|_2\\
 &\leq    C |\Omega|^{1/2}|\Omega|^{-1/2}
 \leq  C.
 \end{aligned}
\end{equation}
We now prove

\begin{eqnarray}\label{SL alpha uniformly bd on outside of Omega}
\int_{(\bigcup \widetilde{R})^c}|S_{F,\triangle^{(1)},\triangle^{(2)}}(a)(x_1,x_2)|dx_2dx_1\leq C.
\end{eqnarray}
From the definition of $a$, we write

\begin{align}\label{DE}
&\int_{(\bigcup \widetilde{R})^c}|S_{F,\triangle^{(1)},\triangle^{(2)}}(a)(x_1,x_2)|dx_2dx_1\\
&\leq \sum_{R\in m(\Omega) } \int_{\widetilde{R}^c}|S_{F,\triangle^{(1)},\triangle^{(2)}}(a_R)(x_1,x_2)|dx_2dx_1\nonumber\\
&\leq \sum_{R\in m(\Omega) } \int_{(100\widetilde{I})^c\times\mathbb{R}^{m}} |S_{F,\triangle^{(1)},\triangle^{(2)}}(a_R)(x_1,x_2)|dx_2dx_1\nonumber\\
&+ \sum_{R\in m(\Omega) } \int_{\mathbb{R}^{n}\times (100\widetilde{J})^c}  |S_{F,\triangle^{(1)},\triangle^{(2)}}(a_R)(x_1,x_2)|dx_2dx_1\nonumber\\
&= \textrm{I}+\textrm{II}.\nonumber
\end{align}
For the term $\textrm{I}$, we have
\begin{eqnarray*}
\int_{(100\widetilde{I})^c\times\mathbb{R}^{m}}|S_{F,\triangle^{(1)},\triangle^{(2)}}(a_R)(x_1,x_2)|dx_2dx_1
 &  =&
\int_{(100\widetilde{I})^c\times 100J} |S_{F,\triangle^{(1)},\triangle^{(2)}}(a_R)(x_1,x_2)|dx_2dx_1 \\[2pt]
&&+  \int_{(100\widetilde{I})^c\times (100J)^c} |S_{F,\triangle^{(1)},\triangle^{(2)}}(a_R)(x_1,x_2)|dx_2dx_1 \\[2pt]
&=& \textrm{I}_1+\textrm{I}_2.
\end{eqnarray*}
 Let us first estimate the term $\textrm{I}_1$.  Set $a_{R,2}=(1\!\!1_1\otimes_2 ({\triangle^{(2)}})^M)b_R$, that is,
$a_R= (({\triangle^{(1)}})^M \otimes_2 1\!\!1_2)a_{R,2}.$
Using H\"older's inequality,
\begin{eqnarray}\label{estimate D1}
\textrm{I}_1 &\leq& C |J|^{{1/2}} \int_{(100\widetilde{I})^c }\Big( \int_{100J}|S_{F,\triangle^{(1)},\triangle^{(2)}}(a_R)(x_1,x_2)|^2dx_2\Big)^{1/2}dx_1.
\end{eqnarray}

Hence, from the definition of $S_{F,\triangle^{(1)},\triangle^{(2)}}$, we have that

\begin{align*}
&\int_{100J}|S_{F,\triangle^{(1)},\triangle^{(2)}}(a_R)(x_1,x_2)|^2dx_2 \\
&\leq \int_{100J}  \int_{\Bbb R^{n+1}_+}\int_{{\Bbb R^{m+1}_+}}\ \int_{\Bbb R^m}\chi^{(1)}_{t_1}(x_1-y_1,x_2-z_2)\chi^{(2)}_{t_2}(z_2-y_2)dz_2 \\
&\qquad\times     \big|(t_2^2\triangle^{(2)}e^{-t_2^2\triangle^{(2)}})\otimes_2\big((t_1^2\triangle^{(1)}e^{-t_1^2\triangle^{(1)}})a_{R}(y_1,\cdot)\big)(y_2)\big|^2{dy_2dt_2\over t_2^{m+1}}{dy_1dt_1\over t_1^{n+m+1}} dx_2 \\
&\leq  \int_{\Bbb R^{n+1}_+}\ \int_{\Bbb R^m}\ \int_{100J}\chi^{(1)}_{t_1}(x_1-y_1,x_2-z_2)dx_2  \\
&\quad\times  \int_{{\Bbb R^{m+1}_+}}    \chi^{(2)}_{t_2}(z_2-y_2) \big|(t_2^2\triangle^{(2)}e^{-t_2^2\triangle^{(2)}})\otimes_2\big((t_1^2\triangle^{(1)}e^{-t_1^2\triangle^{(1)}})a_{R}(y_1,\cdot)\big)(y_2)\big|^2{dy_2dt_2\over t_2^{m+1}}dz_2\ {dy_1dt_1\over t_1^{n+m+1}}  \\
&\leq  \int_0^\infty\int_{|x_1-y_1|\le t_1}\\
&\qquad \int_{\Bbb R^m}
  \int_0^\infty\int_{|z_2-y_2|\le t_2}     \big|(t_2^2\triangle^{(2)}e^{-t_2^2\triangle^{(2)}})\otimes_2\big((t_1^2\triangle^{(1)}e^{-t_1^2\triangle^{(1)}})a_{R}(y_1,\cdot)\big)(y_2)\big|^2{dy_2dt_2\over t_2^{m+1}}dz_2\ {dy_1dt_1\over t_1^{n+1}}  \\
&\leq  \int_0^\infty\int_{|x_1-y_1|\le t_1} \ \int_{\mathbb R^m}      \big| \big((t_1^2\triangle^{(1)}e^{-t_1^2\triangle^{(1)}})a_{R}(y_1,\cdot)\big)(z_2)\big|^2dz_2\ {dy_1dt_1\over t_1^{n+1}},
\end{align*}
where the fourth inequality follows from the Littlewood--Paley $L^2$ estimate of the area function with respect to $\Delta^{(2)}$.
 We then split the range of $t_1$ according to the side-length of $I$  to obtain
\begin{align*}
&\int_{100J}|S_{F,\triangle^{(1)},\triangle^{(2)}}(a_R)(x_1,x_2)|^2dx_2 \\
&\leq C  \int_{\mathbb R^m}  \int_{0}^{\ell(I)}  \int_{|x_1-y_1|<t_1} \\
&\quad\times    \Big[\int_{10I}\int_{10J}  t_1^{-n-m} \exp\Big(-{|y_1-u_1|^2+|z_2-u_2|^2\over ct_1^2}\Big)
|a_{R}(u_1,u_2)|du_1du_2\Big]^2\ {dy_1dz_2 \,dt_1\over t_1^{n+1}}  \\
&+ C  \int_{\mathbb R^m}  \int_{\ell(I)}^\infty  \int_{|x_1-y_1|<t_1}    \big|  \big(t_1^2\triangle^{(1)}\big)^{M+1}e^{-t_1^2\triangle^{(1)}}a_{R,2}(y_1,z_2)\big|^2\ {dy_1dz_2 \,dt_1\over t_1^{n+4M+1}}  \\
 &=:   D_{1} (a_R)(x_1) +D_{2} (a_R)(x_1),
\end{align*}
and the  inequality follows from the kernel estimate of $t_1^2\triangle^{(1)}e^{-t_1^2\triangle^{(1)}}$ and from the property of the atom $a_R$ that  $a_R= (({\triangle^{(1)}})^M \otimes_2 1\!\!1_2)a_{R,2}$.

Let us estimate the term $D_{1} (a_R)(x_1)$. Note that for $x_1\not\in 100\widetilde{I},
$ $0<t_1<\ell(I)$, $|x_1-y_1|<t_1$ and $u_1\in 10I$, then $|y_1-u_1|\geq |x_1-x_I|/2$, where $x_I$ denotes the center of the cube $I$.
Hence
\begin{align*}
  D_{1} (a_R)(x_1)
&\leq C \int_0^{\ell(I)} \int_{|x_1-y_1|<t_1}dy_1\cdot \ t_1^{-2n} \exp\Big(-{|x_1-x_I|^2\over 2ct_1^2}\Big)\\
 &\qquad\times \int_{\mathbb{R}^{m}}\Bigg[\int_{\Bbb R^m}t_1^{-m}\exp\Big(-{|z_2-u_2|^2\over ct_1^2}\Big)
\Big(\int_{10I}|a_{R}(u_1,u_2)|du_1\Big)du_2\Bigg]^2  dz_2 {dt_1\over t_1^{n+1}}\\
&\leq C \int_0^{\ell(I)} t_1^n\cdot \ t_1^{-2n} \exp\Big(-{|x_1-x_I|^2\over 2ct_1^2}\Big) \int_{\mathbb{R}^{m}} \Big[M_2\Big(
\int_{10I}|a_{R}(u_1,\cdot)|du_1\Big)(z_2)\Big]^2  dz_2 {dt_1\over t_1^{n+1}}\\
&\leq C \int_0^{\ell(I)}  t_1^{-2n} \exp\Big(-{|x_1-x_I|^2\over 2ct_1^2}\Big) \int_{\mathbb{R}^{m}}\Big(
\int_{10I}|a_{R}(u_1,z_2)|du_1 \Big)^2 dz_2 {dt_1\over t_1}\\
&\leq C \int_0^{\ell(I)} t_1^{-2n} \exp\Big(-{|x_1-x_I|^2\over 2ct_1^2}\Big) |I|\int_{\mathbb{R}^{m}}
\int_{10I}|a_{R}(u_1,z_2)|^2du_1 dz_2 {dt_1\over t_1}\\
 &\leq  C|I|\int_0^{\ell(I)} t_1^{-2n} \exp\Big(-{|x_1-x_I|^2\over 2ct_1^2}\Big) {dt_1\over t_1} \|a_{R}\|^2_2,
\end{align*}
where in the second inequality, $M_2$ denotes the Hardy--Littlewood maximal function on $\mathbb R^m$.
We then use the fact that $e^{-s}\leq Cs^{-k}$  for any $k>0$ to obtain
\begin{align*}
  D_{1} (a_R)(x_1)
&\leq C|I|
 \int_0^{\ell(I)} t_1^{-2n-1} \Big({ t_1\over |x_1-x_I|}\Big)^{2(n +{1\over 2} )}
   {dt_1 } \|a_{R}\|^2_2\\
&\leq C|I|
 {\ell(I) \over |x_1-x_I|^{2(n +{1\over 2})}}
   \|a_{R}\|^2_2\\
&\leq C  { |I| \ell(I)\over |x_1-x_I|^{2n+1}}
\|a_{R}\|^2_2.
\end{align*}
 In order to estimate the second term $D_{2} (a_R)(x_1)$, we first note that
$$\text{supp}\ a_{R,2}\subset \text{supp}\  (1\!\!1_1\otimes_2 ({\triangle^{(2)}})^M)b_R\subset 10R=10(I\times J).$$
Next we  observe that
  for $x_1\not\in 100\widetilde{I},
$ $\ell(I)\leq t_1 <|x_1-x_I|/4$, $|x_1-y_1|<t_1$ and $u_1\in 10I$, then $|y_1-u_1|\geq  |x_1-x_I|/4$.
Hence, from the kernel estimate of $(t_1^2\triangle^{(1)})^{M+1}e^{-t_1^2\triangle^{(1)}}$ and following similar estimates as in $D_{1} (a_R)(x_1)$ via the Hardy--Littlewood maximal function $M_2$, we split the range of $t_1$ according to $|x_1-x_I|/4$ and continue the estimate of  $D_{2} (a_R)(x_1)$ as follows.
\begin{align*}
 D_{2} (a_R)(x_1)
 &\leq  C|I|\int_{\ell(I)}^{|x_1-x_I|\over4} t_1^{-2n} \exp\Big(-{|x_1-x_I|^2\over 8ct_1^2}\Big) {dt_1\over t_1^{1+4M}} \|a_{R,2}\|^2_2\\
 &\quad+ C|I|\int_{|x_1-x_I|\over4}^{\infty} t_1^{-2n} \exp\Big(-{|x_1-x_I|^2\over 8ct_1^2}\Big) {dt_1\over t_1^{1+4M}} \|a_{R,2}\|^2_2\\
&\leq  C|I| \bigg(
 \int_{\ell(I)}^{\infty} t_1^{-2n-1-4M} \Big({ t_1\over |x_1-x_I|}\Big)^{2(n+2M -{1\over 2})} {dt_1 }\\
&\hskip3cm+  \int_{|x_1-x_I|\over 4}^{\infty} t_1^{-2n-1-4M}    dt_1 \bigg)
 \|a_{R,2}\|^2_2\\[2pt]
&\leq C  {|I|\ell(I) \over |x_1-x_I|^{2n+1}} \ell(I)^{-4M}\ell(J)^{-4M}
\|(1\!\!1_1\otimes_2
(\ell(J)^2{\triangle^{(2)}})^M)b_{R}\|^2_2,
\end{align*}
where in the second inequality we use the condition that $|x_1-x_I|>\ell(I)$, and use the fact that  $e^{-s}\leq Cs^{-k}$  for any $k>0$ for the first term and
the fact that $ t_1^{-2n} \exp\big(-{|x_1-x_I|^2\over 8ct_1^2}\big) \leq C$ when $t_1\geq{|x_1-x_I|\over 4}$.

Combining the estimates of $D_{1} (a_R)(x_1)$ and  $D_{2} (a_R)(x_1)$, we obtain
\begin{align}\label{estimate D11}
 &\int_{100J}|S_{F,\triangle^{(1)},\triangle^{(2)}}(a_R)(x_1,x_2)|^2dx_2\\
 &\lesssim   { |I| \ell(I)\over |x_1-x_I|^{2n+1}}
\|a_{R}\|^2_{L^2(\mathbb{R}^{n}\times\mathbb{R}^{m})}\nonumber\\
&\quad+ {|I|\ell(I) \over |x_1-x_I|^{2n+1}} \ell(I)^{-4M}\ell(J)^{-4M}
\|(1\!\!1_1\otimes_2
(\ell(J)^2{\triangle^{(2)}})^M)b_{R}\|^2_2.\nonumber
\end{align}
Putting \eqref{estimate D11} into the term $\textrm{I}_1$ in (\ref{estimate D1}),
  we have
\begin{align*}
 \textrm{I}_1
 &\lesssim |R|^{1/2} \int_{(100\widetilde{I})^c }   {  \ell(I)^{1\over2}\over |x_1-x_I|^{n+1/2}} dx_1
\|a_{R}\|_2\\
 &\quad+|R|^{1/2} \int_{(100\widetilde{I})^c } {\ell(I)^{1/2 } \over |x_1-x_I|^{n+1/2}}  dx_1  \ell(I)^{-2M}\ell(J)^{-2M}
\|(1\!\!1_1\otimes_2
(\ell(J)^2{\triangle^{(2)}})^M)b_{R}\|_2\\
&\lesssim  |R|^{1/2} \gamma_1(R)^{-1/2} \|a_{R}\|_2
\\
&\quad+|R|^{1/2} \gamma_1(R)^{-1/2}  \ell(I)^{-2M}\ell(J)^{-2M}
\|(1\!\!1_1\otimes_2
(\ell(J)^2{\triangle^{(2)}})^M)b_{R}\|_2.
\end{align*}
Now we turn to estimate the term $\textrm{I}_2$.

One can write

\begin{align*}
\textrm{I}_2&=\int_{(100\widetilde{I})^c\times (100J)^c} |S_{F,\triangle^{(1)},\triangle^{(2)}}(a_R)(x_1,x_2)|dx_1dx_2 \\
   &\leq\sum_{j_1=\tilde j}^\infty\sum_{j_2=6}^\infty\int_{|x_1-x_I|\approx 2^{j_1}\ell(I)} \int_{|x_2-x_J|\approx 2^{j_2}\ell(J)}|S_{F,\triangle^{(1)},\triangle^{(2)}}(a_R)(x_1,x_2)|dx_2dx_1 \\
   &\le \sum_{j_1=\tilde j}^\infty\sum_{j_2=6}^\infty (2^{j_1}\ell(I))^{n/2}(2^{j_2}\ell(J))^{m/2} \\
   &\qquad\times \bigg(\int_{|x_1-x_I|\approx 2^{j_1}\ell(I)} \int_{|x_2-x_J|\approx 2^{j_2}\ell(J)}|S_{F,\triangle^{(1)},\triangle^{(2)}}(a_R)(x_1,x_2)|^2dx_2dx_1\bigg)^{1/2},
\end{align*}
where $\tilde j$ is the smallest integer such that
$2^{\tilde j} I \cap (100\tilde I)^c\not=\emptyset$. We consider the four cases.
\begin{align*}
&\int_{|x_1-x_I|\approx 2^{j_1}\ell(I)} \int_{|x_2-x_J|\approx 2^{j_2}\ell(J)}|S_{F,\triangle^{(1)},\triangle^{(2)}}(a_R)(x_1,x_2)|^2dx_1dx_2 \\
&= \int_{|x_1-x_I|\approx 2^{j_1}\ell(I)} \int_{|x_2-x_J|\approx 2^{j_2}\ell(J)}\Big( \int_0^{\ell(I)}\!\!\int_0^{\ell(J)} +\int_0^{\ell(I)}\!\!\int_{\ell(J)}^\infty +
\int_{\ell(I)}^{\infty}\int_0^{\ell(J)}+ \int_{\ell(I)}^{\infty}\int_{\ell(J)}^\infty \Big)\\
&\quad\times \int_{\mathbb R^{n}}\int_{\mathbb R^{m}}\ \  \int_{\Bbb R^m}\chi^{(1)}_{t_1}(x_1-y_1,x_2-w_2)\chi^{(2)}_{t_2}(w_2-y_2)dw_2\\
&\qquad\times \Big|\big( t_1^2{\triangle^{(1)}}e^{-t_1^2{\triangle^{(1)}}}\,
t_2^2{\triangle^{(2)}}e^{-t_2^2{\triangle^{(2)}}}\big) (a_R) (y_1,y_2)\Big|^2{dy_2dt_2\over t_2^{m+1}} {dy_1dt_1\over
t_1^{n+m+1}}dx_2dx_1\\
&=:\textrm{I}_{21}+\textrm{I}_{22}+\textrm{I}_{23}+\textrm{I}_{24}.
\end{align*}
We first estimate the term $\textrm{I}_{21}$.
For $x_1\not\in 100\widetilde{I}$, $|x_1-y_1|<t_1<\ell(I)$ and $z_1\in 10I$, we have that $|y_1-z_1|\geq |x_1-x_I|/2$. Hence,
from the kernel estimate of $t_1^2\triangle^{(1)}e^{-t_1^2\triangle^{(1)}}$, we have
\begin{align*}
&\textrm{I}_{21}\\
&=\int_{|x_1-x_I|\approx 2^{j_1}\ell(I)} \int_{|x_2-x_J|\approx 2^{j_2}\ell(J)}\int_0^{\ell(I)}\!\!\int_0^{\ell(J)}
    \int_{\mathbb R^{n}}\int_{\mathbb R^{m}} \int_{\Bbb R^m}\chi^{(1)}_{t_1}(x_1-y_1,x_2-w_2)\chi^{(2)}_{t_2}(w_2-y_2)dw_2\\
&\quad\times \Big|\big( t_1^2{\triangle^{(1)}}e^{-t_1^2{\triangle^{(1)}}}\
t_2^2{\triangle^{(2)}}e^{-t_2^2{\triangle^{(2)}}}\big) a_{R} (y_1,y_2)\Big|^2 {dy_2dt_2\over t_2^{m+1}}{dy_1dt_1\over
t_1^{n+m+1}}dx_2dx_1\\
&\lesssim  \int_{|x_1-x_I|\approx 2^{j_1}\ell(I)} \int_{|x_2-x_J|\approx 2^{j_2}\ell(J)}\int_0^{\ell(I)}\!\!\int_0^{\ell(J)} \int_{\mathbb R^{n}}\int_{\mathbb R^{m}} \int_{\Bbb R^m}\chi^{(1)}_{t_1}(x_1-y_1,x_2-w_2)\chi^{(2)}_{t_2}(w_2-y_2)dw_2\\
&\quad\times\Big|\int_{10I}\int_{\Bbb R^m} t_1^{-n-m} \exp\Big(-{|y_1-z_1|^2+|y_2-z_2|^2\over ct_1^2}\Big)
\big(t_2^2{\triangle^{(2)}}e^{-t_2^2{\triangle^{(2)}}}a_{R} (z_1,z_2)\big)dz_2\,dz_1\Big|^2 \\
&\quad\times {dy_2dt_2\over t_2^{m+1}}{dy_1dt_1\over
t_1^{n+m+1}}dx_2dx_1\\
&\lesssim \int_{|x_1-x_I|\approx 2^{j_1}\ell(I)} \int_{|x_2-x_J|\approx 2^{j_2}\ell(J)}\int_0^{\ell(I)}\!\!\int_0^{\ell(J)} \int_{\mathbb R^{n}}\int_{\mathbb R^{m}} \int_{\Bbb R^m}\chi^{(1)}_{t_1}(x_1-y_1,x_2-w_2)\chi^{(2)}_{t_2}(w_2-y_2)dw_2\\
&\quad\times t_1^{-2n}\exp\Big(-{2|x_1-x_I|^2\over ct_1^2}\Big)\Big|\int_{10I}\int_{\Bbb R^m} t_1^{-m} \exp\Big(-{|y_2-z_2|^2\over ct_1^2}\Big)
t_2^2{\triangle^{(2)}}e^{-t_2^2{\triangle^{(2)}}}a_{R} (z_1,z_2)dz_2dz_1\Big|^2 \\
&\quad\times {dy_2dt_2\over t_2^{m+1}} {dy_1dt_1\over
t_1^{n+m+1}}dx_2dx_1.
\end{align*}
It is clear that
\begin{align*}
&\bigg|\int_{10I}\int_{\Bbb R^m} t_1^{-m} \exp\Big(-{|y_2-z_2|^2\over ct_1^2}\Big)
         t_2^2{\triangle^{(2)}}e^{-t_2^2{\triangle^{(2)}}}a_{R} (z_1,z_2)dz_2dz_1\bigg|\\
&=\bigg|\int_{10I}\int_{12J}  t_1^{-m}  \exp\Big(-{|y_2-z_2|^2\over ct_1^2}\Big)
         t_2^2{\triangle^{(2)}}e^{-t_2^2{\triangle^{(2)}}}a_{R} (z_1,z_2)dz_2dz_1\bigg|\\
&\qquad +\bigg| \int_{10I}\int_{(12J)^c}  t_1^{-m}  \exp\Big(-{|y_2-z_2|^2\over ct_1^2}\Big)
         t_2^2{\triangle^{(2)}}e^{-t_2^2{\triangle^{(2)}}}a_{R} (z_1,z_2)dz_2dz_1\bigg|\\
&\le \bigg|\int_{12J} \exp\Big(-{|y_2-z_2|^2\over 2ct_1^2}\Big)\ \  t_1^{-m}  \exp\Big(-{|y_2-z_2|^2\over 2ct_1^2}\Big)
        \int_{10I} t_2^2{\triangle^{(2)}}e^{-t_2^2{\triangle^{(2)}}}a_{R} (z_1,z_2)dz_1\ dz_2\bigg|\\
&\qquad + \int_{10I}\int_{(12J)^c}\ \int_{10J} t_1^{-m}  \exp\Big(-{|y_2-z_2|^2\over ct_1^2}\Big)
          t_2^{-m}  \exp\Big(-{|z_2-u_2|^2\over ct_2^2}\Big)|a_{R} (z_1,u_2)|du_2\,dz_2dz_1.
\end{align*}
Next we point out that there exists a positive constant $C$ such that for every $\alpha>0$,
\begin{align}\label{t12}
\hskip-.5cm\int_{\mathbb R^m}  t_1^{-m} \exp\Big(-{|y_2-z_2|^2\over 2ct_1^2}\Big)t_2^{-m} \exp\Big(-{|z_2-u_2|^2\over 2ct_2^2}\Big) dz_2
\leq  { C\cdot  (\max\{t_1,t_2\})^\alpha\over ( \max\{t_1,t_2\} + |y_2-u_2| )^{m+\alpha}  }.
\end{align}
Note that for $|x_2-x_J|> 100\ell(J)$ and $|x_2-y_2|<{t_1+t_2}<2\ell(J)$, if $z_2\in 12J$, then we have that $|y_2-z_2|>|x_2-x_J|/2$;
if $z_2\in (12J)^c$, since $u_2\in 10J$, we have that $|z_2-u_2|>\ell(J)$. Hence, combining these two cases and the almost orthogonality estimate \eqref{t12}, we have that
\begin{align*}
&\bigg|\int_{10I}\int_{\Bbb R^m} t_1^{-m} \exp\Big(-{|y_2-z_2|^2\over ct_1^2}\Big)
         t_2^2{\triangle^{(2)}}e^{-t_2^2{\triangle^{(2)}}}a_{R} (z_1,z_2)dz_1dz_2\bigg|\\
&\lesssim \exp\Big(-{|x_2-x_J|^2\over 8ct_1^2}\Big) M_2\bigg(\int_{10I} |t_2^2{\triangle^{(2)}}e^{-t_2^2{\triangle^{(2)}}}a_{R} (z_1,\cdot)|dz_1\bigg)(y_2)\\
&\qquad +\int_{10I}\int_{10J}  \exp\Big(-{\ell(J)^2\over ct_2^2}\Big){(\max\{t_1,t_2\})^\alpha\over ( \max\{t_1,t_2\} + |y_2-u_2| )^{m+\alpha}  } |a_{R} (z_1,u_2)|du_2dz_1.
\end{align*}
Plugging the above inequality into $\textrm{I}_{21}$, we have
\begin{align*}
&\textrm{I}_{21}\\
&\lesssim \int_{|x_1-x_I|\approx 2^{j_1}\ell(I)} \int_{|x_2-x_J|\approx 2^{j_2}\ell(J)}\int_0^{\ell(I)}\!\!\int_0^{\ell(J)} \int_{\mathbb R^{n}}\int_{\mathbb R^{m}} \int_{\Bbb R^m}\chi^{(1)}_{t_1}(x_1-y_1, w_2-y_2)\chi^{(2)}_{t_2}(x_2-w_2)dw_2\\
&\quad\times t_1^{-2n}\exp\Big(-{2|x_1-x_I|^2\over ct_1^2}\Big)\exp\Big(-{|x_2-x_J|^2\over 4ct_1^2}\Big) \bigg|M_2\bigg(\int_{10I} |t_2^2{\triangle^{(2)}}e^{-t_2^2{\triangle^{(2)}}}a_{R} (z_1,\cdot)|dz_1\bigg)(y_2)\bigg|^2\\
&\quad \times{dy_1dt_1\over
t_1^{n+m+1}}{dy_2dt_2\over t_2^{m+1}}dx_2dx_1\\
&\ + \int_{|x_1-x_I|\approx 2^{j_1}\ell(I)} \int_{|x_2-x_J|\approx 2^{j_2}\ell(J)}\int_0^{\ell(I)}\!\!\int_0^{\ell(J)} \int_{\mathbb R^{n}}\int_{\mathbb R^{m}} \int_{\Bbb R^m}\chi^{(1)}_{t_1}(x_1-y_1,x_2-w_2)\chi^{(2)}_{t_2}(w_2-y_2)dw_2\\
&\quad\times t_1^{-2n}\exp\Big(-{2|x_1-x_I|^2\over ct_1^2}\Big)|R| \exp\Big(-2{\ell(J)^2\over ct_2^2}\Big){(\max\{t_1,t_2\})^{2\alpha}\over ( \max\{t_1,t_2\} + {|x_2-x_J|} )^{2m+2\alpha}  }  \|a_{R}\|_2^2\\
&\quad \times {dy_1dt_1\over t_1^{n+m+1}}{dy_2dt_2\over t_2^{m+1}}dx_2dx_1\\
&=:\textrm{I}_{211}+\textrm{I}_{212}.
\end{align*}

To estimate $\textrm{I}_{211}$, we use the $L^2({\Bbb R^m})$-boundedness of Hardy--Littlewood maximal function $M_2$ and the $L^2({\Bbb R^m})$-boundedness of Littlewood--Paley square-function to get
\begin{align*}
&\textrm{I}_{211}\\
&\lesssim \int_{|x_1-x_I|\approx 2^{j_1}\ell(I)} \int_{|x_2-x_J|\approx 2^{j_2}\ell(J)}\int_0^{\ell(I)}\!\!\int_0^{\ell(J)} t_1^{-2n}\exp\Big(-{2|x_1-x_I|^2\over ct_1^2}\Big)\exp\Big(-{|x_2-x_J|^2\over 4ct_1^2}\Big) \\
&\qquad\times\int_{\mathbb R^{m}}\bigg| M_2\bigg(\int_{10I} |t_2^2{\triangle^{(2)}}e^{-t_2^2{\triangle^{(2)}}}a_{R} (z_1,\cdot)|dz_1\bigg)(y_2)\bigg|^2dy_2{dt_1\over t_1^{m+1}}{dt_2\over t_2}dx_2dx_1\\
&\lesssim \int_{|x_1-x_I|\approx 2^{j_1}\ell(I)} \int_{|x_2-x_J|\approx 2^{j_2}\ell(J)}\int_0^{\ell(I)}\!\!\int_0^{\ell(J)} \frac{t_1^{2\alpha_1}}{|x_1-x_I|^{2\alpha_1}}\frac{t_1^{2\alpha_2}}{|x_2-x_J|^{2\alpha_2}} \\
&\qquad\times\int_{\mathbb R^{m}}\bigg|\int_{10I} |t_2^2{\triangle^{(2)}}e^{-t_2^2{\triangle^{(2)}}}a_{R} (z_1,y_2)|dz_1\bigg|^2dy_2{dt_1\over t_1^{2n+m+1}}{dt_2\over t_2}dx_2dx_1\\
&\lesssim \int_{|x_1-x_I|\approx 2^{j_1}\ell(I)} \int_{|x_2-x_J|\approx 2^{j_2}\ell(J)}\int_0^{\ell(I)} \frac{t_1^{2\alpha_1}}{|x_1-x_I|^{2\alpha_1}}\frac{t_1^{2\alpha_2}}{|x_2-x_J|^{2\alpha_2}} \\
&\qquad\times|I|\int_{10I}\ \ \int_{\mathbb R^{m}}\Bigg| \bigg(\int_0^\infty|t_2^2{\triangle^{(2)}}e^{-t_2^2{\triangle^{(2)}}}a_{R} (z_1,y_2)|^2{dt_2\over t_2}\bigg)^{1\over2}\Bigg|^2\ dy_2dz_1{dt_1\over t_1^{2n+m+1}}dx_2dx_1\\
&\lesssim \int_{|x_1-x_I|\approx 2^{j_1}\ell(I)} \int_{|x_2-x_J|\approx 2^{j_2}\ell(J)}\int_0^{\ell(I)} \frac{t_1^{2\alpha_1}}{|x_1-x_I|^{2\alpha_1}}\frac{t_1^{2\alpha_2}}{|x_2-x_J|^{2\alpha_2}} \\
&\qquad\times|I| \int_{10I} \int_{10J} |a_{R} (z_1,y_2)|^2dy_2dz_1{dt_1\over t_1^{2n+m+1}}dx_2dx_1.
\end{align*}

Choosing $\alpha_1=n+1/2$ and $\alpha_2=m+1/2$, we have
\begin{align*}
\textrm{I}_{211}&\lesssim \int_{|x_1-x_I|\approx 2^{j_1}\ell(I)}\int_{|x_2-x_J|\approx 2^{j_2}\ell(J)} {dx_1\over (2^{j_1}\ell(I))^{2n+1}}{dx_2\over (2^{j_2}\ell(J))^{2m+1}}  \int_0^{\ell(I)} t_1^{m+1} dt_1
|I| \|a_R\|_2^2\\
&\lesssim (2^{j_1}\ell(I))^{-n-1} (2^{j_2}\ell(J))^{-m-1} \ell(I)^{m+2}\, |I|\|a_{R}\|_2^2\\
&\lesssim (2^{j_1}\ell(I))^{-n} 2^{-j_1} (2^{j_2}\ell(J))^{-m}2^{-j_2}  {\ell(I)^{2} \over \ell(I)\ell(J)}\,|I| \ell(I)^m \|a_{R}\|_2^2\\
&\lesssim (2^{j_1}\ell(I))^{-n} 2^{-j_1} (2^{j_2}\ell(J))^{-m}2^{-j_2} \ |R|\|a_{R}\|_2^2,
\end{align*}
where in the last inequality we use the fact that $\ell(I)\leq \ell(J)$.

We then estimate the term $\textrm{I}_{212}$. We first note that
\begin{align}\label{chi1chi2}
 \int_{\mathbb R^{n}}\int_{\mathbb R^{m}} \int_{\Bbb R^m}\chi^{(1)}_{t_1}(x_1-y_1,x_2-w_2)\chi^{(2)}_{t_2}(w_2-y_2) dy_2\ dw_2dy_1\leq C t_1^{n+m}t_2^m.
\end{align}
Then we have
\begin{align*}
\textrm{I}_{212}
&\lesssim \int_{|x_1-x_I|\approx 2^{j_1}\ell(I)} \int_{|x_2-x_J|\approx 2^{j_2}\ell(J)}\int_0^{\ell(I)}t_1^{-2n}\exp\Big(-{2|x_1-x_I|^2\over ct_1^2}\Big)\int_0^{\ell(J)}\exp\Big(-{\ell(J)^2\over ct_2^2}\Big){dt_2\over t_2}\\
&\qquad\times|R|{ \ell(J)^{2\alpha}\over |x_2-x_J|^{2m+2\alpha}  } \|a_{R}\|_2^2{dt_1\over t_1}dx_2dx_1\\
&\lesssim \int_{|x_1-x_I|\approx 2^{j_1}\ell(I)} \int_{|x_2-x_J|\approx 2^{j_2}\ell(J)}\int_0^{\ell(I)}t_1^{-2n}\exp\Big(-{2|x_1-x_I|^2\over ct_1^2}\Big)\int_0^{\ell(J)} \frac{t_2}{\ell(J)}{dt_2\over t_2}\\
&\qquad\times|R|{ \ell(J)^{2\alpha}\over |x_2-x_J|^{2m+2\alpha}  }  \|a_{R}\|_2^2{dt_1\over t_1}dx_2dx_1\\
&\lesssim \int_{|x_1-x_I|\approx 2^{j_1}\ell(I)} \int_{|x_2-x_J|\approx 2^{j_2}\ell(J)}\int_0^{\ell(I)}t_1^{-2n}\exp\Big(-{2|x_1-x_I|^2\over ct_1^2}\Big)\\
&\qquad\times|R|{  \ell(J)^{2\alpha}\over |x_2-x_J|^{2m+2\alpha}  } \|a_{R}\|_2^2{dt_1\over t_1}dx_2dx_1.
\end{align*}
We use the fact that $e^{-s}\leq Cs^{-k}$  for any $k>0$ to obtain
\begin{align*}
\textrm{I}_{212}\lesssim |R|\|a_{R}\|_2^2\int_{|x_1-x_I|\approx 2^{j_1}\ell(I)} \int_{|x_2-x_J|\approx 2^{j_2}\ell(J)}\int_0^{\ell(I)}\frac{t_1^{2\alpha_1}}{|x_1-x_I|^{2\alpha_1}}{   \ell(J)^{2\alpha}\over |x_2-x_J|^{2m+2\alpha}  } {dt_1\over t_1^{2n+1}}dx_2dx_1.
\end{align*}
Choosing $\alpha_1=n+1/2$ and $\alpha=1/2$, we have
\begin{align*}
\textrm{I}_{212}
&\lesssim |R| \|a_{R}\|_2^2\, (2^{j_1}\ell(I))^{-n-1} (2^{j_2}\ell(J))^{-m-1}\ell(J)\int_0^{\ell(I)}dt_1\\
&\lesssim |R| \|a_{R}\|_2^2\, (2^{j_1}\ell(I))^{-n}2^{-j_1} (2^{j_2}\ell(J))^{-m}2^{-j_2}.
\end{align*}
We then estimate the term $\textrm{I}_{22}$. Set $a_{R,1}=(({\triangle^{(1)}})^M \otimes_2 1\!\!1_2 )b_R$, that is,
$a_R= (1\!\!1_1  \otimes_2 ({\triangle^{(2)}})^M )a_{R,2}.$ Note that  supp $a_{R,1}\subset 10R=10(I\times J)$.
For $x_1\not\in 100\widetilde{I}$, $|x_1-y_1|<t_1<\ell(I)$ and $z_1\in 10I$, we have that $|y_1-z_1|\geq |x_1-x_I|/2$. By the almost orthogonality estimate in \eqref{t12}, we have
\begin{align*}
\textrm{I}_{22}
&=\int_{|x_1-x_I|\approx 2^{j_1}\ell(I)} \int_{|x_2-x_J|\approx 2^{j_2}\ell(J)}\int_0^{\ell(I)}\!\!\int_{\ell(J)}^\infty
    \int_{\mathbb R^{n}}\int_{\mathbb R^{m}} \int_{\Bbb R^m}\chi^{(1)}_{t_1}(x_1-y_1,x_2-w_2)\chi^{(2)}_{t_2}(w_2-y_2)dw_2\\
&\qquad\times \Big|\big( t_1^2{\triangle^{(1)}}e^{-t_1^2{\triangle^{(1)}}}\otimes_2
(t_2^2{\triangle^{(2)}})^{M+1}e^{-t_2^2{\triangle^{(2)}}}\big) a_{R,1} (y_1,y_2)\Big|^2 {dy_2dt_2\over t_2^{m+1+4M}}{dy_1dt_1\over
t_1^{n+m+1}}dx_2dx_1\\
&\lesssim  \int_{|x_1-x_I|\approx 2^{j_1}\ell(I)} \int_{|x_2-x_J|\approx 2^{j_2}\ell(J)}\int_0^{\ell(I)}\!\!\int_{\ell(J)}^\infty \int_{\mathbb R^{n}}\int_{\mathbb R^{m}} \int_{\Bbb R^m}\chi^{(1)}_{t_1}(x_1-y_1,x_2-w_2)\chi^{(2)}_{t_2}(w_2-y_2)dw_2\\
&\qquad\times\Big|\int_{10I}\int_{\Bbb R^m}\int_{\Bbb R^m} t_1^{-n-m} \exp\Big(-{|y_1-z_1|^2+|y_2-z_2|^2\over ct_1^2}\Big)\\
&\qquad\times t_1^{-m} \exp\Big(-{|z_2-u_2|^2\over ct_1^2}\Big) a_{R,1} (z_1,u_2)dz_1dz_2du_2\Big|^2{dy_2dt_2\over t_2^{m+1+4M}} {dy_1dt_1\over
t_1^{n+m+1}}dx_2dx_1\\
&\lesssim \int_{|x_1-x_I|\approx 2^{j_1}\ell(I)} \int_{|x_2-x_J|\approx 2^{j_2}\ell(J)}\int_0^{\ell(I)}\int_{\ell(J)}^\infty \int_{\mathbb R^{n}}\int_{\mathbb R^{m}} \int_{\Bbb R^m}\chi^{(1)}_{t_1}(x_1-y_1,x_2-w_2)\chi^{(2)}_{t_2}(w_2-y_2)dw_2\\
&\qquad\times t_1^{-2n}\exp\Big(-{2|x_1-x_I|^2\over ct_1^2}\Big)\Big|\int_{10I}\int_{\Bbb R^m}{(\max\{t_1,t_2\})^\alpha\over ( \max\{t_1,t_2\} + |y_2-u_2| )^{m+\alpha}  } |a_{R,1} (z_1,u_2)|dz_1du_2\Big|^2 \\
&\qquad\times{dy_2dt_2\over t_2^{m+1+4M}} {dy_1dt_1\over
t_1^{n+m+1}}dx_2dx_1.
\end{align*}
Note that if $|x_2-x_J|>100\ell(J)$, $t_1<t_2<|x_2-x_J|/4$, $|x_2-y_2|<t_1+t_2$ and $u_2\in 10J$, then $|y_2-u_2|\geq  |x_2-x_J|/4$,
and hence we see that $${(\max\{t_1,t_2\})^\alpha\over ( \max\{t_1,t_2\} + |y_2-u_2| )^{m+\alpha}  } \lesssim { t_2^\alpha\over ( t_2 + |x_2-x_J| )^{m+\alpha}  }.$$ If $t_2>|x_2-x_J|/4$, then we see that $${(\max\{t_1,t_2\})^\alpha\over ( \max\{t_1,t_2\} + |y_2-u_2| )^{m+\alpha}  } \lesssim { t_2^\alpha\over ( t_2  )^{m+\alpha}  } \lesssim { 1\over  t_2^{m}  }.$$
Based on these observation, we then use the fact that $e^{-s}\leq Cs^{-k}$  for any $k>0$ to obtain
\begin{align*}
\textrm{I}_{22}
&\lesssim \int_{|x_1-x_I|\approx 2^{j_1}\ell(I)} \int_{|x_2-x_J|\approx 2^{j_2}\ell(J)}\int_0^{\ell(I)}\Big(\int_{\ell(J)}^{|x_2-x_J|\over 4}+\int_{|x_2-x_J|\over 4}^\infty \Big)\\
&\qquad \int_{\mathbb R^{n}}\int_{\mathbb R^{m}} \int_{\Bbb R^m}\chi^{(1)}_{t_1}(x_1-y_1,x_2-w_2)\chi^{(2)}_{t_2}(w_2-y_2)dw_2\\
&\qquad\times t_1^{-2n}\exp\Big(-{2|x_1-x_I|^2\over ct_1^2}\Big)\Big|\int_{10I}\int_{\Bbb R^m}{ t_2^\alpha\over ( t_2 + |x_2-x_J| )^{m+\alpha}  } |a_{R,1} (z_1,u_2)|du_2dz_1\Big|^2 \\
&\qquad\times {dy_1dt_1\over t_1^{n+m+1}}{dy_2dt_2\over t_2^{m+1+4M}}dx_2dx_1\\
&\lesssim \int_{|x_1-x_I|\approx 2^{j_1}\ell(I)} \int_{|x_2-x_J|\approx 2^{j_2}\ell(J)}\int_0^{\ell(I)}\int_{\ell(J)}^{|x_2-x_J|\over 4} \\
&\qquad\times t_1^{-2n}\exp\Big(-{2|x_1-x_I|^2\over ct_1^2}\Big) { t_2^{2\alpha}\over ( t_2 + |x_2-x_J| )^{2m+2\alpha}  } |R|\|a_{R,1}\|_2^2{dt_1\over t_1}{dt_2\over t_2^{1+4M}}dx_2dx_1\\
&\quad+  \int_{|x_1-x_I|\approx 2^{j_1}\ell(I)} \int_{|x_2-x_J|\approx 2^{j_2}\ell(J)}\int_0^{\ell(I)}\int_{|x_2-x_J|\over 4}^\infty \\
&\qquad\times t_1^{-2n}\exp\Big(-{2|x_1-x_I|^2\over ct_1^2}\Big) |R|\|a_{R,1}\|_2^2 {dt_1\over t_1}{dt_2\over t_2^{2m+1+4M}}dx_2dx_1\\
&\lesssim |R|\|a_{R,1}\|_2^2\int_{|x_1-x_I|\approx 2^{j_1}\ell(I)} \int_{|x_2-x_J|\approx 2^{j_2}\ell(J)}\int_0^{\ell(I)}\int_{\ell(J)}^\infty\frac{t_1^{2\alpha_1}}{|x_1-x_I|^{2\alpha_1}} \\
&\qquad\times     { t_2^{2\alpha}\over  |x_2-x_J|^{2m+2\alpha}  }{dt_1\over t_1^{2n+1}}{dt_2\over t_2^{1+4M}}dx_2dx_1\\
&\quad+  |R|\|a_{R,1}\|_2^2\int_{|x_1-x_I|\approx 2^{j_1}\ell(I)} \int_{|x_2-x_J|\approx 2^{j_2}\ell(J)}\int_0^{\ell(I)}\int_{|x_2-x_J|\over4}^\infty \frac{t_1^{2\alpha_1}}{|x_1-x_I|^{2\alpha_1}}\\
&\qquad\times  {dt_1\over t_1^{2n+1}}{dt_2\over t_2^{2m+1+4M}}dx_2dx_1.
\end{align*}
We now choose $\alpha_1=n+{1\over2}$ and $\alpha={1\over2}$, then we obtain that
\begin{align*}
\textrm{I}_{22}
&\lesssim |R|\|a_{R,1}\|_2^2(2^{j_1}\ell(I))^{-n-1}(2^{j_2}\ell(J))^{-m-1}\int_0^{\ell(I)} dt_1 \int_{\ell(J)}^\infty t_2^{-4M}dt_2 \\
&\quad+   |R|\|a_{R,1}\|_2^2(2^{j_1}\ell(I))^{-n-1}(2^{j_2}\ell(J))^{m}\int_0^{\ell(I)} t_1^{2\alpha_1-2n-1}dt_1 \int_{2^{j_2-2}\ell(J)}^\infty t_2^{-2m-1-4M}dt_2 \\
&\lesssim |R| \ell(J)^{-4M} \|a_{R,1}\|_2^2 \ (2^{j_1}\ell(I))^{-n} 2^{-j_1} (2^{j_2}\ell(J))^{-m} 2^{-j_2}  \\
&\quad+  |R|\ell(J)^{-4M}\|a_{R,1}\|_2^2\ \ (2^{j_1}\ell(I))^{-n} 2^{-j_1} (2^{j_2}\ell(J))^{-m} 2^{-4M j_2},
\end{align*}
where we only require that $4M>1$.

We now to estimate $\textrm{I}_{23}$. 
\begin{align*}
\textrm{I}_{23}
&=\int_{|x_1-x_I|\approx 2^{j_1}\ell(I)} \int_{|x_2-x_J|\approx 2^{j_2}\ell(J)}\int_{\ell(I)}^\infty\!\!\int_0^{\ell(J)}
  \!\! \int_{\mathbb R^{n}}\int_{\mathbb R^{m}} \int_{\Bbb R^m}\chi^{(1)}_{t_1}(x_1-y_1,x_2-w_2)\chi^{(2)}_{t_2}(w_2-y_2)dw_2\\
&\qquad\times \Big|\big( (t_1^2{\triangle^{(1)}})^{M+1}e^{-t_1^2{\triangle^{(1)}}}\otimes_2
t_2^2{\triangle^{(2)}}e^{-t_2^2{\triangle^{(2)}}}\big) a_{R,2} (y_1,y_2)\Big|^2 {dy_2dt_2\over t_2^{m+1}}{dy_1dt_1\over
t_1^{n+m+1+4M}}dx_2dx_1\\
&\lesssim  \int_{|x_1-x_I|\approx 2^{j_1}\ell(I)} \int_{|x_2-x_J|\approx 2^{j_2}\ell(J)}\int_{\ell(I)}^\infty\!\!\int_0^{\ell(J)}\!\! \int_{\mathbb R^{n}}\int_{\mathbb R^{m}} \int_{\Bbb R^m}\chi^{(1)}_{t_1}(x_1-y_1,x_2-w_2)\chi^{(2)}_{t_2}(w_2-y_2)dw_2\\
&\qquad\times\Big|\int_{10I}\int_{\Bbb R^m} t_1^{-n-m} \exp\Big(-{|y_1-z_1|^2+|y_2-z_2|^2\over ct_1^2}\Big)
t_2^2{\triangle^{(2)}}e^{-t_2^2{\triangle^{(2)}}}a_{R,2} (z_1,z_2)dz_1dz_2\Big|^2 \\
&\qquad\times {dy_2dt_2\over t_2^{m+1}}{dy_1dt_1\over
t_1^{n+m+1+4M}}dx_2dx_1\\
&=  \int_{|x_1-x_I|\approx 2^{j_1}\ell(I)} \int_{|x_2-x_J|\approx 2^{j_2}\ell(J)}\bigg(\int_{\ell(I)}^{\frac{|x_1-x_I|+|x_2-x_J|}8}+ \int_{\frac{|x_1-x_I|+|x_2-x_J|}8}^\infty \bigg)\!\!\int_0^{\ell(J)}\!\! \int_{\mathbb R^{n}}\int_{\mathbb R^{m}}\\
&\qquad \times \int_{\Bbb R^m}\chi^{(1)}_{t_1}(x_1-y_1,x_2-w_2)\chi^{(2)}_{t_2}(w_2-y_2)dw_2\\
&\qquad\times\Big|\int_{10I}\int_{\Bbb R^m} t_1^{-n-m} \exp\Big(-{|y_1-z_1|^2+|y_2-z_2|^2\over ct_1^2}\Big)
t_2^2{\triangle^{(2)}}e^{-t_2^2{\triangle^{(2)}}}a_{R,2} (z_1,z_2)dz_1dz_2\Big|^2 \\
&\qquad\times {dy_2dt_2\over t_2^{m+1}}{dy_1dt_1\over
t_1^{n+m+1+4M}}dx_2dx_1\\
&=:\textrm{I}_{231}+\textrm{I}_{232}.
\end{align*}
We first estimate the term $\textrm{I}_{231}$ and consider two cases.

\noindent Case (1):  $|x_1-x_I|\ge |x_2-x_J|$.

In this case we have $t_1<(|x_1-x_I|+|x_2-x_J|)/8\le |x_1-x_I|/4$.
Since $|x_1-x_I|>100\ell(I)$, $|x_1-y_1|<t_1$ and $z_1\in 10I$, then $|y_1-z_1|\geq  |x_1-x_I|/4\ge (|x_1-x_I|+|x_2-x_J|)/8$. Hence we have
\begin{align*}
\textrm{I}_{231}
&\lesssim\int_{|x_1-x_I|\approx 2^{j_1}\ell(I)} \int_{|x_2-x_J|\approx 2^{j_2}\ell(J)}\int_{\ell(I)}^{\frac{|x_1-x_I|+|x_2-x_J|}8}\!\!\int_0^{\ell(J)}\!\! \int_{\mathbb R^{n}}\int_{\mathbb R^{m}}
                \\
&\quad \times\!\! \int_{\Bbb R^m}\!\chi^{(1)}_{t_1}(x_1-y_1,x_2-w_2)\chi^{(2)}_{t_2}(w_2-y_2)dw_2 t_1^{-2n}\exp\Big(-{|x_1-x_I|^2\over 8ct_1^2}\Big)\exp\Big(-{|x_2-x_J|^2\over 8ct_1^2}\Big) \\
&\quad \times \Big|\int_{\Bbb R^m} t_1^{-m} \exp\Big(-{|y_2-z_2|^2\over ct_1^2}\Big)
\Big(\int_{10I}t_2^2{\triangle^{(2)}}e^{-t_2^2{\triangle^{(2)}}}  a_{R,2} (z_1,z_2)dz_1\Big)dz_2\Big|^2\\
&\qquad{dy_1dt_1\over
t_1^{n+m+1+4M}}{dy_2dt_2\over t_2^{m+1}}dx_2dx_1\\
&\lesssim\int_{|x_1-x_I|\approx 2^{j_1}\ell(I)} \int_{|x_2-x_J|\approx 2^{j_2}\ell(J)}\int_{\ell(I)}^\infty\!\!\int_0^{\ell(J)} t_1^{-2n}\frac{t_1^{2\alpha_1}}{|x_1-x_I|^{2\alpha_1}}\frac{t_1^{2\alpha_2}}{|x_2-x_J|^{2\alpha_2}} \\
&\quad\times \int_{\Bbb R^m} \Bigg| M_2\bigg(\int_{10I} |t_2^2{\triangle^{(2)}}e^{-t_2^2{\triangle^{(2)}}}a_{R,2} (z_1,\cdot)|dz_1\bigg)(y_2)\Bigg|^2dy_2{dt_1\over
t_1^{m+1+4M}}{dt_2\over t_2}dx_2dx_1\\
&\lesssim\int_{|x_1-x_I|\approx 2^{j_1}\ell(I)} \int_{|x_2-x_J|\approx 2^{j_2}\ell(J)}\int_{\ell(I)}^\infty\!\!\int_0^{\ell(J)} \frac{t_1^{2\alpha_1}}{|x_1-x_I|^{2\alpha_1}}\frac{t_1^{2\alpha_2}}{|x_2-x_J|^{2\alpha_2}} \\
&\quad\times \int_{\Bbb R^m} \bigg|\bigg(\int_{10I} |t_2^2{\triangle^{(2)}}e^{-t_2^2{\triangle^{(2)}}}a_{R,2} (z_1,\cdot)|dz_1\bigg)(y_2)\bigg|^2dy_2{dt_1\over
t_1^{2n+m+1+4M}}{dt_2\over t_2}dx_2dx_1\\
&\lesssim\int_{|x_1-x_I|\approx 2^{j_1}\ell(I)} \int_{|x_2-x_J|\approx 2^{j_2}\ell(J)}\int_{\ell(I)}^\infty \frac{t_1^{2\alpha_1}}{|x_1-x_I|^{2\alpha_1}}\frac{t_1^{2\alpha_2}}{|x_2-x_J|^{2\alpha_2}} \\
&\quad\times |I| \int_{10I} \ \int_{\Bbb R^m} \int_0^\infty |t_2^2{\triangle^{(2)}}e^{-t_2^2{\triangle^{(2)}}}a_{R,2} (z_1,y_2)|^2{dt_2\over t_2}dy_2\ dz_1{dt_1\over
t_1^{2n+m+1+4M}}dx_2dx_1\\
&\lesssim |I|\|a_{R,2}\|_2^2\int_{|x_1-x_I|\approx 2^{j_1}\ell(I)} \int_{|x_2-x_J|\approx 2^{j_2}\ell(J)}\int_{\ell(I)}^\infty \frac{t_1^{2\alpha_1}}{|x_1-x_I|^{2\alpha_1}}\frac{t_1^{2\alpha_2}}{|x_2-x_J|^{2\alpha_2}} {dt_1dx_2dx_1\over t_1^{2n+m+1+4M}},
\end{align*}
where the second inequality follows from the fact that $e^{-s}\leq Cs^{-k}$  for any $k>0$, the third inequality follows from
the $L^2(\mathbb R^m)$-boundedness of the Hardy--Littlewood maximal function $M_2$, the fourth inequality follows from
H\"older's inequality and the last inequality follows from the $L^2(\mathbb R^m)$-boundedness of the Littlewood--Paley square function
with respect to $t_2^2{\triangle^{(2)}}e^{-t_2^2{\triangle^{(2)}}}$.

To continue, we choose $\alpha_1>n$ and $\alpha_2>m$ and $2\alpha_1+2\alpha_2<2n+m+4M$, then we have
\begin{align*}
\textrm{I}_{231}
&\lesssim |I|\|a_{R,2}\|_2^2 (2^{j_1}\ell(I))^{n-2\alpha_1}(2^{j_2}\ell(J))^{m-2\alpha_2}\int_{\ell(I)}^\infty t_1^{2\alpha_1+2\alpha_2-2n-m-4M-1} dt_1\\
&\lesssim |I|\|a_{R,2}\|_2^2 (2^{j_1}\ell(I))^{-n}(2^{j_1}\ell(I))^{2n-2\alpha_1} (2^{j_2}\ell(J))^{-m}(2^{j_2}\ell(J))^{2m-2\alpha_2} \ell(I)^{2\alpha_1+2\alpha_2-2n-m-4M} \\
&\lesssim |I| \ell(I)^{-4M}\|a_{R,2}\|_2^2 (2^{j_1}\ell(I))^{-n}2^{(2n-2\alpha_1)j_1} (2^{j_2}\ell(J))^{-m}2^{(2m-2\alpha_2)j_2} \ell(J)^{m} \\
&\lesssim |R| \ell(I)^{-4M}\|a_{R,2}\|_2^2 (2^{j_1}\ell(I))^{-n}2^{(2n-2\alpha_1)j_1} (2^{j_2}\ell(J))^{-m}2^{(2m-2\alpha_2)j_2}.
\end{align*}
We also note that from these conditions of $\alpha_1$ and $\alpha_2$, we obtain that for the order of the cancellation of the atom $a$, we require that $M>m/4$.

\bigskip

\noindent Case (2):  $|x_1-x_I|< |x_2-x_J|$.

In this case we have $t_1+t_2<(|x_1-x_I|+|x_2-x_J|)/8+\ell(J)\le |x_2-x_J|/2$ and $|x_1-x_I|/4<(|x_1-x_I|+|x_2-x_J|)/8$.
Also note that $|x_1-x_I|>100\ell(I)$, $|x_1-y_1|<t_1<|x_1-x_I|/4$ and that $z_1\in 10I$, so we have $|y_1-z_1|\ge |x_1-x_I|/4$.
Hence
\begin{align*}
\textrm{I}_{231}
&\lesssim \int_{|x_1-x_I|\approx 2^{j_1}\ell(I)} \int_{|x_2-x_J|\approx 2^{j_2}\ell(J)}\int_{\ell(I)}^{\frac{|x_1-x_I|}4}\int_0^{\ell(J)}\!\! \int_{\mathbb R^{n}}\int_{\mathbb R^{m}}
                \\
&\quad\times \int_{\Bbb R^m}\chi^{(1)}_{t_1}(x_1-y_1,x_2-w_2)\chi^{(2)}_{t_2}(w_2-y_2)dw_2\ t_1^{-2n}\exp\Big(-{|x_1-x_I|^2\over 16ct_1^2}\Big) \\
&\quad\times \Big|\int_{10I}\int_{\Bbb R^m} t_1^{-m} \exp\Big(-{|y_2-z_2|^2\over ct_1^2}\Big)
t_2^2{\triangle^{(2)}}e^{-t_2^2{\triangle^{(2)}}}a_{R,2} (z_1,z_2)dz_2dz_1\Big|^2\\
&\qquad{dy_1dt_1\over
t_1^{n+m+1+4M}}{dy_2dt_2\over t_2^{m+1}}dx_2dx_1\\
&+ \int_{|x_1-x_I|\approx 2^{j_1}\ell(I)} \int_{|x_2-x_J|\approx 2^{j_2}\ell(J)}\int_{\frac{|x_1-x_I|}4}^{\frac{|x_1-x_I|+|x_2-x_J|}8}\int_0^{\ell(J)}\!\! \int_{\mathbb R^{n}}\int_{\mathbb R^{m}}
                \\
&\quad\times \int_{\Bbb R^m}\chi^{(1)}_{t_1}(x_1-y_1,x_2-w_2)\chi^{(2)}_{t_2}(w_2-y_2)dw_2\ t_1^{-2n}\\
&\quad\times \Big|\int_{10I}\int_{\Bbb R^m} t_1^{-m} \exp\Big(-{|y_2-z_2|^2\over ct_1^2}\Big)
t_2^2{\triangle^{(2)}}e^{-t_2^2{\triangle^{(2)}}}a_{R,2} (z_1,z_2)dz_2dz_1\Big|^2\\
&\qquad{dy_1dt_1\over
t_1^{n+m+1+4M}}{dy_2dt_2\over t_2^{m+1}}dx_2dx_1\\
&=:\textrm{I}_{2311}+\textrm{I}_{2312}.
\end{align*}
Note that
\begin{align*}
&\bigg|\int_{10I}\int_{\Bbb R^m} t_1^{-m} \exp\Big(-{|y_2-z_2|^2\over ct_1^2}\Big)
         t_2^2{\triangle^{(2)}}e^{-t_2^2{\triangle^{(2)}}}a_{R,2} (z_1,z_2)dz_2dz_1\bigg|\\
&\le \bigg|\int_{10I}\int_{12J}  t_1^{-m}  \exp\Big(-{|y_2-z_2|^2\over ct_1^2}\Big)
         t_2^2{\triangle^{(2)}}e^{-t_2^2{\triangle^{(2)}}}a_{R,2} (z_1,z_2)dz_2dz_1\bigg|\\
&\quad + \bigg|\int_{10I}\int_{(12J)^c}  t_1^{-m}  \exp\Big(-{|y_2-z_2|^2\over ct_1^2}\Big)
         \int_{10J} t_2^{-m}  \exp\Big(-{|z_2-u_2|^2\over ct_2^2}\Big)|a_{R,2} (z_1,u_2)|du_2\ dz_2dz_1\bigg|.
\end{align*}
Since $|x_2-x_J|>100\ell(J)$ and $|x_2-y_2|<t_1+t_2$, if $z_2\in 12I$, then we have $|y_2-z_2|\geq  |x_2-x_J|/4\ge (|x_1-x_I|+|x_2-x_J|)/8$; if $z_2\in (12J)^c$ then for $u_2\in 10J$ we have  $|z_2-u_2|>\ell(J)$. As a consequence, from the almost orthogonality estimate \eqref{t12} we get that
\begin{equation}\label{i23}
\begin{aligned}
&\bigg|\int_{10I}\int_{\Bbb R^m} t_1^{-m} \exp\Big(-{|y_2-z_2|^2\over ct_1^2}\Big)
         t_2^2{\triangle^{(2)}}e^{-t_2^2{\triangle^{(2)}}}a_{R,2} (z_1,z_2)dz_2dz_1\bigg|\\
&\lesssim \exp\Big(-{|x_2-x_J|^2\over 2ct_1^2}\Big) M_2\bigg(\int_{10I} |t_2^2{\triangle^{(2)}}e^{-t_2^2{\triangle^{(2)}}}a_{R,2} (z_1,\cdot)|dz_1\bigg)(y_2)\\
&\qquad + C\int_{10I}\int_{10J}  \exp\Big(-{\ell(J)^2\over 2ct_2^2}\Big){(\max\{t_1,t_2\})^\alpha\over ( \max\{t_1,t_2\} + |y_2-u_2| )^{m+\alpha}  } |a_{R,2} (z_1,u_2)|du_2dz_1.
\end{aligned}
\end{equation}
Plugging \eqref{i23} into $\textrm{I}_{2311}$, we have
\begin{align*}
&\textrm{I}_{2311} \\
&\lesssim \int_{|x_1-x_I|\approx 2^{j_1}\ell(I)} \int_{|x_2-x_J|\approx 2^{j_2}\ell(J)}\int_{\ell(I)}^{\frac{|x_1-x_I|}4}\!\!\int_0^{\ell(J)}\!\! \int_{\mathbb R^{n}}\int_{\mathbb R^{m}}
                \\
&\quad\times \int_{\Bbb R^m}\chi^{(1)}_{t_1}(x_1-y_1,x_2-w_2)\chi^{(2)}_{t_2}(w_2-y_2)dw_2 t_1^{-2n}\exp\Big(-{|x_1-x_I|^2\over 16ct_1^2}\Big)\exp\Big(-{|x_2-x_J|^2\over 2ct_1^2}\Big) \\
&\quad\times  \bigg|M_2\bigg(\int_{10I} |t_2^2{\triangle^{(2)}}e^{-t_2^2{\triangle^{(2)}}}a_{R,2} (z_1,\cdot)|dz_1\bigg)(y_2)\bigg|^2{dy_2dt_2\over t_2^{m+1}}{dy_1dt_1\over t_1^{n+m+1+4M}}dx_2dx_1\\
&\ \,+C \int_{|x_1-x_I|\approx 2^{j_1}\ell(I)} \int_{|x_2-x_J|\approx 2^{j_2}\ell(J)}\int_{\ell(I)}^{\frac{|x_1-x_I|}4}\!\!\int_0^{\ell(J)} \int_{\mathbb R^{n}}\int_{\mathbb R^{m}}\\
&\quad\times \int_{\Bbb R^m}\chi^{(1)}_{t_1}(x_1-y_1,x_2-w_2)\chi^{(2)}_{t_2}(w_2-y_2)dw_2 t_1^{-2n}\exp\Big(-{|x_1-x_I|^2\over 16ct_1^2}\Big)|R| \exp\Big(-{\ell(J)^2\over ct_2^2}\Big)\\
&\quad \times
    {(\max\{t_1,t_2\})^{2\alpha}\over ( \max\{t_1,t_2\} + |x_2-x_J| )^{2m+2\alpha}  } \|a_{R,2}\|_2^2{dy_2dt_2\over t_2^{m+1}}{dy_1dt_1\over t_1^{n+m+1+4M}}dx_2dx_1\\
&=:\textrm{I}_{23111}+\textrm{I}_{23112}.
\end{align*}
To estimate $\textrm{I}_{23111}$, we use the same method in case (1) to obtain
$$\textrm{I}_{23111}
\lesssim |R| \ell(I)^{-4M}\|a_{R,2}\|_2^2 (2^{j_1}\ell(I))^{-n}2^{(2n-2\alpha_1)j_1} (2^{j_2}\ell(J))^{-m}2^{(2m-2\alpha_2)j_2}
,$$
by choosing $n<\alpha_1$, $m<\alpha_2$ and $2\alpha_1+2\alpha_2<2n+m+4M$.

For the term $\textrm{I}_{23112}$, we use the fact that $e^{-s}\leq Cs^{-k}$  for any $k>0$ to show
\begin{align*}
\textrm{I}_{23112}
&\lesssim |R|\int_{|x_1-x_I|\approx 2^{j_1}\ell(I)} \int_{|x_2-x_J|\approx 2^{j_2}\ell(J)}\int_{\ell(I)}^{\frac{|x_1-x_I|}4}t_1^{-2n}\frac{t_1^{2\alpha_1}}{|x_1-x_I|^{2\alpha_1}}\\
&\quad\times\!\!\int_0^{\ell(J)}\exp\Big(-2{\ell(J)^2\over ct_2^2}\Big){dt_2\over t_2}
    {(\max\{t_1,\ell(J)\})^{2\alpha}\over |x_2-x_J|^{2m+2\alpha}  } \|a_{R,2}\|_2^2{dt_1\over t_1^{1+4M}}dx_2dx_1.
\end{align*}
Let $\alpha={1\over2}$ and {$0<2\alpha_1-2n<4M-1$}.
\begin{align*}
\textrm{I}_{23112}
&\lesssim |R|\|a_{R,2}\|_2^2\int_{|x_1-x_I|\approx 2^{j_1}\ell(I)} \int_{|x_2-x_J|\approx 2^{j_2}\ell(J)}\int_{\ell(I)}^{\frac{|x_1-x_I|}4}t_1^{-2n}\frac{t_1^{2\alpha_1}}{|x_1-x_I|^{2\alpha_1}}\\
&\qquad\times
    {t_1+\ell(J) \over  |x_2-x_J| ^{2m+1}  } {dt_1\over t_1^{1+4M}}dx_2dx_1\\
    &\lesssim |R| \|a_{R,2}\|_2^2 (2^{j_1}\ell(I))^{n-2\alpha_1} (2^{j_2}\ell(J))^{-m-1}\int_{\ell(I)}^\infty t_1^{2\alpha_1-2n-4M}dt_1\\
&\qquad + |R| \|a_{R,2}\|_2^2 (2^{j_1}\ell(I))^{n-2\alpha_1} (2^{j_2}\ell(J))^{-m-1}\ell(J)\int_{\ell(I)}^\infty t_1^{2\alpha_1-2n-4M-1}dt_1\\
&\lesssim |R| \ell(I)^{-4M}\|a_{R,2}\|_2^2 (2^{j_1}\ell(I))^{-n} 2^{(2n-2\alpha_1)j_1} (2^{j_2}\ell(J))^{-m}2^{-j_2}. 
\end{align*}
This finishes the estimate for the term $\textrm{I}_{2311}$.

We plug \eqref{i23} into $\textrm{I}_{2312}$ to get
\begin{align*}
&\textrm{I}_{2312} \\
&\lesssim \int_{|x_1-x_I|\approx 2^{j_1}\ell(I)} \int_{|x_2-x_J|\approx 2^{j_2}\ell(J)}\int_{\frac{|x_1-x_I|}4}^{\frac{|x_1-x_I|+|x_2-x_J|}8}\!\!\int_0^{\ell(J)}\!\! \int_{\mathbb R^{n}}\int_{\mathbb R^{m}}
                \\
&\quad\times \int_{\Bbb R^m}\chi^{(1)}_{t_1}(x_1-y_1,x_2-w_2)\chi^{(2)}_{t_2}(w_2-y_2)dw_2\ t_1^{-2n}\exp\Big(-{|x_2-x_J|^2\over ct_1^2}\Big) \\
&\qquad\times  \bigg|M_2\bigg(\int_{10I} |t_2^2{\triangle^{(2)}}e^{-t_2^2{\triangle^{(2)}}}a_{R,2} (z_1,\cdot)|dz_1\bigg)(y_2)\bigg|^2{dy_1dt_1\over t_1^{n+m+1+4M}}{dy_2dt_2\over t_2^{m+1}}dx_2dx_1\\
&\ \ + \int_{|x_1-x_I|\approx 2^{j_1}\ell(I)} \int_{|x_2-x_J|\approx 2^{j_2}\ell(J)}\int_{\frac{|x_1-x_I|}4}^{\frac{|x_1-x_I|+|x_2-x_J|}8}\!\!\int_0^{\ell(J)} \int_{\mathbb R^{n}}\int_{\mathbb R^{m}}\\
&\quad\times \int_{\Bbb R^m}\chi^{(1)}_{t_1}(x_1-y_1,x_2-w_2)\chi^{(2)}_{t_2}(w_2-y_2)dw_2\ t_1^{-2n} \exp\Big(-{\ell(J)^2\over ct_2^2}\Big)\\
&\qquad \times
    \bigg|\int_{10I}\int_{10J}  {(\max\{t_1,t_2\})^{\alpha}\over ( \max\{t_1,t_2\} + |y_2-u_2| )^{m+\alpha}  }   |a_{R,2} (z_1,u_2)|du_2dz_1\bigg|^2\ {dy_1dt_1\over t_1^{n+m+1+4M}}{dy_2dt_2\over t_2^{m+1}}dx_2dx_1\\
&=:\textrm{I}_{23121}+\textrm{I}_{23122}.
\end{align*}

To estimate $\textrm{I}_{23121}$, we use the same method in case (1) to obtain
\begin{align*}
\textrm{I}_{23121}
&\lesssim  |I|(2^{j_1}\ell(I))^{2\alpha_2-2n-m-4M}(2^{j_2}\ell(J))^{m-2\alpha_2} \|a_{R,2}\|_2^2\\
&\lesssim |R| \ell(I)^{-4M}\|a_{R,2}\|_2^2 (2^{j_1}\ell(I))^{-n}2^{(2\alpha_2-n-m-4M)j_1} (2^{j_2}\ell(J))^{-m}2^{(2m-2\alpha_2)j_2}
\end{align*}
by choosing $m<\alpha_2$ and $2\alpha_2<2n+m+4M$.

We then estimate the term $\textrm{I}_{23122}$. Note that in this case
$|x_2-y_2|<t_1+t_2<\frac{|x_1-x_I|}4+\ell(J) < \frac{|x_2-x_J|}4+\ell(J)$. Hence, we have
$|y_2-u_2|> |x_2-x_J| - |x_2-y_2|-|u_2-x_J|> |x_2-x_J| -\frac{|x_2-x_J|}4- 2\ell(J) >\frac{|x_2-x_J|}2 $.
Thus, we get that
\begin{align*}
\textrm{I}_{23122}
&\lesssim \int_{|x_1-x_I|\approx 2^{j_1}\ell(I)} \int_{|x_2-x_J|\approx 2^{j_2}\ell(J)}\int_{\frac{|x_1-x_I|}4}^\infty\!\!\int_0^{\ell(J)}  \exp\Big(-2{\ell(J)^2\over ct_2^2}\Big)\\
&\qquad \times  { \max\{t_1^{2\alpha},t_2^{2\alpha}\}\over  {|x_2-x_J|} ^{2m+2\alpha}  }  |R|\|a_{R,2}\|_2^2{dt_1\over t_1^{2n+1+4M}}{dt_2\over t_2}dx_2dx_1\\
&\lesssim  |R| \|a_{R,2}\|_2^2 (2^{j_1}\ell(I))^{-n}(2^{j_2}\ell(J))^{-m}  (2^{j_1}\ell(I))^{2n}(2^{j_2}\ell(J))^{-2\alpha}\\
&\qquad\times \int_{\frac{|x_1-x_I|}4}^\infty {1\over t_1^{2n+1+4M-2\alpha}} dt_1\int_0^{\ell(J)}\exp\Big(-2{\ell(J)^2\over ct_2^2}\Big){dt_2\over t_2}  \\
&\quad+ |R| \|a_{R,2}\|_2^2 (2^{j_1}\ell(I))^{-n}(2^{j_2}\ell(J))^{-m}  (2^{j_1}\ell(I))^{2n}(2^{j_2}\ell(J))^{-2\alpha}\\
&\qquad\times \int_{\frac{|x_1-x_I|}4}^\infty {1\over t_1^{2n+1+4M}} dt_1\int_0^{\ell(J)}\exp\Big(-2{\ell(J)^2\over ct_2^2}\Big) t_2^{2\alpha}{dt_2\over t_2} \\
&\lesssim |R| \|a_{R,2}\|_2^2 (2^{j_1}\ell(I))^{-n}(2^{j_2}\ell(J))^{-m}  2^{(-4M+2\alpha)j_1}2^{-2\alpha j_2},
\end{align*}
where we choose $0<\alpha<2M$, and  the last inequality follows from the fundamental estimates that
$\int_0^{\ell(J)}\exp\big(-2{\ell(J)^2\over ct_2^2}\big) {dt_2\over t_2}\lesssim 1$ and that $\int_0^{\ell(J)}\exp\big(-2{\ell(J)^2\over ct_2^2}\big) t_2^{2\alpha}{dt_2\over t_2}\lesssim \ell(J)^{2\alpha}$.

Now we consider the term $\textrm{I}_{232}$. Note that in this case, there is no lower bound for $|y_1-z_1|$, and hence we can only use the fact that $\exp\big(-{|y_1-z_1|^2\over ct_1^2}\big)\leq1$. Then, from this observation and from the $L^2(\mathbb R^m)$-boundedness of the Hardy--Littlewood Maximal function and the Littlewood--Paley square function, we obtain that
\begin{align*}
\textrm{I}_{232}
&\lesssim  \int_{|x_1-x_I|\approx 2^{j_1}\ell(I)} \int_{|x_2-x_J|\approx 2^{j_2}\ell(J)}\int_{\frac{|x_1-x_I|+|x_2-x_J|}8}^\infty\!\!\int_0^{\ell(J)}\!\! \int_{\mathbb R^{n}}\int_{\mathbb R^{m}}\\
&\qquad \times \int_{\Bbb R^m}\chi^{(1)}_{t_1}(x_1-y_1,x_2-w_2)\chi^{(2)}_{t_2}(w_2-y_2)dw_2\\
&\qquad\times\Big|\int_{10I}\int_{\Bbb R^m} t_1^{-n-m} \exp\Big(-{|y_1-z_1|^2+|y_2-z_2|^2\over ct_1^2}\Big)
t_2^2{\triangle^{(2)}}e^{-t_2^2{\triangle^{(2)}}}a_{R,2} (z_1,z_2)dz_2dz_1\Big|^2 \\
&\qquad\times {dy_2dt_2\over t_2^{m+1}}{dy_1dt_1\over
t_1^{n+m+1+4M}}dx_2dx_1\\
&\lesssim  \int_{|x_1-x_I|\approx 2^{j_1}\ell(I)} \int_{|x_2-x_J|\approx 2^{j_2}\ell(J)}\int_{\frac{|x_1-x_I|+|x_2-x_J|}8}^\infty\!\!\int_0^{\ell(J)}\!\! \int_{\mathbb R^{m}} t_1^n t_2^m \\
&\qquad\times  t_1^{-2n}\bigg|\int_{\Bbb R^m} t_1^{-m} \exp\Big(-{ |y_2-z_2|^2\over ct_1^2}\Big)
\bigg(\int_{10I}t_2^2{\triangle^{(2)}}e^{-t_2^2{\triangle^{(2)}}}a_{R,2} (z_1,z_2)dz_1\bigg)\ dz_2\bigg|^2 \\
&\qquad\times {dy_2dt_2\over t_2^{m+1}}{dt_1\over
t_1^{n+m+1+4M}}dx_2dx_1\\
&\lesssim  \int_{|x_1-x_I|\approx 2^{j_1}\ell(I)} \int_{|x_2-x_J|\approx 2^{j_2}\ell(J)}\int_{\frac{|x_1-x_I|+|x_2-x_J|}8}^\infty \\
&\qquad\times  \int_{\mathbb R^{m}}\, \int_0^\infty \Big|M_2\Big(\int_{10I}|t_2^2{\triangle^{(2)}}e^{-t_2^2{\triangle^{(2)}}}a_{R,2} (z_1,\cdot)|dz_1\Big)(y_2)\Big|^2{dt_2\over t_2} \, dy_2\ {dt_1\over t_1^{2n+m+1+4M}}dx_2dx_1\\
&\lesssim  \int_{|x_1-x_I|\approx 2^{j_1}\ell(I)} \int_{|x_2-x_J|\approx 2^{j_2}\ell(J)}\int_{\frac{|x_1-x_I|+|x_2-x_J|}8}^\infty \\
&\qquad\times  \int_{\mathbb R^{m}}\, \Big|\int_{10I}|a_{R,2} (z_1,y_2)|dz_1\Big|^2 dy_2\ {dt_1dx_2dx_1\over t_1^{2n+m+1+4M}}\\
&\lesssim |I|\|a_{R,2}\|_2^2 \int_{|x_1-x_I|\approx 2^{j_1}\ell(I)} \int_{|x_2-x_J|\approx 2^{j_2}\ell(J)} \frac1{(|x_1-x_I|+|x_2-x_J|)^{2n+m+4M}} dx_2dx_1\\
&\lesssim  |I| \|a_{R,2}\|_2^2 \frac{ (2^{j_1}\ell(I))^n (2^{j_2}\ell(J))^m }{(2^{j_1}\ell(I)+2^{j_2}\ell(J))^{2n+m+4M}} \\
&\lesssim |R| \ell(I)^{-4M}\|a_{R,2}\|_2^2 (2^{j_1}\ell(I))^{-n} (2^{j_2}\ell(J))^{-m}    2^{-2Mj_1}2^{-(m-2M)j_2},
\end{align*}
where in the last inequality we use the fact that $\ell(I)\leq \ell(J)$ and we also require that $M>{m\over2}$.

Finally we estimate the term $\textrm{I}_{24}$. Note that  $a_R=(({\triangle^{(1)}})^M\otimes_2({\triangle^{(2)}})^M)b_R$ and supp $b_{R}\subset 10R=10(I\times J)$.
Let $A(x_1,x_2)=\frac{|x_1-x_I|+|x_2-x_J|}8.$ Then
\begin{align*}
&\textrm{I}_{24}\\
&= \int_{|x_1-x_I|\approx 2^{j_1}\ell(I)} \int_{|x_2-x_J|\approx 2^{j_2}\ell(J)}\bigg(\int_{\ell(I)}^{A(x_1,x_2)}\int_{\ell(J)}^{A(x_1,x_2)}+\int_{\ell(I)}^{A(x_1,x_2)}\int_{A(x_1,x_2)}^\infty +\int_{A(x_1,x_2)}^\infty\int_{\ell(J)}^{A(x_1,x_2)}\\
 &\quad   +\int_{A(x_1,x_2)}^\infty\int_{A(x_1,x_2)}^\infty \bigg)   \int_{\mathbb R^{n}}\int_{\mathbb R^{m}} \int_{\Bbb R^m}\chi^{(1)}_{t_1}(x_1-y_1,x_2-w_2)\chi^{(2)}_{t_2}(w_2-y_2)dw_2\\
&\qquad\times \Big|\big((t_1^2{\triangle^{(1)}})^{M+1}e^{-t_1^2{\triangle^{(1)}}}\otimes_2 (t_2^2{\triangle^{(2)}})^{M+1}e^{-t_2^2{\triangle^{(2)}}}\big) b_R (y_1,y_2)\Big|^2 {dy_2dt_2\over t_2^{m+4M+1}}{dy_1dt_1\over
t_1^{n+m+4M+1}} dx_2dx_1\\
&=:\textrm{I}_{241}+\textrm{I}_{242}+\textrm{I}_{243}+\textrm{I}_{244}.
\end{align*}
To estimate the terms $\textrm{I}_{241}$, 
we consider two cases:

\noindent Case(3):  $|x_1-x_I|\ge |x_2-x_J|$.

In this case we see that  $t_1<(|x_1-x_I|+|x_2-x_J|)/8\le |x_1-x_I|/4$.
Since $|x_1-x_I|>100\ell(I)$, $|x_1-y_1|<t_1$ and $z_1\in 10I$, then $|y_1-z_1|\geq  |x_1-x_I|/4\ge (|x_1-x_I|+|x_2-x_J|)/8$. Based on these observations, we get that
\begin{align*}
&\textrm{I}_{241}\\
&\lesssim \int_{|x_1-x_I|\approx 2^{j_1}\ell(I)} \int_{|x_2-x_J|\approx 2^{j_2}\ell(J)}\int_{\ell(I)}^{A(x_1,x_2)}\!\!\int_{\ell(J)}^\infty\!\! \int_{\mathbb R^{n}}\int_{\mathbb R^{m}} \\
&\quad\times \int_{\Bbb R^m}\chi^{(1)}_{t_1}(x_1-y_1,x_2-w_2)\chi^{(2)}_{t_2}(w_2-y_2)dw_2 t_1^{-2n}\exp\Big(-{|x_1-x_I|^2\over ct_1^2}\Big)\exp\Big(-{|x_2-x_J|^2\over ct_1^2}\Big) \\
&\quad\times \Big|\int_{10I}\int_{\Bbb R^m}\int_{\Bbb R^m} t_1^{-m} \exp\Big(-{|y_2-z_2|^2\over ct_1^2}\Big)
t_2^{-m} \exp\Big(-{|z_2-u_2|^2\over ct_1^2}\Big)|b_{R} (z_1,u_2)|du_2dz_2dz_1\Big|^2\\
&\qquad{dy_1dt_1\over
t_1^{n+m+1+4M}}{dy_2dt_2\over t_2^{4M+1}}dx_2dx_1\\
&\lesssim \int_{|x_1-x_I|\approx 2^{j_1}\ell(I)} \int_{|x_2-x_J|\approx 2^{j_2}\ell(J)}\int_{\ell(I)}^\infty\!\!\int_{\ell(J)}^\infty t_1^{-2n}\exp\Big(-{|x_1-x_I|^2\over ct_1^2}\Big)\exp\Big(-{|x_2-x_J|^2\over ct_1^2}\Big) \\
&\quad\times {\max\{t_1,t_2\}^{2\alpha}\over (\max\{t_1,t_2\})^{2m+2\alpha}} |R|\|b_R\|_2^2 {dt_1\over t_1^{4M+1}}{dt_2\over t_2^{4M+1}}dx_2dx_1,
\end{align*}
where the last inequality follows by using the almost orthogonality estimate as in \eqref{t12} and by skipping the $|y_2-u_2|$ in this estimate.
We then use the fact that $e^{-s}\leq Cs^{-k}$  for any $k>0$ to obtain
\begin{align*}
\textrm{I}_{241}
&\lesssim \int_{|x_1-x_I|\approx 2^{j_1}\ell(I)} \int_{|x_2-x_J|\approx 2^{j_2}\ell(J)}\int_{\ell(I)}^\infty\!\!\int_{\ell(J)}^\infty {1\over|x_1-x_I|^{2\alpha_1}} {1\over|x_2-x_J|^{2\alpha_2}} \\
&\quad\times  (t_1^{-2m}+t_2^{-2m}) |R|\|b_R\|_2^2 {dt_1\over t_1^{2n-2\alpha_1-2\alpha_2+1+4M}}{dt_2\over t_2^{4M+1}}dx_2dx_1\\
&\lesssim  |R|\ell(I)^{-4M}\ell(J)^{-4M}\|b_R\|_2^2\,(2^{j_1}\ell(I))^{-n}(2^{j_2}\ell(J))^{-m} 2^{(2n-2\alpha_1)j_1}2^{(2m-2\alpha_2)j_2},
\end{align*}
where the last inequality follows from the standard integration estimate, and we require that $n<\alpha_1$, $m<\alpha_2$ and $2\alpha_1+2\alpha_2<2n+4M$. This also implies that $M>{m\over2}$.

\smallskip
\noindent Case (4):  $|x_1-x_I|< |x_2-x_J|$.

In this case we have $t_1+t_2<(|x_1-x_I|+|x_2-x_J|)/4\le |x_2-x_J|/2$ and $|x_1-x_I|/4<(|x_1-x_I|+|x_2-x_J|)/8$.
Now we consider the lower bound for $|y_1-z_1|$. We first consider the case $t_1<|x_1-x_I|/4$. Since $|x_1-x_I|>100\ell(I)$, $|x_1-y_1|<t_1<|x_1-x_I|/4$ and $z_1\in 10I$, then $|y_1-z_1|\ge |x_1-x_I|/4$. For the case  $t_1>|x_1-x_I|/4$, although we know that $t_1<(|x_1-x_I|+|x_2-x_J|)/8$, there is no specific estimate for the lower bound for $|y_1-z_1|$, thus we can only use the fact that $\exp\big(-{|y_1-z_1|^2\over ct_1^2}\big)\leq1$.
Moreover, since $|x_2-x_J|>100\ell(J)$, $|x_2-y_2|<t_1+t_2< |x_2-x_J|/2$, then for $u_2\in 10J$ we have $|y_2-u_2|\geq  |x_2-x_J|/4\ge (|x_1-x_I|+|x_2-x_J|)/8$. Now based on these observations, by splitting the range of $t_1$ and by using the almost orthogonality estimate as in \eqref{t12} similar to Case (3) above, we have that
\begin{align*}
\textrm{I}_{241}
&\lesssim \int_{|x_1-x_I|\approx 2^{j_1}\ell(I)} \int_{|x_2-x_J|\approx 2^{j_2}\ell(J)}\bigg(\int_{\ell(I)}^{\frac{|x_1-x_I|}4}\int_{\ell(J)}^{A(x_1,x_2)}+\int_{\frac{|x_1-x_I|}4}^{A(x_1,x_2)}\int_{\ell(J)}^{A(x_1,x_2)} \bigg) \int_{\mathbb R^{n}}\int_{\mathbb R^{m}}\\
 &\qquad\times   \int_{\Bbb R^m}\chi^{(1)}_{t_1}(x_1-y_1,x_2-w_2)\chi^{(2)}_{t_2}(w_2-y_2)dw_2\\ 
&\qquad\times \Big|\int_{10I}\int_{\Bbb R^m}\int_{\Bbb R^m} t_1^{-n-m} \exp\Big(-{|x_1-x_I|^2+|y_2-z_2|^2\over ct_1^2}\Big)
\\
&\qquad\qquad\qquad \ t_2^{-m} \exp\Big(-{|z_2-u_2|^2\over ct_1^2}\Big)|b_{R} (z_1,u_2)|du_2dz_2dz_1\Big|^2{dy_2dt_2\over t_2^{m+4M+1}}{dy_1dt_1\over
t_1^{n+m+4M+1}}dx_2dx_1\\
&\lesssim  \int_{|x_1-x_I|\approx 2^{j_1}\ell(I)} \int_{|x_2-x_J|\approx 2^{j_2}\ell(J)}\int_{\ell(I)}^{\frac{|x_1-x_I|}4}\int_{\ell(J)}^{A(x_1,x_2)} t_1^{-2n}\exp\Big(-{|x_1-x_I|^2\over ct_1^2}\Big)\\
&\qquad\times \frac{(\max\{t_1,t_2\})^{2\alpha}}{|x_2-x_J|^{2m+2\alpha}}|R|\|b_R \|_2^2 {dt_1\over t_1^{4M+1}}{dt_2\over t_2^{4M+1}}dx_2dx_1\\
&\quad + \int_{|x_1-x_I|\approx 2^{j_1}\ell(I)} \int_{|x_2-x_J|\approx 2^{j_2}\ell(J)}\\
&\qquad\qquad\int_{\frac{|x_1-x_I|}4}^{A(x_1,x_2)}\int_{\ell(J)}^\infty\frac{(\max\{t_1,t_2\})^{2\alpha}}{|x_2-x_J|^{2m+2\alpha}}|R|\|b_R \|_2^2 {dt_1\over t_1^{2n+4M+1}}{dt_2\over t_2^{4M+1}}dx_2dx_1.
\end{align*}
We use the facts that $e^{-s}\leq Cs^{-k}$  for any $k>0$ and that $\ell(I)\leq \ell(J)$ to obtain that
\begin{align*}
\textrm{I}_{241}
&\lesssim  |R| \ell(I)^{-4M}\ell(J)^{-4M} \|b_R \|_2^2(2^{j_1}\ell(I))^{-n}(2^{j_2}\ell(J))^{-m} 2^{-(2n-2\alpha_1)j_1} 2^{-2\alpha j_2}\\
&\quad+  |R| \ell(I)^{-4M}\ell(J)^{-4M} \|b_R \|_2^2(2^{j_1}\ell(I))^{-n}(2^{j_2}\ell(J))^{-m} 2^{-(4M-2\alpha)j_1} 2^{-2\alpha j_2}\\
&\quad+  |R| \ell(I)^{-4M}\ell(J)^{-4M} \|b_R \|_2^2(2^{j_1}\ell(I))^{-n}(2^{j_2}\ell(J))^{-m} 2^{-4Mj_1} 2^{-2\alpha j_2}
\end{align*}
for $n<\alpha_1$, $0<\alpha$, $\alpha<2M$ and $\alpha_1+\alpha<n+2M$.

We then estimate the term $\textrm{I}_{242}$. Similar to the estimate of  the term $\textrm{I}_{241}$, we also consider Case (3) and  Case (4) for the comparison of $|x_1-x_I|$ and $|x_2-x_J|$. And we also
 note that $t_1<(|x_1-x_I|+|x_2-x_J|)/8 \le t_2$. Hence, there is no lower bound for $|y_2-u_2|$ and the almost orthogonality estimate
 appearing in   the estimates for term $\textrm{I}_{242}$ will be replaced by
 ${1\over t_2^{2m}}$. Then we have
\begin{align*}
\textrm{I}_{242}
&\lesssim \int_{|x_1-x_I|\approx 2^{j_1}\ell(I)} \int_{|x_2-x_J|\approx 2^{j_2}\ell(J)}\int_{\ell(I)}^\infty\!\!\int_{|x_1-x_I|+|x_2-x_J|\over8}^\infty t_1^{-2n}\exp\Big(-{|x_1-x_I|^2\over ct_1^2}\Big) \\
&\quad\times {1\over t_2^{2m}} |R|\|b_R\|_2^2 {dt_1\over t_1^{4M+1}}{dt_2\over t_2^{4M+1}}dx_2dx_1
\\
&\quad + \int_{|x_1-x_I|\approx 2^{j_1}\ell(I)} \int_{|x_2-x_J|\approx 2^{j_2}\ell(J)}\int_{\frac{|x_1-x_I|}4}^{A(x_1,x_2)}\\
&\qquad\qquad\int_{{|x_1-x_I|+|x_2-x_J|\over8}}^\infty\frac{1}{ t_2^{2m}}|R|\|b_R \|_2^2 {dt_1\over t_1^{2n+4M+1}}{dt_2\over t_2^{4M+1}}dx_2dx_1
\\
&\lesssim  |R| \ell(I)^{-4M}\ell(J)^{-4M} \|b_R \|_2^2(2^{j_1}\ell(I))^{-n}(2^{j_2}\ell(J))^{-m} 2^{-(2\alpha_1-2n)j_1} 2^{-4M j_2}\\
&\quad + |R| \ell(I)^{-4M}\ell(J)^{-4M} \|b_R \|_2^2(2^{j_1}\ell(I))^{-n}(2^{j_2}\ell(J))^{-m} 2^{-4Mj_1} 2^{-4M j_2}.
\end{align*}
Here we require that {$n<\alpha_1<4M+2n.$}

We now consider the term $\textrm{I}_{243}$. And we also
 note that in this case $t_1> t_2$. Hence, there is no lower bound for $|y_1-z_1|$. So we will use the fact that $\exp\big(-{|y_1-z_1|^2\over ct_1^2}\big)\leq1$. Moreover, there is also no lower bound for $|y_2-z_2|$.  And hence the almost orthogonality estimate
 appearing in  the estimates for term $\textrm{I}_{242}$ will be replaced by
 ${1\over t_1^{2m}}$. Then we have
\begin{align*}
\textrm{I}_{243}
&\lesssim  |R|\|b_R\|_2^2\int_{|x_1-x_I|\approx 2^{j_1}\ell(I)} \int_{|x_2-x_J|\approx 2^{j_2}\ell(J)}\int_{A(x_1,x_2)}^\infty {dt_1\over t_1^{2n+2m+4M+1}} \int_{\ell(J)}^\infty {dt_2\over t_2^{4M+1}}dx_1dx_2\\
&\lesssim  |R|\|b_R\|_2^2 (2^{j_1}\ell(I))^n ( 2^{j_2}\ell(J))^m {1\over (2^{j_1}\ell(I)+ 2^{j_2}\ell(J))^{2n+2m+4M}} {1\over \ell(J)^{4M}}\\
&\lesssim  |R|\ell(I)^{-4M}\ell(J)^{-4M}\|b_R\|_2^2 (2^{j_1}\ell(I))^{-n}(2^{j_2}\ell(J))^{-m}  2^{-2Mj_1}2^{-2Mj_2}.
\end{align*}

We finally estimate the term $\textrm{I}_{244}$. And we also
 note that in this case there are no lower bounds for $|y_1-z_1|$ or $|y_2-z_2|$. So we will use the fact that $\exp\big(-{|y_1-z_1|^2\over ct_1^2}\big)\leq1$.  And the almost orthogonality estimate
 appearing in  the estimates for term $\textrm{I}_{241}$ will be replaced by
 ${1\over \max\{ t_1,t_2\}^{2m}}$. Then we have
\begin{align*}
\textrm{I}_{244}
&\lesssim  \int_{|x_1-x_I|\approx 2^{j_1}\ell(I)} \int_{|x_2-x_J|\approx 2^{j_2}\ell(J)}\int_{A(x_1,x_2)}^\infty\int_{A(x_1,x_2)}^\infty t_1^{-2n}(\max\{t_1,t_2\})^{-2m}\\
&\qquad\times  |R|\|b_R\|_2^2{dt_1\over t_1^{4M+1}}{dt_2\over t_2^{4M+1}}dx_1dx_2\\
&\lesssim  |R|\ell(I)^{-4M}\ell(J)^{-4M}\|b_R\|_2^2 (2^{j_1}\ell(I))^{-n}(2^{j_2}\ell(J))^{-m}  2^{-2Mj_1}2^{-2Mj_2}\\
&\quad+|R|\ell(I)^{-4M}\ell(J)^{-4M}\|b_R\|_2^2 (2^{j_1}\ell(I))^{-n}(2^{j_2}\ell(J))^{-m}  2^{-4Mj_1}2^{-4Mj_2}.
\end{align*}

Combing the estimates of $\textrm{I}_{2i}$ for $i=1,2,3, 4$, we can show that
\begin{align*}
\textrm{I}_2&\lesssim |R|^{1/2} \gamma_1(R)^{-\delta}\ell(I)^{-2M}\ell(J)^{-2M}\Big(\|( (\ell(I)^2{\triangle^{(1)}})^M) \otimes_2
(\ell(J)^2{\triangle^{(2)}})^M)b_{R}\|_2\\
& +  \|((\ell(I)^2{\triangle^{(1)}})^M\otimes_2
1\!\!1_2) b_{R}\|_2+ \|(1\!\!1_1\otimes_2
(\ell(J)^2{\triangle^{(2)}})^M)b_{R}\|_2+ \|b_{R}\big\|_2\Big)
\end{align*}
for $\delta>0$. Estimates of $\textrm{I}_1$ and $\textrm{I}_2$, together with H\"older's inequality and Journ\'e's covering lemma, show that
\begin{align*}
\textrm{I}&\leq \sum_{R\in m(\Omega) }
 \int_{(100\widetilde{I})^c\times\mathbb{R}^{m}} |S_{F,\triangle^{(1)},\triangle^{(2)}}(a_R)(x_1,x_2)|dx_2dx_1\\
&\lesssim \sum_{R\in m(\Omega) }|R|^{1/2} \gamma_1(R)^{-\delta}
  \ell(I)^{-2M}\ell(J)^{-2M}\Big(\|( (\ell(I)^2{\triangle^{(1)}})^M) \otimes_2
(\ell(J)^2{\triangle^{(2)}})^M)b_{R}\|_2\\
  &+ \|((\ell(I)^2{\triangle^{(1)}})^M\otimes_2
1\!\!1_2) b_{R}\|_2+\|(1\!\!1_1\otimes_2
(\ell(J)^2\triangle^{(2)})^M)b_{R}\|_2+\|b_{R}\big\|_2\Big)\\
&\lesssim  \bigg(\sum_{R\in m(\Omega) } \ell(I)^{-4M}\ell(J)^{-4M} \Big(\|( (\ell(I)^2{\triangle^{(1)}})^M) \otimes_2
(\ell(J)^2{\triangle^{(2)}})^M)b_{R}\|^2_2\\
&\qquad +  \|((\ell(I)^2{\triangle^{(1)}})^M\otimes_2
1\!\!1_2) b_{R}\|^2_2 + \|(1\!\!1_1\otimes_2
(\ell(J)^2\triangle^{(2)})^M)b_{R}\|^2_2\\
&\qquad   +\|b_{R}\big\|^2_2\Big)\bigg)^{1/2}\Big(\sum_{R\in m(\Omega) }|R|  \gamma_1(R)^{-2\delta} \Big)^{1/2}  \\
&\lesssim  |\Omega|^{-{1\over 2}}|\Omega|^{{1\over 2}}\lesssim  1.
\end{align*}
For the term $\textrm{II}$, we have
\begin{eqnarray*}
\int_{\mathbb{R}^{n}\times (100\widetilde{J})^c}|S_{F,\triangle^{(1)},\triangle^{(2)}}(a_R)(x_1,x_2)|dx_2dx_1
 &  =&
\int_{100I \times (100\widetilde{J})^c} |S_{F,\triangle^{(1)},\triangle^{(2)}}(a_R)(x_1,x_2)|dx_2dx_1 \\[2pt]
&&+  \int_{(100I)^c\times (100\widetilde{J})^c} |S_{F,\triangle^{(1)},\triangle^{(2)}}(a_R)(x_1,x_2)|dx_2dx_1 \\[2pt]
&=& \textrm{II}_1+\textrm{II}_2.
\end{eqnarray*}
The estimate of $\textrm{II}_2$ is symmetric to the estimate of $\textrm{I}_2$, since
one can write
\begin{align*}
\textrm{II}_2&=\int_{(100I)^c\times (100\widetilde{J})^c} |S_{F,\triangle^{(1)},\triangle^{(2)}}(a_R)(x_1,x_2)|dx_2dx_1 \\
   &\leq\sum_{j_1=6}^\infty\sum_{j_2=\tilde j}^\infty\int_{|x_1-x_I|\approx 2^{j_1}\ell(I)} \int_{|x_2-x_J|\approx 2^{j_2}\ell(J)}|S_{F,\triangle^{(1)},\triangle^{(2)}}(a_R)(x_1,x_2)|dx_2dx_1 \\
   &\le \sum_{j_1=6}^\infty\sum_{j_2=\tilde j}^\infty (2^{j_1}\ell(I))^{n/2}(2^{j_2}\ell(J))^{m/2} \\
   &\qquad\times \bigg(\int_{|x_1-x_I|\approx 2^{j_1}\ell(I)} \int_{|x_2-x_J|\approx 2^{j_2}\ell(J)}|S_{F,\triangle^{(1)},\triangle^{(2)}}(a_R)(x_1,x_2)|^2dx_2dx_1\bigg)^{1/2},
\end{align*}
where $\tilde j$ is the smallest integer such that
$2^{\tilde j} J \cap (100\tilde J)^c\not=\emptyset$. Hence, following the approach and technique in the estimate of $\textrm{I}_2$,
we obtain that
\begin{align*}
\textrm{II}_2&\lesssim |R|^{1/2} \gamma_1(R)^{-\delta}\ell(I)^{-2M}\ell(J)^{-2M}\Big(\|( (\ell(I)^2{\triangle^{(1)}})^M) \otimes_2
(\ell(J)^2{\triangle^{(2)}})^M)b_{R}\|^2_2\\
& \qquad\qquad+  \|(1\!\!1_1\otimes_2
(\ell(J)^2{\triangle^{(2)}})^M)b_{R}\|^2_2+ \|b_{R}\big\|^2_2\Big)
\end{align*}
for $\delta>0$.
So we just estimate the term $\textrm{II}_1$.
We note that the estimate for this term is essentially different from the tensor product setting. We will use the 
techniques that we developed in the estimate of the term $\textrm{I}_2$. We begin by using H\"older's
inequality.
{
\begin{eqnarray*}
\textrm{II}_1
 &  \leq& \sum_{j_2= \tilde{\tilde j}}^\infty
\int_{100I} \int_{ |x_2-x_J|\approx 2^{j_2}\ell(J) } |S_{F,\triangle^{(1)},\triangle^{(2)}}(a_R)(x_1,x_2)|dx_2dx_1\\
 &  \lesssim&\sum_{j_2= \tilde{\tilde j}}^\infty  |I|^{1\over2} |2^{j_2}\ell(J) |^{m\over2} 
\bigg(\int_{100I} \int_{ |x_2-x_J|\approx 2^{j_2}\ell(J) } |S_{F,\triangle^{(1)},\triangle^{(2)}}(a_R)(x_1,x_2)|^2dx_2dx_1\bigg)^{1\over2},
\end{eqnarray*}
where $\tilde{\tilde j}$ is the smallest integer such that
$2^{\tilde{\tilde j}} I \cap (100\tilde J)^c\not=\emptyset$. We consider the four cases.
\begin{align*}
&\int_{100I} \int_{|x_2-x_J|\approx 2^{j_2}\ell(J)}|S_{F,\triangle^{(1)},\triangle^{(2)}}(a_R)(x_1,x_2)|^2dx_2dx_1 \\
&= \int_{100I} \int_{|x_2-x_J|\approx 2^{j_2}\ell(J)}\Big( \int_0^{\ell(I)}\!\!\int_0^{\ell(J)} +\int_0^{\ell(I)}\!\!\int_{\ell(J)}^\infty +
\int_{\ell(I)}^{\infty}\int_0^{\ell(J)}+ \int_{\ell(I)}^{\infty}\int_{\ell(J)}^\infty \Big)\\
&\quad\times \int_{\mathbb R^{n}}\int_{\mathbb R^{m}}\ \  \int_{\Bbb R^m}\chi^{(1)}_{t_1}(x_1-y_1,x_2-w_2)\chi^{(2)}_{t_2}(w_2-y_2)dw_2\\
&\qquad\times \Big|\big( t_1^2{\triangle^{(1)}}e^{-t_1^2{\triangle^{(1)}}}\,
t_2^2{\triangle^{(2)}}e^{-t_2^2{\triangle^{(2)}}}\big) (a_R) (y_1,y_2)\Big|^2{dy_2dt_2\over t_2^{m+1}} {dy_1dt_1\over
t_1^{n+m+1}}dx_2dx_1\\
&=:\textrm{II}_{11}+\textrm{II}_{12}+\textrm{II}_{13}+\textrm{II}_{14}.
\end{align*}

We first consider the term $\textrm{II}_{11}$.
For $t_1<\ell(I)$ and $t_2< \ell(J)$, we have that $\chi^{(1)}_{t_1}(x_1-y_1,x_2-w_2)\chi^{(2)}_{t_2}(w_2-y_2)$ gives
$|x_2-y_2|\le t_1+t_2\le 2\ell(J)$. Since $|x_2-x_J|\approx 2^{j_2}\ell(J)$, we have $|y_2-x_J|\approx 2^{j_2}\ell(J)$
and $|y_2-z_2|\ge |y_2-x_J|-|z_2-x_J|\ge |x_2-x_J|/4$ if $z_2\in B(x_J, |x_2-x_J|/4)$. 
Hence, from the estimate in \eqref{t12} and the $L^2(\Bbb R^{m+n})$-boundedness of area square function,
\begin{align*}
&\textrm{II}_{11} \\
&\leq\int_{100I} \int_{|x_2-x_J|\approx 2^{j_2}\ell(J)} \int_0^{\ell(I)}\!\!\int_0^{\ell(J)}
\int_{\mathbb R^{n}}\int_{\Bbb R^m}  \int_{\Bbb R^m}\chi^{(1)}_{t_1}(x_1-y_1,x_2-w_2)\chi^{(2)}_{t_2}(w_2-y_2)dw_2\\
&\qquad\times \Big|\big( t_1^2{\triangle^{(1)}}e^{-t_1^2{\triangle^{(1)}}}\ \chi_{B(x_J, |x_2-x_J|/4)}\ 
t_2^2{\triangle^{(2)}}e^{-t_2^2{\triangle^{(2)}}}\big) (a_R) (y_1,y_2)\Big|^2{dy_2dt_2\over t_2^{m+1}} {dy_1dt_1\over
t_1^{n+m+1}}dx_2dx_1\\
&\ \ +\int_{100I} \int_{|x_2-x_J|\approx 2^{j_2}\ell(J)} \int_0^{\ell(I)}\!\!\int_0^{\ell(J)}
\int_{\mathbb R^{n}}\int_{\Bbb R^m}  \int_{\Bbb R^m}\chi^{(1)}_{t_1}(x_1-y_1,x_2-w_2)\chi^{(2)}_{t_2}(w_2-y_2)dw_2\\
&\qquad\times \Big|\big( t_1^2{\triangle^{(1)}}e^{-t_1^2{\triangle^{(1)}}}\ \chi_{B(x_J, |x_2-x_J|/4)^c}\ 
t_2^2{\triangle^{(2)}}e^{-t_2^2{\triangle^{(2)}}}\big) (a_R) (y_1,y_2)\Big|^2{dy_2dt_2\over t_2^{m+1}} {dy_1dt_1\over
t_1^{n+m+1}}dx_2dx_1\\
&\lesssim  \int_{100I} \int_{|x_2-x_J|\approx 2^{j_2}\ell(J)}\int_0^{\ell(I)}\!\!\int_0^{\ell(J)} \int_{\mathbb R^{n}}\int_{\mathbb R^{m}} \int_{\Bbb R^m}\chi^{(1)}_{t_1}(x_1-y_1,x_2-w_2)\chi^{(2)}_{t_2}(w_2-y_2)dw_2\\
&\quad\times\Big|\int_{10I}\int_{B(x_J, |x_2-x_J|/4)} t_1^{-n-m} \exp\Big(-{|y_1-z_1|^2+|y_2-z_2|^2\over ct_1^2}\Big)
\big(t_2^2{\triangle^{(2)}}e^{-t_2^2{\triangle^{(2)}}}a_{R} (z_1,z_2)\big)dz_2\,dz_1\Big|^2 \\
&\quad\times {dy_2dt_2\over t_2^{m+1}}{dy_1dt_1\over
t_1^{n+m+1}}dx_2dx_1\\
&\ \ + \int_{|x_2-x_J|\approx 2^{j_2}\ell(J)} \!\!\int_0^{\ell(J)} \ \ 
\\
&\quad\times \int_{\mathbb R^{n}}\int_{\Bbb R^m} \int_0^{\ell(I)}  \Big|\big( t_1^2{\triangle^{(1)}}e^{-t_1^2{\triangle^{(1)}}}\ \chi_{B(x_J, |x_2-x_J|/4)^c}\ 
t_2^2{\triangle^{(2)}}e^{-t_2^2{\triangle^{(2)}}}\big) (a_R) (y_1,y_2)\Big|^2 {dt_1\over
t_1}dy_1dy_2 {dt_2dx_2\over t_2^{m+1}} \\
&\lesssim \int_{100I} \int_{|x_2-x_J|\approx 2^{j_2}\ell(J)}\int_0^{\ell(I)}\!\!\int_0^{\ell(J)} \int_{\mathbb R^{n}}\int_{\mathbb R^{m}} \int_{\Bbb R^m}\chi^{(1)}_{t_1}(x_1-y_1, w_2-y_2)\chi^{(2)}_{t_2}(x_2-w_2)dw_2\\
&\quad\times t_1^{-2n}\exp\Big(-{|x_2-x_J|^2\over 4ct_1^2}\Big) \bigg|M_2\bigg(\int_{10I} |t_2^2{\triangle^{(2)}}e^{-t_2^2{\triangle^{(2)}}}a_{R} (z_1,\cdot)|dz_1\bigg)(y_2)\bigg|^2\\
&\quad \times{dy_1dt_1\over
t_1^{n+m+1}}{dy_2dt_2\over t_2^{m+1}}dx_2dx_1\\
&\ \ + \int_{|x_2-x_J|\approx 2^{j_2}\ell(J)} \int_0^{\ell(J)}  
 \int_{\mathbb R^{n}}\int_{\Bbb R^m}  \Big| \chi_{B(x_J, |x_2-x_J|/4)^c}
t_2^2{\triangle^{(2)}}e^{-t_2^2{\triangle^{(2)}}} (a_R) (y_1,y_2)\Big|^2dy_1dy_2 {dt_2dx_2\over t_2^{m+1}} \\
&=:\textrm{II}_{111}+\textrm{II}_{112}.
\end{align*}

For the term $\textrm{II}_{111}$, by using the $L^2(\mathbb R^m)$ boundedness of the Hardy--Littlewood maximal function, we obtain that
\begin{align*}
\textrm{II}_{111}
&\lesssim |I| \int_{|x_2-x_J|\approx 2^{j_2}\ell(J)}\int_0^{\ell(I)}\!\!\int_0^{\ell(J)}t_1^{-2n} \Big({t_1 \over |x_2-x_J|}\Big)^\alpha   \\
&\quad\times \int_{\mathbb R^{m}}\bigg|M_2\bigg(\int_{10I} |t_2^2{\triangle^{(2)}}e^{-t_2^2{\triangle^{(2)}}}a_{R} (z_1,\cdot)|dz_1\bigg)(y_2)\bigg|^2dy_2{dt_1\over
t_1^{m+1}}{dt_2\over t_2}dx_2\\
&\lesssim |I| \int_{|x_2-x_J|\approx 2^{j_2}\ell(J)}\Big({1 \over 2^{j_2}\ell(J)}\Big)^\alpha  \int_0^{\ell(I)}\!\!\int_0^{\ell(J)}    \\
&\quad\times \int_{\mathbb R^{m}}\bigg|\int_{10I} |t_2^2{\triangle^{(2)}}e^{-t_2^2{\triangle^{(2)}}}a_{R} (z_1,y_2)|dz_1\bigg|^2dy_2{dt_1\over
t_1^{2n+m-\alpha+1}}{dt_2\over t_2}dx_2\\
&\lesssim |I| (2^{j_2}\ell(J))^m \Big({1 \over 2^{j_2}\ell(J)}\Big)^\alpha  \ell(I)^{\alpha-2n-m}   |I|   \\
&\quad\times \int_{10I} \int_0^{\ell(J)} \int_{\mathbb R^{m}} |t_2^2{\triangle^{(2)}}e^{-t_2^2{\triangle^{(2)}}}a_{R} (z_1,y_2)|^2 dy_2{dt_2\over t_2}dz_1\\
&\lesssim  \Big({1 \over 2^{j_2}}\Big)^{\alpha-m}  \int_{10I}\int_{\mathbb R^{m}} |a_{R} (z_1,y_2)|^2 dy_2dz_1\\
&= \Big({1 \over 2^{j_2}}\Big)^{\alpha-m}  \|a_{R} \|_2^2,
\end{align*}
where we chose $\alpha>2n+2m$, and we applied the Littlewood--Paley estimate in the last inequality.

For the term $\textrm{II}_{112}$, we have $w_2\in B(x_J, |x_2-x_J|/4)^c$ and $u_2\in 10J$ give $|w_2-u_2|\ge |x_2-x_J|/8$.
Hence
\begin{align*}
\textrm{II}_{112}
&\lesssim \int_0^{\ell(J)}\int_{|x_2-x_J|\approx 2^{j_2}\ell(J)} \int_{10I} \int_{B(x_J, |x_2-x_J|/4)^c} \\
&\qquad\times  \bigg|\int_{10J} t_2^{-m}\exp\Big(-{|w_2-u_2|^2\over 4ct_2^2}\Big)a_R(x_1,u_2)\bigg|^2 dx_1dw_2dx_2{dt_2\over t_2^{m+1}}\\
&\lesssim \int_0^{\ell(J)}\int_{|x_2-x_J|\approx 2^{j_2}\ell(J)}\int_{B(x_J, |x_2-x_J|/4)^c}  t_2^{-m}\exp\Big(-{|w_2-u_2|^2\over 4ct_2^2}\Big)dw_2 \\
&\qquad\times  |J|\|a_R\|_2^2\frac{t_2^{2\alpha}}{|x_2-x_J|^{2\alpha}} dx_2{dt_2\over t_2^{2m+1}}\\
&\lesssim ( 2^{j_2}\ell(J))^{m-2\alpha_2}|J|\|a_R\|_2^2\ell(J)^{2\alpha_2-2m} \\
&\lesssim ( 2^{j_2})^{m-2\alpha_2}\|a_R\|_2^2,
\end{align*}
for $\alpha_2>m.$

To estimate $\textrm{II}_{12}$, we write
\begin{align*}
&\textrm{II}_{12}= \int_{100I} \int_{|x_2-x_J|\approx 2^{j_2}\ell(J)}\int_0^{\ell(I)}\Big(\int_{\ell(J)}^{\frac{|x_2-x_J|}8} +\int_{\frac{|x_2-x_J|}8}^\infty \Big)\\
&\quad\times \int_{\mathbb R^{n}}\int_{\mathbb R^{m}}\ \  \int_{\Bbb R^m}\chi^{(1)}_{t_1}(x_1-y_1,x_2-w_2)\chi^{(2)}_{t_2}(w_2-y_2)dw_2\\
&\qquad\times \Big|\big( t_1^2{\triangle^{(1)}}e^{-t_1^2{\triangle^{(1)}}}\,
t_2^2{\triangle^{(2)}}e^{-t_2^2{\triangle^{(2)}}}\big) (a_R) (y_1,y_2)\Big|^2{dy_2dt_2\over t_2^{m+1}} {dy_1dt_1\over
t_1^{n+m+1}}dx_2dx_1\\
&=:\textrm{II}_{121}+\textrm{II}_{122}.
\end{align*}

For $t_1<\ell(I)$ and $t_2< \frac{|x_2-x_J|}8$, we have that $\chi^{(1)}_{t_1}(x_1-y_1,x_2-w_2)\chi^{(2)}_{t_2}(w_2-y_2)$ gives
$|x_2-y_2|\le t_1+t_2\le  \frac{|x_2-x_J|}4$. Since $|x_2-x_J|\approx 2^{j_2}\ell(J)$, we have $|y_2-x_J|\ge |x_2-x_J|-|x_2-y_2|\ge  \frac{|x_2-x_J|}2$
and $|y_2-z_2|\ge |y_2-x_J|-|z_2-x_J|\ge |x_2-x_J|/4$ if $z_2\in B(x_J, |x_2-x_J|/4)$. 
Hence, from the estimate in \eqref{t12} and the $L^2(\Bbb R^{m+n})$-boundedness of area square function,
\begin{align*}
&\textrm{II}_{121} \\
&\leq\int_{100I} \int_{|x_2-x_J|\approx 2^{j_2}\ell(J)} \int_0^{\ell(I)}\!\!\int_{\ell(J)}^{\frac{|x_2-x_J|}8} 
\int_{\mathbb R^{n}}\int_{\Bbb R^m}  \int_{\Bbb R^m}\chi^{(1)}_{t_1}(x_1-y_1,x_2-w_2)\chi^{(2)}_{t_2}(w_2-y_2)dw_2\\
&\qquad\times \Big|\big( t_1^2{\triangle^{(1)}}e^{-t_1^2{\triangle^{(1)}}}\ \chi_{B(x_J, |x_2-x_J|/4)}\ 
t_2^2{\triangle^{(2)}}e^{-t_2^2{\triangle^{(2)}}}\big) (a_R) (y_1,y_2)\Big|^2{dy_2dt_2\over t_2^{m+1}} {dy_1dt_1\over
t_1^{n+m+1}}dx_2dx_1\\
&\ \ +\int_{100I} \int_{|x_2-x_J|\approx 2^{j_2}\ell(J)} \int_0^{\ell(I)}\!\!\int_{\ell(J)}^{\frac{|x_2-x_J|}8} 
\int_{\mathbb R^{n}}\int_{\Bbb R^m}  \int_{\Bbb R^m}\chi^{(1)}_{t_1}(x_1-y_1,x_2-w_2)\chi^{(2)}_{t_2}(w_2-y_2)dw_2\\
&\qquad\times \Big|\big( t_1^2{\triangle^{(1)}}e^{-t_1^2{\triangle^{(1)}}}\ \chi_{B(x_J, |x_2-x_J|/4)^c}\ 
t_2^2{\triangle^{(2)}}e^{-t_2^2{\triangle^{(2)}}}\big) (a_R) (y_1,y_2)\Big|^2{dy_2dt_2\over t_2^{m+1}} {dy_1dt_1\over
t_1^{n+m+1}}dx_2dx_1\\
&\lesssim  \int_{100I} \int_{|x_2-x_J|\approx 2^{j_2}\ell(J)}\int_0^{\ell(I)}\!\!\int_{\ell(J)}^{\frac{|x_2-x_J|}8}  \int_{\mathbb R^{n}}\int_{\mathbb R^{m}} \int_{\Bbb R^m}\chi^{(1)}_{t_1}(x_1-y_1,x_2-w_2)\chi^{(2)}_{t_2}(w_2-y_2)dw_2\\
&\quad\times\Big|\int_{10I}\int_{B(x_J, |x_2-x_J|/4)} t_1^{-n-m} \exp\Big(-{|y_1-z_1|^2+|y_2-z_2|^2\over ct_1^2}\Big)
\big(t_2^2{\triangle^{(2)}}e^{-t_2^2{\triangle^{(2)}}}a_{R} (z_1,z_2)\big)dz_2\,dz_1\Big|^2 \\
&\quad\times {dy_2dt_2\over t_2^{m+1}}{dy_1dt_1\over
t_1^{n+m+1}}dx_2dx_1\\
&\ \ + \int_{|x_2-x_J|\approx 2^{j_2}\ell(J)} \!\!\int_{\ell(J)}^{\frac{|x_2-x_J|}8} \ \ 
\\
&\quad\times \int_{\mathbb R^{n}}\int_{\Bbb R^m} \int_0^{\ell(I)}  \Big|\big( t_1^2{\triangle^{(1)}}e^{-t_1^2{\triangle^{(1)}}}\ \chi_{(12J)^c}\ 
t_2^2{\triangle^{(2)}}e^{-t_2^2{\triangle^{(2)}}}\big) (a_R) (y_1,y_2)\Big|^2 {dt_1\over
t_1}dy_1dy_2 \ dx_2{dt_2\over t_2^{m+1}} \\
&\lesssim \int_{100I} \int_{|x_2-x_J|\approx 2^{j_2}\ell(J)}\int_0^{\ell(I)}\!\!\int_{\ell(J)}^{\frac{|x_2-x_J|}8} \int_{\mathbb R^{n}}\int_{\mathbb R^{m}} \int_{\Bbb R^m}\chi^{(1)}_{t_1}(x_1-y_1, w_2-y_2)\chi^{(2)}_{t_2}(x_2-w_2)dw_2\\
&\quad\times t_1^{-2n}\exp\Big(-{|x_2-x_J|^2\over 4ct_1^2}\Big) \bigg|M_2\bigg(\int_{10I} |t_2^2{\triangle^{(2)}}e^{-t_2^2{\triangle^{(2)}}}a_{R} (z_1,\cdot)|dz_1\bigg)(y_2)\bigg|^2\\
&\quad \times{dy_1dt_1\over
t_1^{n+m+1}}{dy_2dt_2\over t_2^{m+1}}dx_2dx_1\\
&\ \ + \int_{|x_2-x_J|\approx 2^{j_2}\ell(J)} \int_{\ell(J)}^{\frac{|x_2-x_J|}8}
 \int_{\mathbb R^{n}}\int_{\Bbb R^m}  \Big| \chi_{B(x_J, |x_2-x_J|/4)^c}
t_2^2{\triangle^{(2)}}e^{-t_2^2{\triangle^{(2)}}} (a_R) (y_1,y_2)\Big|^2dy_1dy_2 \ dx_2{dt_2\over t_2^{m+1}} \\
&=:\textrm{II}_{1211}+\textrm{II}_{1212}.
\end{align*}

For the term $\textrm{II}_{1211}$, by using the $L^2(\mathbb R^m)$ boundedness of the Hardy--Littlewood maximal function, we obtain that
\begin{align*}
\textrm{II}_{1211}
&\lesssim |I| \int_{|x_2-x_J|\approx 2^{j_2}\ell(J)}\int_0^{\ell(I)}\!\!\int_{\ell(J)}^{\frac{|x_2-x_J|}8}t_1^{-2n} \Big({t_1 \over |x_2-x_J|}\Big)^\alpha   \\
&\quad\times \int_{\mathbb R^{m}}\bigg|M_2\bigg(\int_{10I} |t_2^2{\triangle^{(2)}}e^{-t_2^2{\triangle^{(2)}}}a_{R} (z_1,\cdot)|dz_1\bigg)(y_2)\bigg|^2dy_2{dt_1\over
t_1^{m+1}}{dt_2\over t_2}dx_2\\
&\lesssim |I| \int_{|x_2-x_J|\approx 2^{j_2}\ell(J)}\Big({1 \over 2^{j_2}\ell(J)}\Big)^\alpha  \int_0^{\ell(I)}\!\!\int_{\ell(J)}^{\frac{|x_2-x_J|}8}    \\
&\quad\times \int_{\mathbb R^{m}}\bigg|\int_{10I} |t_2^2{\triangle^{(2)}}e^{-t_2^2{\triangle^{(2)}}}a_{R} (z_1,y_2)|dz_1\bigg|^2dy_2{dt_1\over
t_1^{2n+m-\alpha+1}}{dt_2\over t_2}dx_2\\
&\lesssim |I| (2^{j_2}\ell(J))^m \Big({1 \over 2^{j_2}\ell(J)}\Big)^\alpha  \ell(I)^{\alpha-2n-m}   |I|   \\
&\quad\times \int_{10I} \int_0^\infty\int_{\mathbb R^{m}} |t_2^2{\triangle^{(2)}}e^{-t_2^2{\triangle^{(2)}}}a_{R} (z_1,y_2)|^2 dy_2{dt_2\over t_2}dz_1\\
&\lesssim  \Big({1 \over 2^{j_2}}\Big)^{\alpha-m}  \int_{10I}\int_{\mathbb R^{m}} |a_{R} (z_1,y_2)|^2 dy_2dz_1\\
&= \Big({1 \over 2^{j_2}}\Big)^{\alpha-m}  \|a_{R} \|_2^2,
\end{align*}
where we chose $\alpha>2n+2m$, and we applied the Littlewood--Paley estimate in the last inequality.

For the term $\textrm{II}_{1212}$, we have $w_2\in B(x_J, |x_2-x_J|/4)^c$ and $u_2\in 10J$ give $|w_2-u_2|\ge |x_2-x_J|/8$.
Hence
\begin{align*}
&\textrm{II}_{1212}\\
&=\int_{|x_2-x_J|\approx 2^{j_2}\ell(J)} \int_{\ell(J)}^{\frac{|x_2-x_J|}8}
 \int_{\mathbb R^{n}}\int_{\Bbb R^m}\\
&\quad\times   \Big| \chi_{B(x_J, |x_2-x_J|/4)^c}(y) 
(t_2^2{\triangle^{(2)}})^{M+1}e^{-t_2^2{\triangle^{(2)}}} (a_{R,1}) (y_1,y_2)\Big|^2dy_1dy_2 dx_2{dt_2\over t_2^{m+4M+1}}\\
&\lesssim \int_{\ell(J)}^{\frac{|x_2-x_J|}8}\int_{|x_2-x_J|\approx 2^{j_2}\ell(J)} \int_{10I} \int_{B(x_J, |x_2-x_J|/4)^c} \\
&\qquad\times  \bigg|\int_{10J} t_2^{-m}\exp\Big(-{|w_2-u_2|^2\over 4ct_2^2}\Big)a_{R,1}(x_1,u_2)\bigg|^2 
              dx_1dw_2dx_2{dt_2\over t_2^{m+4M+1}}\\
&\lesssim  \int_{\ell(J)}^\infty\int_{|x_2-x_J|\approx 2^{j_2}\ell(J)}\int_{B(x_J, |x_2-x_J|/4)^c}  t_2^{-m}\exp\Big(-{|w_2-u_2|^2\over     4ct_2^2}\Big)dw_2  dx_2{dt_2\over t_2^{2m+4M+1}} dx_2{dt_2\over t_2^{2m+4M+1}}\\
&\lesssim ( 2^{j_2}\ell(J))^{m-2\alpha_2}|J|\|a_{R,1}\|_2^2\ell(J)^{2\alpha_2-2m-4M} \\
&\lesssim ( 2^{j_2})^{m-2\alpha_2}\ell(J)^{-4M}\|a_{R,1}\|_2^2
\end{align*}
for $m<\alpha_2<2M$.

By the $L^2(\Bbb R^{m+n})$-boundedness of square function,
\begin{align*}
\textrm{II}_{122}
&\lesssim \int_{10I} \int_{|x_2-x_J|\approx 2^{j_2}\ell(J)}\int_{\Bbb R^m}\int_{\frac{|x_2-x_J|}8}^\infty \Big|t_2^2{\triangle^{(2)}}e^{-t_2^2{\triangle^{(2)}}} (a_{R,1}) (x_1,w_2)\Big|^2{dt_2\over t_2^{m+1}} dx_2dx_1dw_2\\
&\lesssim  \int_{|x_2-x_J|\approx 2^{j_2}\ell(J)}\int_{\frac{|x_2-x_J|}8}^\infty\int_{\Bbb R^m}t_2^{-m}\exp\Big(-{|w_2-u_2|^2\over 4ct_2^2}\Big)dw_2 |J| \|a_{R,1}\|_2^2 dx_2{dt_2\over t_2^{2m+4M+1}} \\
&\lesssim (2^{j_2}\ell(J))^{-m-4M} |J| \|a_{R,1}\|_2^2\\
&\lesssim (2^{j_2})^{-m-4M}\ell(J)^{-4M} \|a_{R,1}\|_2^2.
\end{align*}

We now to estimate $\textrm{II}_{13}$. 
\begin{align*}
\textrm{II}_{13}
&=\int_{100I} \int_{|x_2-x_J|\approx 2^{j_2}\ell(J)}\int_{\ell(I)}^\infty\!\!\int_0^{\ell(J)}
  \!\! \int_{\mathbb R^{n}}\int_{\mathbb R^{m}} \int_{\Bbb R^m}\chi^{(1)}_{t_1}(x_1-y_1,x_2-w_2)\chi^{(2)}_{t_2}(w_2-y_2)dw_2\\
&\qquad\times \Big|\big( (t_1^2{\triangle^{(1)}})^{M+1}e^{-t_1^2{\triangle^{(1)}}}\otimes_2
t_2^2{\triangle^{(2)}}e^{-t_2^2{\triangle^{(2)}}}\big) a_{R,2} (y_1,y_2)\Big|^2 {dy_2dt_2\over t_2^{m+1}}{dy_1dt_1\over
t_1^{n+m+1+4M}}dx_2dx_1\\
&\lesssim  \int_{100I} \int_{|x_2-x_J|\approx 2^{j_2}\ell(J)}\int_{\ell(I)}^\infty\!\!\int_0^{\ell(J)}\!\! \int_{\mathbb R^{n}}\int_{\mathbb R^{m}} \int_{\Bbb R^m}\chi^{(1)}_{t_1}(x_1-y_1,x_2-w_2)\chi^{(2)}_{t_2}(w_2-y_2)dw_2\\
&\qquad\times\Big|\int_{10I}\int_{\Bbb R^m} t_1^{-n-m} \exp\Big(-{|y_1-z_1|^2+|y_2-z_2|^2\over ct_1^2}\Big)
t_2^2{\triangle^{(2)}}e^{-t_2^2{\triangle^{(2)}}}a_{R,2} (z_1,z_2)dz_1dz_2\Big|^2 \\
&\qquad\times {dy_2dt_2\over t_2^{m+1}}{dy_1dt_1\over
t_1^{n+m+1+4M}}dx_2dx_1\\
&=  \int_{100I} \int_{|x_2-x_J|\approx 2^{j_2}\ell(J)}\bigg(\int_{\ell(I)}^{\frac{|x_2-x_J|}8}+ \int_{\frac{|x_2-x_J|}8}^\infty \bigg)\!\!\int_0^{\ell(J)}\!\! \int_{\mathbb R^{n}}\int_{\mathbb R^{m}}\\
&\qquad \times \int_{\Bbb R^m}\chi^{(1)}_{t_1}(x_1-y_1,x_2-w_2)\chi^{(2)}_{t_2}(w_2-y_2)dw_2\\
&\qquad\times\Big|\int_{10I}\int_{\Bbb R^m} t_1^{-n-m} \exp\Big(-{|y_1-z_1|^2+|y_2-z_2|^2\over ct_1^2}\Big)
t_2^2{\triangle^{(2)}}e^{-t_2^2{\triangle^{(2)}}}a_{R,2} (z_1,z_2)dz_1dz_2\Big|^2 \\
&\qquad\times {dy_2dt_2\over t_2^{m+1}}{dy_1dt_1\over
t_1^{n+m+1+4M}}dx_2dx_1\\
&=:\textrm{II}_{131}+\textrm{II}_{132}.
\end{align*}
We first estimate the term $\textrm{II}_{131}$.

Recall that $\ell(I)\leq \ell(J)< |x_2-x_J|$.
We have $t_1+t_2<|x_2-x_J|/8+\ell(J)\le |x_2-x_J|/2$.
Hence
\begin{align*}
\textrm{II}_{131}
&\lesssim \int_{100I} \int_{|x_2-x_J|\approx 2^{j_2}\ell(J)}\int_{\ell(I)}^{\frac{|x_2-x_J|}8}\int_0^{\ell(J)}\!\! \int_{\mathbb R^{n}}\int_{\mathbb R^{m}}
                \\
&\quad\times \int_{\Bbb R^m}\chi^{(1)}_{t_1}(x_1-y_1,x_2-w_2)\chi^{(2)}_{t_2}(w_2-y_2)dw_2\ t_1^{-2n} \\
&\quad\times \Big|\int_{10I}\int_{\Bbb R^m} t_1^{-m} \exp\Big(-{|y_2-z_2|^2\over ct_1^2}\Big)
t_2^2{\triangle^{(2)}}e^{-t_2^2{\triangle^{(2)}}}a_{R,2} (z_1,z_2)dz_2dz_1\Big|^2\\
&\qquad{dy_1dt_1\over
t_1^{n+m+1+4M}}{dy_2dt_2\over t_2^{m+1}}dx_2dx_1.
\end{align*}
Note that
\begin{align*}
&\bigg|\int_{10I}\int_{\Bbb R^m} t_1^{-m} \exp\Big(-{|y_2-z_2|^2\over ct_1^2}\Big)
         t_2^2{\triangle^{(2)}}e^{-t_2^2{\triangle^{(2)}}}a_{R,2} (z_1,z_2)dz_2dz_1\bigg|\\
&\le \bigg|\int_{10I}\int_{12J}  t_1^{-m}  \exp\Big(-{|y_2-z_2|^2\over ct_1^2}\Big)
         t_2^2{\triangle^{(2)}}e^{-t_2^2{\triangle^{(2)}}}a_{R,2} (z_1,z_2)dz_2dz_1\bigg|\\
&\quad + \bigg|\int_{10I}\int_{(12J)^c}  t_1^{-m}  \exp\Big(-{|y_2-z_2|^2\over ct_1^2}\Big)
         \int_{10J} t_2^{-m}  \exp\Big(-{|z_2-u_2|^2\over ct_2^2}\Big)|a_{R,2} (z_1,u_2)|du_2\ dz_2dz_1\bigg|.
\end{align*}
Since $|x_2-x_J|>100\ell(J)$ and $|x_2-y_2|<t_1+t_2$, if $z_2\in 12I$, then we have $|y_2-z_2|\geq  |x_2-x_J|/4\ge (|x_1-x_I|+|x_2-x_J|)/8$; if $z_2\in (12J)^c$ then for $u_2\in 10J$ we have  $|z_2-u_2|>\ell(J)$. As a consequence, from the almost orthogonality estimate \eqref{t12} we get that
\begin{equation}\label{ii131}
\begin{aligned}
&\bigg|\int_{10I}\int_{\Bbb R^m} t_1^{-m} \exp\Big(-{|y_2-z_2|^2\over ct_1^2}\Big)
         t_2^2{\triangle^{(2)}}e^{-t_2^2{\triangle^{(2)}}}a_{R,2} (z_1,z_2)dz_2dz_1\bigg|\\
&\lesssim \exp\Big(-{|x_2-x_J|^2\over 2ct_1^2}\Big) M_2\bigg(\int_{10I} |t_2^2{\triangle^{(2)}}e^{-t_2^2{\triangle^{(2)}}}a_{R,2} (z_1,\cdot)|dz_1\bigg)(y_2)\\
&\qquad + \int_{10I}\int_{10J}  \exp\Big(-{\ell(J)^2\over 2ct_2^2}\Big){(\max\{t_1,t_2\})^\alpha\over ( \max\{t_1,t_2\} + |y_2-u_2| )^{m+\alpha}  } |a_{R,2} (z_1,u_2)|du_2dz_1.
\end{aligned}
\end{equation}
Plugging \eqref{ii131} into the estimate of $\textrm{II}_{131}$ above, we have
\begin{align*}
\textrm{II}_{131} 
&\lesssim \int_{100I} \int_{|x_2-x_J|\approx 2^{j_2}\ell(J)}\int_{\ell(I)}^{\frac{|x_2-x_J|}8}\!\!\int_0^{\ell(J)}\!\! \int_{\mathbb R^{m}}
                \\
&\quad\times \int_{\mathbb R^{n}}\int_{\Bbb R^m}\chi^{(1)}_{t_1}(x_1-y_1,x_2-w_2)\chi^{(2)}_{t_2}(w_2-y_2)dw_2dy_1\  \ t_1^{-2n}\exp\Big(-{|x_2-x_J|^2\over 2ct_1^2}\Big) \\
&\quad\times  \bigg|M_2\bigg(\int_{10I} |t_2^2{\triangle^{(2)}}e^{-t_2^2{\triangle^{(2)}}}a_{R,2} (z_1,\cdot)|dz_1\bigg)(y_2)\bigg|^2{dy_2dt_2\over t_2^{m+1}}{dt_1\over t_1^{n+m+1+4M}}dx_2dx_1\\
&\ \,+ \int_{100I} \int_{|x_2-x_J|\approx 2^{j_2}\ell(J)}\int_{\ell(I)}^{\frac{|x_2-x_J|}8}\!\!\int_0^{\ell(J)} \\
&\quad\times \int_{\mathbb R^{n}}\int_{\mathbb R^{m}}\int_{\Bbb R^m}\chi^{(1)}_{t_1}(x_1-y_1,x_2-w_2)\chi^{(2)}_{t_2}(w_2-y_2)dw_2dy_2dy_1\ \ t_1^{-2n}\exp\Big(-{|x_1-x_I|^2\over 16ct_1^2}\Big)\\
&\quad \times|R| \exp\Big(-{\ell(J)^2\over ct_2^2}\Big)
    {(\max\{t_1,t_2\})^{2\alpha}\over ( \max\{t_1,t_2\} + |x_2-x_J| )^{2m+2\alpha}  } \|a_{R,2}\|_2^2{dy_2dt_2\over t_2^{m+1}}{dt_1\over t_1^{n+m+1+4M}}dx_2dx_1\\
&=:\textrm{II}_{1311}+\textrm{II}_{1312}.
\end{align*}
To estimate $\textrm{II}_{1311}$, by using the $L^2$ boundedness of the Hardy--Littlewood maximal function, we get that
\begin{align*}
\textrm{II}_{1311}
&\lesssim\int_{100I} \int_{|x_2-x_J|\approx 2^{j_2}\ell(J)}\int_{\ell(I)}^\infty\!\!\int_0^{\ell(J)} t_1^{-2n}\frac{t_1^{2\alpha_2}}{|x_2-x_J|^{2\alpha_2}} \\
&\quad\times \int_{\Bbb R^m} \Bigg| M_2\bigg(\int_{10I} |t_2^2{\triangle^{(2)}}e^{-t_2^2{\triangle^{(2)}}}a_{R,2} (z_1,\cdot)|dz_1\bigg)(y_2)\Bigg|^2dy_2{dt_2\over t_2}{dt_1\over
t_1^{m+1+4M}}dx_2dx_1\\
&\lesssim\int_{100I} \int_{|x_2-x_J|\approx 2^{j_2}\ell(J)}\int_{\ell(I)}^\infty\!\!\int_0^{\ell(J)} \frac{t_1^{2\alpha_2}}{|x_2-x_J|^{2\alpha_2}} \\
&\quad\times \int_{\Bbb R^m} \bigg|\bigg(\int_{10I} |t_2^2{\triangle^{(2)}}e^{-t_2^2{\triangle^{(2)}}}a_{R,2} (z_1,\cdot)|dz_1\bigg)(y_2)\bigg|^2dy_2{dt_2\over t_2}{dt_1\over
t_1^{2n+m+1+4M}}dx_2dx_1\\
&\lesssim\int_{100I} \int_{|x_2-x_J|\approx 2^{j_2}\ell(J)}\int_{\ell(I)}^\infty \frac{t_1^{2\alpha_2}}{|x_2-x_J|^{2\alpha_2}} \\
&\quad\times |I| \int_{10I} \ \int_0^\infty\int_{\Bbb R^m}  |t_2^2{\triangle^{(2)}}e^{-t_2^2{\triangle^{(2)}}}a_{R,2} (z_1,y_2)|^2dy_2 {dt_2\over t_2}\ dz_1\ {dt_1\over
t_1^{2n+m+1+4M}}dx_2dx_1\\
&\lesssim |I|\|a_{R,2}\|_2^2\int_{100I} \int_{|x_2-x_J|\approx 2^{j_2}\ell(J)}\int_{\ell(I)}^\infty \frac{t_1^{2\alpha_2}}{|x_2-x_J|^{2\alpha_2}} {dt_1dx_2dx_1\over t_1^{2n+m+1+4M}}\\
&\lesssim |I|^2\|a_{R,2}\|_2^2  (2^{j_2}\ell(J))^m  \frac{1}{(2^{j_2}\ell(J))^{2\alpha_2}} {1\over \ell(I)^{2n+m+4M-2\alpha_2}}\\
&\lesssim  {  \ell(I)^{-4M}\|a_{R,2}\|_2^2   \frac{1}{(2^{j_2}\ell(J))^{2\alpha_2-m}} {1\over \ell(I)^{m-2\alpha_2}}  } \\
&\lesssim  { \ell(I)^{-4M}\|a_{R,2}\|_2^2   \frac{1}{(2^{j_2})^{2\alpha_2-m}  } }
\end{align*}
by choosing $m<\alpha_2<4M$.

For the term $\textrm{II}_{1312}$, we have
\begin{align*}
\textrm{II}_{1312}
&\lesssim |R|\int_{100I} \int_{|x_2-x_J|\approx 2^{j_2}\ell(J)}\int_{\ell(I)}^{\frac{|x_2-x_J|}8}t_1^{-2n} \\
&\quad\times\!\!\int_0^{\ell(J)}\exp\Big(-2{\ell(J)^2\over ct_2^2}\Big){dt_2\over t_2}
    {(\max\{t_1,\ell(J)\})^{2\alpha}\over |x_2-x_J|^{2m+2\alpha}  } \|a_{R,2}\|_2^2{dt_1\over t_1^{1+4M}}dx_2dx_1.
\end{align*}
Let $\alpha={1\over2}$. We further get
\begin{align*}
\textrm{II}_{1312}
&\lesssim |R|\|a_{R,2}\|_2^2\int_{100I} \int_{|x_2-x_J|\approx 2^{j_2}\ell(J)}\int_{\ell(I)}^{\frac{|x_2-x_J|}8}t_1^{-2n}
    {t_1+\ell(J) \over  |x_2-x_J| ^{2m+1}  } {dt_1\over t_1^{1+4M}}dx_2dx_1\\
    &\lesssim |R| \|a_{R,2}\|_2^2\ |I|\ (2^{j_2}\ell(J))^{-m-1}\int_{\ell(I)}^\infty {dt_1\over t_1^{2n+4M}}\\
    &\quad + |R| \|a_{R,2}\|_2^2\ |I|\ (2^{j_2}\ell(J))^{-m-1}\ \ell(J)\ \int_{\ell(I)}^\infty {dt_1\over t_1^{2n+1+4M}}\\
&\lesssim  \ell(I)^{-4M}\|a_{R,2}\|_2^2  {1\over 2^{j_2(m+1)}}.
\end{align*}
This finishes the estimate for the term $\textrm{II}_{131}$.

We plug \eqref{i23} into $\textrm{II}_{132}$ to get
\begin{align*}
&\textrm{II}_{132} \\
&\lesssim \int_{100I} \int_{|x_2-x_J|\approx 2^{j_2}\ell(J)}\int_{\frac{|x_2-x_J|}8}^\infty\!\!\int_0^{\ell(J)}\!\! \int_{\mathbb R^{n}}\int_{\mathbb R^{m}}
                \\
&\quad\times \int_{\Bbb R^m}\chi^{(1)}_{t_1}(x_1-y_1,x_2-w_2)\chi^{(2)}_{t_2}(w_2-y_2)dw_2\ t_1^{-2n}\exp\Big(-{|x_2-x_J|^2\over ct_1^2}\Big) \\
&\qquad\times  \bigg|M_2\bigg(\int_{10I} |t_2^2{\triangle^{(2)}}e^{-t_2^2{\triangle^{(2)}}}a_{R,2} (z_1,\cdot)|dz_1\bigg)(y_2)\bigg|^2{dy_1dt_1\over t_1^{n+m+1+4M}}{dy_2dt_2\over t_2^{m+1}}dx_2dx_1\\
&\ \ + \int_{100I} \int_{|x_2-x_J|\approx 2^{j_2}\ell(J)}\int_{\frac{|x_2-x_J|}8}^\infty\!\int_0^{\ell(J)} \int_{\mathbb R^{n}}\int_{\mathbb R^{m}}\\
&\quad\times \int_{\Bbb R^m}\chi^{(1)}_{t_1}(x_1-y_1,x_2-w_2)\chi^{(2)}_{t_2}(w_2-y_2)dw_2\ t_1^{-2n} \exp\Big(-{\ell(J)^2\over ct_2^2}\Big)\\
&\qquad \times
    \bigg|\int_{10I}\int_{10J}  {(\max\{t_1,t_2\})^{\alpha}\over ( \max\{t_1,t_2\} + |y_2-u_2| )^{m+\alpha}  }   |a_{R,2} (z_1,u_2)|du_2dz_1\bigg|^2\ {dy_1dt_1\over t_1^{n+m+1+4M}}{dy_2dt_2\over t_2^{m+1}}dx_2dx_1\\
&=:\textrm{II}_{1321}+\textrm{II}_{1322}.
\end{align*}

To estimate $\textrm{II}_{1321}$, we use the same method  in $\textrm{II}_{1311}$ to obtain
\begin{align*}
\textrm{II}_{1321}
&\lesssim  { \ell(I)^{-4M}\|a_{R,2}\|_2^2   \frac{1}{(2^{j_2})^{2\alpha_2-m}  } }
\end{align*}
by choosing $m<\alpha_2<4M$.

We then estimate the term $\textrm{II}_{1322}$. 
\begin{align*}
\textrm{II}_{1322}
&\lesssim \int_{100I} \int_{|x_2-x_J|\approx 2^{j_2}\ell(J)}\int_{\frac{|x_2-x_J|}8}^\infty\!\!\int_0^{\ell(J)}  \exp\Big(-2{\ell(J)^2\over ct_2^2}\Big)  {1\over  t_1 ^{2m}  }  |R|\|a_{R,2}\|_2^2{dt_1\over t_1^{2n+1+4M}}{dt_2\over t_2}dx_2dx_1\\
&\lesssim  \ell(I)^{-4M} \|a_{R,2}\|_2^2 {1\over (2^{j_2})^{2n+m+4M}},
\end{align*}
where we choose $0<\alpha<2M$, and  the last inequality follows from the fundamental estimates that
$\int_0^{\ell(J)}\exp\big(-2{\ell(J)^2\over ct_2^2}\big) {dt_2\over t_2}\lesssim 1$.

This finishes the estimate for the term $\textrm{II}_{132}$.

Finally we estimate the term $\textrm{II}_{14}$. Note that  $a_R=(({\triangle^{(1)}})^M\otimes_2({\triangle^{(2)}})^M)b_R$ and supp $b_{R}\subset 10R=10(I\times J)$.
Let $A(x_2)=\frac{|x_2-x_J|}8.$ Then
\begin{align*}
\textrm{II}_{14}
&= \int_{100I} \int_{|x_2-x_J|\approx 2^{j_2}\ell(J)}\bigg(\int_{\ell(I)}^{A(x_2)}\int_{\ell(J)}^{A(x_2)}+\int_{\ell(I)}^{A(x_2)}\int_{A(x_2)}^\infty +\int_{A(x_2)}^\infty\int_{\ell(J)}^{A(x_2)}\\
 &\ \  +\int_{A(x_2)}^\infty\int_{A(x_2)}^\infty \bigg)   \int_{\mathbb R^{n}}\int_{\mathbb R^{m}} \int_{\Bbb R^m}\chi^{(1)}_{t_1}(x_1-y_1,x_2-w_2)\chi^{(2)}_{t_2}(w_2-y_2)dw_2\\
&\quad\times \Big|\big((t_1^2{\triangle^{(1)}})^{M+1}e^{-t_1^2{\triangle^{(1)}}}\otimes_2 (t_2^2{\triangle^{(2)}})^{M+1}e^{-t_2^2{\triangle^{(2)}}}\big) b_R (y_1,y_2)\Big|^2 {dy_2dt_2\over t_2^{m+4M+1}}{dy_1dt_1\over
t_1^{n+m+4M+1}} dx_2dx_1\\
&=:\textrm{II}_{141}+\textrm{II}_{142}+\textrm{II}_{143}+\textrm{II}_{144}.
\end{align*}
For the terms $\textrm{II}_{141}$, although we know that $t_1<|x_2-x_J|/8$, there is no specific estimate for the lower bound for $|y_1-z_1|$, thus we can only use the fact that $\exp\big(-{|y_1-z_1|^2\over ct_1^2}\big)\leq1$.
Moreover, since $|x_2-x_J|>100\ell(J)$, $|x_2-y_2|<t_1+t_2< |x_2-x_J|/4$, then for $u_2\in 10J$ we have $|y_2-u_2|\geq  |x_2-x_J|/4$. Now based on these observations, by using the almost orthogonality estimate as in \eqref{t12}, we have that
\begin{align*}
\textrm{II}_{141}
&\lesssim \int_{100I} \int_{|x_2-x_J|\approx 2^{j_2}\ell(J)}\int_{\ell(I)}^{A(x_2)}\int_{\ell(J)}^{A(x_2)}\int_{\mathbb R^{n}}\int_{\mathbb R^{m}}\\
 &\qquad\times   \int_{\Bbb R^m}\chi^{(1)}_{t_1}(x_1-y_1,x_2-w_2)\chi^{(2)}_{t_2}(w_2-y_2)dw_2\\ 
&\qquad\times \Big|\int_{10I}\int_{\Bbb R^m}\int_{\Bbb R^m} t_1^{-n-m} \exp\Big(-{|x_1-y_1|^2+|y_2-z_2|^2\over ct_1^2}\Big)
\\
&\qquad\qquad\qquad \ t_2^{-m} \exp\Big(-{|z_2-u_2|^2\over ct_1^2}\Big)|b_{R} (z_1,u_2)|du_2dz_2dz_1\Big|^2{dy_2dt_2\over t_2^{m+4M+1}}{dy_1dt_1\over
t_1^{n+m+4M+1}}dx_2dx_1\\
&\lesssim  \int_{100I} \int_{|x_2-x_J|\approx 2^{j_2}\ell(J)}\int_{\ell(I)}^{A(x_2)}\int_{\ell(J)}^{A(x_2)} t_1^{-2n}
    \frac{(\max\{t_1,t_2\})^{2\alpha}}{|x_2-x_J|^{2m+2\alpha}}|R|\|b_R \|_2^2 {dt_1\over t_1^{4M+1}}{dt_2\over t_2^{4M+1}}dx_2dx_1\\
&\lesssim  \int_{100I} \int_{|x_2-x_J|\approx 2^{j_2}\ell(J)}\int_{\ell(I)}^\infty\int_{\ell(J)}^\infty t_1^{-2n}
    \frac{t_1^{2\alpha}}{|x_2-x_J|^{2m+2\alpha}}|R|\|b_R \|_2^2 {dt_1\over t_1^{4M+1}}{dt_2\over t_2^{4M+1}}dx_2dx_1\\
&\quad+\int_{100I} \int_{|x_2-x_J|\approx 2^{j_2}\ell(J)}\int_{\ell(I)}^\infty\int_{\ell(J)}^\infty t_1^{-2n}
    \frac{t_2^{2\alpha}}{|x_2-x_J|^{2m+2\alpha}}|R|\|b_R \|_2^2 {dt_1\over t_1^{4M+1}}{dt_2\over t_2^{4M+1}}dx_2dx_1
\end{align*}
We use the fact that $\ell(I)\leq \ell(J)$ to obtain that
\begin{align*}
\textrm{II}_{141}
&\lesssim  \ell(I)^{-4M}\ell(J)^{-4M} \|b_R \|_2^2  2^{-(m+2\alpha) j_2},
\end{align*}
where we choose $0<\alpha<2M$.

We then estimate the term $\textrm{II}_{142}$. 
Note that $t_1<|x_2-x_J|/8 \le t_2$. Hence, there is no lower bound for $|y_2-u_2|$ and the almost orthogonality estimate
 appearing in the estimates for term $\textrm{II}_{142}$ will be replaced by
 ${1\over t_2^{2m}}$. Then we have
\begin{align*}
\textrm{II}_{142}
&\lesssim  \int_{100I} \int_{|x_2-x_J|\approx 2^{j_2}\ell(J)}\int_{\ell(I)}^{A(x_2)}
 \int_{{|x_2-x_J|\over8}}^\infty\frac{1}{ t_2^{2m}}|R|\|b_R \|_2^2 {dt_1\over t_1^{2n+4M+1}}{dt_2\over t_2^{4M+1}}dx_2dx_1
\\
&\lesssim  \ell(I)^{-4M}\ell(J)^{-4M} \|b_R \|_2^2 2^{-(m+4M) j_2}.
\end{align*}

We now consider the term $\textrm{II}_{143}$ and use the fact that $\exp\big(-{|y_1-z_1|^2\over ct_1^2}\big)\leq1$. 
Moreover, there is also no lower bound for $|y_2-z_2|$.  And hence the almost orthogonality estimate
 appearing in  the estimates for term $\textrm{I}_{242}$ will be replaced by
 ${1\over t_1^{2m}}$ since  $t_1> t_2$ in this case. Then we have
\begin{align*}
\textrm{II}_{143}
&\lesssim  |R|\|b_R\|_2^2\int_{100I} \int_{|x_2-x_J|\approx 2^{j_2}\ell(J)}\int_{A(x_2)}^\infty {dt_1\over t_1^{2n+2m+4M+1}} \int_{\ell(J)}^\infty {dt_2\over t_2^{4M+1}}dx_1dx_2\\
&\lesssim  |R|\|b_R\|_2^2 \ell(I)^n ( 2^{j_2}\ell(J))^m (2^{j_2}\ell(J))^{-2n-2m-4M}\ell(J)^{-4M}\\
&\lesssim  \ell(I)^{-4M}\ell(J)^{-4M}\|b_R\|_2^2  2^{-(m+2M)j_2}.
\end{align*}

We finally estimate the term $\textrm{II}_{144}$. And we also
 note that in this case there are no lower bounds for $|y_1-z_1|$ or $|y_2-z_2|$. So we will use the fact that $\exp\big(-{|y_1-z_1|^2\over ct_1^2}\big)\leq1$.  And the almost orthogonality estimate
 appearing in  the estimates for term $\textrm{I}_{241}$ will be replaced by
 ${1\over \max\{ t_1,t_2\}^{2m}}$. Then we have
\begin{align*}
\textrm{II}_{144}
&\lesssim  \int_{100I} \int_{|x_2-x_J|\approx 2^{j_2}\ell(J)}\int_{A(x_2)}^\infty\int_{A(x_2)}^\infty t_1^{-2n}(\max\{t_1,t_2\})^{-2m}
   |R|\|b_R\|_2^2{dt_1\over t_1^{4M+1}}{dt_2\over t_2^{4M+1}}dx_1dx_2\\
&\lesssim  \ell(I)^{-4M}\ell(J)^{-4M}\|b_R\|_2^2  2^{-(m+8M)j_2}.
\end{align*}

Combing the estimates of $\textrm{II}_{1i}$ for $i=1,2,3, 4$, we can show that
\begin{align*}
\textrm{II}_1&\lesssim |R|^{1/2} \gamma_1(R)^{-\delta}\ell(I)^{-2M}\ell(J)^{-2M}\Big(\|( (\ell(I)^2{\triangle^{(1)}})^M) \otimes_2
(\ell(J)^2{\triangle^{(2)}})^M)b_{R}\|_2\\
& +  \|((\ell(I)^2{\triangle^{(1)}})^M\otimes_2
1\!\!1_2) b_{R}\|_2+ \|(1\!\!1_1\otimes_2
(\ell(J)^2{\triangle^{(2)}})^M)b_{R}\|_2+ \|b_{R}\big\|_2\Big)
\end{align*}
for $\delta>0$.
}

By H\"older's inequality and Journ\'e's covering lemma, we obtain that
$\textrm{II}\lesssim1$.
Hence the proof of Proposition \ref{leAtom} is completed by
 (\ref{SL alpha uniformly bd on outside of Omega}) and (\ref{SL alpha uniformly bd on inside of Omega}).
\end{proof}

We now turn to {\bf Step 2}.
Our goal is to show that every
$f\in H^1_{F,\triangle^{(1)},\triangle^{(2)} }(\mathbb{R}^{n}\times\mathbb{R}^{m})\cap L^2(\mathbb{R}^{n+m}
)$ has a $(1, 2, M)$-atom representation,
with appropriate quantitative control of the coefficients.
To be more specific,
\begin{prop}\label{prop-product H-SL subset H-at}
Suppose $M\geq 1$. If $f\in
H^1_{F,\triangle^{(1)},\triangle^{(2)} }(\mathbb{R}^{n}\times\mathbb{R}^{m})\cap L^2(\mathbb{R}^{n+m})$, then there exist a
family of $(1, 2, M)$-atoms $\{a_j\}_{j=0}^\infty$ and a sequence of
numbers $\{\lambda_j\}_{j=0}^\infty\in \ell^1$ such that $f$ can be
represented in the form $f=\sum_{j}\lambda_ja_j$, with the sum
converging in $L^2(\mathbb{R}^{n+m})$, and
$$ \|f\|_{\mathbb{H}^1_{F,at,M}(\mathbb{R}^{n}\times\mathbb{R}^{m})}\leq C\sum_{j=0}^\infty|\lambda_j|
\leq
C\|f\|_{H^1_{F,\triangle^{(1)},\triangle^{(2)} }(\mathbb{R}^{n}\times\mathbb{R}^{m})},
$$
where $C$ is independent of $f$. In particular,
$$ H^1_{F,\triangle^{(1)},\triangle^{(2)} }(\mathbb{R}^{n}\times\mathbb{R}^{m})\cap
L^2(\mathbb{R}^{n+m})\ \ \subseteq\
\mathbb{H}^1_{F,at,M}(\mathbb{R}^{n}\times\mathbb{R}^{m}).
$$
\end{prop}

\begin{proof}

Let $f\in H^1_{F,\triangle^{(1)},\triangle^{(2)} }(\mathbb{R}^{n}\times\mathbb{R}^{m})\cap
L^2(\mathbb{R}^{n+m})$.
 For each $\ell\in\mathbb Z$, we define
 \begin{eqnarray*}
    \Omega_\ell&:=&\{(x_1,x_2)\in \mathbb{R}^{n}\times\mathbb{R}^{m}: S_{F,\triangle^{(1)},\triangle^{(2)}}(f)(x_1,x_2)>2^\ell \},\\
    B_\ell&:=&\Big\{R=I\times J:\ \ell(J)\geq\ell(I),\
        |R\cap \Omega_\ell|>{1\over 2}|R|,\  | R\cap \Omega_{\ell+1}|\leq {1\over 2}|R| \Big\}, {\rm\ and}\\
    \widetilde{\Omega}_\ell&:=&\Big\{(x_1,x_2)\in \mathbb{R}^{n}\times\mathbb{R}^{m}: \mathcal{M}_s(\chi_{\Omega_\ell})>{1\over10} \Big\}.
\end{eqnarray*}
For each dyadic rectangle $R=I\times J$ in $\mathbb{R}^{n}\times\mathbb{R}^{m}$, the tent $T(R)$ is defined as
$$    T(R)
    := \big\{ (y_1,y_2,t_1,t_2):\ (y_1,y_2)\in R, t_1\in (\ell(I)/2, \ell(I)], t_2\in (\ell(J)/2,\ell(J)]\big\}.
$$
For brevity, in what follows we will write $\chi_{T(R)}$ for
$\chi_{T(R)}(y_1, y_2, t_1, t_2)$.

Using the reproducing formula, we can write
\begin{align}\label{e2 in section 5.3.3}
    & f(x_1,x_2)\nonumber\\
    &= \int_0^\infty\!\!\int_0^\infty\nonumber
        \psi^{(1)}(t_1\sqrt{\triangle^{(1)}})\psi^{(2)}(t_2\sqrt{\triangle^{(2)}})(t_1^2\triangle^{(1)}e^{-t_1^2\triangle^{(1)}}\otimes_2 t_2^2\triangle^{(2)}e^{-t_2^2\triangle^{(2)}})(f)(x_1,x_2){dt_1dt_2\over t_1t_2}\\
    &=\int_0^\infty\!\!\int_0^\infty\!\! \int_{\mathbb R^n}\int_{\mathbb R^m}\ \int_{\mathbb R^m}
        K_{\psi(t_1\sqrt{\triangle^{(1)}})}(x_1,y_1,x_2,z_2)K_{\psi(t_2\sqrt{\triangle^{(2)}})}(z_2,y_2)dz_2\nonumber\\
    &\hskip1cm (t_1^2\triangle^{(1)}e^{-t_1^2\triangle^{(1)}} t_2^2\triangle^{(2)}e^{-t_2^2\triangle^{(2)}})(f)(y_1,y_2)dy_2dy_1{dt_1dt_2\over t_1t_2}\nonumber\\
    &= \sum_{\ell\in\mathbb{Z}}\sum_{R\in B_\ell}  \int_{T(R)} \int_{\mathbb R^m}
        K_{\psi(t_1\sqrt{\triangle^{(1)}})}(x_1,y_1,x_2,z_2)K_{\psi(t_2\sqrt{\triangle^{(2)}})}(z_2,y_2)dz_2\nonumber\\
    &\hskip1cm (t_1^2\triangle^{(1)}e^{-t_1^2\triangle^{(1)}} t_2^2\triangle^{(2)}e^{-t_2^2\triangle^{(2)}})(f)(y_1,y_2)dy_2dy_1{dt_1dt_2\over t_1t_2}\nonumber\\
    &= \sum_{\ell\in\mathbb{Z}}\lambda_{\ell} \bigg( {1\over \lambda_{\ell}}\sum_{\bar R\in B_\ell, \bar R {\rm\ max}}\    \sum_{R\in B_\ell, R\subset \bar R } \int_{T(R)} \int_{\mathbb R^m}
        K_{\psi(t_1\sqrt{\triangle^{(1)}})}(x_1,y_1,x_2,z_2)K_{\psi(t_2\sqrt{\triangle^{(2)}})}(z_2,y_2)dz_2\nonumber\\
    &\hskip1cm (t_1^2\triangle^{(1)}e^{-t_1^2\triangle^{(1)}} t_2^2\triangle^{(2)}e^{-t_2^2\triangle^{(2)}})(f)(y_1,y_2)dy_2dy_1{dt_1dt_2\over t_1t_2}\bigg)\nonumber\\
    &=:  \sum_{\ell\in\mathbb{Z}}\lambda_{\ell} a_{\ell}(x_1,x_2), \nonumber
\end{align}
where
\begin{align*}
&\lambda_\ell:=\bigg\|\bigg( \sum_{\bar R\in B_\ell } \int_{0}^\infty\int_{0}^\infty
        \big|(t_1^2\triangle^{(1)}e^{-t_1^2\triangle^{(1)}} t_2^2\triangle^{(2)}e^{-t_2^2\triangle^{(2)}})(f)(y_1,y_2)\big|^2 \chi_{T(R)}{dt_1dt_2\over t_1t_2}\bigg)^{1\over2}\bigg\|_2 |\widetilde{\Omega}_\ell|^{1\over2}.
\end{align*}

We now first claim that each $a_{\ell}$ is a flag atom.
First, it is direct to see that for each $\ell$,
$$a_{\ell}(x_1,x_2)  =  \sum_{\bar R\in B_\ell, \bar R {\rm\ max}} a_{\ell, \bar R}(x_1,x_2).$$
Next, for each $\ell$ and $\bar R\in B_\ell$ with  $\bar R {\rm\ max}$, we further have
$$a_{\ell, \bar R}(x_1,x_2): = \bigg(\big({\triangle^{(1)}}\big)^M \big( {\triangle^{(2)}}\big)^M\bigg)\big( b_{\ell, \bar R}\big)(x_1,x_2),$$
where
\begin{eqnarray}\label{b ell}
 b_{\ell, \bar R}(x_1,x_2)&:=&  {1\over \lambda_{\ell}}   \sum_{R\in B_\ell, R\subset \bar R } \int_{T(R)} t_1^{2M}t_2^{2M}\\
 &&\quad\quad \int_{\mathbb R^m}
        K_{\varphi^{(1)}(t_1\sqrt{\triangle^{(1)}})}(x_1,y_1,x_2,z_2)K_{\varphi^{(2)}(t_2\sqrt{\triangle^{(2)}})}(z_2,y_2)dz_2\nonumber\\
        &&\hskip1cm (t_1^2\triangle^{(1)}e^{-t_1^2\triangle^{(1)}} t_2^2\triangle^{(2)}e^{-t_2^2\triangle^{(2)}})(f)(y_1,y_2)dy_2dy_1{dt_1dt_2\over t_1t_2}\nonumber
\end{eqnarray}
and $\varphi^{(1)}, \varphi^{(2)}$ are the function mentioned in Lemma
\ref{lemma finite speed}. Then it follows from Lemma \ref{lemma
finite speed} that the integral kernel
$K_{(t_i^2\triangle^{(i)})^{k}\varphi^{(i)}(t_i\sqrt{\triangle^{(i)}})}$ of the operator
$(t_i^2\triangle^{(i)})^{k}\varphi^{(i)}(t_i\sqrt{\triangle^{(i)}})$ satisfy
\begin{eqnarray}\label{suppK1}
{\rm
supp}\,K_{(t_1^2\triangle^{(1)})^{k}\varphi^{(1)}(t_1\sqrt{\triangle^{(1)}})}\subset
\big\{(x,y)\in{\Bbb R}^{n+m}\times {\Bbb R}^{n+m}:
|x-y|<t_1\big\}
\end{eqnarray}
 and
\begin{eqnarray}\label{suppK2}
{\rm supp}\, K_{(t_2^2\triangle^{(2)})^{k}\varphi^{(2)}(t_2\sqrt{\triangle^{(2)}})}\subset
\big\{(u,v)\in{\Bbb R}^{m}\times {\Bbb R}^{m}:
|u-v|<t_2\big\}.
\end{eqnarray}
We now consider the support of $\big(({\triangle^{(1)}})^{k_1} \otimes_2({\triangle^{(2)}})^{k_2}\big)b_{\ell, \bar R}$.
From the definition of $b_{\ell, \bar R}$ as in \eqref{b ell}, we have that
\begin{eqnarray*}
&&\big(({\triangle^{(1)}})^{k_1} \otimes_2({\triangle^{(2)}})^{k_2}\big)( b_{\ell, \bar R})(x_1,x_2)\\
&&\ \ :=  {1\over \lambda_{\ell}}   \sum_{R\in B_\ell, R\subset \bar R } \int_{T(R)} t_1^{2M-2k}t_2^{2M-2k}\\
&&\qquad\int_{\mathbb R^m}
         K_{ (t_1^2\triangle^{(1)})^{k} \varphi^{(1)}(t_1\sqrt{\triangle^{(1)}})}(x_1,y_1,x_2,z_2)K_{(t_2^2\triangle^{(2)})^{k}\varphi^{(2)}(t_2\sqrt{\triangle^{(2)}})}(z_2,y_2)dz_2\nonumber\\
        &&\hskip1.5cm (t_1^2\triangle^{(1)}e^{-t_1^2\triangle^{(1)}} t_2^2\triangle^{(2)}e^{-t_2^2\triangle^{(2)}})(f)(y_1,y_2)dy_2dy_1{dt_1dt_2\over t_1t_2}.\nonumber
\end{eqnarray*}

Now from the following term on the right-hand side of the above equality
\begin{eqnarray}\label{suppKK}
 \int_{\mathbb R^m }K_{(t_1^2\triangle^{(1)})^{k}\varphi^{(1)}(t_1\sqrt{\triangle^{(1)}})}(x_1,y_1,x_2,z_2) K_{(t_2^2\triangle^{(2)})^{k}\varphi^{(2)}(t_2\sqrt{\triangle^{(2)}})}(z_2,y_2)dz_2
 \end{eqnarray}
and from the support conditions \eqref{suppK1} and \eqref{suppK2}, we obtain that for $(x_1,x_2)$ in \eqref{suppKK},
$$|x_1-y_1|\le 3\ell(I_R),\ |x_2-z_2|\le 3\ell(I_R),\quad {\rm and}\quad |z_2-y_2|\le 3\ell(J_R)$$
since
$(y_1,y_2,t_1,t_2)\in  T(R)$.
Hence we obtain that
 $$|x_1-y_1|\le 3\ell(I_R)\quad {\rm and}\quad |x_2-y_2|\le 3\ell(I_R)+3\ell(J_R).$$

As a consequence, we have
that for every $k_1,k_2=0,1,\dots,M$,
\begin{eqnarray}\label{eq1005}
{\rm supp}\, \big(({\triangle^{(1)}})^{k_1} \otimes_2({\triangle^{(2)}})^{k_2}\big)b_{\ell, \bar R}\subseteq 10\frak R,
\end{eqnarray}
where $\ell(I_{\frak R})=\ell(I_R)$ and $\ell(J_{\frak R})=\ell(I_R)+\ell(J_R)$.

Then based on the support condition above and on the definition of $\widetilde{\Omega}_\ell$, we obtain that
$$ {\rm supp}\  a_{\ell} \subset  \widetilde{\Omega}_\ell.$$

Next we estimate $\|a_{\ell}\|_2$.
Taking $g\in L^2(\mathbb{R}^{n+m})$
with  $\|g\|_2=1$, from the definition of $a_{\ell}$, we have
\begin{eqnarray*}
&&\hspace{-.4cm}
 \Big|\int_{\mathbb{R}^{n}\times\mathbb{R}^{m}}
a_{\ell}(x_1,x_2)g(x_1,x_2)dx_2dx_1 \Big|\\
&&=\Bigg| \bigg(  {1\over \lambda_{\ell}}\sum_{\bar R\in B_\ell, \bar R {\rm\ max}}\  \sum_{R\in B_\ell, R\subset \bar R } \int_{T(R)}
        \psi(t_1\sqrt{\triangle^{(1)}})\psi(t_2\sqrt{\triangle^{(2)}})(g)(y_1,y_2)\nonumber\\
    &&\hskip1cm (t_1^2\triangle^{(1)}e^{-t_1^2\triangle^{(1)}} t_2^2\triangle^{(2)}e^{-t_2^2\triangle^{(2)}})(f)(y_1,y_2)dy_2dy_1{dt_1dt_2\over t_1t_2}\bigg) \Bigg|\\
&&\leq     {1\over \lambda_{\ell}}  \int_{\mathbb{R}^{n}\times\mathbb{R}^{m}}    \bigg( \sum_{\bar R\in B_\ell } \int_{0}^\infty\int_{0}^\infty
        |\psi(t_1\sqrt{\triangle^{(1)}})\psi(t_2\sqrt{\triangle^{(2)}})(g)(y_1,y_2)|^2 \chi_{T(R)}{dt_1dt_2\over t_1t_2}\bigg)^{1\over2}\\
    &&\hskip1cm  \bigg( \sum_{\bar R\in B_\ell } \int_{0}^\infty\int_{0}^\infty
        \big|(t_1^2\triangle^{(1)}e^{-t_1^2\triangle^{(1)}} t_2^2\triangle^{(2)}e^{-t_2^2\triangle^{(2)}})(f)(y_1,y_2)\big|^2 \chi_{T(R)}{dt_1dt_2\over t_1t_2}\bigg)^{1\over2}dy_2dy_1\\
&&\leq     {1\over \lambda_{\ell}}  \|g\|_2 \bigg\|\bigg( \sum_{\bar R\in B_\ell } \int_{0}^\infty\int_{0}^\infty
        \big|(t_1^2\triangle^{(1)}e^{-t_1^2\triangle^{(1)}} t_2^2\triangle^{(2)}e^{-t_2^2\triangle^{(2)}})(f)(y_1,y_2)\big|^2 \chi_{T(R)}{dt_1dt_2\over t_1t_2}\bigg)^{1\over2}\bigg\|_2\\
  &&\leq  |\widetilde{\Omega}_\ell|^{-{1\over2}},
 \end{eqnarray*}
and hence, we have  $
\|a\|_2 \leq C
|\widetilde {\Omega}_\ell|^{-{1\over 2}}$.

A similar argument to that above shows that for every  $0\leq k_1,k_2\leq M$,
\begin{eqnarray*}
\sum_{\bar R\in B_\ell, \bar R {\rm\ max}}\ell(I_R)^{-4M}\ell(J_{\frak R})^{-4M}
\|\big(\ell(I_R)^2\triangle^{(1)}\big)^{k_1}\otimes_2\big(\ell(J_{\frak  R})^2\triangle^{(2)}\big)^{k_2}b_{\ell,\bar R}\|^2_2
  \leq  C |\Omega|^{-1} .
\end{eqnarray*}

 Combining all the estimates above, we can see that $a$ is a $(1, 2, M)$-atom
 as in Definition \ref{def of product atom}  up to some constant depending only on $M,\psi$.

To see that the atomic decomposition $\sum_\ell \lambda_\ell
a_\ell$ converges to $f$ in the $L^2(\mathbb{R}^{n+m})$ norm, we
only need to show that $\|\sum_{|\ell|>G} \lambda_\ell
a_\ell\|_2\rightarrow 0$ as $G$ tends to
infinity. To see this, first note that
$$
    \Big\|\sum_{|\ell|>G} \lambda_\ell a_\ell\Big\|_2
    =\sup_{h:\, \|h\|_2}
    \Big|\big\langle \sum_{|\ell|>G} \lambda_\ell a_\ell, h\big\rangle_{L^2(\mathbb{R}^{n+m})}\Big|.
$$
Next, we have
\begin{align*}
    &\Big|\big\langle \sum_{|\ell|>G} \lambda_\ell a_\ell, h\big\rangle_{L^2(\mathbb{R}^{n+m})}\Big|\\
    &=\bigg| \sum_{|\ell|>G} \sum_{R\in B_\ell}  \int_{T(R)} \psi(t_1\sqrt{\triangle^{(1)}})\psi(t_2\sqrt{\triangle^{(2)}})(h)(y_1,y_2)\nonumber\\
    &\hskip1.5cm \times(t_1^2\triangle^{(1)}e^{-t_1^2\triangle^{(1)}} t_2^2\triangle^{(2)}e^{-t_2^2\triangle^{(2)}})(f)(y_1,y_2)dy_1dy_2{dt_1dt_2\over t_1t_2}\ \bigg|\\
    &\leq \int_{\mathbb{R}^{n}\times\mathbb{R}^{m}}\bigg( \sum_{|\ell|>G}\sum_{R\in B_\ell}
        \int_{0}^\infty\!\!\int_{0}^\infty |\psi(t_1\sqrt{\triangle^{(1)}})\psi(t_2\sqrt{\triangle^{(2)}})(h)(y_1,y_2)|^2\chi_{T(R)}{dt_1dt_2\over t_1t_2} \bigg)^{1\over2}\nonumber\\
    &\hskip.78cm\bigg( \sum_{|\ell|>G}\sum_{R\in B_\ell}\!
        \int_{0}^\infty\!\!\!\int_{0}^\infty \!\!\big|(t_1^2\triangle^{(1)}e^{-t_1^2\triangle^{(1)}}\otimes_2 t_2^2\triangle^{(2)}e^{-t_2^2\triangle^{(2)}})(f)(y_1,y_2)\big|^2\chi_{T(R)}  {dt_1dt_2\over t_1t_2}\bigg)^{1\over2}\!dy_1dy_2\\
    &\leq C\|h\|_2 \bigg\|\bigg( \sum_{|\ell|>G}\sum_{R\in B_\ell}
        \int_{0}^\infty\!\!\int_{0}^\infty \big|(t_1^2\triangle^{(1)}e^{-t_1^2\triangle^{(1)}}\otimes_2 t_2^2\triangle^{(2)}e^{-t_2^2\triangle^{(2)}})(f)\big|^2\chi_{T(R)}  {dt_1dt_2\over t_1t_2}\bigg)^{1\over2}\bigg\|_2\\
    &\rightarrow 0
\end{align*}
as $G$ tends to $\infty$, since $\|S_{F,\triangle^{(1)},\triangle^{(2)}}f\|_2 < \infty$.  This implies that $f = \sum_\ell \lambda_\ell a_\ell$ in the
sense of $L^2(\mathbb R^{n+m})$.

Next, we verify the estimate for the series
$\sum_{\ell}|\lambda_\ell|$. To deal with this, we claim that for
each $\ell\in\mathbb{Z}$,
\begin{eqnarray*}
    \sum_{R\in B_\ell}\int_{T(R)} \big|(t_1^2\triangle^{(1)}e^{-t_1^2\triangle^{(1)}}\otimes_2 t_2^2\triangle^{(2)}e^{-t_2^2\triangle^{(2)}})(f)(y_1,y_2)\big|^2 dy_2dy_1{dt_1dt_2\over t_1t_2}
    \leq C2^{2(\ell+1)}|\widetilde{\Omega}_\ell|.
\end{eqnarray*}

\noindent First we note that
\begin{equation*}
    \int_{\widetilde{\Omega}_\ell\backslash \Omega_{\ell+1}} S_{F,\triangle^{(1)},\triangle^{(2)}}(f)^2(x_1,x_2)\, dx_2dx_1
    \leq  2^{2(\ell+1)}|\widetilde{\Omega}_\ell|.
\end{equation*}
Also we point out that
\begin{eqnarray*}
    && \int_{\widetilde{\Omega}_\ell\backslash \Omega_{\ell+1}} S_{F,\triangle^{(1)},\triangle^{(2)}}(f)^2(x_1,x_2)\, dx_2dx_1\\
    &&= \int_{\widetilde{\Omega}_\ell\backslash \Omega_{\ell+1}} \int_{\Bbb R^{n+1}_+}\int_{\Bbb R^{m+1}_+} \chi_{t_1,t_2}(x_1-y_1,x_2-y_2) \\
    &&\hskip1cm\big|\big(t_1^2\triangle^{(1)}e^{-t_1^2\triangle^{(1)}}\otimes_2
    t_2^2\triangle^{(2)}e^{-t_2^2\triangle^{(2)}}\big)f(y_1,y_2)\big|^2{dy_2dt_2\, dy_1dt_1\over t_2^{m+1} t_1^{n+m+1}}\, dx_2dx_1\\
    &&=   \int_{\Bbb R^{n+1}_+}\int_{\Bbb R^{m+1}_+}\big|\big(t_1^2\triangle^{(1)}e^{-t_1^2\triangle^{(1)}}\otimes_2
    t_2^2\triangle^{(2)}e^{-t_2^2\triangle^{(2)}}\big)(f)(y_1,y_2)\big|^2\\
    &&\hskip1cm\times |\{(x_1,x_2)\in\widetilde{\Omega}_\ell\backslash \Omega_{\ell+1}:\, |x_1-y_1|<t_1,|x_2-y_2|<t_1+t_2\}|
        {dy_2dt_2\, dy_1dt_1\over t_2^{m+1} t_1^{n+m+1} } \\
    &&\geq \sum_{R\in B_\ell}\int_{T(R)}\big|\big(t_1^2\triangle^{(1)}e^{-t_1^2\triangle^{(1)}}\otimes_2
    t_2^2\triangle^{(2)}e^{-t_2^2\triangle^{(2)}}\big)(f)(y_1,y_2)\big|^2\\
    &&\hskip1cm\times |\{(x_1,x_2)\in\widetilde{\Omega}_\ell\backslash \Omega_{\ell+1}:\, |x_1-y_1|<t_1,|x_2-y_2|<t_1+t_2\}|
        {dy_2dt_2\, dy_1dt_1\over  t_2^{m+1} t_1^{n+m+1}} \\
    &&\geq C\sum_{R\in B_\ell}\int_{T(R)} \big|(t_1^2\triangle^{(1)}e^{-t_1^2\triangle^{(1)}}\otimes_2 t_2^2\triangle^{(2)}e^{-t_2^2\triangle^{(2)}})(f)(y_1,y_2)\big|^2
        {dy_1dy_2dt_1dt_2\over t_1t_2},
\end{eqnarray*}
where the last inequality follows from the definition of $B_\ell$. This shows that the claim holds.

As a consequence, we have
\begin{align*}
    &\sum_\ell|\lambda_\ell| \\
    &\leq C\!\sum_\ell\!\bigg\|\bigg( \sum_{R\in B_\ell}
        \int_{0}^\infty\!\!\int_{0}^\infty \big|(t_1^2\triangle^{(1)}e^{-t_1^2\triangle^{(1)}}\otimes_2 t_2^2\triangle^{(2)}e^{-t_2^2\triangle^{(2)}})(f)\big|^2\chi_{T(R)}
        {dt_1dt_2\over t_1t_2}\bigg)^{1/2}\bigg\|_2|\widetilde{\Omega}_\ell|^{1\over2}\\
    &\leq C\!\sum_\ell\!\bigg(\!\sum_{R\in B_\ell}\int_{T(R)}\! \big|(t_1^2\triangle^{(1)}e^{-t_1^2\triangle^{(1)}}\otimes_2 t_2^2\triangle^{(2)}e^{-t_2^2\triangle^{(2)}})(f)(y_1,y_2)\big|^2
        dy_1dy_2{dt_1dt_2\over t_1t_2}\bigg)^{1\over2}  |\widetilde{\Omega}_\ell|^{1\over2}\\
    &\leq C\sum_\ell2^{\ell+1}|\widetilde{\Omega}_\ell|\leq C\sum_\ell2^{\ell}|\Omega_\ell|\\
    &\leq C\|S_{F,\triangle^{(1)},\triangle^{(2)}}(f)\|_1\\
    &=C\|f\|_{H^1_{F,\triangle^{(1)},\triangle^{(2)}}(\mathbb R^n\times \mathbb R^m)}.
\end{align*}
 This completes the proof of Proposition \ref{prop-product H-SL subset H-at}.
\end{proof}

\section{Estimate of Riesz transform and area function via atomic decomposition}

In this section we prove
\begin{align}\label{Riesz Area}
\sum_{j=1}^{n+m}\sum_{k=1}^m\|R_{j,k}(f)\|_1+\|f\|_1\lesssim \|S_F(f)\|_1.
\end{align}

Based on the estimate in \eqref{SFLaplace SF}, to prove the estimate \eqref{Riesz Area},
it suffices to prove that there exists a positive constant $C$ such that for  $j=0,1,\ldots,n+m$, and for $k=0,\ldots,m,$
\begin{align}\label{eeee1}
\|R_{j,k}(f)\|_{1} \leq C\|S_{F,\triangle^{(1)},\triangle^{(2)}}(f)\|_1. 
\end{align}
Indeed, as mentioned, $R_{j,k}$ is the composition of $R^{(1)}_j$ and $R^{(2)}_k,$ and hence $R_{j,k}$ is bounded on $L^p(\R^{n+m}), 1<p<\infty.$
The flag Riesz transform can also be defined by $T:=\nabla^{(1)} ({\triangle^{(1)}})^{-1/2}\otimes_2 \nabla^{(2)} ({\triangle^{(2)}})^{-1/2}$ as follows via functional calculus,
\begin{eqnarray}\label{e8.16}
Tf(x_1, x_2)={1\over 4 {\pi}}\int_0^{\infty}\int_0^{\infty}\big(\nabla^{(1)} e^{-t_1\triangle^{(1)}}\otimes_2 \nabla^{(2)}
e^{-t_2\triangle^{(2)}}\big)f(x_1, x_2){dt_1 dt_2\over \sqrt{t_1 t_2}}.
\end{eqnarray}

The estimate $\sum_{j=1}^{n+m}\sum_{k=1}^m\|R_{j,k}(f)\|_1+\|f\|_1\lesssim \|S_F(f)\|_1$ follows from the following theorem.
\begin{theorem}\label{Riesz}
The flag Riesz transform $\nabla^{(1)} ({\triangle^{(1)}})^{-1/2}\otimes_2 \nabla^{(2)} ({\triangle^{(2)}})^{-1/2}$
extends to a bounded operator from
$H_{F,\triangle^{(1)},\triangle^{(2)}}^1(\mathbb{R}^{n}\times\mathbb{R}^{m})$ to
$L^1(\mathbb{R}^{n+m})$.
\end{theorem}

\begin{proof}
Let $T:=\nabla^{(1)} ({\triangle^{(1)}})^{-1/2}\otimes_2 \nabla^{(2)} ({\triangle^{(2)}})^{-1/2}$.
It suffices to show that   $T$  is uniformly bounded on each $(1, 2, M)$ atom $a$ with
$M>\max\{n, m\}/2$, and  there exists a constant
$C>0$ independent of $a$ such that
\begin{eqnarray}\label{T alpha uniformly bd}
\|T(a)\|_1\leq C.
\end{eqnarray}

From the definition of $(1, 2, M)$ atom, it follows
that $a$ is supported in some $\Omega\subset
\mathbb{R}^{n}\times\mathbb{R}^{m}$ and $a$ can be further
decomposed into $a=\sum_{R\in m(\Omega)}a_R$. For any $R=I \times
J\subset\Omega$, let $l$ be the biggest dyadic cube containing  $I$,
so that $l\times J\subset\widetilde{\Omega}$, where
$\widetilde{\Omega}=\{x\in\mathbb{R}^{n}\times\mathbb{R}^{m}:\
M_s(\chi_{\Omega})(x)>1/2\}$. Next, let $Q$ be the biggest dyadic
cube containing $J$, so that $l\times Q\subset
\widetilde{\widetilde{\Omega}}$, where
$\widetilde{\widetilde{\Omega}}=\{x\in
\mathbb{R}^{n}\times\mathbb{R}^{m}:\
M_s(\chi_{\widetilde{\Omega}})(x)>1/2\}$. Now let $\widetilde{R}$ be
the 100-fold dilate of $l\times Q$ concentric with $l\times Q$.
Clearly, an application of the strong maximal function theorem shows
that $\big|\bigcup\limits_{R\subset\Omega} \widetilde{R}\big|\leq
C|\widetilde{\widetilde{\Omega}}|\leq C|\widetilde{\Omega}|\leq
C|\Omega|$. From (iii) in the definition of $(1, 2, M)$
atom, we can obtain that
$$
\int_{\cup \widetilde{R}}|T(a)(x)|dx \leq \big|\cup
\widetilde{R}\big|^{1/2}\|T(a)\|_2\leq
C|\Omega|^{1/2}
\|a\|_2
 \leq   C|\Omega|^{1/2}|\Omega|^{-1/2}
 \leq  C.$$
Therefore,  the proof of  (\ref{T alpha uniformly bd}) reduces to showing
that
\begin{eqnarray}\label{T alpha uniformly bd on outside of Omega}
\int_{\big(\cup \widetilde{R}\big)^c}|T(a)(x_1,x_2)|dx_2dx_1 \leq C.
\end{eqnarray}

\noindent Since  $a=\sum_{R\in m(\Omega)}a_R$, we have
\begin{align*}
&\int_{\big(\cup \widetilde{R}\big)^c}|T(a)(x_1,x_2)|dx_2dx_1
\\
&\leq \sum_{R\in
m(\Omega)} \int_{{\widetilde{R}}^c}|T(a_R)(x_1,x_2)|dx_2dx_1\\
&\leq \sum_{R\in
m(\Omega)}\int_{(100l)^c\times\mathbb{R}^{m}}|T(a_R)(x_1,x_2)|dx_2dx_1+\sum_{R\in
m(\Omega)}\int_{\mathbb{R}^{n}\times(100S)^c}|T(a_R)(x_1,x_2)|dx_2dx_1 \\
&=\textrm{I}+\textrm{II}.
\end{align*}

For term $\textrm{I}$, we observe that
\begin{align*}
\int_{(100l)^c\times\mathbb{R}^{m}}|T(a_R)(x_1,x_2)|dx_2dx_1
&=\Big(\int_{(100l)^c\times 100J} +\int_{(100l)^c\times (100J)^c}\Big) |T(a_R)(x_1,x_2)|dx_2dx_1 \\
&=\textrm{I}_1+\textrm{I}_2.
 \end{align*}

Let us first estimate the term $\textrm{I}_1$. H\"older's inequality gives
\begin{align*}
\textrm{I}_1&= \int_{(100l)^c}\int_{100J}  \bigg| \nabla^{(2)} ({\triangle^{(2)}})^{-1/2} \otimes_2 \nabla^{(1)}({\triangle^{(1)}})^{-1/2}a_R(x_1, x_2)\bigg|dx_2dx_1\\
&\lesssim |J|^{1/2}\int_{(100l)^c}\bigg(\int_{100J}  \bigg| \nabla^{(2)} ({\triangle^{(2)}})^{-1/2}\big( \nabla^{(1)}({\triangle^{(1)}})^{-1/2}a_R(x_1, \cdot)\big)(x_2)\bigg|^2dx_2\bigg)^{1/2}dx_1.
\end{align*}
By using the $L^2$-boundedness of $\nabla^{(2)} ({\triangle^{(2)}})^{-1/2}$, we obtain that
\begin{align*}
\textrm{I}_1
&\lesssim |J|^{1/2}\int_{(100l)^c}\bigg(\int_{\Bbb R^m}  \bigg|\nabla^{(1)}({\triangle^{(1)}})^{-1/2}a_R(x_1, x_2)\bigg|^2dx_2\bigg)^{1/2}dx_1\\
&\lesssim |J|^{1/2}\int_{(100l)^c}\bigg(\int_{100J}  \bigg| \int_0^{\infty}\nabla^{(1)} e^{-t_1\triangle^{(1)}}  a_R(x_1, x_2) {dt_1 \over \sqrt{t_1 }}\bigg|^2dx_2\bigg)^{1/2}dx_1\\
&\quad + |J|^{1/2}\int_{(100l)^c}\bigg(\int_{(100J)^c}  \bigg| \int_0^{\infty}\nabla^{(1)} e^{-t_1\triangle^{(1)}}  a_R(x_1, x_2) {dt_1 \over \sqrt{t_1 }}\bigg|^2dx_2\bigg)^{1/2}dx_1\\
&=\textrm{I}_{11}+\textrm{I}_{12}
\end{align*}

We first handle $\textrm{I}_{11}$.
Let $a_R= (({\triangle^{(1)}})^M \otimes_2 1\!\!1_2)a_{R,2}$ where $a_{R,2}=(1\!\!1_1 \otimes_2({\triangle^{(2)}})^M)b_R$.
\begin{align*}
\textrm{I}_{11} &\lesssim |J|^{1/2}\int_{(100l)^c}\bigg(\int_{100J}  \bigg|\int_{10I}\int_{10J} \int_0^{\infty} q_{t_1}^{(1)}(x_1-y_1,x_2-y_2)a_R(y_1, y_2) dy_2dy_1{dt_1 \over {t_1 }}\bigg|^2dx_2\bigg)^{1/2}dx_1\\
&= |J|^{1/2}\int_{(100l)^c}\bigg(\int_{100J}  \bigg|\int_{10I}\int_{10J} \int_0^{\infty} \big(t_1\triangle^{(1)}\big)^Mq_{t_1}^{(1)}(x_1-y_1,x_2-y_2)a_{R,2}(y_1, y_2)dy_2dy_1 \\
 &\hskip6cm \times {dt_1 \over {t_1^{1+M} }}\bigg|^2dx_2\bigg)^{1/2}dx_1.
\end{align*}
By the heat kernel estimate and Hardy-Littlewood Maximal function on $\Bbb R^m$,
\begin{align*}
&\textrm{I}_{11}\\
&\lesssim |J|^{1/2}\int_{(100l)^c}\bigg(\int_{100J}  \bigg|\int_{10I}\int_{10J} \int_0^{\infty} {1\over t_1^{n+m\over2}}e^{-{|(x_1,x_2)-(y_1,y_2)|^2\over t_1}} |a_{R,2}(y_1, y_2)|dy_2dy_1 {dt_1 \over t_1^{1+M}}\bigg|^2dx_2\bigg)^{1/2}dx_1\\
&= |J|^{1/2}\int_{(100l)^c}\bigg(\int_{100J}  \bigg|\int_{10J} t_1^{-\frac m2}e^{-|x_2-y_2|^2\over t_1}\int_{10I}\int_0^{\infty} {1\over t_1^{n\over2}}e^{-{|x_1-y_1|^2\over t_1}} |a_{R,2}(y_1, y_2)|dy_1 {dt_1 \over t_1^{1+M}}dy_2\bigg|^2dx_2\bigg)^{1/2}dx_1\\
&\lesssim  |J|^{1/2}\int_{(100l)^c}\bigg(\int_{100J}  \bigg|M_2\bigg(\int_{10I}\int_0^{\infty} {1\over t_1^{n\over2}}e^{-{|x_1-y_1|^2\over t_1}} |a_{R,2}(y_1, \cdot)|dy_1 {dt_1 \over t_1^{1+M}}\bigg)(x_2)\bigg|^2dx_2\bigg)^{1/2}dx_1\\
&\lesssim  |J|^{1/2}\int_{(100l)^c}\bigg(\int_{\Bbb R^m}  \bigg|\int_{10I}\int_0^{\infty} {1\over t_1^{n\over2}}e^{-{|x_1-y_1|^2\over t_1}} |a_{R,2}(y_1, x_2)|dy_1 {dt_1 \over t_1^{1+M}}\bigg|^2dx_2\bigg)^{1/2}dx_1\\
&\lesssim  |J|^{1/2}\int_{(100l)^c}\bigg(\int_{10J}  \bigg|\int_{10I}\int_0^{\ell(I)^2} {1\over t_1^{n\over2}}e^{-{|x_1-y_1|^2\over t_1}} |a_{R,2}(y_1, x_2)|dy_1 {dt_1 \over t_1^{1+M}}\bigg|^2dx_2\bigg)^{1/2}dx_1\\
&\qquad +  |J|^{1/2}\int_{(100l)^c}\bigg(\int_{10J}  \bigg|\int_{10I}\int_{\ell(I)^2}^\infty {1\over t_1^{n\over2}}e^{-{|x_1-y_1|^2\over t_1}} |a_{R,2}(y_1, x_2)|dy_1 {dt_1 \over t_1^{1+M}}\bigg|^2dx_2\bigg)^{1/2}dx_1\\
&=:\textrm{I}_{111} +\textrm{I}_{112}.
 \end{align*}

We first consider $\textrm{I}_{111}$ and write
\begin{align*}
\textrm{I}_{111}
&\lesssim |J|^{1/2}\sum_{j_1=\tilde j}^\infty \int_{|x_1-x_I|\approx 2^{j_1}\ell(I)}\bigg(\int_{10J}  \bigg|\int_{10I}\int_0^{\ell(I)^2} {1\over t_1^{n\over2}}e^{-{|x_1-y_1|^2\over t_1}} |a_{R,2}(y_1, x_2)|dy_1 {dt_1 \over t_1^{1+M}}\bigg|^2dx_2\bigg)^{1/2}dx_1\\
&\lesssim |J|^{1/2}\sum_{j_1=\tilde j}^\infty \int_{|x_1-x_I|\approx 2^{j_1}\ell(I)}\\
&\qquad\bigg(\int_{10J}  \bigg|\int_{10I}\int_0^{\ell(I)^2} {1\over t_1^{n\over2}} \bigg({ t_1\over|x_1-x_I|^2}\bigg)^{\alpha_1} |a_{R,2}(y_1, x_2)|dy_1 {dt_1 \over t_1^{1+M}}\bigg|^2dx_2\bigg)^{1/2}dx_1\\
&\lesssim |J|^{1/2}\sum_{j_1=\tilde j}^\infty \int_{|x_1-x_I|\approx 2^{j_1}\ell(I)}\frac{\ell(I)^{2\alpha_1-n-2M}}{|x_1-x_I|^{2\alpha_1}}\bigg(\int_{10J}  \bigg|\int_{10I} |a_{R,2}(y_1, x_2)|dy_1\bigg|^2dx_2\bigg)^{1/2}dx_1\\
&\lesssim |J|^{1/2}\sum_{j_1=\tilde j}^\infty \int_{|x_1-x_I|\approx 2^{j_1}\ell(I)}\frac{\ell(I)^{2\alpha_1-n-2M}}{|x_1-x_I|^{2\alpha_1}}\bigg(\int_{10J}  |I|\int_{10I} |a_{R,2}(y_1, x_2)|^2dy_1dx_2\bigg)^{1/2}dx_1
 \end{align*}
where we choose $2\alpha_1-n>2M$ and $\tilde j$ is the smallest integer such that
$2^{\tilde j} I \cap (100l)^c\not=\emptyset$.
 Hence,
\begin{align*}
\textrm{I}_{111}
&\lesssim |R|^{1/2}\sum_{j_1=\tilde j}^\infty \int_{|x_1-x_I|\approx 2^{j_1}\ell(I)}\frac{\ell(I)^{2\alpha_1-n}}{|x_1-x_I|^{2\alpha_1}} \ell(I)^{-2M}\|a_{R,2}\|_2\\
&\lesssim |R|^{1/2} \sum_{j_1=\tilde j}^\infty\frac{\ell(I)^{2\alpha_1-n}}{(2^{j_1}\ell(I))^{2\alpha_1-n}}  \ell(I)^{-2M}\|a_{R,2}\|_2\\
&\lesssim |R|^{1\over2}   \gamma_1(R)^{-(2\alpha_1-n)}    \ell(I)^{-2M} \|a_{R,2}\|_2,
 \end{align*}
where the last inequality follows from the fact that
\begin{equation}\label{tilde j}
2^{\tilde j} \approx {\ell(l)\over \ell(I)}.
\end{equation}

To consider $\textrm{I}_{112}$, we choose $0<2\alpha_1-n<2M$ and hence
\begin{align*}
&\textrm{I}_{112}\\
&\lesssim |J|^{1/2}\sum_{j_1=\tilde j}^\infty \int_{|x_1-x_I|\approx 2^{j_1}\ell(I)}\bigg(\int_{10J}  \bigg|\int_{10I}\int_{\ell(I)^2}^\infty {1\over t_1^{n\over2}}e^{-{|x_1-y_1|^2\over t_1}} |a_{R,2}(y_1, x_2)|dy_1 {dt_1 \over t_1^{1+M}}\bigg|^2dx_2\bigg)^{1/2}dx_1\\
&\lesssim |J|^{1/2}\sum_{j_1=\tilde j}^\infty \int_{|x_1-x_I|\approx 2^{j_1}\ell(I)}\\
&\qquad\bigg(\int_{10J}  \bigg|\int_{10I}\int_{\ell(I)^2}^\infty {1\over t_1^{n\over2}} \bigg({ t_1\over|x_1-x_I|^2}\bigg)^{\alpha_1} |a_{R,2}(y_1, x_2)|dy_1 {dt_1 \over t_1^{1+M}}\bigg|^2dx_2\bigg)^{1/2}dx_1\\
&\lesssim |J|^{1/2}\sum_{j_1=\tilde j}^\infty \int_{|x_1-x_I|\approx 2^{j_1}\ell(I)}\frac{\ell(I)^{2\alpha_1-n-2M}}{|x_1-x_I|^{2\alpha_1}}\bigg(\int_{10J}  \bigg|\int_{10I} |a_{R,2}(y_1, x_2)|dy_1\bigg|^2dx_2\bigg)^{1/2}dx_1\\
&\lesssim |J|^{1/2}\sum_{j_1=\tilde j}^\infty \int_{|x_1-x_I|\approx 2^{j_1}\ell(I)}\frac{\ell(I)^{2\alpha_1-n-2M}}{|x_1-x_I|^{2\alpha_1}}\bigg(\int_{10J}  |I|\int_{10I} |a_{R,2}(y_1, x_2)|^2dy_1dx_2\bigg)^{1/2}dx_1\\
&\lesssim |R|^{1\over2}   \gamma_1(R)^{-(2\alpha_1-n)}    \ell(I)^{-2M} \|a_{R,2}\|_2,
 \end{align*}
where again the last inequality follows from \eqref{tilde j}

We now handle $\textrm{I}_{12}$. Similar to the term $\textrm{I}_{11}$ we have
\begin{align*}
\textrm{I}_{12} &\lesssim |J|^{1/2}\int_{(100l)^c}\bigg(\int_{(100J)^c}  \bigg|\int_{10I}\int_{10J} \int_0^{\infty} q_{t_1}^{(1)}(x_1-y_1,x_2-y_2)a_R(y_1, y_2) dy_2dy_1{dt_1 \over {t_1 }}\bigg|^2dx_2\bigg)^{1/2}dx_1\\
&= |J|^{1/2}\int_{(100l)^c}\bigg(\int_{(100J)^c}  \bigg|\int_{10I}\int_{10J} \int_0^{\infty} \big(t_1\triangle^{(1)}\big)^Mq_{t_1}^{(1)}(x_1-y_1,x_2-y_2)a_{R,2}(y_1, y_2)dy_2dy_1 \\
 &\hskip6cm \times {dt_1 \over {t_1^{1+M} }}\bigg|^2dx_2\bigg)^{1/2}dx_1\\
&\lesssim |J|^{1/2}\int_{(100l)^c}\bigg(\int_{(100J)^c}  \bigg|\int_{10I}\int_{10J} \int_0^{\infty} \frac1{t_1^{\frac{n+m}2}}e^{-\frac{|x_1-y_1|^2+|x_2-y_2|^2}{t_1}}|a_{R,2}(y_1, y_2)|dy_2dy_1 \\
 &\hskip6cm \times {dt_1 \over {t_1^{1+M} }}\bigg|^2dx_2\bigg)^{1/2}dx_1\\
&\lesssim |J|^{1/2}\int_{(100l)^c}\bigg(\int_{(100J)^c}  \bigg|\int_{10I}\int_{10J} \int_0^{\infty} \frac1{t_1^{\frac n2}}\Big(\frac{t_1}{|x_1-y_1|^2}\Big)^{\alpha_1} \frac1{t_1^{\frac m2}}\Big(\frac{t_1}{|x_2-y_2|^2}\Big)^{\alpha_2}\\
 &\hskip6cm \times |a_{R,2}(y_1, y_2)|dy_2dy_1 {dt_1 \over {t_1^{1+M} }}\bigg|^2dx_2\bigg)^{1/2}dx_1
\end{align*}
for any $\alpha_1,\alpha_2>0$. Next, similar to the estimate in $\textrm{I}_{11}$, let $\tilde j$ is the smallest integer such that
$2^{\tilde j} I \cap (100l)^c\not=\emptyset$. Then we have

\begin{align*}
\textrm{I}_{12}
&\lesssim |J|^{1/2}\sum_{j_1=\tilde j}^\infty \int_{|x_1-x_I|\approx 2^{j_1}\ell(I)}\bigg(\sum_{j_2=6}^\infty\int_{|x_2-y_J|\approx 2^{j_2}\ell(J)} \\
&\quad\times \bigg|\int_{10I}\int_{10J} \int_0^{\ell(I)^2} \frac1{t_1^{\frac n2}}\Big(\frac{t_1}{|x_1-y_1|^2}\Big)^{\alpha_1} \frac1{t_1^{\frac m2}}\Big(\frac{t_1}{|x_2-y_2|^2}\Big)^{\alpha_2}|a_{R,2}(y_1, y_2)|dy_2dy_1{dt_1 \over {t_1^{1+M} }}\bigg|^2dx_2\bigg)^{1/2}dx_1\\
&\ + |J|^{1/2}\sum_{j_1=\tilde j}^\infty \int_{|x_1-x_I|\approx 2^{j_1}\ell(I)}\bigg(\sum_{j_2=6}^\infty\int_{|x_2-y_J|\approx 2^{j_2}\ell(J)} \\
&\quad\times \bigg|\int_{10I}\int_{10J} \int_{\ell(I)^2}^{\infty} \frac1{t_1^{\frac n2}}\Big(\frac{t_1}{|x_1-y_1|^2}\Big)^{\alpha_1} \frac1{t_1^{\frac m2}}\Big(\frac{t_1}{|x_2-y_2|^2}\Big)^{\alpha_2}|a_{R,2}(y_1, y_2)|dy_2dy_1{dt_1 \over {t_1^{1+M} }}\bigg|^2dx_2\bigg)^{1/2}dx_1\\
&=\textrm{I}_{121} + \textrm{I}_{122}.
\end{align*}
As for $\textrm{I}_{121}$, by choosing $2\alpha_1>n+2M$ and $2\alpha_2>m$, and then by taking the integration of $t_1$ as
$$\int_0^{\ell(I)^2} t_1^{\alpha_1-{\frac n2}+\alpha_2-{\frac m2}-M-1}dt_1=\ell(I)^{2\alpha_1-n+2\alpha_2-m-2M},$$
we have
\begin{align*}
\textrm{I}_{121}
&\lesssim |J|^{1/2}\sum_{j_1=\tilde j}^\infty \int_{|x_1-x_I|\approx 2^{j_1}\ell(I)}\bigg(\sum_{j_2=6}^\infty\int_{|x_2-y_J|\approx 2^{j_2}\ell(J)} \\
&\quad\times \bigg|\ell(I)^{2\alpha_1-n+2\alpha_2-m-2M}\frac1{2^{2j_1\alpha_1}\ell(I)^{2\alpha_1}}\frac1{2^{2j_2\alpha_2}\ell(J)^{2\alpha_2}} \|a_{R,2}\|_1\bigg|^2dx_2\bigg)^{1/2}dx_1\\
&\lesssim |J|^{1/2}\sum_{j_1=\tilde j}^\infty \int_{|x_1-x_I|\approx 2^{j_1}\ell(I)}\ell(I)^{2\alpha_1-n+2\alpha_2-m-2M}\frac1{2^{2j_1\alpha_1}\ell(I)^{2\alpha_1}} \|a_{R,2}\|_1 \\
&\quad\times \bigg(\sum_{j_2=6}^\infty \frac {2^{j_2m}\ell(J)^m}{2^{4j_2\alpha_2}\ell(J)^{4\alpha_2}}\bigg)^{1/2}dx_1\\
&\lesssim |J|^{1/2}\sum_{j_1=\tilde j}^\infty  \ell(I)^{2\alpha_1-n+2\alpha_2-m-2M}\frac{2^{j_1n}\ell(I)^n}{2^{2j_1\alpha_1}\ell(I)^{2\alpha_1}} |R|^{1/2}\|a_{R,2}\|_2\ell(J)^{\frac m2-2\alpha_2} \\
&\lesssim \gamma_1(R)^{-(2\alpha_1-n)}|R|^{1/2}\ell(I)^{-2M}\|a_{R,2}\|_2,
\end{align*}
where again the last inequality inequality follows from \eqref{tilde j}.

As for $\textrm{I}_{121}$, by choosing $n<2\alpha_1<n+2M$ and $m/4<\alpha_2<m/2$, and then by taking the integration of $t_1$ as
$$\int_{\ell(I)^2}^\infty t_1^{\alpha_1-{\frac n2}+\alpha_2-{\frac m2}-M-1}dt_1=\ell(I)^{2\alpha_1-n+2\alpha_2-m-2M},$$
we have
\begin{align*}
\textrm{I}_{121}
&\lesssim |J|^{1/2}\sum_{j_1=\tilde j}^\infty \int_{|x_1-x_I|\approx 2^{j_1}\ell(I)}\bigg(\sum_{j_2=6}^\infty\int_{|x_2-y_J|\approx 2^{j_2}\ell(J)} \\
&\quad\times \bigg|\ell(I)^{2\alpha_1-n+2\alpha_2-m-2M}\frac1{2^{2j_1\alpha_1}\ell(I)^{2\alpha_1}}\frac1{2^{2j_2\alpha_2}\ell(J)^{2\alpha_2}} \|a_{R,2}\|_1\bigg|^2dx_2\bigg)^{1/2}dx_1\\
&\lesssim |J|^{1/2}\sum_{j_1=\tilde j}^\infty \int_{|x_1-x_I|\approx 2^{j_1}\ell(I)}\ell(I)^{2\alpha_1-n+2\alpha_2-m-2M}\frac1{2^{2j_1\alpha_1}\ell(I)^{2\alpha_1}} \|a_{R,2}\|_1 \\
&\quad\times \bigg(\sum_{j_2=6}^\infty \frac {2^{j_2m}\ell(J)^m}{2^{4j_2\alpha_2}\ell(J)^{4\alpha_2}}\bigg)^{1/2}dx_1\\
&\lesssim |J|^{1/2}\sum_{j_1=\tilde j}^\infty  \ell(I)^{2\alpha_1-n+2\alpha_2-m-2M}\frac{2^{j_1n}\ell(I)^n}{2^{2j_1\alpha_1}\ell(I)^{2\alpha_1}} |R|^{1/2}\|a_{R,2}\|_2\ell(J)^{\frac m2-2\alpha_2} \\
&\lesssim \gamma_1(R)^{-(2\alpha_1-n)}|R|^{1/2}\ell(I)^{-2M}\|a_{R,2}\|_2,
\end{align*}
where again the last inequality inequality follows from \eqref{tilde j}.

We now consider $\textrm{I}_2$. We write
\begin{align*}
\textrm{I}_2&\leq \int_{(100l)^c\times (100J)^c }  \bigg|  \nabla^{(1)}({\triangle^{(1)}})^{-1/2} \otimes_2 \nabla^{(2)} ({\triangle^{(2)}})^{-1/2} a_R(x_1, x_2)\bigg|dx_2dx_1\\
&= \int_{(100l)^c\times (100J)^c }  \bigg|  {1\over 2 {\sqrt\pi}} \int_0^{\infty}  \nabla^{(1)}e^{-t_1\triangle^{(1)}}    \nabla^{(2)} ({\triangle^{(2)}})^{-1/2} a_R(x_1, x_2)  {dt_1 \over \sqrt{t_1 }}\bigg|dx_2dx_1\\
&=  \int_{(100l)^c\times (100J)^c }  \bigg|  {1\over 2 {\sqrt\pi}} \int_0^{\ell(I)^2} \nabla^{(1)} e^{-t_1\triangle^{(1)}}  \nabla^{(2)} ({\triangle^{(2)}})^{-1/2} a_R(x_1, x_2) {dt_1 \over \sqrt{t_1 }}\bigg|dx_2dx_1\\
&\quad+ \int_{(100l)^c\times (100J)^c}  \bigg|  {1\over 2 {\sqrt\pi}} \int_{\ell(I)^2}^\infty \nabla^{(1)} e^{-t_1\triangle^{(1)}}  \nabla^{(2)} ({\triangle^{(2)}})^{-1/2} a_R(x_1, x_2) {dt_1 \over \sqrt{t_1 }}\bigg|dx_2dx_1\\
&=:\textrm{I}_{21}+\textrm{I}_{22}.
 \end{align*}

We first consider $\textrm{I}_{21}$.
From the heat kernel estimate and the support condition of $a_R$, it is clear that
\begin{align*}
\textrm{I}_{21}
&\lesssim \int_{(100l)^c\times (100J)^c }  \bigg|  {1\over 2 {\sqrt\pi}} \int_0^{\ell(I)^2} \\
&\qquad  \int_{\mathbb R^{n+m}}q^{(1)}_{t_1}(x_1-y_1,x_2-y_2) \nabla^{(2)} ({\triangle^{(2)}})^{-1/2} a_R(y_1, y_2) dy_1dy_2{dt_1 \over {t_1 }}\bigg|dx_2dx_1\\
&\lesssim   \int_{(100l)^c\times (100J)^c }  \bigg|  {1\over 2 {\sqrt\pi}} \int_0^{\ell(I)^2} \\
&\qquad  \int_{10I}\int_{10J} q^{(1)}_{t_1}(x_1-y_1,x_2-y_2) \nabla^{(2)} ({\triangle^{(2)}})^{-1/2} a_R(y_1, y_2) dy_1dy_2{dt_1 \over {t_1 }}\bigg|dx_2dx_1\\
&\quad+   \int_{(100l)^c\times (100J)^c}  \bigg|  {1\over 2 {\sqrt\pi}} \int_0^{\ell(I)^2} \sum_{k_2=0}^\infty\\
&\qquad  \int_{10I}\int_{|y_2-y_J|\approx 2^{k_2}\ell(J) }  q^{(1)}_{t_1}(x_1-y_1,x_2-y_2) \nabla^{(2)} ({\triangle^{(2)}})^{-1/2} a_R(y_1, y_2) dy_1dy_2{dt_1 \over {t_1 }}\bigg|dx_2dx_1\\
&=: \textrm{I}_{211}+\textrm{I}_{212}.
 \end{align*}

For the term $\textrm{I}_{211}$, H\"older's inequality gives
\begin{align*}
\textrm{I}_{211}
&\lesssim \int_{(100l)^c\times (100J)^c }    \int_0^{\ell(I)^2}\int_{10I}\bigg(\int_{10J} |q^{(1)}_{t_1}(x_1-y_1,x_2-y_2)|^2dy_2\bigg)^{1/2} \\
&\qquad\quad\bigg(\int_{10J} |\nabla^{(2)} ({\triangle^{(2)}})^{-1/2} a_R(y_1, y_2)|^2dy_2\bigg)^{1/2} {dt_1 \over {t_1 }}dx_2dx_1\\
&\lesssim \int_{(100l)^c\times (100J)^c}    \int_0^{\ell(I)^2}\int_{10I}\bigg(\int_{10J}  \bigg({1\over t_1^{n+m\over2}}e^{-{|(x_1,x_2)-(y_1,y_2)|^2\over t_1}}\bigg)^2dy_2\bigg)^{1/2} \\
&\qquad\quad\bigg(\int_{10J} |a_R(y_1, y_2)|^2dy_2\bigg)^{1/2} {dt_1 \over {t_1 }}dy_1dx_2dx_1.
\end{align*}
Choosing $2\alpha_1-n>0$ and $2\alpha_2-m>0$, the fact $\ell(I)\le \ell(J)$ implies
\begin{align*}
\textrm{I}_{211}
&\lesssim \sum_{j_1=\tilde j}^\infty \int_{|x_1-x_I|\approx 2^{j_1}\ell(I)}\sum_{j_2=6}^\infty \int_{|x_2-y_J|\approx 2^{j_2}\ell(J)}|J|^{1/2}\int_0^{\ell(I)^2} \\
&\qquad    {1\over t_1^{n+m\over2}} \bigg({ t_1\over|x_1-x_I|^2}\bigg)^{\alpha_1} \bigg({ t_1\over|x_2-y_J|^2}\bigg)^{\alpha_2}\int_{10I} \bigg(\int_{10J} |a_R(y_1, y_2)|^2dy_2\bigg)^{1/2} {dt_1 \over {t_1 }}dy_1dx_2dx_1\\
&\lesssim |R|^{1/2}\sum_{j_1=\tilde j}^\infty \int_{|x_1-x_I|\approx 2^{j_1}\ell(I)}\sum_{j_2=6}^\infty \int_{|x_2-y_J|\approx 2^{j_2}\ell(J)}\int_0^{\ell(I)^2} t_1^{\alpha_1+\alpha_2-\frac{n+m}2-1}dt_1\\
&\qquad     \bigg({ 1\over|x_1-x_I|^2}\bigg)^{\alpha_1} \bigg({ 1\over|x_2-y_J|^2}\bigg)^{\alpha_2}\|a_R\|_2dx_2dx_1\\
&\lesssim |R|^{1/2}\sum_{j_1=\tilde j}^\infty\sum_{j_2=6}^\infty \ell(I)^{2(\alpha_1+\alpha_2-\frac{n+m}2)}{ 1\over \big(2^{j_1}\ell(I)\big)^{2\alpha_1-n}}{ 1\over \big(2^{j_2}\ell(J)\big)^{2\alpha_2-m}}\|a_R\|_2\\
&\lesssim |R|^{1/2}\gamma_1(R)^{-(2\alpha_1-n)}    \|a_R\|_2,
 \end{align*}
where again the last inequality follows from \eqref{tilde j}

Let $a_R= (1\!\!1_1 \otimes_2({\triangle^{(2)}})^M)a_{R,1}$ where $a_{R,1}=(({\triangle^{(1)}})^M \otimes_2 1\!\!1_2)b_R$.
To estimate $\textrm{I}_{212}$, we have
\begin{align*}
&\textrm{I}_{212} \\
&\lesssim \int_{(100l)^c\times (100J)^c }  \bigg| \int_0^{\ell(I)^2}\int_0^\infty \sum_{k_2=0}^\infty \int_{10I}\int_{|y_2-z_2|\approx 2^{k_2}\ell(J) }\\
&\qquad  \times\int_{10J}  q^{(1)}_{t_1}(x_1-y_1,x_2-y_2)(t_2{\triangle^{(2)}})^M q^{(2)}_{t_1}(y_2-z_2)  a_{R,1}(y_1, z_2) dz_2dy_1dy_2{dt_1 \over t_1} {dt_2 \over t_2^{1+M}}\bigg|dx_2dx_1\\
&\lesssim \int_{(100l)^c\times (100J)^c} \int_0^{\ell(I)^2}\int_0^\infty\sum_{k_2=0}^\infty   {1\over t_1^{n+m\over2}}e^{-{|(x_1,x_2)-(y_1,y_2)|^2\over t_1}} \\
&\qquad  \int_{10I}\int_{|y_2-z_2|\approx 2^{k_2}\ell(J) }\int_{10J}   {1\over t_2^{m\over2}}e^{-{|y_2-z_2|^2\over t_2}}dy_2 |a_{R,1}(y_1, z_2)| dz_2dy_1{dt_1 \over {t_1 }} {dt_2 \over {t_2^{1+M} }}dx_2dx_1\\
&\lesssim \int_{(100l)^c\times (100J)^c } \int_0^{\ell(I)^2}\bigg(\int_0^{t_1}+\int_{t_1}^\infty \bigg) {1\over t_1^{n+m\over2}}e^{-{|x_1-y_1|^2\over t_1}}  \int_{10I} \int_{10J}   {1\over t_2^{m\over2}}\\
&\qquad \sum_{k_2=0}^\infty\int_{|y_2-z_2|\approx 2^{k_2}\ell(J) }e^{-{|x_2-y_2|^2\over t_1}}e^{-{|y_2-z_2|^2\over t_2}}dy_2 |a_{R,1}(y_1, z_2)| dz_2dy_1{dt_1 \over {t_1 }} {dt_2 \over {t_2^{1+M} }}dx_2dx_1\\
&=:\textrm{I}_{2121}+\textrm{I}_{2122}.
 \end{align*}

By the heat kernel estimate, we choose $2\alpha_1-n>0$ and $2\alpha_2-m>0$ to obtain
\begin{align*}
\textrm{I}_{2121}
&\lesssim \int_{(100l)^c\times (100J)^c} \int_0^{\ell(I)^2}{1\over t_1^{n+m\over2}}e^{-{|x_1-y_1|^2\over t_1}}\int_0^{t_1}{1\over t_2^{m\over2}}e^{-{\ell(J)^2}\over 2t_2}{dt_2 \over {t_2^{1+M} }}  \\
&\qquad  \int_{10I} \int_{10J} \sum_{k_2=0}^\infty\int_{|y_2-z_2|\approx 2^{k_2}\ell(J) } e^{-{|x_2-y_2|^2\over 2 t_1}}e^{-{|y_2-z_2|^2\over 2t_1}}dy_2 |a_{R,1}(y_1, z_2)| dz_2dy_1{dt_1 \over {t_1 }}dx_2dx_1\\
&\lesssim \int_{(100l)^c\times (100J)^c } \int_0^{\ell(I)^2} {1\over t_1^{n+m\over2}}e^{-{|x_1-y_1|^2\over t_1}}\int_0^{t_1} \ell(J)^{-m-2M-1}dt_2 \\
&\qquad  \int_{10I} \int_{10J}   t_1^{{m\over2}} e^{-{|x_2-z_2|^2\over 2t_1}}|a_R(y_1, z_2)| dz_2dy_1{dt_1 \over {t_1 }}dx_2dx_1\\
&\lesssim\sum_{j_1=\tilde j}^\infty \int_{|x_1-x_I|\approx 2^{j_1}\ell(I)}\sum_{j_2=6}^\infty \int_{|x_2-y_J|\approx 2^{j_2}\ell(J)}\int_0^{\ell(I)^2} \int_{10I}\int_{10J} t_1^{-\frac n2}\ell(J)^{-m-2M}\\
&\quad\times  \bigg({ t_1\over|x_1-x_I|^2}\bigg)^{\alpha_1} \bigg({ t_1\over|x_2-y_J|^2}\bigg)^{\alpha_2} |a_{R,1}(y_1, z_2)| dz_2dy_1 {dt_1 \over {t_1 }}dx_2dx_1.
\end{align*}
Then by using the H\"older's inequality we get that
\begin{align*}
\textrm{I}_{2121}
&\lesssim|R|^{1/2}\sum_{j_1=\tilde j}^\infty\sum_{j_2=6}^\infty \int_0^{\ell(I)^2}t_1^{\alpha_1+\alpha_2-\frac{n}2-1}dt_1\ell(J)^{-m-2M} \\
&\qquad\times   { 1\over \big(2^{j_1}\ell(I)\big)^{2\alpha_1-n}}{ 1\over \big(2^{j_2}\ell(J)\big)^{2\alpha_2-m}}\|a_{R,1}\|_2  \\
&\lesssim |R|^{1/2}\gamma_1(R)^{-(2\alpha_1-n)}\ell(J)^{-2M}    \|a_{R,1}\|_2.
\end{align*}
Choosing $2\alpha_1>n+m$ and $m/2<\alpha_2<M$, we get
\begin{align*}
&\textrm{I}_{2122}\\
&\lesssim \int_{(100l)^c\times (100J)^c} \int_0^{\ell(I)^2}{1\over t_1^{n+m\over2}}e^{-{|x_1-y_1|^2\over t_1}}\int_{t_1}^\infty {1\over t_2^{m\over2}}e^{-{\ell(J)^2}\over 2t_2} \\
&\qquad  \int_{10I} \int_{10J} \sum_{k_2=0}^\infty\int_{|y_2-z_2|\approx 2^{k_2}\ell(J) } e^{-{|x_2-y_2|^2\over 2 t_2}}e^{-{|y_2-z_2|^2\over 2t_2}}dy_2 |a_{R,1}(y_1, z_2)| dz_2dy_1{dt_1 \over {t_1 }}{dt_2 \over {t_2^{1+M} }} dx_2dx_1\\
&\lesssim \int_{(100l)^c\times (100J)^c} \int_0^{\ell(I)^2} {1\over t_1^{n+m\over2}}e^{-{|x_1-y_1|^2\over t_1}}\int_{t_1}^\infty {1\over t_2^{m\over2}}e^{-{\ell(I)^2}\over 2t_2} \\
&\qquad  \int_{10I} \int_{10J}   t_2^{{m\over2}} e^{-{|x_2-z_2|^2\over 2t_2}}|a_{R,1}(y_1, z_2)| dz_2dy_1{dt_1 \over {t_1 }}{dt_2 \over {t_2^{1+M} }} dx_2dx_1\\
&\lesssim\sum_{j_1=\tilde j}^\infty \int_{|x_1-x_I|\approx 2^{j_1}\ell(I)}\sum_{j_2=6}^\infty \int_{|x_2-y_J|\approx 2^{j_2}\ell(J)}\int_0^{\ell(I)^2} \int_{10I}\int_{10J} t_1^{-\frac {n+m}2}\int_{t_1}^\infty e^{-{\ell(I)^2}\over 2t_2} \\
&\quad\times  \bigg({ t_1\over|x_1-x_I|^2}\bigg)^{\alpha_1} \bigg({ t_2\over|x_2-y_J|^2}\bigg)^{\alpha_2} |a_{R,1}(y_1, z_2)| dz_2dy_1 {dt_1 \over {t_1 }}{dt_2 \over {t_2^{1+M} }}dx_2dx_1\\
&\lesssim|R|^{1/2}\sum_{j_1=\tilde j}^\infty\sum_{j_2=6}^\infty \int_0^{\ell(I)^2}t_1^{\alpha_1-\frac{n+m}2-1}\bigg(\int_{t_1}^{\ell(J)^2} \Big(\frac{t_2}{\ell(J)^2}\Big)^{1+M}t_2^{\alpha_2-M-1}dt_2+\int_{\ell(J)^2}^\infty t_2^{\alpha_2-M-1}dt_2\bigg)dt_1\\
&\qquad\times   { 1\over \big(2^{j_1}\ell(I)\big)^{2\alpha_1-n}}{ 1\over \big(2^{j_2}\ell(J)\big)^{2\alpha_2-m}}\|a_{R,1}\|_2  \\
&\lesssim |R|^{1/2}\gamma_1(R)^{-(2\alpha_1-n)}\ell(J)^{-2M}   \|a_{R,1}\|_2.
\end{align*}
We now consider $\textrm{I}_{22}$.
From the heat kernel estimate and the support condition of $a_R$, it is clear that
\begin{align*}
\textrm{I}_{22}
&\lesssim \int_{(100l)^c\times (100J)^c}  \bigg|  {1\over 2 {\sqrt\pi}} \int_{\ell(I)^2}^\infty \\
&\qquad  \int_{\mathbb R^{n+m}}q^{(1)}_{t_1}(x_1-y_1,x_2-y_2) \nabla^{(2)} ({\triangle^{(2)}})^{-1/2} a_R(y_1, y_2) dy_1dy_2{dt_1 \over {t_1 }}\bigg|dx_2dx_1\\
&\lesssim \int_{(100l)^c\times (100J)^c }  \bigg|  {1\over 2 {\sqrt\pi}} \int_{\ell(I)^2}^\infty \\
&\qquad  \int_{10I}\int_{10J} q^{(1)}_{t_1}(x_1-y_1,x_2-y_2) \nabla^{(2)} ({\triangle^{(2)}})^{-1/2} a_R(y_1, y_2) dy_1dy_2{dt_1 \over {t_1 }}\bigg|dx_2dx_1\\
&\quad+   \int_{(100l)^c\times (100J)^c}  \bigg|  {1\over 2 {\sqrt\pi}} \int_{\ell(I)^2}^\infty \sum_{k_2=0}^\infty\\
&\qquad  \int_{10I}\int_{|y_2-z_2|\approx 2^{k_2}\ell(J) }  q^{(1)}_{t_1}(x_1-y_1,x_2-y_2) \nabla^{(2)} ({\triangle^{(2)}})^{-1/2} a_R(y_1, y_2) dy_1dy_2{dt_1 \over {t_1 }}\bigg|dx_2dx_1\\
&=: \textrm{I}_{221}+\textrm{I}_{222}.
 \end{align*}

Let $a_R= (({\triangle^{(1)}})^M \otimes_2 1\!\!1_2)a_{R,2}$ where $a_{R,2}=(1\!\!1_1 \otimes_2({\triangle^{(2)}})^M)b_R$. For the term $\textrm{I}_{221}$, H\"older's inequality gives
\begin{align*}
\textrm{I}_{221}
&=  \int_{(100l)^c\times (100J)^c }  \bigg|  {1\over 2 {\sqrt\pi}} \int_{\ell(I)^2}^\infty \\
&\qquad  \int_{10I}\int_{10J} (t_1\triangle^{(1)})^M q^{(1)}_{t_1}(x_1-y_1,x_2-y_2) \nabla^{(2)} ({\triangle^{(2)}})^{-1/2} a_{R,2}(y_1, y_2) dy_1dy_2{dt_1 \over {t_1^{1+M} }}\bigg|dx_2dx_1\\
&\lesssim \int_{(100l)^c\times (100J)^c } \int_{\ell(I)^2}^\infty   \int_I\bigg(\int_J | (t_1\triangle^{(1)})^Mq^{(1)}_{t_1}(x_1-y_1,x_2-y_2)|^2dy_2\bigg)^{1/2} \\
&\qquad\quad\bigg(\int_{10J} |\nabla^{(2)} ({\triangle^{(2)}})^{-1/2} a_R(y_1, y_2)|^2dy_2\bigg)^{1/2}{dt_1 \over {t_1^{1+M} }}dx_2dx_1\\
&\lesssim|J|^{1/2} \int_{(100l)^c\times (100J)^c } \int_{\ell(I)^2}^\infty\int_{10I} {1\over t_1^{n+m\over2}}e^{-{|(x_1,x_2)-(x_I,y_J)|^2\over t_1}}\\
&\qquad\quad\bigg(\int_J |a_{R,2}(y_1, y_2)|^2dy_2\bigg)^{1/2} {dt_1 \over {t_1^{1+2M} }}dy_1dx_2dx_1.
\end{align*}
Hence,
\begin{align*}
\textrm{I}_{221}
&\lesssim|R|^{1/2} \sum_{j_1=\tilde j}^\infty \int_{|x_1-x_I|\approx 2^{j_1}\ell(I)}\sum_{j_2=6}^\infty \int_{|x_2-y_J|\approx 2^{j_2}\ell(J)} \\
&\qquad\quad  \int_{\ell(I)^2}^\infty {1\over t_1^{n+m\over2}}e^{-{|(x_1,x_2)-(x_I,y_J)|^2\over t_1}} {dt_1 \over {t_1^{1+M} }}dx_1dx_2\|a_{R,2}\|_2 \\
&\lesssim|R|^{1/2}\sum_{j_1=\tilde j}^\infty \int_{|x_1-x_I|\approx 2^{j_1}\ell(I)}\sum_{j_2=6}^\infty \int_{|x_2-y_J|\approx 2^{j_2}\ell(J)}\int_{\ell(I)^2}^\infty  t_1^{\alpha_1+\alpha_2-\frac{n+m}2-1-M}dt_1\\
&\qquad     \bigg({ 1\over|x_1-x_I|^2}\bigg)^{\alpha_1} \bigg({ 1\over|x_2-y_J|^2}\bigg)^{\alpha_2}\|a_{R,2}\|_2dx_1dx_2\\
&\lesssim|R|^{1/2}\sum_{j_1=\tilde j}^\infty\sum_{j_2=6}^\infty \ell(I)^{2(\alpha_1+\alpha_2-\frac{n+m}2-M)}{ 1\over \big(2^{j_1}\ell(I)\big)^{2\alpha_1-n}}{ 1\over \big(2^{j_2}\ell(J)\big)^{2\alpha_2-m}}\|a_{R,2}\|_2\\
&\lesssim |R|^{1/2}\gamma_1(R)^{-(2\alpha_1-n)}\ell(I)^{-2M}\|a_{R,2}\|_2,
\end{align*}
where $2\alpha_1>n$, $2\alpha_2>m$ and $2\alpha_1+2\alpha_2<n+m+2M$.

Let $a_R= (({\triangle^{(1)}})^M \otimes_2 ({\triangle^{(2)}})^M)b_R$.
To estimate $\textrm{I}_{222}$, we have
\begin{align*}
&\textrm{I}_{222} \\
&\lesssim \int_{(100l)^c\times (100J)^c }  \bigg| \int_{\ell(I)^2}^\infty\int_0^\infty \sum_{k_2=0}^\infty\int_{10I}\int_{|y_2-z_2|\approx 2^{k_2}\ell(J) }\\
&\quad\int_{10J}  (t_1\triangle^{(1)})^Mq^{(1)}_{t_1}(x_1-y_1,x_2-y_2) (t_2\triangle^{(2)})^Mq^{(2)}_{t_1}(y_2-z_2)  b_R(y_1, z_2) dz_2dy_1dy_2{dt_1 \over t_1^{1+M}} {dt_2 \over t_2^{1+M}}\bigg|dx_2dx_1\\
&\lesssim \int_{(100l)^c\times (100J)^c} \int_{\ell(I)^2}^\infty\int_0^\infty\sum_{k_2=0}^\infty   {1\over t_1^{n+m\over2}}e^{-{|(x_1,x_2)-(y_1,y_2)|^2\over t_1}} \\
&\quad  \int_{10I}\int_{|y_2-z_2|\approx 2^{k_2}\ell(J) }\int_{10J}   {1\over t_2^{m\over2}}e^{-{|y_2-z_2|^2\over t_2}}dy_2 |b_R(y_1, z_2)| dz_2dy_1{dt_1 \over {t_1^{1+M} }} {dt_2 \over {t_2^{1+M} }}dx_2dx_1\\
&\lesssim \int_{(100l)^c\times (100J)^c} \int_{\ell(I)^2}^\infty\bigg(\int_0^{t_1}+\int_{t_1}^\infty \bigg) {1\over t_1^{n+m\over2}}e^{-{|x_1-y_1|^2\over t_1}} \\
&\quad  \int_{10I} \int_{10J}   {1\over t_2^{m\over2}}\sum_{k_2=0}^\infty\int_{|y_2-z_2|\approx 2^{k_2}\ell(J) }e^{-{|x_2-y_2|^2\over t_1}}e^{-{|y_2-z_2|^2\over t_2}}dy_2 |b_R(y_1, z_2)| dz_2dy_1{dt_1 \over {t_1^{1+M} }} {dt_2 \over {t_2^{1+M} }}dx_2dx_1\\
&=:\textrm{I}_{2221}+\textrm{I}_{2222}.
 \end{align*}

{Note that $M>m/2$. By the heat kernel estimate, we choose $2\alpha_1-n>0, 2\alpha_2-m>0$, $\beta_2>0$ and $\alpha_1+\alpha_2+\beta_2<\frac n2+M$ }\ to obtain
\begin{align*}
\textrm{I}_{2221}
&\lesssim \int_{(100l)^c\times (100J)^c} \int_{\ell(I)^2}^\infty{1\over t_1^{n+m\over2}}e^{-{|x_1-y_1|^2\over t_1}}\int_0^{t_1} \Big(\frac{t_2}{\ell(J)^2}\Big)^{\beta_2+\frac m2+M}t_2^{-\frac m2-M-1}dt_2  \\
&\qquad  \int_I \int_J \sum_{k_2=0}^\infty\int_{|y_2-z_2|\approx 2^{k_2}\ell(J) } e^{-{|x_2-y_2|^2\over 2 t_1}}e^{-{|y_2-z_2|^2\over 2t_1}}dy_2 |b_R(y_1, z_2)| dz_2dy_1{dt_1 \over {t_1^{1+M} }}dx_2dx_1\\
&\lesssim \int_{(100l)^c\times (100J)^c } \int_{\ell(I)^2}^\infty {1\over t_1^{n+m\over2}}e^{-{|x_1-y_1|^2\over t_1}}t_1^{\beta_2}\ell(J)^{-2\beta_2-m-2M} \\
&\qquad  \int_I \int_J   t_1^{{m\over2}} e^{-{|x_2-z_2|^2\over 2t_1}}|b_R(y_1, z_2)| dz_2dy_1{dt_1 \over {t_1 }}dx_2dx_1\\
&\lesssim\sum_{j_1=\tilde j}^\infty \int_{|x_1-x_I|\approx 2^{j_1}\ell(I)}\sum_{j_2=6}^\infty \int_{|x_2-y_J|\approx 2^{j_2}\ell(J)}\int_{\ell(I)^2}^\infty \int_I\int_J t_1^{-\frac n2+\beta_2}\ell(J)^{-2\beta_2-m-2M}\\
&\quad\times  \bigg({ t_1\over|x_1-x_I|^2}\bigg)^{\alpha_1} \bigg({ t_1\over|x_2-y_J|^2}\bigg)^{\alpha_2} |b_R(y_1, z_2)| dz_2dy_1 {dt_1 \over {t_1^{1+M} }}dx_2dx_1\\
&\lesssim|R|^{1/2}\sum_{j_1=\tilde j}^\infty\sum_{j_2=6}^\infty \int_{\ell(I)^2}^\infty t_1^{\alpha_1+\alpha_2+\beta_2-\frac{n}2-M-1}dt_1\ell(J)^{-m-2M} \\
&\qquad\times   { 1\over \big(2^{j_1}\ell(I)\big)^{2\alpha_1-n}}{ 1\over \big(2^{j_2}\ell(J)\big)^{2\alpha_2-m}}\|b_R\|_2  \\
&\lesssim |R|^{1/2}\gamma_1(R)^{-(2\alpha_1-n)} \ell(I)^{-2M} \ell(J)^{-2M}   \|b_R\|_2.
\end{align*}

Choosing $n<2\alpha_1<n+m+2M$ and $m/2<\alpha_2<M<\beta_2$, we get
\begin{align*}
&\textrm{I}_{2222}\\
&\lesssim \int_{(100l)^c\times (100J)^c} \int_{\ell(I)^2}^\infty {1\over t_1^{n+m\over2}}e^{-{|x_1-y_1|^2\over t_1}}\int_{t_1}^\infty {1\over t_2^{m\over2}}e^{-{\ell(J)^2}\over 2t_2} \\
&\qquad  \int_I \int_J \sum_{k_2=0}^\infty\int_{|y_2-z_2|\approx 2^{k_2}\ell(J) } e^{-{|x_2-y_2|^2\over 2 t_2}}e^{-{|y_2-z_2|^2\over 2t_2}}dy_2 |b_R(y_1, z_2)| dz_2dy_1{dt_1 \over {t_1^{1+M} }}{dt_2 \over {t_2^{1+M} }} dx_2dx_1\\
&\lesssim \int_{(100l)^c\times (100J)^c } \int_{\ell(I)^2}^\infty {1\over t_1^{n+m\over2}}e^{-{|x_1-y_1|^2\over t_1}}\int_{t_1}^\infty {1\over t_2^{m\over2}}e^{-{\ell(J)^2}\over 2t_2} \\
&\qquad  \int_I \int_J   t_2^{{m\over2}} e^{-{|x_2-z_2|^2\over 2t_2}}|b_R(y_1, z_2)| dz_2dy_1{dt_1 \over {t_1^{1+M} }}{dt_2 \over {t_2^{1+M} }} dx_2dx_1\\
&\lesssim\sum_{j_1=\tilde j}^\infty \int_{|x_1-x_I|\approx 2^{j_1}\ell(I)}\sum_{j_2=6}^\infty \int_{|x_2-y_J|\approx 2^{j_2}\ell(J)}\int_{\ell(I)^2}^\infty \int_I\int_J t_1^{-\frac {n+m}2}\int_{t_1}^\infty e^{-{\ell(I)^2}\over 2t_2} \\
&\quad\times  \bigg({ t_1\over|x_1-x_I|^2}\bigg)^{\alpha_1} \bigg({ t_2\over|x_2-y_J|^2}\bigg)^{\alpha_2} |b_R(y_1, z_2)| dz_2dy_1 {dt_1 \over {t_1^{1+M}}}{dt_2 \over {t_2^{1+M} }}dx_2dx_1\\
&\lesssim|R|^{1/2}\sum_{j_1=\tilde j}^\infty\sum_{j_2=6}^\infty \int_{\ell(I)^2}^\infty t_1^{\alpha_1+\alpha_2-\frac{n+m}2-M-1}\bigg(\int_{t_1}^{\ell(J)^2} \Big(\frac{t_2}{\ell(J)^2}\Big)^{\beta_2}t_2^{\alpha_2-M-1}dt_2+\int_{\ell(J)^2}^\infty t_2^{\alpha_2-M-1}dt_2\bigg)dt_1 \\
&\qquad\times   { 1\over \big(2^{j_1}\ell(I)\big)^{2\alpha_1-n}}{ 1\over \big(2^{j_2}\ell(J)\big)^{2\alpha_2-m}}\|b_R\|_2 \\
&\lesssim |R|^{1/2}\gamma_1(R)^{-(2\alpha_1-n)} \ell(I)^{-2M}\ell(J)^{-2M}    \|b_R\|_2.
\end{align*}

Combining the above estimates, there exists a positive constant $\delta_1$ such that
\begin{align*}
&\int_{(100l)^c\times\mathbb{R}^{m}}|T(a_R)(x_1,x_2)|dx_2dx_1 \\
&\lesssim|R|^{1/2}\gamma_1(R)^{-\delta_1}\bigg(\|a_R\|_2+\ell(J)^{-2M}\|a_{R,1}\|_2\\
&\qquad+\ell(I)^{-2M}\|a_{R,2}\|_2+\ell(I)^{-2M}\ell(J)^{-2M}    \|b_R\|_2\bigg)\\
&=|R|^{1/2}\gamma_1(R)^{-\delta_1}\bigg(\|a_R\|_2+\ell(I)^{-2M}\|(({\triangle^{(1)}})^M\otimes_2 1\!\!1_2)b_R\|_2\\
&\qquad\qquad +\ell(I)^{-2M}\|(1\!\!1_1 \otimes_2({\triangle^{(2)}})^M)b_R\|_2+\ell(I)^{-2M}\ell(J)^{-2M}    \|b_R\|_2\bigg).
\end{align*}

Using H\"older's inequality, Journ\'e's covering lemma and the properties of flag atoms, we have
\begin{align*}
\textrm{I}&:= \sum_{R\in m(\Omega)}\int_{(100l)^c\times\mathbb{R}^{m}}|T(a_R)(x_1,x_2)|dx_2dx_1 \\
&\leq \sum_{R\in m(\Omega) }|R|^{1/2} \gamma_1(R)^{-\delta_1}
  \ell(I)^{-2M}\ell(J)^{-2M}\Big(\|((\ell(I)^2\triangle^{(1)})^M)\otimes_2
(\ell(J)^2\triangle^{(2)})^M)b_{R}\|_2\\
  &\quad + \|((\ell(I)^2\triangle^{(1)})^M)\otimes_2
    1\!\!1_2)b_{R}\|_2+ \|(1\!\!1_1 \otimes_2(\ell(J)^2{\triangle^{(2)}})^M)b_R\|_2+ \|b_R\|_2\Big)\\
&\lesssim\Big(\sum_{R\in m(\Omega) }|R|  \gamma_1(R)^{-2\delta_1} \Big)^{1/2}  \bigg(\sum_{R\in m(\Omega) } \ell(I)^{-4M}\ell(J)^{-4M}\\
&\qquad\times \Big(\|((\ell(I)^2\triangle^{(1)})^M)\otimes_2
(\ell(J)^2\triangle^{(2)})^M)b_{R}\|^2_2+ \|((\ell(I)^2\triangle^{(1)})^M)\otimes_2
    1\!\!1_2)b_{R}\|^2_2\\
&\qquad+ \|(1\!\!1_1 \otimes_2(\ell(J)^2{\triangle^{(2)}})^M)b_R\|^2_2+ \|b_R\|^2_2\Big)\bigg)^{1/2}\\
&\lesssim |\Omega|^{{1\over 2}} |\Omega|^{-{1\over 2}}
 \lesssim1.
\end{align*}

For term $\textrm{II}$, we observe that
\begin{align*}
&\int_{\mathbb{R}^{n}\times (100S)^c}|T(a_R)(x_1,x_2)|dx_2dx_1\\
&=\Big(\int_{{100I} \times (100S)^c } +\int_{(100I)^c\times (100S)^c}\Big) |T(a_R)(x_1,x_2)|dx_2dx_1 \\
&=\textrm{II}_1+\textrm{II}_2.
 \end{align*}

The estimate of term $\textrm{II}_2$ is the same with the estimate of $\textrm{I}_2$,
\begin{align*}
\textrm{II}_2
&\lesssim |R|^{1/2}\gamma_2(R)^{-\delta} \bigg(\|a_R\|_2 + \ell(J)^{-2M}\|a_{R,1}\|_2 \\
&\qquad + \ell(I)^{-2M}\|a_{R,2}\|_2 + \ell(I)^{-2M}\ell(J)^{-2M}    \|b_R\|_2 \bigg),
\end{align*}
where $\delta>0$.
So we just estimate the term $\textrm{II}_1$.
Then
\begin{align*}
\textrm{II}_1&\leq \int_{{100I} \times (100S)^c}  \bigg|  \nabla^{(1)}({\triangle^{(1)}})^{-1/2} \otimes_2 \nabla^{(2)} ({\triangle^{(2)}})^{-1/2} a_R(x_1, x_2)\bigg|dx_2dx_1\\
&= \int_{{100I} \times (100S)^c}  \bigg|  {1\over 2 {\sqrt\pi}}\int_0^{\infty} \nabla^{(2)} e^{-t_2\triangle^{(2)}}   \big(\nabla^{(1)} ({\triangle^{(1)}})^{-1/2} a_R(x_1, x_2) \big) {dt_2 \over \sqrt{t_2 }}\bigg|dx_2dx_1\\
&= \int_{{100I} \times (100S)^c}  \bigg|  {1\over 2 {\sqrt\pi}}\int_0^{\ell(J)^2} \nabla^{(2)} e^{-t_2\triangle^{(2)}}   \big(\nabla^{(1)} ({\triangle^{(1)}})^{-1/2} a_R(x_1, x_2) \big) {dt_2 \over \sqrt{t_2 }}\bigg|dx_2dx_1\\
&\quad+ \int_{{100I} \times (100S)^c}  \bigg|  {1\over 2 {\sqrt\pi}}\int_{\ell(J)^2}^{\infty} \nabla^{(2)} e^{-t_2\triangle^{(2)}}   \big(\nabla^{(1)} ({\triangle^{(1)}})^{-1/2} a_R(x_1, x_2) \big) {dt_2 \over \sqrt{t_2 }}\bigg|dx_2dx_1\\
&=:\textrm{II}_{11}+\textrm{II}_{12}.
 \end{align*}

We first consider $\textrm{II}_{11}$. We write
\begin{align*}
\textrm{II}_{11}&\leq  \int_{{100I} \times (100S)^c}  \bigg|  {1\over 2 {\sqrt\pi}}\int_0^{\ell(J)^2}\nabla^{(2)} e^{-t_2\triangle^{(2)}}   \big(\nabla^{(1)} ({\triangle^{(1)}})^{-1/2} a_R(x_1, x_2) \big) {dt_2 \over \sqrt{t_2 }}\bigg|dx_2dx_1\\
&\lesssim\sum_{j_2=\bar j}^\infty \int_{100I} \int_{|x_2-y_J|\approx 2^{j_2}\ell(J)} \int_0^{\ell(J)^2}\bigg(\int_{100J} + \int_{|y_2-y_J|\ge50\ell(J)}\bigg)\\
&\qquad\qquad t_2^{-\frac{m}2}e^{-{|x_2-y_2|^2\over t_2}} \big|\big(\nabla^{(1)} ({\triangle^{(1)}})^{-1/2} a_R(x_1, y_2) \big)\big| dy_2{dt_2 \over t_2 }dx_2dx_1\\
&=:\textrm{II}_{111}+\textrm{II}_{112}.
 \end{align*}
By the $L^2(\Bbb R^{n+m})$ boundedness of $\nabla^{(1)} ({\triangle^{(1)}})^{-1/2}$,
\begin{align*}
\textrm{II}_{111}
&\lesssim|R|^{1/2}\sum_{j_2=\bar j}^\infty\int_{|x_2-y_J|\approx 2^{j_2}\ell(J)} \int_0^{\ell(J)^2} t_2^{-\frac{m}2}e^{-{|x_2-y_J|^2\over t_2}}{dt_2 \over t_2 }dx_2 \|a_R\|_2 \\
&\lesssim|R|^{1/2}\sum_{j_2=\bar j}^\infty\int_{|x_2-y_J|\approx 2^{j_2}\ell(J)} \int_0^{\ell(J)^2} t_2^{-\frac{m}2}\bigg({ t_2\over|x_2-y_J|^2}\bigg)^{\alpha_2}{dt_2 \over t_2 }dx_2 \|a_R\|_2\\
&\lesssim|R|^{1/2}\sum_{j_2=\bar j}^\infty\ell(J)^{2\alpha_2-m}{ 1\over (2^{j_2}\ell(J))^{2\alpha_2-m}}\|a_R\|_2\\
&\lesssim|R|^{1\over2}   \gamma_2(R)^{-(2\alpha_2-m)}    \|a_R\|_2,
 \end{align*}
where $\alpha_2>m/2$ and the last inequality follows from the fact that
$$ 2^{\bar j} \approx {\ell(S)\over \ell(J)}. $$
For the term of $\textrm{II}_{112}$, the heat kernel estimate gives
\begin{align*}
\textrm{II}_{112}
&\lesssim\sum_{j_2=\bar j}^\infty \sum_{k_2=0}^\infty \int_{100I} \int_{|x_2-y_J|\approx 2^{j_2}\ell(J)} \int_0^{\ell(J)^2}\int_{|y_2-y_J|\approx 2^{k_2}\ell(J)} t_2^{-\frac{m}2}e^{-{|x_2-y_2|^2\over t_2}}\\
&\qquad\int_0^\infty\int_{10I}\int_{10J} t_1^{-\frac{n+m}2}e^{-{|(x_1,y_2)-(z_1,z_2)|^2\over t_1}} |a_R(z_1, z_2)|dz_1dz_2{dt_1 \over t_1 } dy_2{dt_2 \over t_2 }dx_1dx_2\\
&\lesssim|I|\sum_{j_2=\bar j}^\infty \sum_{k_2=0}^\infty  \int_{|x_2-y_J|\approx 2^{j_2}\ell(J)} \int_0^{\ell(J)^2}\int_{|y_2-y_J|\approx 2^{k_2}\ell(J)} t_2^{-\frac{m}2}e^{-{|x_2-y_2|^2\over t_2}}\\
&\qquad \bigg(\int_0^{t_2}+\int_{t_2}^\infty\bigg)\int_{10I}\int_{10J} t_1^{-\frac{n+m}2}e^{-{|y_2-z_2|^2\over 2t_1}} |a_R(z_1, z_2)|dz_1dz_2{dt_1 \over t_1 } dy_2{dt_2 \over t_2 }dx_2\\
&=:\textrm{II}_{1121}+\textrm{II}_{1122}.
\end{align*}
We do the integral for the variable $y_2$ to get
\begin{align*}
\textrm{II}_{1121}
&\lesssim|I|\sum_{j_2=\bar j}^\infty \int_{|x_2-y_J|\approx 2^{j_2}\ell(J)} \int_0^{\ell(J)^2}\int_0^{t_2} t_1^{-\frac{n+m}2}e^{-\ell(J)^2\over t_1}{dt_1 \over t_1 } e^{-{|x_2-y_J|^2\over 2t_2}}{dt_2 \over t_2 }\\
&\qquad\qquad   \int_{10I}\int_{10J}  |a_R(z_1, z_2)|dz_1dz_2dx_2 \\
&\lesssim|R|^{1/2}|I|\sum_{j_2=\bar j}^\infty \int_{|x_2-y_J|\approx 2^{j_2}\ell(J)} \int_0^{\ell(J)^2} t_2^{\alpha_1-\frac{n+m}2}\ell(J)^{-2\alpha_1} e^{-{|x_2-y_J|^2\over 2t_2}}{dt_2 \over t_2 } \|a_R\|_2\\
&\lesssim|R|^{1/2}|I|\ell(J)^{-2\alpha_1}\sum_{j_2=\bar j}^\infty \frac{\ell(J)^{2\alpha_1+2\alpha_2-n-m}}{(2^{j_2}\ell(J))^{2\alpha_2-m}} \|a_R\|_2\\
&\lesssim|R|^{1\over2}   \gamma_2(R)^{-(2\alpha_2-m)}    \|a_R\|_2,
\end{align*}
where $2\alpha_1>n+m$ and $2\alpha_2>m$. Similarly,
\begin{align*}
\textrm{II}_{1122}
&\lesssim|R|^{1/2}|I|\|a_R\|_{L^2(\mathbb R^{n+m})} \\
&\qquad \sum_{j_2=\bar j}^\infty \int_{|x_2-y_J|\approx 2^{j_2}\ell(J)} \int_0^{\ell(J)^2}e^{-\ell(J)^2\over 2t_2} \int_{t_2}^\infty t_1^{-\frac{n+m}2} e^{-{|x_2-y_J|^2\over 2t_1}}\frac{dt_1}{t_1}\frac{dt_2}{t_2}dx_2\\
&\lesssim|R|^{1/2}|I|\|a_R\|_2 \sum_{j_2=\bar j}^\infty \int_0^{\ell(J)^2}e^{-\ell(J)^2\over 2t_2} \int_{t_2}^\infty t_1^{-\frac{n+m}2}\frac{t_1^{\alpha_1}}{(2^{j_2}\ell(J))^{2\alpha_1-m}}\frac{dt_1}{t_1}\frac{dt_2}{t_2}\\
&\lesssim|R|^{1/2}|I| \gamma_2(R)^{-(2\alpha_1-m)} \|a_R\|_2 \ell(J)^{m-2\alpha_1}\int_0^{\ell(J)^2} t_2^{\alpha_1-\frac{n+m}2}e^{-\ell(J)^2\over 2t_2}\frac{dt_2}{t_2}\\
&\lesssim|R|^{1\over2}   \gamma_2(R)^{-(2\alpha_2-m)}    \|a_R\|_2,
\end{align*}
where $m<2\alpha_1<n+m$.

Let $a_R= (1\!\!1_1 \otimes_2 ({\triangle^{(2)}})^M) a_{R,1}$, where $a_{R,1}=( ({\triangle^{(1)}})^M\otimes_2 1\!\!1_2) b_R$
We now consider $\textrm{II}_{12}$ and write
\begin{align*}
\textrm{II}_{12}&\leq  \int_{{100I} \times (100S)^c}  \bigg|  {1\over 2 {\sqrt\pi}}\int_{\ell(J)^2}^\infty \nabla^{(2)} (t_2\triangle^{(2)})^M e^{-t_2\triangle^{(2)}}   \big(\nabla^{(1)} ({\triangle^{(1)}})^{-1/2} a_{R,1}(x_1, x_2) \big) {dt_2 \over {t_2^{\frac12+M} }}\bigg|dx_2dx_1\\
&\lesssim\sum_{j_2=\bar j}^\infty \int_{100I} \int_{|x_2-y_J|\approx 2^{j_2}\ell(J)} \int_{\ell(J)^2}^\infty \bigg(\int_{100J} + \int_{|y_2-y_J|\ge50\ell(J)}\bigg)\\
&\qquad\qquad t_2^{-\frac{m}2}e^{-{|x_2-y_2|^2\over t_2}} \big|\big(\nabla^{(1)} ({\triangle^{(1)}})^{-1/2} a_{R,1}(x_1, y_2) \big)\big| dy_2{dt_2 \over t_2^{1+M} }dx_2dx_1\\
&=:\textrm{II}_{121}+\textrm{II}_{122}.
 \end{align*}

By the $L^2(\Bbb R^{n+m})$ boundedness of $\nabla^{(1)} ({\triangle^{(1)}})^{-1/2}$,
\begin{align*}
\textrm{II}_{121}
&\lesssim|R|^{1/2}\sum_{j_2=\bar j}^\infty\int_{|x_2-y_J|\approx 2^{j_2}\ell(J)} \int_{\ell(J)^2}^\infty t_2^{-\frac{m}2-M}e^{-{|x_2-y_J|^2\over t_2}}{dt_2 \over t_2 }dx_2 \|a_{R,1}\|_2 \\
&\lesssim|R|^{1/2}\sum_{j_2=\bar j}^\infty (2^{j_2}\ell(J)^{m-2\alpha_2} \int_{\ell(J)^2}^\infty t_2^{\alpha_2-\frac{m}2-M}{dt_2 \over t_2 } \|a_{R,1}\|_2\\
&\lesssim|R|^{1\over2}   \gamma_2(R)^{-(2\alpha_2-m)}  \ell(J)^{-2M}  \|a_{R,1}\|_2,
 \end{align*}
where $0<2\alpha_2-m<M$.

For the term of $\textrm{II}_{122}$, the heat kernel estimate gives
\begin{align*}
\textrm{II}_{122}
&\lesssim\sum_{j_2=\bar j}^\infty \sum_{k_2=0}^\infty \int_{100I} \int_{|x_2-y_J|\approx 2^{j_2}\ell(J)} \int_{\ell(J)^2}^\infty \int_{|y_2-y_J|\approx 2^{k_2}\ell(J)} t_2^{-\frac{m}2-M}e^{-{|x_2-y_2|^2\over t_2}}\\
&\qquad\qquad\int_0^\infty\int_{10I}\int_{10J} t_1^{-\frac{n+m}2}e^{-{|(x_1,y_2)-(z_1,z_2)|^2\over t_2}} |a_{R,1}(z_1, z_2)|dz_1dz_2{dt_1 \over t_1 } dy_2{dt_2 \over t_2 }dx_2dx_1\\
&\lesssim|I|\sum_{j_2=\bar j}^\infty \sum_{k_2=0}^\infty  \int_{|x_2-y_J|\approx 2^{j_2}\ell(J)} \int_{\ell(J)^2}^\infty\int_{|y_2-y_J|\approx 2^{k_2}\ell(J)} t_2^{-\frac{m}2-M}e^{-{|x_2-y_2|^2\over t_2}}\\
&\qquad\qquad \bigg(\int_0^{t_2}+\int_{t_2}^\infty\bigg)\int_{10I}\int_{10J} t_1^{-\frac{n+m}2}e^{-{|y_2-z_2|^2\over 2t_1}} |a_{R,1}(z_1, z_2)|dz_1dz_2{dt_1 \over t_1 } dy_2{dt_2 \over t_2 }dx_2\\
&=:\textrm{II}_{1221}+\textrm{II}_{1222}.
\end{align*}

We integrate with respect to $y_2$

\begin{align*}
\textrm{II}_{1221}
&\lesssim|I|\sum_{j_2=\bar j}^\infty \int_{|x_2-y_J|\approx 2^{j_2}\ell(J)} \int_{\ell(J)^2}^\infty \int_0^{t_2} t_1^{-\frac{n+m}2}e^{-\ell(J)^2\over 2t_1}{dt_1 \over t_1 }t_2^{-M} e^{-{|x_2-y_J|^2\over 2t_2}}{dt_2 \over t_2 }\\
&\qquad\qquad   \int_{10I}\int_{10J}  |a_R(z_1, z_2)|dz_1dz_2dx_2 \\
&\lesssim|R|^{1/2}|I|\sum_{j_2=\bar j}^\infty (2^{j_2}\ell(J))^{m-2\alpha_2}\int_{\ell(J)^2}^\infty t_2^{\alpha_1+\alpha_2-\frac{n+m}2-M}\ell(J)^{-2\alpha_1}{dt_2 \over t_2 } \|a_{R,1}\|_2\\
&=|R|^{1/2}|I|\sum_{j_2=\bar j}^\infty (2^{j_2}\ell(J))^{m-2\alpha_2}\ell(J)^{2\alpha_1+2\alpha_2-n-m-2M}\ell(J)^{-2\alpha_1} \|a_{R,1}\|_2\\
&\lesssim|R|^{1\over2}   \gamma_2(R)^{-(2\alpha_2-m)}\ell(J)^{-2M}    \|a_{R,1}\|_2,
\end{align*}
where $2\alpha_1>n+m$ and $2\alpha_2>m$ and $2\alpha_1+2\alpha_2<n+m+2M$. Similarly,
\begin{align*}
\textrm{II}_{1222}
&\lesssim|R|^{1/2}|I|\|a_{R,1}\|_2 \\
&\qquad \sum_{j_2=\bar j}^\infty \int_{|x_2-y_J|\approx 2^{j_2}\ell(J)} \int_{\ell(J)^2}^\infty t_2^{-M}  \int_{t_2}^\infty t_1^{-\frac{n+m}2} e^{-{|x_2-y_J|^2\over 2t_1}}\frac{dt_1}{t_1}\frac{dt_2}{t_2}dx_2\\
&\lesssim|R|^{1/2}|I|\|a_{R,1}\|_2 \sum_{j_2=\bar j}^\infty \int_{\ell(J)^2}^\infty  t_2^{-M}  \int_{t_2}^\infty t_1^{-\frac{n+m}2}\frac{t_1^{\alpha_1}}{(2^{j_2}\ell(J))^{2\alpha_1-m}}\frac{dt_1}{t_1}\frac{dt_2}{t_2}\\
&\lesssim|R|^{1/2}|I| \gamma_2(R)^{-(2\alpha_1-m)} \|a_{R,1}\|_2 \ell(J)^{m-2\alpha_1}\int_{\ell(J)^2}^\infty t_2^{\alpha_1-\frac{n+m}2-M}\frac{dt_2}{t_2}\\
&\lesssim|R|^{1\over2}   \gamma_2(R)^{-(2\alpha_1-m)} \ell(J)^{-2M}   \|a_{R,1}\|_2,
\end{align*}
where $m<2\alpha_1<n+m$.

Combining the above estimates, there exists a positive constant $\delta_2$ such that
\begin{align*}
&\int_{\mathbb{R}^{n}\times (100S)^c}|T(a_R)(x_1,x_2)|dx_2dx_1 \\
&\lesssim|R|^{1/2}\gamma_2(R)^{-\delta_2}\bigg(\|a_R\|_2+\ell(I)^{-2M}\|a_{R,2}\|_2\\
&\qquad+\ell(J)^{-2M}\|a_{R,1}\|_2+\ell(I)^{-2M}\ell(J)^{-2M}    \|b_R\|_2\bigg)\\
&=|R|^{1/2}\gamma_2(R)^{-\delta_2}\bigg(\|a_R\|_2+\ell(I)^{-2M}\|(({\triangle^{(1)}})^M\otimes_2 1\!\!1_2)b_R\|_2\\
&\qquad\qquad +\ell(I)^{-2M}\|(1\!\!1_1 \otimes_2({\triangle^{(2)}})^M)b_R\|_2+\ell(I)^{-2M}\ell(J)^{-2M}    \|b_R\|_2\bigg).
\end{align*}

Using the H\"older's inequality, Journe\'s covering lemma and the properties of flag atoms, we have
\begin{align*}
\textrm{II}&:= \sum_{R\in m(\Omega)}\int_{\mathbb{R}^{n}\times (100S)^c}|T(a_R)(x_1,x_2)|dx_1dx_2  \\
&\leq \sum_{R\in m(\Omega) }|R|^{1/2} \gamma_2(R)^{-\delta_2}
  \ell(I)^{-2M}\ell(J)^{-2M}\Big(\|((\ell(I)^2\triangle^{(1)})^M)\otimes_2
(\ell(J)^2\triangle^{(2)})^M)b_{R}\|_2\\
  &\quad + \|((\ell(I)^2\triangle^{(1)})^M)\otimes_2
    1\!\!1_2)b_{R}\|_2+ \|(1\!\!1_1 \otimes_2(\ell(J)^2{\triangle^{(2)}})^M)b_R\|_2+ \|b_R\|_2\Big)\\
&\lesssim\Big(\sum_{R\in m(\Omega) }|R|  \gamma_2(R)^{-2\delta_2} \Big)^{1/2}  \bigg(\sum_{R\in m(\Omega) } \ell(I)^{-4M}\ell(J)^{-4M}\\
&\qquad\times \Big(\|((\ell(I)^2\triangle^{(1)})^M)\otimes_2
(\ell(J)^2\triangle^{(2)})^M)b_{R}\|^2_2+ \|((\ell(I)^2\triangle^{(1)})^M)\otimes_2
    1\!\!1_2)b_{R}\|^2_2\\
&\qquad+ \|(1\!\!1_1 \otimes_2(\ell(J)^2{\triangle^{(2)}})^M)b_R\|^2_2+ \|b_R\|^2_2\Big)\bigg)^{1/2}\\
&\lesssim |\Omega|^{{1\over 2}} |\Omega|^{-{1\over 2}}
 \lesssim1.
 \end{align*}

Therefore,
\begin{align*}
&\int_{\big(\cup \widetilde{R}\big)^c}|T(a)(x_1,x_2)|dx_2dx_1
\\
&\leq \sum_{R\in
m(\Omega)}\int_{(100l)^c\times\mathbb{R}^{m}}|T(a_R)(x_1,x_2)|dx_2dx_1+\sum_{R\in
m(\Omega)}\int_{\mathbb{R}^{n}\times(100S)^c}|T(a_R)(x_1,x_2)|dx_2dx_1\\
&\le C.
\end{align*}
The inequality \eqref{T alpha uniformly bd on outside of Omega} is done and the proof is completed.
\end{proof}

Based on the result above, we already showed that $\sum_{j=1}^{n+m}\sum_{k=1}^m\|R_{j,k}(f)\|_1+\|f\|_1\lesssim \|S_F(f)\|_1$,
which, together with all estimates provided from Section 2 to Section 4, gives that
\begin{eqnarray*}
\|S_F(f)\|_1& \lesssim & \|S_F(U)\|_1 \lesssim \|U^*\|_1\lesssim \|M^{*}_{\Phi}(f)\|_1 \lesssim  \|U^*\|_1\\
&\lesssim &  \|U^+\|_1\lesssim \|M^{+}_{\Phi}(f)|_1  \lesssim  \|U^+\|_1\\
&\lesssim &   \sum_{j=1}^{n+m}\sum_{k=1}^m\|R_{j,k}(f)\|_1+\|f\|_1\\
&\lesssim& \|S_F(f)\|_1.
\end{eqnarray*}

\newpage

\noindent{\bf Acknowledgement:}
The authors would like to thank the referees for all the helpful comments and suggestions, which made this paper much more clear and accurate.

BDW's research supported in part by National Science Foundation grants DMS \# 1560955 and DMS \# 1800057.  JL supported by ARC DP 160100153 and Macquarie University New Staff Grant.
MYL is supported by Ministry of Science and Technology, R.O.C. under Grant \#MOST 106-2115-M-008-003-MY2 as well as supported by National Center for Theoretical Sciences of Taiwan.
This paper started in 2007 when YSH visited JL at Sun Yat-Sen University, and subsequent work was done when MYL visited JL at Sun Yat-Sen University in 2012 and 2013,
and JL visited MYL at National Central University, Taiwan in 2014 and when JL visited BDW at Washington University in STL in 2015.

\bigskip

\bigskip
Department of Mathematics,  Auburn University,  Auburn, AL 36849-5310, U.S.A.
\smallskip

{\it E-mail}: \texttt{hanyong@auburn.edu}

\bigskip
\bigskip

Department of Mathematics, National Central University, Chung-Li 320, Taiwan.

\noindent\&

National Center for Theoretical Sciences, 1 Roosevelt Road, Sec. 4, National Taiwan University, Taipei 106, Taiwan

\smallskip

{\it E-mail}: \texttt{mylee@math.ncu.edu.tw}

\bigskip
\bigskip

Department of Mathematics,
         Macquarie University,
         NSW,  2109,
        Australia

\smallskip
{\it E-mail}: \texttt{ ji.li@mq.edu.au}

\bigskip
\bigskip

Department of Mathematics, Washington University--St. Louis, St. Louis, MO 63130-4899, U.S.A.

\smallskip

{\it E-mail}: \texttt{wick@math.wustl.edu}

\end{document}